\documentclass[11pt, oneside]{book} 	% use "amsart" instead of "article" for AMSLaTeX format
\usepackage{geometry} 
\usepackage{afterpage}
\usepackage{graphicx}				% Use pdf, png, jpg, or eps§ with pdflatex; use eps in DVI mode
\usepackage{float}		
\usepackage[toc,page]{appendix}
\usepackage{amssymb}
\usepackage{amsmath}
\usepackage{amsthm} %% make other styles for definitions and remarks
\usepackage{mathrsfs}
\usepackage{tikz}
\usepackage{tikz-cd}
\usepackage{mathtools}
\usepackage{xfrac}
\usepackage{enumitem}
\usepackage{caption}
\usepackage{subcaption}
\usepackage[colorinlistoftodos,prependcaption,textsize=tiny]{todonotes}
\usepackage[xcolor,pdftex,rightbars]{changebar}
\usepackage{xargs}
\usepackage[hidelinks]{hyperref}
\usepackage{cleveref}

\newtheorem{theorem}{Theorem}[section]
\newtheorem{lemma}[theorem]{Lemma}
\newtheorem{proposition}[theorem]{Proposition}
\newtheorem{corollary}[theorem]{Corollary}
\theoremstyle{definition}
\newtheorem{definition}[theorem]{Definition}
\newtheorem{example}[theorem]{Example}

\newtheorem{remark}[theorem]{Remark}
\newtheorem{claim}{Claim}[theorem]

\newtheorem{definitionconstruction}[theorem]{Construction}

\cbcolor{red}
\newcommandx{\change}[2][1=]{\todo[linecolor=red,backgroundcolor=red!25,bordercolor=red,#1]{#2}}
\newcommandx{\maychange}[2][1=]{\todo[linecolor=blue,backgroundcolor=blue!25,bordercolor=blue,#1]{#2}}
\newcommandx{\talk}[2][1=]{\todo[linecolor=OliveGreen,backgroundcolor=OliveGreen!25,bordercolor=OliveGreen,#1]{#2}}
\newcommandx{\improvement}[2][1=]{\todo[linecolor=Plum,backgroundcolor=Plum!25,bordercolor=Plum,#1]{#2}}

\setenumerate{label = \normalfont(\roman*)}
\title{Stratified Homotopy Theory and a Combinatorial Whitehead Group for Stratified Spaces}
\author{Lukas Waas}
\begin{document}
\begin{titlepage}
	\centering
	\vspace*{1.5cm} 
	\begin{center} \large 
		{\Large Master Thesis}\\
		\vspace*{2.3cm}
		\textbf{\huge{Stratified Homotopy Theory}\\\vspace*{0.4cm} \Large and a \vspace*{0.4cm}\\ \huge{ Whitehead Group for Stratified Spaces}}\\
		\vspace*{1.5cm}
		Lukas Waas\\
		\vspace*{2cm}

		\begin{align*}	
		\text{Supervisor: }&\text{Prof. Dr. Markus Banagl}\\
		\end{align*}
		Faculty of Mathematics and Computer Science\\
		University of Heidelberg\\
		\vspace*{1cm}
		October 21, 2020\\
		Last Updated on February 13, 2021
	\end{center}
\end{titlepage}
\thispagestyle{empty}
\chapter*{\centering \begin{normalsize}Abstract\end{normalsize}}
\thispagestyle{empty}
\begin{quotation}
	\noindent \small Simple homotopy theory deals with the question of when a homotopy equivalence $f$ between sufficiently combinatorial topological spaces $X$ and $Y$ can be represented through a sequence of elementary combinatorial moves, called elementary expansions and collapses. It turns out, this question is answered completely by an obstruction element, the Whitehead torsion of $f$, inside of an algebraic group, the Whitehead group of $X$. In this master thesis, we extend this classical perspective to the world of stratified homotopy theory. To obtain a well established framework to work in, we prove a series of results on two model categories of simplicial sets and topological spaces, both equipped with a notion of filtration, introduced by Sylvain Douteau in his PHD thesis. Weak equivalences in these categories are essentially stratum preserving maps that induce weak homotopy equivalences on strata and higher homotopy links. Making use of a filtered version of the simplicial approximation theorem, we show that as long as one restricts to finite filtered simplicial sets, their homotopy theory provides a good model for the filtered topological one. In particular, we show that there is a fully faithful embedding of homotopy categories from the (finite) simplicial into the topological setting. We also use these results to characterize the morphisms in the topological filtered homotopy category between filtered spaces that are triangulable and stratified in some very general sense - which includes most definitions of stratified spaces - as stratified homotopy classes of stratum preserving maps. Having then gained a good theoretical understanding of how stratified homotopy theory can be modeled in the world of simplicial sets, we propose a class of combinatorial elementary expansions for filtered simplicial sets that generalize both the classical ones as well as a class of stratified expansions suggested by Banagl et al. We prove that they fulfill a series of axioms, suggested by Eckmann and Siebenmann for the construction of simple homotopy theory. In doing so, we obtain a combinatorially defined Whitehead group and torsion for filtered simplicial sets (and hence also for triangulable stratified spaces). We then begin a detailed investigation of their formal properties, proving for example that a Mayer-Vietoris formula holds. Using these results we show that every filtered simplicial set has the simple homotopy type (in the sense induced by these expansions) of a filtered simplicial complex of the same dimensions. We then apply the results we obtained on the connection between filtered simplicial sets and filtered topological spaces to obtain a more topological description of the Whitehead group and to generalize the Whitehead torsion to arbitrary stratum preserving maps of triangulated filtered spaces. Finally, we prove that our simple homotopy theory is a generalization of the classical one, in the sense that it agrees with the latter when one considers trivially filtered simplicial sets as CW-complexes.
\end{quotation}
\clearpage
\chapter*{\centering \begin{normalsize}Zusammenfassung\end{normalsize}}
\thispagestyle{empty}
\begin{quotation}
	\noindent \small Einfache Homotopietheorie beschäftigt sicht mit der Frage, wann eine Homotopie-\\äquivalenz $f$ zwischen hinreichend kombinatorisch gearteten Räumen $X$ und $Y$ durch eine Abfolge von elementaren kombinatorischen Operationen, genannt elementare Erweiterungen und Kollapse, dargestellt werden kann. Eine vollständige Antwort auf diese Frage liefert ein Obstruktionselement in der sogenannten Whiteheadgruppe von $X$, die Whitehead torsion von $f$. In dieser Masterarbeit verallgemeinern wir diese klassischen Ergebnisse auf stratifizierte Homotopietheorie. Unser homotopietheoretisches Framework hierfür sind zwei Modellkategorien von filtrierten Objekten - eine von simplizialen Mengen, eine von topologischen Räumen - eingeführt von Sylvain Douteau in seiner Promotion. In beiden Kategorien sind schwache Äquivalenzen dadurch charakterisiert, dass sie schwache Homotopieäqui-\\valenzen auf Strata und (höheren) Homotopielinks induzieren. Mit Hilfe einer stratifizierten Version des simplizialen Approximationssatzes beweisen wir, dass sich die simpliziale Homotopiekategorie (unter leichten Endlichkeitsannahmen) volltreu in die topologische einbettet, also die simplizial filtrierte Welt ein gutes Homotopiemodell für die topologische bietet. Weiterhin charakterisieren wir die Morphismen in den topologischen Homotopiekategorie zwischen einer Klasse von endlich triangulierbaren stratifizierten Räumen, die beinahe alle klassischen Beispiele stratifizierter Räume beinhaltet. Sie sind durch Homotopieklassen strata erhaltender Abbildungen unter stratifizierter homotopie gegeben. Mit diesen Ergebnissen ausgestattet, beginnen wir dann unsere Untersuchung einfacher stratifizierter Homotopietheorie. Wir schlagen eine Klasse elementarer Erweiterungen vor, welche sowohl die klassischen elementaren Erweiterungen als auch eine Klasse stratifizierter Erweiterungen die kürzlichen von Banagl et al. vorgeschlagen wurden, verallgemeinert. Für diese Beweisen wir, dass sie einen Axiomensatz erfüllen, der von Eckmann und Siebenmann für die Konstruktion einer einfachen Homotopietheorie vorgeschlagen wurde. Inbesondere erhalten wir eine kombinatorische Whiteheadgruppe und Whitehead torsion für filtrierte simpliziale Mengen, und damit auch für triangulierbare stratifizierte Räume. Als nächstes folgt eine detailierte Untersuchung von formalen Eigenschaften dieser Whiteheadgruppe und der zu ihr korrespondierenden einfachen Äquivalenzen. Inbesondere beweisen wir eine Mayer-Vietoris Formel und zeigen mit ihrer Hilfe, dass jede filtrierte simpliziale Menge den einfachen Homotopietyp eines filtrierten Simplizialkomplexes hat. Dann nutzen wir die im ersten Teil der Arbeit gewonnen Resultate, um die Whitehead torsion auf triangulierte filtrierte Räume zu verallgemeinern und erhalten so eine Reihe äquivalenter Interpretationen der stratifizierten Whiteheadgruppe. Zu guter Letzt beweisen wir mit Hilfe dieser Charakterisierungen, dass unsere Konstruktion der Whiteheadgruppe im Falle trivialer Filtration natürlich isomorph zu der klassischen ist.
\end{quotation}
\clearpage
\thispagestyle{empty}
\chapter*{\centering
\begin{normalsize}Acknowledgements\end{normalsize}}
\thispagestyle{empty}
\begin{quotation}
\noindent \small First and foremost I want to thank my advisor, Professor Markus Banagl. He was the one who awakened my interest in topology during my first calculus course as a freshman and has accompanied my studies of the subject ever since. Next, I want to thank the Studienstiftung des Deutschen Volkes for their continued financial and intellectual support. Finally, big thanks go out to my friends Ricardo, Jonas, Dario, Tim and Tim and my girlfriend Rike for their help in proofreading and hunting down the commas I tend to spice up every sentence with.
\end{quotation}
\clearpage
\thispagestyle{empty}

\thispagestyle{empty}
\tableofcontents
\addtocontents{toc}{\protect\thispagestyle{empty}}
\thispagestyle{empty}
\setcounter{chapter}{-1}
\setcounter{page}{0}
\chapter{Introduction}
The goal of this work is the extension of methods of a rather classical field of algebraic topology, simple homotopy theory, to one that has recently seen a lot of development, stratified homotopy theory. While our investigation is purely theoretical, we want to motivate it by an example originating from topological data analysis (TDA).\\
\\
The pipeline of TDA is often described as follows. Given some dataset, for example a point cloud in euclidean space, one wants to identify some of its topological features. As a point cloud on its own is discrete, the first step to do so is to replace it by an appropriate, more geometrically interesting object. This is usually done by connecting its points through simplices, thus obtaining a simplicial complex (or a family of simplicial complexes, depending on certain parameters), following certain rules that consider, for example, the distance between points. Finally, one applies a topological invariant to this resulting simplicial complex, obtaining some algebraic data, that one can interpret. This invariant is typically (co)homology with coefficients in some field. Homology of course is a homotopy invariant. Recently, however, there has been a lot of interest in replacing homology in this final step by some other (not necessarily topological) invariant (see for example \cite{bendichHarer2011}). Doing so, one hopes to pick up on different, and potentially finer features than homology can identify (see for example \cite{rieck2019persistent} and \cite{bendichHarer2011}). One such invariant is intersection homology, as suggested in \cite{bendichHarer2011}. Intersection homology is an invariant of spaces equipped that are equipped with a filtration that is particularly adapted to working with spaces with singularities. In this setting, the classical pipeline begins with a point cloud in which some points have been marked as belonging to a singularity. The simplicial complex $K^1$ obtained from this point cloud (for example through the Vietoris-Rips construction) is then naturally equipped with a subcomplex $K^0$ given by the full subcomplex spanned by the points marked as singular. Such a pair is called a filtered simplicial complex (filtered over $\{0,1\}$). For this filtered simplicial complex $K = (K^0 \subset K^1)$, one then computes the intersection homology (with respect to some perversity and formal codimension, in the sense of \cite{friedman2020}).
\begin{figure}[H]
	\centering
	\includegraphics[width=0.45\linewidth]{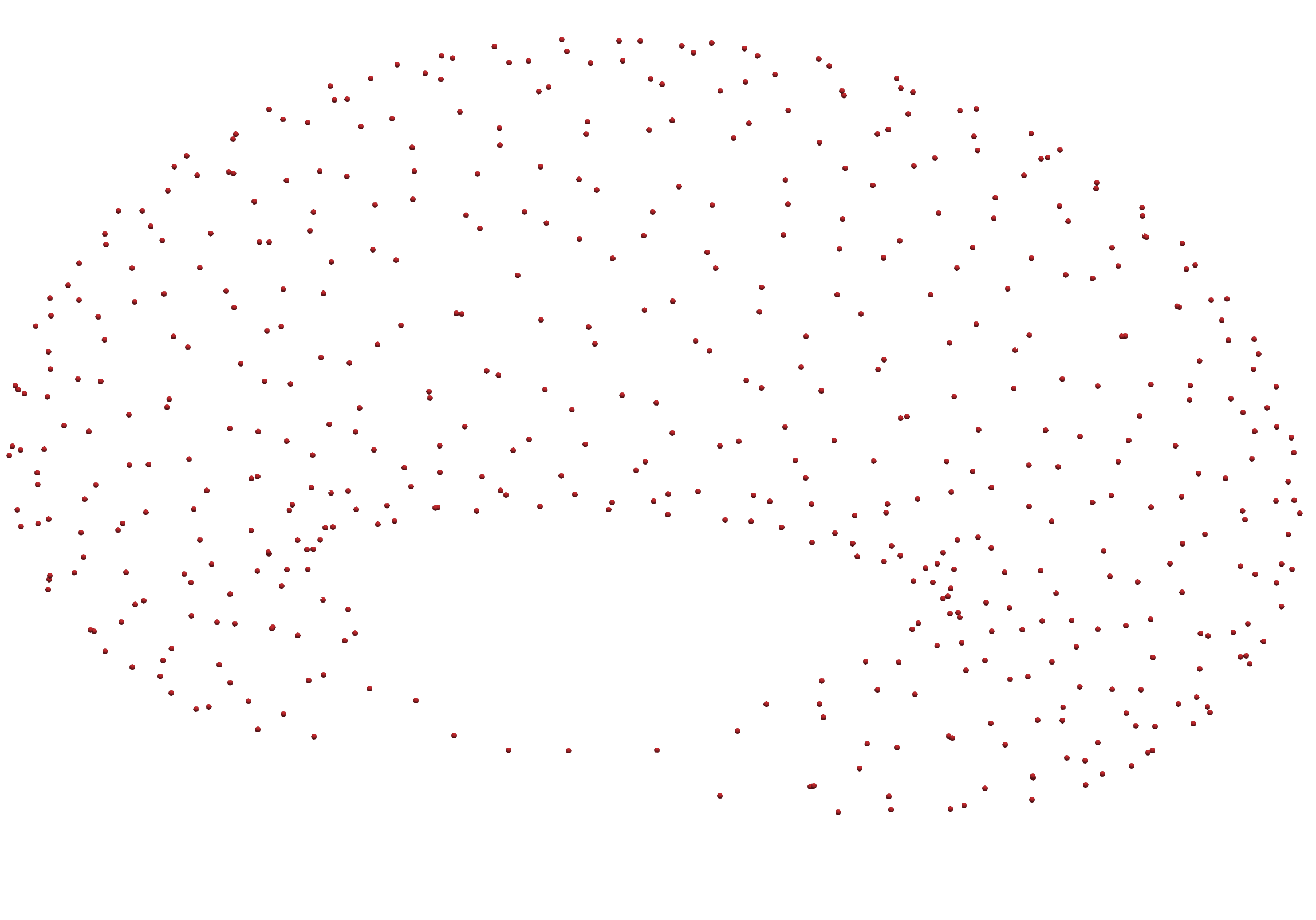}
	\includegraphics[width=0.45\linewidth]{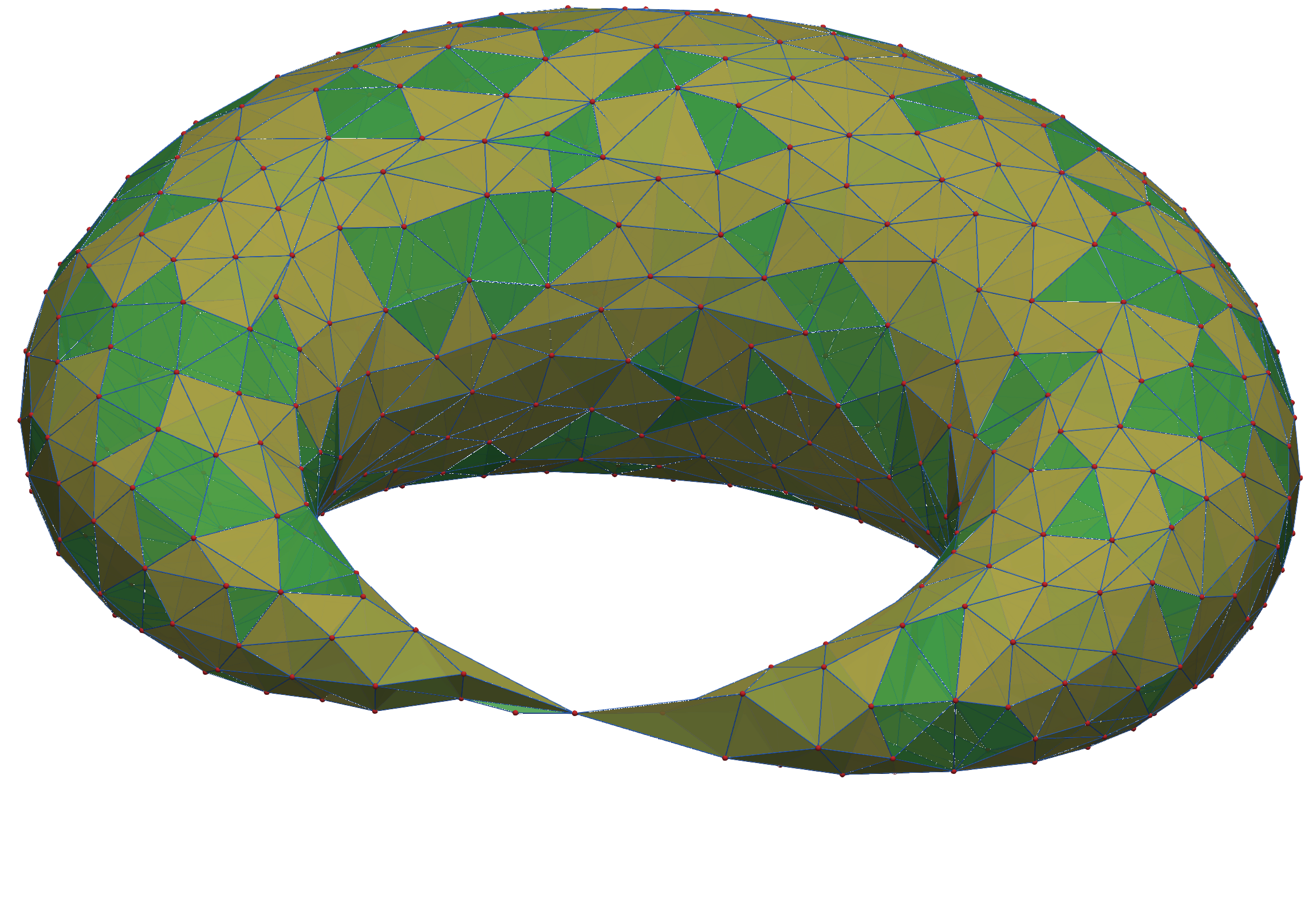}
	\caption{The left shows a point cloud in $\mathbb R^3$ sampled from a pinched torus. The right shows a Vietoris-Rips complex built from it, up to dimension $3$. Three-dimensional simplices are marked in a lighter green. $K^0$ is taken to be the pinched point.}
	\label{fig:exVRc}
\end{figure}
In practice, the filtered complex $K$ can be rather large and high dimensional even if the point cloud is located in low dimensional euclidean space. For computational reasons, one may therefore ask the question whether it is possible to replace $K$ by a smaller filtered complex, without changing its intersection homology. This is essentially the approach of \cite{banagl2020stratified}. There, the authors characterize certain elementary combinatorial moves that do not change the intersection homology - the stratified homotopy type, to be more precise - of a filtered simplicial complex. These moves are a generalization of elementary collapses of simplicial complexes, due to Whitehead. One says that a simplicial complex $K$ collapses elementarily to $L$ or equivalently, $L$ expands elementarily to $K$ if $K$ is obtained from $L$ through filling a horn. That is, $K$ is obtained from $L$ by adding a simplex that has all but one of its proper faces present in $L$ together with the missing face. \begin{figure}[H]
	\centering
	$\vcenter{\hbox{\includegraphics[width=0.45\linewidth]{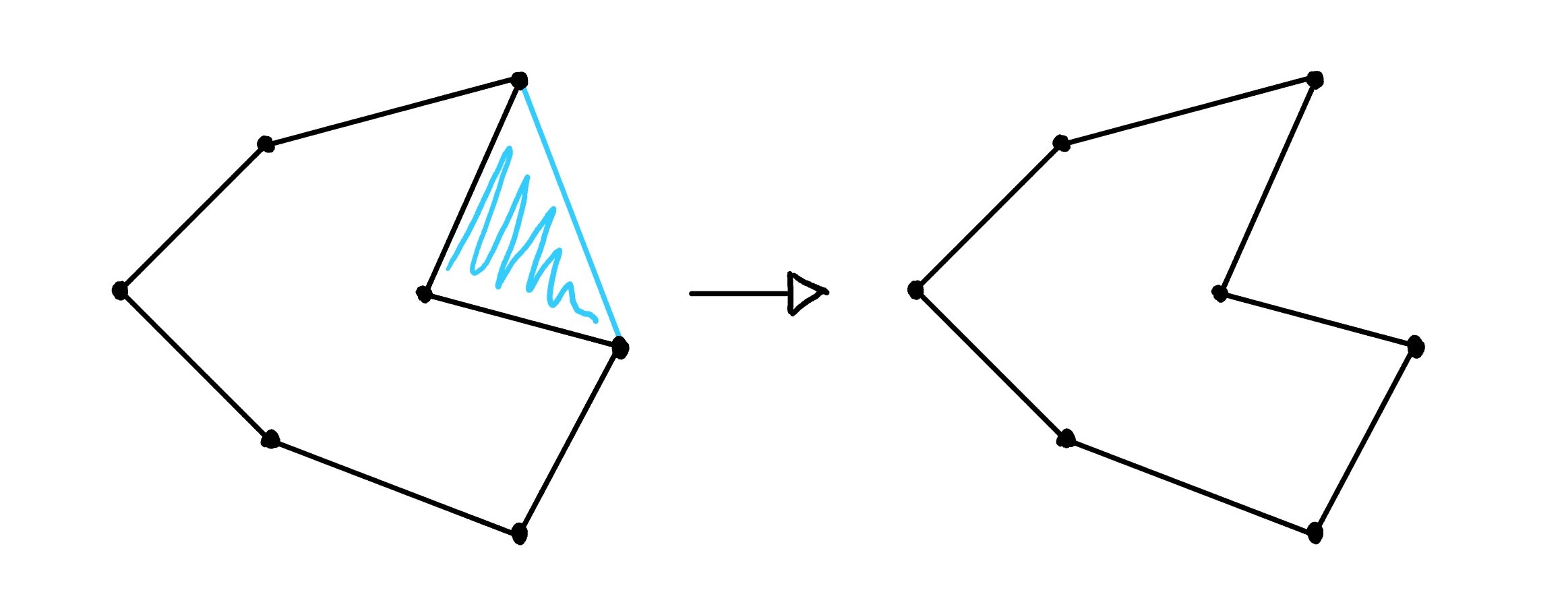}}}$ \hspace*{.2in}
	$\vcenter{\hbox{\includegraphics[width=0.45\linewidth]{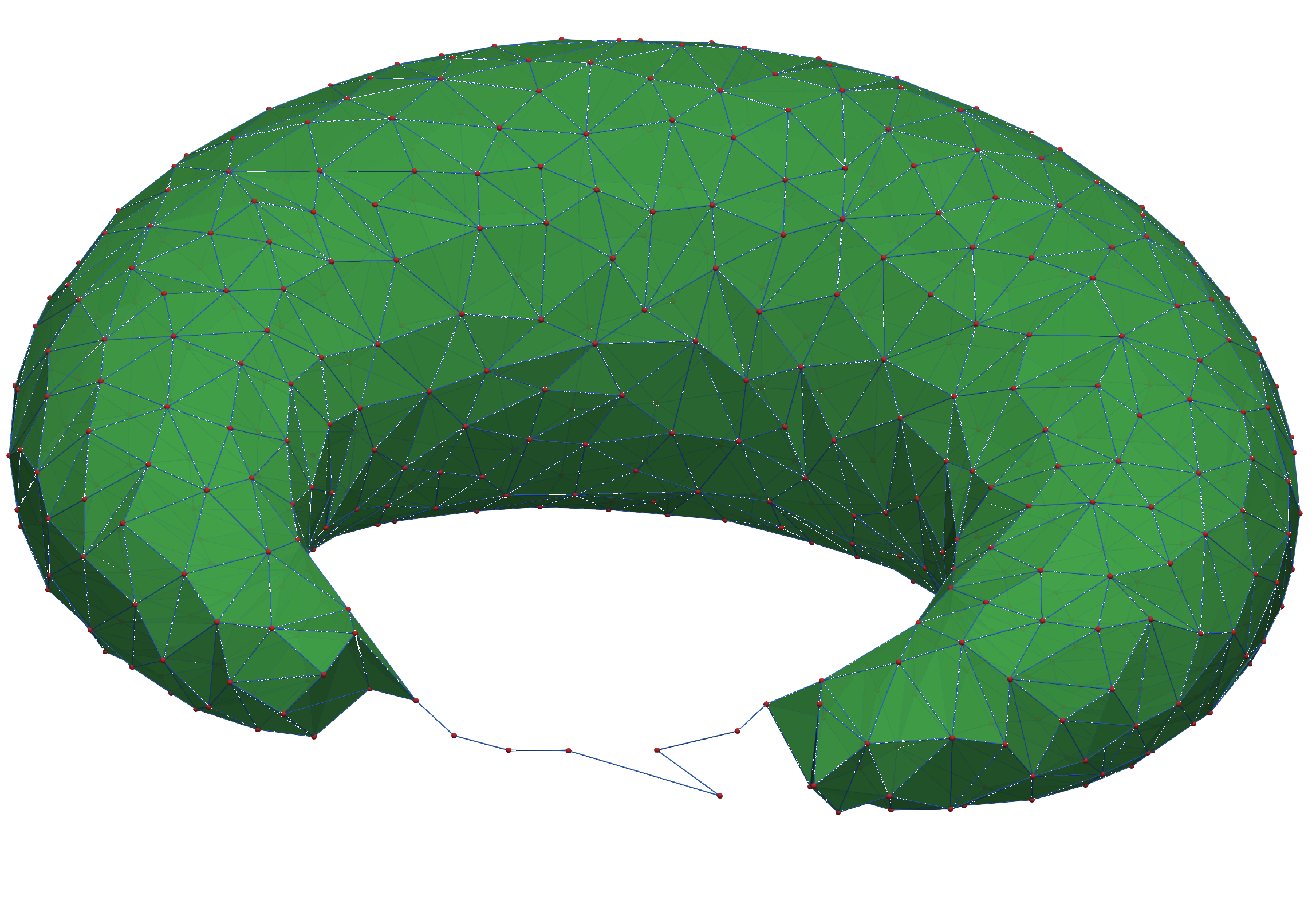}}}$
	\caption{To the left is an elementary collapse in the classical setting. The added simplex and its face are marked in blue. The picture on the right shows a reduced complex obtained from the filtered complex in \Cref{fig:exVRc} through the collapses described in \cite{banagl2020stratified}.}
	\label{fig:introElemCol}
\end{figure} The realization of the inclusion $L \hookrightarrow K$ is then a homotopy equivalence. A map $f:|K_0| \to |K_1|$ homotopic to the composition of the realizations of such elementary expansions and their inverses is called a simple homotopy equivalence. The field concerned with the question of the relationship between simple homotopy equivalences and general homotopy equivalences is simple homotopy theory. At the same time, the field concerned with the question of finding a homotopy theoretical setting to formulate intersection homology in is stratified homotopy theory. Thus, the study of such combinatorial operations in the stratified (filtered) setting lies in the intersection of these two fields. If one wants to obtain a deeper understanding of how to use such operations in practice, what other possible combinatorial manipulations might be useful, and what their theoretical limitations are - that is, which homotopy equivalences (in some appropriate stratified sense) can be represented through such operations - this intersection is what one needs to study. Let us call it simple stratified homotopy theory.\\
\\
A filtered space is (classically) a topological space, $X$, together with a filtration by closed subspaces $X^0 \subset ... \subset X^n = X$. Such a space $X$ then decomposes (on a set level) into the disjoint union of the subspaces $X_k = X^k\setminus X^{k-1}$. These spaces are called strata. Filtered spaces arise naturally in numerous fields of mathematics. Usually, they can be thought of as specifying singularities in a space. For example, a filtration can be obtained by repeatedly taking the singular locus in a complex algebraic variety. In these types of settings, the filtrations are often particularly "nice" when it comes to the shape of their strata (for a complex variety they are smooth complex manifolds) and the interactions between the strata. Depending on what kind of niceness conditions one focuses on, one arrives at various different notions of a stratified spaces (for a good introduction, we recommend for example \cite{banagl2007topological} and \cite{friedman2020}). Stratified homotopy theory is concerned with the question of finding a good homotopy theoretical setting for these kind of objects. Intersection homology for example is generally not a homotopy invariant, however, it is invariant under homotopy equivalences that respect the stratification in a appropriate way, called stratified homotopy equivalences (see \cite[Prop. 4.1.10.]{friedman2020}).\\
\\
In the non-stratified setting, i.e. the classical one, the question whether a homotopy equivalence $f:|K| \to |L|$ between two realized (finite) simplicial complexes (or CW-complexes, with the geometrically analogous notions of expansions and collapses, see \cite{cohenCourse}) is a simple homotopy equivalence is entirely answered by a certain element of the
so called Whitehead group, denoted $Wh(K)$, of $K$. This element is called the Whitehead torsion of $f$, $\tau(f)$. It turns out that $\tau(f)$ is $0$ if and only if $f$ is a simple homotopy equivalence. One might think of this result as the fundamental theorem of simple homotopy theory. The Whitehead group admits numerous equivalent constructions ranging from the purely geometric-combinatorial to algebraic constructions, entirely dependent on the fundamental group (groupoid to be more precise) of $K$ (see for ex. \cite{cohenCourse}). Similarly, there are several different interpretations of what $\tau(f)=0$ entails geometrically. In case where $f$ is an $h$-cobordism - that is, $K$ is (a CW-structure on) a closed, connected manifold $M$, $|L|$ is a compact manifold with boundary $M \sqcup N$, and $f$ is the inclusion into the boundary, in addition to being a homotopy equivalence - then if $\textnormal{dim}(M)\geq 5$, $\tau(f)$ essentially measures whether $f$ is given by a boundary inclusion $M \hookrightarrow M \times I$. This result is called the $s$-cobordism theorem, and it is widely regarded to be one of the greatest contributions of simple homotopy theory (see for example \cite{waldhausen2000spaces} for a detailed study of this phenomenon in the topological, smooth and PL setting). 
%Another geometrical perspective, as to what the Whitehead torsion obstructs, is the theory of cell like maps. $f$ is called cell like if its restrictions of the shape $f^{-1}(U) \to U$, for $U$ open, are also homotopy equivalences. Then, $f$ being cell like implies $\tau(f) = 0$. It is from these two perspectives that most of the deeper theoretical investigations into simple stratified homotopy theory, known to us, have taken place (see for example \cite{quinn1988homotopically},). In 
Much of the work on simple stratified homotopy theory has been done with this kind of result in mind. See for example \cite{weinberger1994topological} for a good overview. However, this is not the approach we are going to take. The goal of this master thesis is to develop the combinatorial perspective - that is, the one induced by collapses and expansions - for the stratified setting. To be more precise, we identify a certain class of elementary expansions that preserve the (weak) stratified homotopy type (in the sense of \cite{douteauEnTop}). We then prove the existence of a Whitehead group and Whitehead torsion for this setting, which behave analogously to the classical perspective, that is, they measure whether an equivalence can be built from these simple combinatorial operations. This is be the content of \Cref{chII}. To the best of our knowledge, no such work has been done by previous authors. Comparing our formulation of simple stratified homotopy theory to the one of previous authors will certainly prove an interesting topic for future work.\\
\\
To formulate our combinatorial perspective, we make use of the category of simplicial sets filtered over a poset $P$, $\textnormal{s\textbf{Set}}_P$, introduced in \cite{douSimp}. Much of the work of this master thesis consists of showing what they model (homotopically speaking) on the topological side. This is necessary if one wants to understand what the combinatorial operations and the induced notions of Whitehead group and Torsion measure from a more topological perspective. In \cite{douteauFren}, the author has begun this comparison process. He has introduced model structures on $\textnormal{s\textbf{Set}}_P$ and on $\textnormal{\textbf{Top}}_{P}$, the category of topological spaces filtered over $P$, and studied their interaction. We expand on this work. In fact, we show that the realization functor between these categories preserves weak equivalences and, as long as one restricts to a sufficiently compact setting, induces an equivalence of categories on the homotopy categories. These types of investigations are the content of \Cref{chI}. While a priori not necessary for the construction of the Whitehead invariants, this allows us to interpret them from the topological perspective. Independently from this particular application, we think that they might add to a deeper understanding of stratified homotopy theory, in particular when it comes to a homotopy theoretical formulation of intersection homology. 
\section{Statement of results and structural overview}
As we already mentioned in the previous paragraph, the content of this thesis is split up into two partially independent investigations. 
In this subsection, we give a rough overview of our main results. They are marked in bold font. When stating the content theorems, we often state them in a more conceptual form to avoid too many definitions in this introductory statement. The precise statements are of course found under the respective references. \\
\\
\textbf{ \Cref{chI}: }This chapter is an investigation into two suggested model categories for stratified homotopy theory, suggested in \cite{douteauEnTop} and \cite{douSimp}. It can be read entirely independently from any questions on simple homotopy theory, however, the choice of setting is strongly motivated by such questions. For example, we restrict to the case of a fixed underlying type of stratification. This is sufficient for the study of most notions of equivalence, but of course not for a homotopy theory of more general maps of stratified spaces. In \cite{douSimp} and \cite{douteauEnTop}, Douteau introduced a model category of simplicial sets, $\textnormal{s\textbf{Set}}_P$, and one of topological spaces, $\textnormal{\textbf{Top}}_{P}$, both filtered over some fixed partially ordered set $P$. The objects in $\textnormal{\textbf{Top}}_{P}$ are (certain) topological spaces, filtered by a family of closed subspaces indexed over $P$. This induces a notion of strata, indexed over $P$, as in the case where $P$ is linear. The morphisms are then such maps that retain the stratification index.
The two categories are connected through a realization, singular simplicial set adjunction, $|-|_P \dashv \operatorname{Sing}_P$. The notion of weak equivalences in both cases can be expressed by inducing isomorphisms on analogues of homotopy groups, adapted to the stratified setting. Passing to these weak equivalences that are slightly more general than general stratified homotopy equivalences allows for the extra degree of freedom necessary for our combinatorial formulation of simple stratified homotopy theory. \Cref{chI} is mainly concerned with connecting these two model categories.
To be precise, our main results here are the following. \\
\\First off, it is the content of\textbf{ \Cref{thrmWeakEquSustain}} that the realization functor $|-|_P: \textnormal{s\textbf{Set}}_P \to \textnormal{\textbf{Top}}_{P}$ preserves weak equivalences.
%\begin{theorem}\Cref{thrmWeakEquSustain}
%	The realization functor $|-|_P: \textnormal{s\textbf{Set}}_P \to \textnormal{\textbf{Top}}_{P}$ retains weak equivalences.
%\end{theorem}
We show this result by using results of \cite{douteauFren} in a more rigid filtered setting and translating these to the usual filtered setting by a comparisson of homotopy links argument, \textbf{\Cref{thrmHolWeakEq}}. \\
\\
In particular, as an immediate consequence of this, we have that $|-|_P$ induces a functor of the respective homotopy categories, denoted by adding an $\mathcal H$. Under the additional assumption that we restrict to the full subcategory of $\mathcal H \textnormal{s\textbf{Set}}_P$, $\mathcal H \textnormal{s\textbf{Set}}_P^{fin}$, given by simplicial sets with finitely many non-degenerate simplices, we further show \textbf{\Cref{thrmFullyFaithful}}: The induced functor of homotopy categories $\mathcal H \textnormal{s\textbf{Set}}_P^{fin} \to \mathcal H \textnormal{\textbf{Top}}_{P}$ is fully faithful.\\
\\
To show the latter result, we make use of a filtered analogue of the simplicial approximation theorem. This result is not entirely new. The proof and the theorem are based on methods and results found in \cite{schwartz1971}. However, there are some technical problems in the source which we were able to circumvent. Furthermore, we keep explicit track of subdivisions, which is be necessary for our purposes. The proof is decidedly more involved than the non-filtered setting. We finally obtain \textbf{\Cref{thrmSimplicialApproximationB}}, a filtered analogue to the simplicial approximation theorem, which also states that stratification respecting homotopies between realizations of morphisms in $\textnormal{s\textbf{Set}}_P$, between appropriately filtered ordered simplicial complexes, are witnessed by a sufficiently high degree of subdivision.
This result, independently from the remainder of this work, should be interesting when it comes to comparing the filtered PL-setting to the topological one.
We finally reduce the question of general filtered simplicial sets to that of filtered simplicial complexes, by showing that, at least in a finite setting, every $P$-filtered simplicial set can be replaced by a $P$-filtered ordered simplicial complex, up to simple equivalence. This is a consequence of our result \textbf{\Cref{thrmSSvSC}}, which states that this even holds up to an appropriate notion of simple equivalence. \\
\\
\\As a side effect of the simplicial approximation theorem, we are also able to actually characterize the morphisms in $\mathcal H \textnormal{\textbf{Top}}_{P}$ between objects that are not cofibrant with respect to the model structure. At least between compact, piecewise linear filtered spaces that are stratified in some appropriate sense, we show that they are actually given by stratum respecting homotopy classes between maps of $P$-filtered spaces (\textbf{\Cref{thrmHoClassofFSpace}}). This result hints at the fact that there might be an additional model structure on $\textnormal{\textbf{Top}}_{P}$, having the same equivalences, where the fibrant objects are actually spaces that are "nicely" stratified. We hope to prove such a result in later work.\\
\\
\textbf{ \Cref{chII}: }Having gained a better understanding of our stratified homotopy categories, we then move on to \Cref{chII}, the simple stratified setting. We do so in the setting of $\textnormal{s\textbf{Set}}_P$. Hence, most of the chapter only requires \Cref{subsecNotation} and \Cref{subsecPsset} - \Cref{subsecRepRes} to obtain an abstract understanding. Only at \Cref{secTopWh}, do the results of \Cref{chI} come into play. However, they are of course useful to have in mind as they justify that our combinatorial musings actually have relevance to the topological realm. \\
\\
Our approach to simple homotopy theory is the usage of a general category theoretical framework for simple homotopy theory developed independently by Eckmann and Siebenmann (\cite{eckmann2006} and \cite{siebenmannInfinite}). For this framework to be applicable, one needs to specify a class of morphisms (taking the role of expansions), that fulfill a certain set of axioms (see \Cref{thrmEckSieb}). 
Our analogue to the elementary expansions, described in the classical setting in our introduction before this subsection, are given by pushouts of what Douteau calls ``admissible horn inclusions''. These admissible horn inclusions are a generating set for the trivial cofibrations in $\textnormal{s\textbf{Set}}_P$. The induced notion of expansion is called (finite) filtered strong anodyne extensions, FSAE, for short. They are a more general class of operations than the ones introduced in \cite{banagl2020stratified} and generally only provide weak equivalences in $\textnormal{s\textbf{Set}}_P$ and $\textnormal{\textbf{Top}}_{P}$, not stratified homotopy equivalences (i.e. homotopy equivalences in the sense in the sense of stratum preserving homotopies of stratum preserving maps). However, whenever two realizations of filtered simplicial sets that are "nicely" stratified, not just filtered, are connected by a zigzag of such expansions, they are actually stratified homotopy equivalent, by an analogue to the Whitehead theorem (\Cref{thrmWhitehead}) or an appeal to \Cref{thrmHoClassofFSpace}. At the same time, this additional degree of freedom makes them behave more similar to the classical expansions. We begin illustrating these advantages in \Cref{subsecElemExp}. We then translate methods developed in the non-stratified setting by Moss (\cite{MossSae}) to the stratified one to obtain a series of equivalent characterizations of FSAEs which are often easier to verify (\textbf{\Cref{propEqCharSaeTot}}). \\
\\We use these to show that FSAEs interact with mapping cylinders in the simplicial setting much like expansions of CW-complexes interact with mapping cylinders of cellular maps. This is the content of \textbf{\Cref{propCylIncs}}. 
 \\As a consequence of these types of results, we are then able to show that the axioms of \Cref{thrmEckSieb} are fulfilled in the setting of $\textnormal{s\textbf{Set}}_P^{fin}$ and FSAEs (\textbf{\Cref{thrmEckSiebAxAreTrue}}). In particular, this gives us a of Whitehead group and Whitehead torsion that measures whether a morphism in the homotopy category $\mathcal H \textnormal{s\textbf{Set}}_P^{fin}$ is given by a zigzag of FSAEs, i.e. is a simple equivalence in this setting. We show that these behave much like the classical torsion and Whitehead group, thus laying a solid theoretical groundwork for further, more geometrically minded analysis. This work does not yet contain any attempts at computing our Whitehead group. However, a first step in this direction, allowing for a possible reduction to the classical case, is that we have shown a Mayer-Vietoris formula (\textbf{\Cref{propBigSumFormula}}).\\
 \\
 Next, we obtain a $P$-filtered simplicial set analogue to the well known result that every finite CW-complex has the weak homotopy type of a finite simplicial complex. This is the content of \textbf{\Cref{thrmSSvSC}}, which can be proven independently from any of our results in \Cref{chI}, and in fact was the deciding argument there for the reduction from filtered simplicial sets to filtered ordered simplicial complexes.
 \\
	Then, we move on to obtaining an interpretation of the Whitehead torsion for morphisms in the topological filtered homotopy category $\textnormal{\textbf{Top}}_{P}$. This is of course through the use of methods from \Cref{chI}. In \Cref{propCharOfWh}, we then summarize all the equivalent descriptions of our Whitehead group and Whitehead torsion, that we obtained. Among them is that the Whitehead torsion of a weak equivalence between realizations of finite $P$-filtered simplicial sets, $f: |X|_P \to |Y|_P$, disappears if and only if it comes from a zigzag of (realizations of) elementary expansions in $\mathcal H\textnormal{\textbf{Top}}_{P}$. The Whitehead group of $|X|_P$ (as before) is then described by the class of isomorphism $|X|_P \to |Y|_P$ in $\mathcal H\textnormal{s\textbf{Set}}_P$ modulo post composition with such simple equivalences.\\
	\\
	Finally, we prove that our theory is an extension of the classical one. That is, for the special case where $P = \star$ is a one-point set we show that our Whitehead group is naturally isomorphic to the non filtered one, under the embedding of simplicial sets into CW-complexes. This is the content of \textbf{\Cref{thrmOldNewAgree}}.
\section{Some conventions and notation}
We use this subsection to fix some notation and conventions. We freely make use of standard language of category theory, trying to add a reference when we deviate from the most basic notions.
	\begin{itemize}
			\item Throughout all of this thesis, we denote by $P$ some fixed, possibly infinite, partially ordered set. By a \textit{flag} in $P$, we mean a finite, linearly ordered subset of $P$. Flags are denoted in the form $\{p_0 \leq ... \leq p_n\}$, for $p_i \in P$. We also need flags that allow for repetition. These are called \textit{d-flags}, where the $d$ stands for degenerate. They are denoted in the form $(p_0 \leq ... \leq p_n)$, for $p_i \in P$. We avoid the more standard $[ - ]$ notation as it interferes with the standard notation of setting $[n] := \{0, ..., n\}$.
			\item By $\textbf{Set}$, $\textbf{Ab}$ and $\textbf{AbMon}$ we denote the categories of sets, abelian groups and abelian monoids respectively.
			\item When we speak of topological spaces, we do not mean arbitrary ones, but such topological spaces that are $\Delta$-generated (see \cite{nlab:delta-generated_topological_space}). This means that they have the final topology with respect to all continuous maps coming from realizations of standard simplices (or equivalently from finite dimensional euclidean space). The category of such spaces, together with (continuous) maps as morphisms is denoted by $\textnormal{\textbf{Top}}$. This is a convenient category of topological spaces (in the sense of \cite{nlab:convenient_category_of_topological_spaces}). The reason we do not go with the more standard compactly-generated weakly Hausdorff spaces, is that we are constantly be working with partially ordered sets equipped with the Alexandroff topology (the closed sets are the down-sets). These are only $\operatorname{T}_1$ when they are discrete. Limits in the $\Delta$-generated category are obtained, by taking limits in $\textnormal{\textbf{Top}}$ and then taking the final topology with respect to all maps from simplices into this space. Colimits are computed just as in the naive topological category. We will take care to mention potential issues arising from this, when they occur. However, most constructions we use are purely categorical here and the differences in topoology are basically irrelevant from a homotopy theoretical perspective. This is so, as there is no difference to these settings as long as one ``probes'' with $\Delta$-generated objects such as spheres. The reader having any remaining doubts is recommended a look at \cite{duggerDelta} and at \cite{nlab:delta-generated_topological_space}.
			\item We freely make use of much of the standard language of algebraic topology. Standard notation of spaces includes $I$ for the unit interval and $S^n$ and $D^n$ for the $n$-dimensional sphere and disk respectively.
			\item
			Besides standard results and notions of algebraic topology, we freely make use of the language of simplicial sets and of simplicial homotopy theory. The category of simplicial sets is denotes by $\textnormal{s\textbf{Set}}$. We use most of the standard notation, found for example in \cite{goerss2012simplicial}. 
			\item Further, combinatorial objects of interest are (abstract) ordered and nonordered simplicial complexes. By a simplicial complex $K$ we mean a set $K^{(0)}$ (of vertices), together with a subset of its power set that only contains finite sets and is closed under the subset relation. We have allowed the empty simplex here as it makes some notation more canonical, for example when it comes to joins, but of course this is purely a matter of taste. We freely use most of the language associated to simplicial complexes (simplices, subdivisions, faces etc.), found for example in \cite[Ch. 3]{spanier1989algebraic}. By $\textnormal{\textbf{sCplx}}$, we denote the category of simplicial complexes, with morphisms given by simplicial maps (maps of simplicial complexes), i.e. maps on the underlying vertex sets, such that the image of each simplex is again a simplex. An ordered simplicial complex is a simplicial complex that is additionally equipped with a linear ordering on each simplex such that these orderings are compatible under inclusions of faces. A map of ordered simplicial complexes is a simplicial map that is additionally order preserving on the simplices. The category of ordered simplicial complexes is denoted by $\textnormal{\textbf{sCplx}}^{\operatorname{o}}$. 
			\item
			Finally, we make free use of the language of model categories, in particular simplicial ones. Most of this can be found for example in \cite{hirschhornModel}, but we give references when we first use a new terminology deviating from the most basic definitions. The homotopy category of a model category, obtained by localizing the weak equivalences, is denoted by adding a $\mathcal H$ to the name of the category. We sometimes use this notation for types of homotopy categories, that do not necessarily originate from a notion of model category as well. Morphisms in these homotopy categories that come from ones in the original category are denoted in the form $[f]$. If the model category is simplicial we denote by $[X,Y]$ the set of simplicial homotopy classes between $X$ and $Y$. For the classical simplicial structure on topological spaces, for example, these are just the homotopy classes of continuous maps.
	\end{itemize}
\chapter{Stratified Homotopy Theory} \label{chI}
In this chapter, we are going to set and describe the framework, in which our stratified simple homotopy theory is formulated. As already mentioned, this is based on work of Douteau in \cite{douteauEnTop, douSimp, douteauFren}. Since we have just given a rather detailed overview of the structure and content of this chapter in the overview of results, we refrain from giving another such summary here. The respective sections and subsections are each equipped with such a short summary anyway.
\subsection{Some notation and language for filtered objects}\label{subsecNotation}
Throughout this work, we are constantly be working with objects that are in some sense filtered over a partially ordered set. As it would be quite a hassle to introduce the analogous notation and language over and over again, we summarize some generalities in this subsection. We follow the choice of language in \cite{douteauFren} here. This is in no way intended to be a category theoretical theory of what filtrations are, and purely intended as fixing some language and general constructions. In fact, the language is somewhat conflicting with what is usually defined to be a category of filtered objects (see \Cref{remBadLang}). While writing out such a general theory might be an interesting undertaking, we did not find that it really adds anything mathematically here. In particular, if one deviates too far from the examples we give here, this language can turn out to be somewhat nonsensical. It is not intended in this way. In other words: Is this construction phrased appropriately for a general theory of filtered objects? Definitely not. Is it sufficient to quickly introduce language for a multitude of scenarios we care about? Yes, it is.\\
\\ 
Denote by $\textnormal{\textbf{Poset}}$ the category of partially ordered sets, with morphisms given by order preserving maps.
\begin{definitionconstruction}\label{conFilteredCat}
	Let $\mathcal C$ be some category. (The reader should have the categories of ($\Delta$-generated) topological spaces, simplicial sets, and simplicial complexes in mind). Furthermore, let $F: \textnormal{\textbf{Poset}} \longrightarrow \mathcal C$ be some functor. Then, for each $P \in \textnormal{\textbf{Poset}}$, consider the over category $\mathcal C_{\slash F(P)}$, over $F(P)$. That is, the category with objects given by arrows $p_C:C \to F(P)$ in $\mathcal C$ and morphisms given by commutative diagrams:
	\begin{center}
		\begin{tikzcd}
		C \arrow[rd ,"p_C", swap] \arrow[rr]& & C' \arrow[ld, "p_{C'}"]\\
		&F(P)&
		\end{tikzcd}.
	\end{center}
	We call $\mathcal C_{\slash F(P)}$ the \textit{category of $P$-filtered objects with \textit{stratum preserving morphisms} in $\mathcal C$ (with respect to $F$)}. We mostly omit the "with respect to $F$" from here on out, but one should note that it is important to keep track of it if one wants to obtain a notion of strata. If the objects and morphisms in $\mathcal{C}$ have particular names, let us say spaces and maps, the objects and morphism in $\mathcal C_{\slash F(P)}$ are called \textit{$P$-filtered spaces} and {stratum preserving maps} respectively. We usually omit explicitly mentionining that the morphisms are to be stratum preserving, but one should keep in mind that there are other notions of morphisms of filtered objects (see \cref{remBadLang}). By a slight abuse of notation, we often refer to a $P$-filtered object by its underlying object in $\mathcal C$. The arrow $p_C$ is then referred to as \textit{the filtration} of $C$.
\end{definitionconstruction}
This construction is functorial in a number of ways, summarized in \Cref{conFiltFun}. We start, however, with
 a series of examples of such categories so that the reader has some examples in mind. In our cases, the generated filtered categories have somewhat more concrete descriptions than just being some arrow category. We give such descriptions in the respective subsections.
\begin{example}\label{exFilteredObj}
	Our main examples for $F$ as in \Cref{conFilteredCat} and hence categories of filtered objects are the following. 
	\begin{itemize}
		\item \textbf{Filtered topological spaces: }Given a partially ordered set $P$ we can equip it with its Alexandroff topology; that is, the topology with closed sets given by the sets that are closed below under the partial order on $P$. By abuse of notation, we also refer to this space as $P$, as it contains all the data of the original partially ordered set. Then $P$ is in fact a $\Delta$-generated space. To see this, we need to show that a set whose inverse image under every map from the realization of a $n$-simplex $|\Delta^n| \to P$ is closed is a downset. Let $A$ be such a set. In fact, it suffices to probe with intervals instead of arbitrary simplices. For $p \in A$, $p' \leq p$ consider the map \begin{align*}
			\sigma: I &\longrightarrow P \\
			[0,0.5] \ni t &\longmapsto p'\\
			(0.5, 1] \ni t &\longmapsto p.
		\end{align*}
		This is clearly continuous. But as $(0.5,1]$ is not closed, we have $$p \in A \implies p' \in A.$$ Thus, the Alexandroff topology construction induces a functor $$\textnormal{\textbf{Poset}} \longrightarrow \textnormal{\textbf{Top}},$$ taken as the identity on morphisms. In particular, using this functor as in \Cref{conFilteredCat}, we obtain categories of filtered ($\Delta$-generated) topological spaces. For a fixed poset $P$, we denote this category by $\textnormal{\textbf{Top}}_{P}$. $\textnormal{\textbf{Top}}_{P}$ is studied in detail in the next subsection.
		\item \textbf{Filtered simplicial complexes: } Given a partially ordered set, the subset of its power set, given by all finite linearly ordered subset (flags) of $P$, $\textnormal{sd}(P)$ is a simplicial complex, called the nerve of $P$. Even more, it is an ordered simplicial complex (see \Cref{subsecOrdered}) in the obvious way. The induced covariant map on power sets makes this construction functorial, inducing functors 
		\begin{align*}
			\textnormal{sd}(-): \textnormal{\textbf{Poset}} &\longrightarrow \textnormal{\textbf{sCplx}} \\
			\textnormal{sd}(-): \textnormal{\textbf{Poset}} &\longrightarrow \textnormal{\textbf{sCplx}}^{\operatorname{o}}. 
		\end{align*}
		We denote the induced categories of (ordered) filtered simplicial complexes by $\textnormal{\textbf{sCplx}}_P,(\textnormal{\textbf{sCplx}}^{\operatorname{o}}_P)$. They are studied in detail in \Cref{secSimApp}.
		\item \textbf{Filtered simplicial sets: } Consider the fully faithful embedding of the simplex category $\Delta$ into $\textnormal{\textbf{Poset}}$. Using this embedding, one obtains the nerve functor $$N(-): \textnormal{\textbf{Poset}} \to \textnormal{s\textbf{Set}},$$ by sending $P$ to $$ \Delta \hookrightarrow \textnormal{\textbf{Poset}} \xrightarrow{\textnormal{Hom}_{\textnormal{\textbf{Poset}}}(-,P)} \textnormal{\textbf{Set}}.$$ 
		Explicitly, the $n$-simplices of $N(P)$ are given by (not necessarily strictly) increasing sequences in $P$ of length $n$. In other words, they are given by the d-flags in $P$. The length $n$ of such a d-flag $\mathcal J$ is denoted by $\#\mathcal J$. Note that $N(P)$ can be alternatively thought of as the image of $\textnormal{sd}(P)$ under a Yoneda-style embedding $\textnormal{\textbf{sCplx}}^{\operatorname{o}} \hookrightarrow \textnormal{s\textbf{Set}}$. We do this in detail in \Cref{subsecOrdered}. The induced category of $P$-filtered simplicial sets is be denoted by $\textnormal{s\textbf{Set}}_P$. If $\Delta^n \to N(P)$ is the filtration of an $n$-simplex, then by definition of $N(P)$, it is uniquely determined by a $d$-flag $\mathcal J = (p_0 \leq ... \leq p_n)$. We denote such a filtered simplex by $\Delta^{\mathcal J}$. We begin with a more detailed description of $\textnormal{s\textbf{Set}}_P$ in \Cref{subsecPsset}.
	\end{itemize}
\end{example}
\begin{example}
	It can also be interesting to consider different cases of $P$, for the \Cref{conFilteredCat}. For example if $P =\star$, the one-point set, and $F$ is a functor that preserves terminal objects, then the category $\mathcal C_{\slash F(\star)}$ is the arrow category over a terminal object, which is isomorphic to the original category $\mathcal C$ under the forgetful functor. Hence, in \Cref{exFilteredObj}, we just obtain the categories 
		$\textnormal{\textbf{Top}}$, $\textnormal{\textbf{sCplx}},$ $\textnormal{\textbf{sCplx}}^{\operatorname{o}}$ and $\textnormal{s\textbf{Set}}$. 
\end{example}
\Cref{conFilteredCat} is functorial in the following sense. The reader will of course notice that the following construction can be decomposed into a few varieties of smaller functorialities, but as the examples relevant to us are mostly of the following particular shape this summarized one makes it easier to refer to this construction.
\begin{definitionconstruction}\label{conFiltFun}
	Let $\mathcal C$ and $\mathcal D$ be two categories and $P,Q \in \textnormal{\textbf{Poset}}$. Consider a diagram of functors
\begin{center}
		\begin{tikzcd}
		\mathcal{C} \arrow[rr, "L" {name = L}]& & \mathcal{D} \\
		& \arrow[lu, "F"{name=F}]\textnormal{\textbf{Poset}} \arrow[ru, "G"{name=G}, swap]& 
	\end{tikzcd}
\end{center}
together with a morphism $f: LF(P) \to G(Q)$. Then, for each $P \in \textnormal{\textbf{Poset}}$, we obtain a functor $$\mathcal C_{\slash F(P)} \to \mathcal D_{\slash G(Q)},$$ by applying $L$ to arrows, and then post composing them with $f$. Dual to this, if we are given a diagram of functors 
\begin{center}
	\begin{tikzcd}
		\mathcal{C} & & \arrow[ll, "R" {name = R}]\mathcal{D} \\
		& \arrow[lu, "F"{name=F}]\textnormal{\textbf{Poset}} \arrow[ru, "G"{name=G}, swap]& 
	\end{tikzcd}
\end{center}
together with a morphism $g:F(P) \to RG(Q)$ which $\mathcal C$ has pullbacks along, then we obtain a functor $$\mathcal D_{\slash G(Q)} \to \mathcal C_{\slash F(P)}$$ by first applying $R$ to each arrow, and then changing base (i.e. pulling back) along $g$. 
\\
This is particularly interesting in the case, where $L \dashv R$ are an adjoint pair and $f$ and $g$ correspond to each other under this adjunction. Then, as base change and postcomposition are also adjoint, one obtains that the induced functors also form an adjoint pair (with the same left-right positions).
\end{definitionconstruction}
The first main example one should consider is the case where $Q \subset P$ is a subposet and $R$ is given by the identity functor. In this case this leads to an abstract definition of strata.
\begin{definition}\label{defStrata}
	In the setting of \Cref{conFilteredCat}, fix some partially ordered sets $Q \subset F$. Then we denote by 
	$$(-)_Q: \mathcal C_{\slash P} \to \mathcal C_{\slash Q}$$
	the functor obtained by basechanging along $F(Q) \to F(P)$ (as in \Cref{conFiltFun}). If $Q=\{p\}$ is a one point set, and $F$ preserves terminal objects we can also think of $(-)_{\{p\}}$ as having target in $\mathcal C_{\slash F(\{p\})} \cong \mathcal C$. Then, for $C \in \mathcal C$, we denote this by $C_{p}$ and call it \textit{the $p$-th stratum} of $C$. Furthermore, also for $p \in P$, we denote $C_{\leq p}:=C_{\{q \leq p\}}$. 
\end{definition}
\begin{example}
	If $\mathcal C = \textnormal{\textbf{Top}}$ and $P = [n]$ is some finite linearly ordered set, then a filtered space is a filtered space in the sense of Friedman (see for \cite{friedman2020} (up to being Hausdorff)). Then the strata in our sense are the ($\Delta$-fication of) the strata in the usual sense. The $p$-th filtration space, denoted $T^{p}$ in Friedmans book, is then given by (the underlying space of) $T_{\leq p}$. 
\end{example}
\begin{remark}
	We should make another remark on notation at this point. The reason we chose not to go with the classical notation - of using lower indices for strata and superscript indices for filtration degrees - is that in any sort of inductive proof one quickly runs out of places where one can put their indices. In fact, our situation is already confusing enough, as it is standard to use the lower index for the set of $n$-simplices in a simplicial set $X$. We always denote this set by $X([n])$ to avoid any possible confusion. The superscript is reserved for enumeration of objects.
\end{remark}
Other frequently occurring examples of functors arising through \Cref{conFiltFun} are the following.
\begin{example}\label{exFiltFun} We consider a few examples of functors induced in the filtered setting. 
	\begin{itemize}
		\item \textbf{Subdivisions}: Recall the barycentric subdivision functors, constructed as follows. On a simplex, take its power set as a partially ordered set, then apply $\textnormal{sd}(-)$ or $N(-)$ respectively, inducing a functor $\Delta^{op} \to \mathcal C$, where $\mathcal{C}$ is either $\textnormal{s\textbf{Set}}, \textnormal{\textbf{sCplx}}^{\operatorname{o}}$ or $\textnormal{\textbf{sCplx}}$. This is then left Kan extended to the category of (ordered) simplicial complexes and simplicial sets respectively. More explicitly for simplicial sets one extends via $\textnormal{sd}(X) = \varinjlim_{\Delta^n \to X, n \in \mathbb N} \textnormal{sd}(\Delta^n)$ and for a simplicial complex $K$, one just takes $\textnormal{sd}(K)$ to be the functor of \Cref{exFilteredObj} applied to the set of simplices of $K$ ordered by inclusion. For ordered simplicial complexes and more generally ordered simplicial sets, $\textnormal{sd}$ comes together with a last vertex map $$\textnormal{l.v.}:\textnormal{sd}(\Delta^n) \longrightarrow \Delta^n,$$ induced on the simplex level by sending a vertex in the subdivision $\sigma = (x_0 \leq ... \leq x_k)$ to $x_k$, the maximal vertex with respect to the linear ordering on $\sigma$. This induces a natural transformation $$\textnormal{l.v.}: \textnormal{sd}(-) \to 1_{\mathcal C},$$ where $\mathcal C$ is either $\textnormal{s\textbf{Set}}$ or $\textnormal{\textbf{sCplx}}^{\operatorname{o}}$. Now, in \Cref{conFiltFun}, for $\mathcal C$ either of the three categories above, $L=\textnormal{sd}$ the respective subdivision functor and $f$ the corresponding last vertex map \begin{align*}
			\textnormal{sd}(N(P)) &\to N(P), \\
			\textnormal{sd}^2(P) &\to \textnormal{sd}(P)
		\end{align*} we obtain filtered barycentric subdivision functors. Explicitly, the filtration of a subdivision of a simplex of $X$ is given, by assigning to each vertex $\sigma$ the maximal $p$ in $p_{X}(\sigma)$.
		\begin{figure}[H]
	\centering
	\includegraphics[width=80mm]{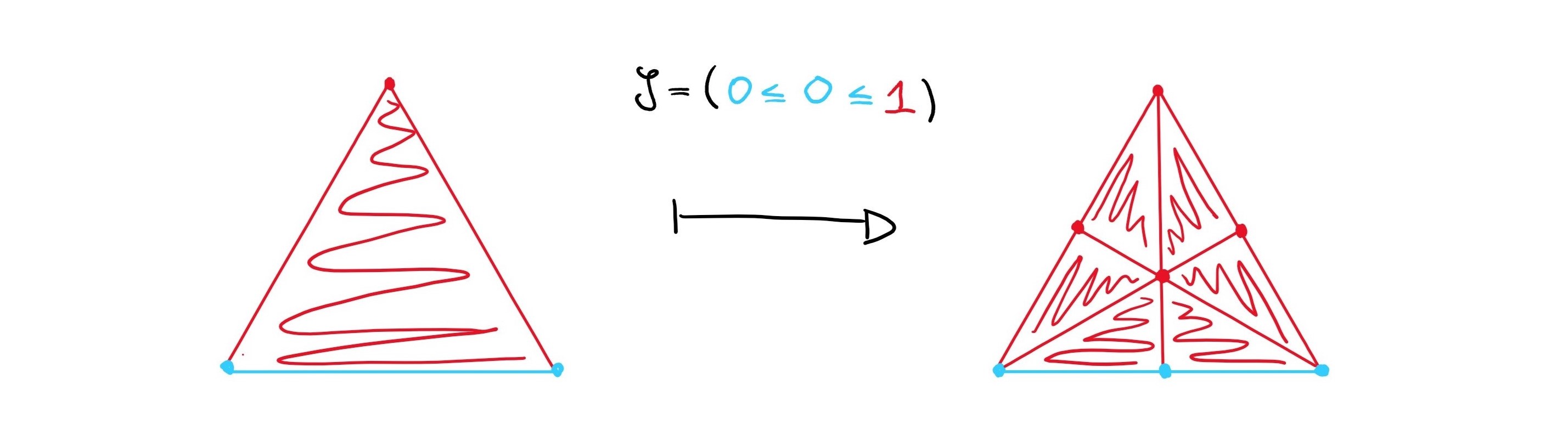}
	\caption{Subdivision of the filtered simplex over $\{0,1\}$ given by $\mathcal J = (0 \leq 0 \leq 1)$.}
	\label{fig:exFiltFunSd}
	\end{figure}
		 By abuse of notation, we denote all of these functors by $\textnormal{sd}(-)$. This is also to distinguish them, from another type of subdivision functor that we introduce later on in \Cref{conSDp}. In the case where $\mathcal C$ is either the category of filtered simplicial sets or of ordered simplicial complexes, the last vertex map is filtered and hence induces a natural transformation $$\textnormal{l.v.}:\textnormal{sd}(-) \longrightarrow 1_{\mathcal C_{\slash P}}.$$ This particular type of subdivision is of course be central to our investigation of the filtered simplicial approximation theorem in \Cref{secSimApp}.
		\item \textbf{Realizations}: Let $|-|$ be either of the classical realization functors from $\textnormal{s\textbf{Set}}$, $ \textnormal{\textbf{sCplx}}^{\operatorname{o}}$, $\textnormal{\textbf{sCplx}}$ into $\textnormal{\textbf{Top}}$. 
		Then the realizations of $|\textnormal{sd}(P)| = |N(P)|$ are naturally filtered as follows. Each realized simplex in $|\textnormal{sd}(P)|$ is given by the convex span 
		$$\Big \{\sum_{p \in \mathcal J} t_p |p| \subset \mathbb{R}^{\mathcal J} \mid t_p \geq 0, \sum_{p \in \mathcal J} t_p = 1 \Big \},$$ 
		where $\mathcal J$ is some flag in $P$ and we denote by $|p|$ the unit vector in $\mathbb R^{\mathcal J}$ corresponding to $p$. 
		Then, we send such an element $ \sum_{p \in \mathcal J} t_p |p|$ to $\max\{p \in \mathcal J \mid t_p > 0\}$. This defines a natural transformation between $|\textnormal{sd}(-)|$ and the Alexandroff space functor. Taking $L$ to be $|-|$ and $f$ to be this natural transformation in \Cref{conFiltFun} then induces filtered realization functors $$|-|_P: \mathcal{C}_{P} \longrightarrow \textnormal{\textbf{Top}}_{P},$$ again all denoted the same. Explicitly, the interior of a simplex $\sigma \in X$ is assigned to the stratum corresponding to $\max\{p_X(\sigma)\}$. 
		\begin{figure}[H]
			\centering
			\includegraphics[width=120mm]{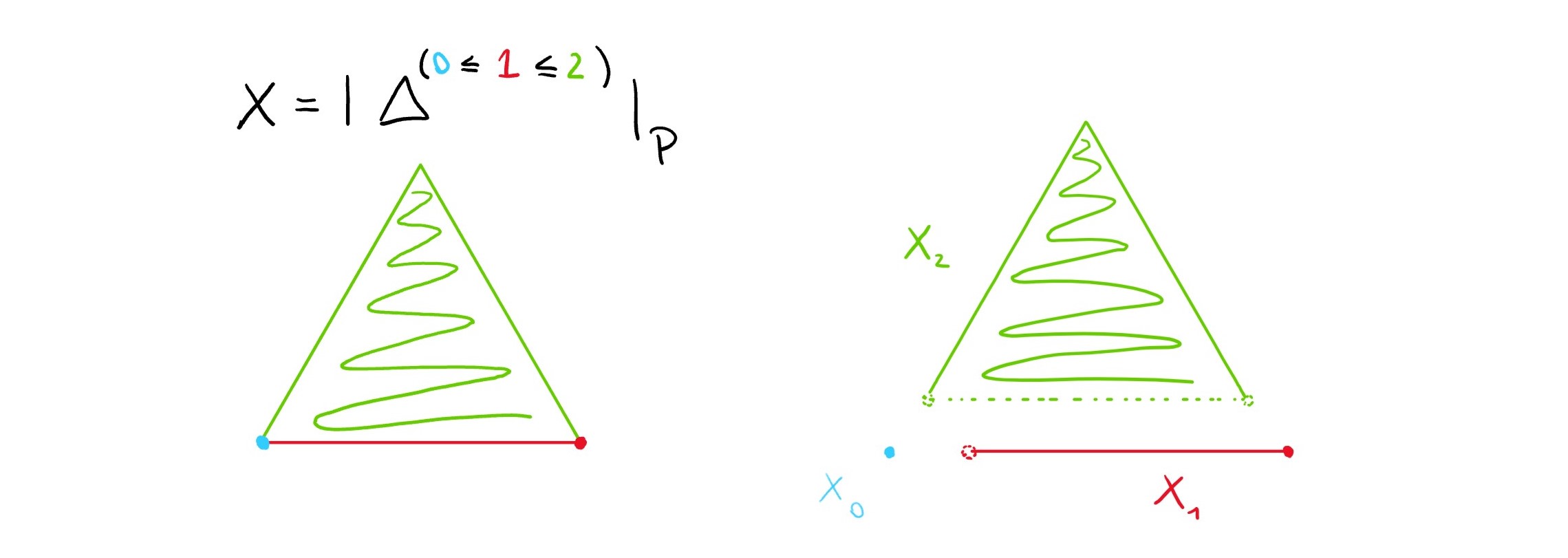}
			\caption{Filtration and induced strata of the realization of a standard simplex corresponding to the flag $\mathcal J= \{0 \leq 1 \leq2\}$ over $P = \{0, 1,2\}$.}
			\label{fig:exFiltFunReal}
		\end{figure}These functors are of course compatible with the various subdivision and embedding functors between different choices of $\mathcal C$, described in more detail in \Cref{secSimApp}. 
		\item \textbf{Singular simplicial sets: }$|-|$ is of particular interest when $\mathcal C$ is the category of simplicial sets, as then it admits a right adjoint given by the singular simplicial set functor $$\operatorname{Sing}: \textnormal{\textbf{Top}} \longrightarrow \textnormal{s\textbf{Set}}$$ (see for ex \cite{goerss2012simplicial}). Now, take $g: N(P) \to \operatorname{Sing}(P)$ to be the adjoint map from the natural map $|N(P)| \to P$, from the last bulletpoint and $R = \operatorname{Sing}(-)$. Then \Cref{conFiltFun} induces a right adjoint $\operatorname{Sing}_P(-)$ to $|-|_P$. Explicitly, $\operatorname{Sing}_P(X) \subset \operatorname{Sing}(X)$ is given only by such singular simplices, that are of the shape $\sigma : |\Delta^\mathcal J|_P \to X$, for d-flags $\mathcal J$. The filtration is then given by mapping $\sigma$ to $\mathcal J$.
	\end{itemize}
	$|-|_P$ and $\operatorname{Sing}_P(-)$ are of course be the integral instruments of connecting homotopy theory in the simplicial and the topological setting. 
\end{example}
To finish this section, we should make a quick note on how limits and colimits are computed in these categories of filtered objects.
\begin{remark}\label{remLimandColim}
	Let $\mathcal C$ be an arbitrary category and $c \in \mathcal C$. It is an easily verified fact on over categories that if $\mathcal C$ is complete (cocomplete) (as $\textnormal{\textbf{Top}}$ and $\textnormal{s\textbf{Set}}$ are), then so is $\mathcal C_{\slash c}$. Limits of a diagram $F: \mathcal J \to \mathcal C_{\slash c}$ are computed by first composing with the forgetful functor into $\mathcal C$, then adding $c$ and the arrows into it to the resulting diagram, and finally taking the limit of this diagram in $\mathcal{C}$ and its structure morphism into $c$ as the limit in $\mathcal C_{\slash c}$. For example, the product in $\textnormal{s\textbf{Set}}_P$ of $X$ and $Y$ is given by the diagonal of the pullback square
	\begin{center}
		\begin{tikzcd}
			X\times_{N(P)} Y \arrow[d] \arrow[r] & X \arrow[d, "p_X"]\\
			Y \arrow[r, "p_Y"] & N(P)
		\end{tikzcd}.
	\end{center}
	Colimits are computed, by again taking the composition with the forgetful functor, and taking the colimit in $\mathcal{C}$ together with its induced arrow into $c$. That is, the forgetful functor $\mathcal C_{\slash c} \to \mathcal C$ preserves and reflect colimits. In particular, colimits in $\textnormal{s\textbf{Set}}_P$ or $\textnormal{\textbf{Top}}_{P}$ are computed on the underlying simplicial sets (spaces), and then naturally filtered.
\end{remark}
We should end this subsection with a quick remark on possible linguistic confusion that may arise.
\begin{remark}\label{remBadLang}
	Sadly, the nomenclature around stratified and filtered objects is somewhat inconsistent. This is in particular the case, when it comes to when something should be called stratified as opposed to filtered. To just give a small outlook on the choice of possible nomenclature, here is a taste: In \cite{weinberger1994topological}, a map of spaces filtered over the same $P$ is said to be stratified, if it is stratum preserving in our sense. In \cite{douteauFren}, such a map is called a filtered map. A filtered map, in \cite{weinberger1994topological} however, is one that satisfies $f(X_{\leq k}) \subset Y_{\leq k}$ instead of $f(X_k) \subset Y_k$. The term stratified, is reserved for maps between filtered spaces over different posets, given by commutative diagrams 
	\begin{center}
		\begin{tikzcd}
			X \arrow[r] \arrow[d, "p_X"]& Y \arrow[d, "p_Y"] \\
			P \arrow[r] &Q.
		\end{tikzcd}
	\end{center}
	 In \cite{haine2018homotopy}, a stratified space is just what we defined to be a filtered space. In \cite{chataur2018intersection} and \cite{friedman2020}, they are however defined to be filtered space that additionally satisfy the so called Frontier Condition. Classically, stratified spaces are often expected to have strata that at least homologically are manifold like. As to not add to the confusion, we only use the term stratified space with additional qualifiers such as \textit{homotopically} or \textit{locally conelike.} When it comes to maps, we decided to go with stratum preserving, simply because it seems to be the most descriptive. Sadly, this has the awkward side effect that the objects and morphisms in our categories are not named analogously. In particular, when we are citing from \cite{douSimp} and \cite{douteauEnTop}, one should be careful to do the necessary translations.
\end{remark}
\section{Model structures for stratified homotopy theory}
Our study of stratified homotopy theory does heavily apply the language of model categories. We do not start the undertaking of giving an introduction to this theory here, as it would go beyond the scope of this master thesis. We recommend \cite{hirschhornModel} for an introduction. There are several approaches that attempt to use model categories to study stratified homotopy theory out there at the moment. We have decided to go with the one recently published by Douteau in his thesis \cite{douteauFren} which is partially summarized in \cite{douteauEnTop} and \cite{douSimp}. Roughly speaking, the perspective of homotopy theory come down to defining weak equivalences to be morphisms that induce weak equivalences (in the classical sense) on all strata and (generalized) links (the spaces connecting the strata), as opposed to stratified homotopy equivalences. Throughout this thesis, we try to motivate why we think this setting is appropriate in particular for our study of simple homotopy theory. In this section, we give a quick introduction into the work done in these publications. The intention however is mostly one of giving and motivating definitions, and for all proofs we refer to the respective sources. The reader familiar with the results in \cite{douteauEnTop} and \cite{douSimp} can probably skip this section as our own work begins from the section following it.
\subsection{The Henrique-Douteau model structure on $P$-filtered topological spaces}
It is a frequent phenomenon in mathematics that for the successful study of a restrictive class of objects it can be immensely helpful to pass to a larger class. Doing so one gains access to a greater range of operations without constantly having to worry about leaving the restrictive setting. In some sense, this is what filtered spaces are to stratified spaces. Intersection homology for example, while originally defined and intended only for pseudomanifolds, extends to all spaces filtered over a finite linear set (see \cite{king1985topological}). Furthermore, if one wants to enjoy all the advantages of a model category, one needs to pass to a large enough class of objects such that it admits limits and colimits. This is also be necessary for the development of our simple homotopy theory, as the filling of horns does in general not sustain the additional properties required for the various definitions of stratified spaces. Our setting, for the topological investigation of stratified homotopy theory, is be the category of filteres spaces over $P$, $\textnormal{\textbf{Top}}_{P}$, introduced in \Cref{exFilteredObj}. We start this subsection by giving a less abstract description of objects and morphisms in $\textnormal{\textbf{Top}}_{P}$ and introduce the various enrichments of it. We then start investigating notions of homotopy in this setting. In doing so, we also summarize the most important results on the model structure defined by Douteau on this category in \cite{douteauEnTop} and \cite{douteauFren}. We mostly follow \cite{douteauFren} for this.\\
\\
In \Cref{exFilteredObj} we defined the category of $P$-filtered topological spaces to be the over-category of $\textnormal{\textbf{Top}}$ over $P$, where the latter is thought of as a topological space, equipped with the Alexandroff topology. 
\begin{remark}
	A $P$-filtered topological space $X$ can alternatively be defined as a space $X \in \textnormal{\textbf{Top}}$ together with a family $(X_{\leq p})_{p \in P}$ of closed subspaces such that $$X_p \subset X_{p'} \iff p \leq p',$$ for $p,p' \in P$. The strata $X_p$ of $X$ are then explicitly given by the subspace (in the $\Delta$-generated sense) of $X$ given by $$X_p = X_{\leq p} - \bigcup_{p' < p} X_{\leq p'}.$$ In particular, most of the various definitions of stratified spaces can be thought of as objects in $\textnormal{\textbf{Top}}_{P}$ for appropriate $P$. Taking this point of view, a morphism in $\textnormal{\textbf{Top}}_{P}$ $f: X \to Y$ alternatively described as a map of the underlying topological spaces such that $$f(X_p) \subset Y_p \textnormal{ ,for all }p \in P.$$ This justifies calling them stratum preserving.
	\begin{figure}[H]
		\centering
		\includegraphics[width=120mm]{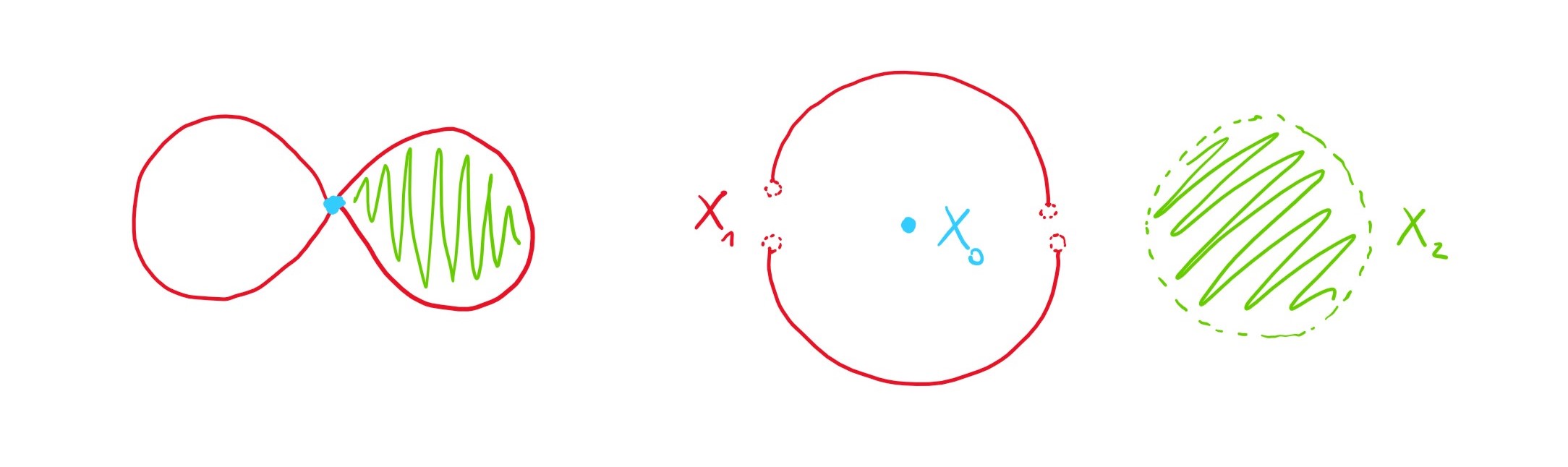}
		\caption{Visualization of a filtered space over $P = [2]$ and its strata, obtained by taking the wedge sum of a $S^1$ and a $D^2$.}
		\label{fig:exFiltSp}
	\end{figure}
\end{remark}
The category $\textnormal{\textbf{Top}}_{P}$ is enriched and copowered (see \cite{nlab:enriched_category} and \cite{nlab:copower} for definitions) over the closed monoidal category $\textnormal{\textbf{Top}}$. This is accomplished as follows.
\begin{definitionconstruction}\label{conEnTop}
	The category $\textnormal{\textbf{Top}}$ is a closed monoidal category enriched over itself. 
	To be more precise, one obtains a monoidal structure through the product, $- \times -$. 
	And an inner hom-functor by equipping $\textnormal{Hom}_{\textnormal{\textbf{Top}}}(T,T')$ with the $\Delta$-ification of the compact open topology, for $T,T' \in \textnormal{\textbf{Top}}$ (see also \cite{duggerDelta}).
 	We denote the resulting space by $C^0(T,T')$. We now equip $\textnormal{Hom}_{\textnormal{\textbf{Top}}_{P}}(X,Y)$ with the ($\Delta$-generated) subspace topology with respect to the inclusion $$\textnormal{Hom}_{\textnormal{\textbf{Top}}_{P}}(X,Y) \subset C^0 (X,Y),$$ where we omitted the forgetful functor into $\textnormal{\textbf{Top}}$ in the right mapping space. 
 	We denote these spaces by $C^0_P(X,Y)$. Furthermore, $\textnormal{\textbf{Top}}_{P}$ is copowered over $\textnormal{\textbf{Top}}$, with the copower $- \otimes X \dashv C^0_P(X,-)$ given by the $P$-filtered space $$ T \times X \xrightarrow{\pi_X} X \xrightarrow{p_X} P.\,$$ for $T \in \textnormal{\textbf{Top}}$ and $X \in \textnormal{\textbf{Top}}_{P}$. For details see \cite[Sec. 5.2]{douteauFren}.
\end{definitionconstruction} 
The copowering with $\textnormal{\textbf{Top}}$ naturally induces a notion of homotopy on $\textnormal{\textbf{Top}}_{P}$.
\begin{definition}
	Two stratum preserving maps of $P$-filtered spaces $f_0,f_1: X \to Y$ are said to be \textit{stratum preserving homotopic} or also \textit{stratified homotopic} if there exists a stratum preserving map $H: I \otimes X \to Y$ such that the diagram
	\begin{center}
		\begin{tikzcd}[column sep = large]			X \sqcup X \cong ( \star \sqcup \star) \otimes X \arrow[r, "f_0 \sqcup f_1"] \arrow[d, hook, "i \otimes 1"]& Y\\
			I \otimes X \arrow[ru, dashed, "H", swap] 
		\end{tikzcd}
	\end{center}
	commutes. $i: \star \sqcup \star \hookrightarrow I$ denotes the inclusion of the endpoints here. It is an easy verification that this notion of homotopy behaves much like the classical one, that is, being homotopic is symmetric, reflexive, transitive and compatible with compositions etc. We use most of the standard nomenclature, also used in the non-filtered setting, adding the prefix \textit{stratum preserving} or \textit{stratified}. So by nomenclature such as\textit{ stratified homotopy classes, stratum preserving homotopy equivalence, etc.} we mean the obvious thing. Stratum preserving homotopy equivalences are also often referred to as stratified homotopy equivalences.
	For $X,Y \in \textnormal{\textbf{Top}}_{P}$, we denote by $[X,Y]_P$ the set of \textit{stratified homotopy classes} of stratum preserving maps $X \to Y$. Furthermore, stratum preserving homotopy between morphisms in $\textnormal{\textbf{Top}}_{P}$ is denoted by $\simeq_P$. 
\end{definition}
It is a straighforward verification using standard properties of mapping spaces that this notion of homotopy is compatible with the enriched structure in the following sense.
\begin{lemma}\label{lemC0retainHo}
	The enriched hom functor $$C^0_P(-,-):\textnormal{\textbf{Top}}_{P}^{op} \times \textnormal{\textbf{Top}}_{P} \to \textnormal{\textbf{Top}}$$ sends stratified homotopic morphisms in $\textnormal{\textbf{Top}}_{P}$ into homotopic ones. In particular, stratified homotopy equivalence induce homotopy equivalences on $C_P^0(-,-)$.
\end{lemma}
\begin{remark}
	For many intents and purposes, stratum preserving homotopy equivalence is a very useful notion of isomorphism. Singular intersection homology for example, descends to a functor on stratified homotopy classes, and hence is a stratified homotopy equivalence invariant \cite[Prop. 4.1.10.]{friedman2020}. However, for many other intends and purposes, such as the simple homotopy theory (see also our comments in \Cref{subsecElemExp}) we define in this work, stratified homotopy equivalence seems to be too discrete of an equivalence relation. Consider for example the map indicated in the following. 
	\begin{figure}[H]
		\centering
		\includegraphics[width=120mm]{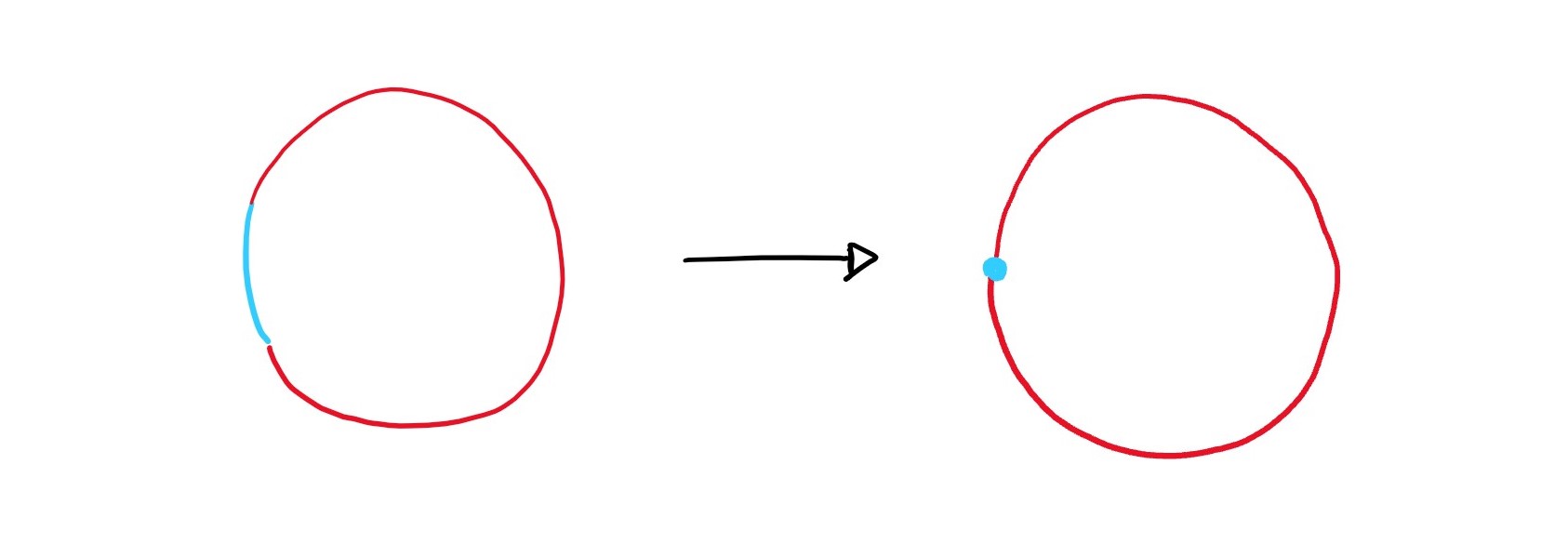}
		\caption{Illustration of a stratum preserving map between two filtrations of $S^1$. It is given by collapsing the blue region to a point.}
		\label{fig:ExNotStratEq}
	\end{figure}
	This map clearly induces a homotopy equivalence of the underlying spaces. However, it is not a stratum preserving homotopy equivalence. Geometrically, this is since there is no way of continuously mapping the cycle on the right to the one on the left, without mapping the red into the blue part.
	An easy way to see this more formally is to compute the intersection homology (see for ex. \cite{friedman2020}), $I^0H_1(-,\mathbb Q)$ (simplicially) for both sides. As the left hand side gives a filtration of $S^1$ that turns it into CS set by invariance of filtration (see \cite[Ch. 2]{friedman2020}) we obtain: 
	$$ I^0H_1(lhs,\mathbb Q) = H_1(S^1,\mathbb Q) = \mathbb Q.$$
	For the right hand side we can not make a similar argument, as the $0$-stratum is of the wrong dimension. But an easy computation, or the geometric argument that no $1$-cycle can intersect the $0$-stratum transversally, shows that $$I^0H_1(rhs,\mathbb Q) = 0.$$ 
	 Similarly to passing from homotopy equivalences to weak homotopy equivalences, it turns out to be useful to broaden the notion of stratum preserving homotopy equivalence a little, to a notion lying between stratum preserving homotopy equivalence and (weak) homotopy equivalence of the underlying spaces.
\end{remark}
To do so, we need the notion of homotopy links.
\begin{definitionconstruction}\label{defHomotopyLink}
	Let $\mathcal J$ be a flag in $P$ and $X \in \textnormal{\textbf{Top}}_{P}$. Consider the space $C^0_P(|\Delta^\mathcal J|_P, X) \in \textnormal{\textbf{Top}}$. If $\mathcal{J} = \{p\}$, then this gives the $p$-th stratum. If $\mathcal{J} = (p_0 < p_1)$, then this is what is usually referred to as the homotopy link (see \cite{quinn1988homotopically}) of the $p_0$-th stratum in the $p_1$-th. Following this nomenclature, we call $C^0_P(|\Delta^{\mathcal J}|_P, X)$ the $\mathcal{J}${-th homotopy link} of $X$ and denote it by $$\textnormal{Hol}_{P}(\mathcal J,X).$$ We also call such objects generalized homotopy links. Functoriality of $C^0_P$ and $|-|_P$ induces a functor $$\textnormal{Hol}_{P}: \textnormal{\textbf{Top}}_{P} \longrightarrow \textnormal{\textbf{Top}}^{\textnormal{sd}(P)^{op}}$$ into the category of space valued presheaves on $\textnormal{sd}(P)$, where we think of $\textnormal{sd}(P)$ as a category by taking the category induced by the partial order of inclusion of simplices. 
\end{definitionconstruction}
\begin{remark}
	One should motivate a little bit why the space in \Cref{defHomotopyLink} are called links. Recall that, for a simplicial complex $K$ with a full subcomplex $S \subset K$, the link of $S$ in $K$ is the subcomplex given by all simplices that do not intersect $K$ but are contained in a simplex doing so. For an illustration, see \Cref{fig:exLinkandHol}.
	In a sense, the link contains geometric information about how $S$ sits in $K$. If one takes a first barycentric subdivision first, then furthermore the link has a natural map into $S$, given by sending a vertex $\sigma \in \textnormal{sd}(K)$ given by a simplex $\sigma $ of $K$ to its intersection with $S$. Then the link gives the boundary of a mapping cylinder neighbourhood of $|S|$ in $|K|$, induced by this map. For piecewise linear filtered spaces (that is, PL spaces filtered by closed PL spaces) one can always construct a regular neighbourhood in this way and the link is unique up to PL homeomorphism. 
	\begin{figure}[H]
		\centering
		\includegraphics[width=120mm]{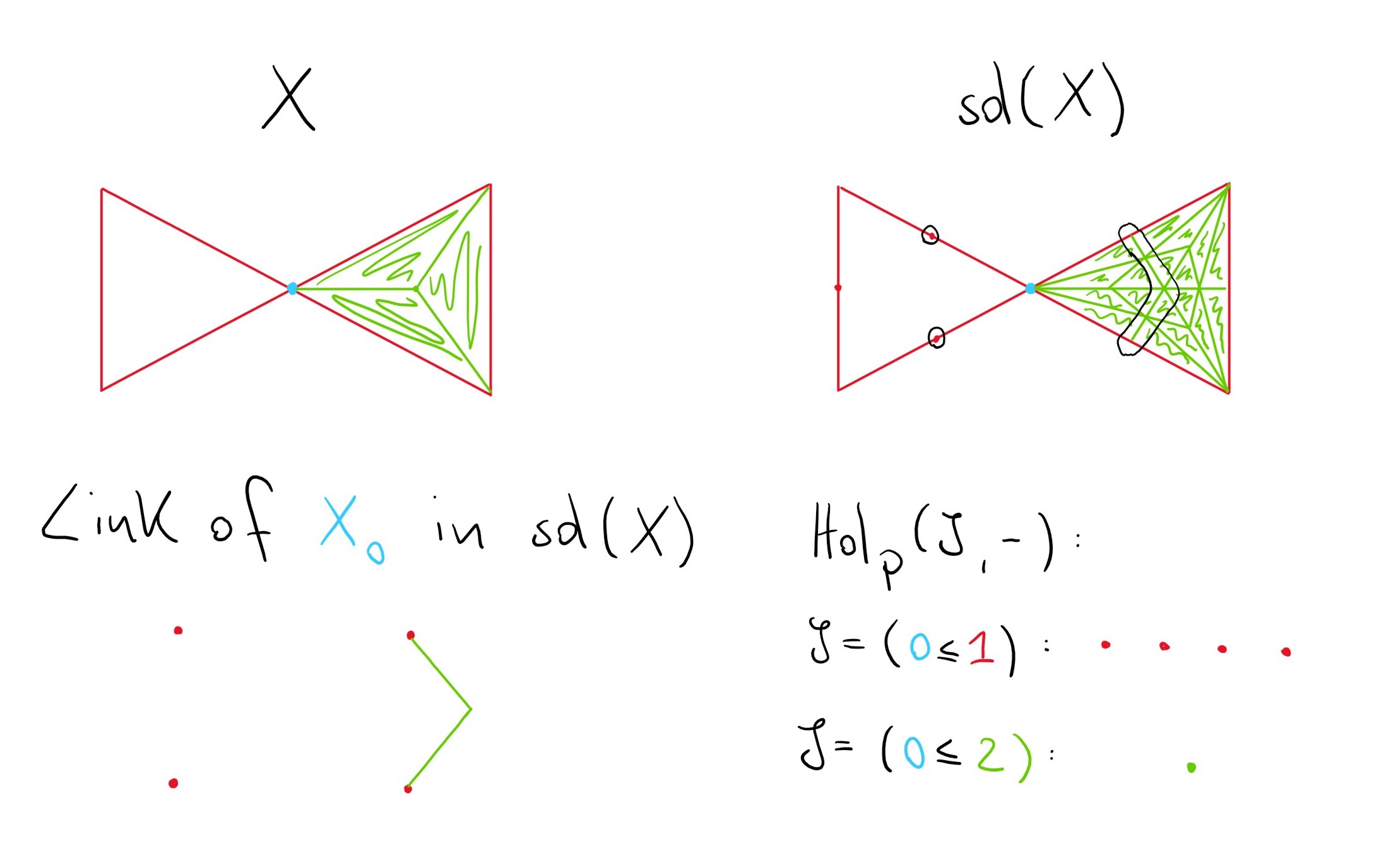}
		\caption{An illustration of the link and of (the homotopy type of the) homotopy links of a PL-model of the space in \Cref{fig:exFiltSp}. This is computed using \Cref{propQuinHol}.}
		\label{fig:exLinkandHol}
	\end{figure}
	For general spaces however, such a mapping cylinder neighbourhood might of course not even exist. The homotopy link (in the case of $\mathcal J = (p_0 < p_1)$) gives a functorial replacement for this. It is in fact a good replacement in the following sense. 
	\begin{proposition}\label{propQuinHol}
		Let $X$ be a metric space filtered over $P= \{0,1\}$. Let $N$ be a neighbourhood of $X_0$ in $X$, such that there exists a deformation retraction $r:N \times I \to N$ of $N$ into $X_0$ which is stratum preserving up to $t=1$. Then $N$ (with the induced filtration) is filtered homotopy equivalent to the $P$-filtered space $\Big ( X_0 \subset Cyl(f) \Big )$, where $Cyl(f)$ is the mapping cylinder (with the teardrop topology, see \cite{quinn1988homotopically}) of the starting point evaluation map \begin{align*}
		f:\textnormal{Hol}_P\big( (0 \leq 1), X \big) &\longrightarrow \textnormal{Hol}_P\big((0), X \big) \cong X_0
		\end{align*}
		induced by the inclusion $(0) \hookrightarrow (0 \leq 1)$.
	\end{proposition}
	For a proof see \cite[Lem, 2.4]{quinn1988homotopically} with a correction found in \cite[A. 1]{friedman2003stratified}. In particular, the homotopy link is actually homotopy equivalent to the link in the PL sense, wherever that exists.
\end{remark}
It is an immediate consequence of \Cref{lemC0retainHo} that $\textnormal{Hol}_{P}(\mathcal J, -)$ sends stratum preserving homotopy equivalences into (weak) homotopy equivalences. 
It turns out that, for most examples of stratified spaces, the converse holds. Such a result was first stated by Miller in \cite[Thm. 6.3]{miller2013} and it has been the starting point of a wealth of investigations into stratified homotopy theory. We use the following Whitehead-Theorem style version of such a result, shown in \cite{douSimp}. First however, we have need for a few definitions in the simplicial setting. 
\begin{definition}\label{defAdmHorn}
	Let $\mathcal J = (p_0 \leq ... \leq p_n)$ be a $d$-flag in $P$ of length $n$. Consider a horn inclusion $ \Lambda_k^n \hookrightarrow \Delta^n$. Denote by $\Lambda_k^\mathcal J \subset \Delta ^\mathcal{J}$ the $P$-filtered space given by $$\Lambda_k^n \hookrightarrow \Delta^n \xrightarrow{p_{\Delta^{\mathcal J}}} N(P).$$ We call such a horn inclusion $\Lambda^\mathcal J_k \hookrightarrow \Delta^{\mathcal J}$ \textit{admissible} if $p_k= p_{k+1}$ or $p_k = p_{k-1}$; that is, if $\mathcal J$ is degenerate with the $k$-th vertex repeated. We also say that $k$ is $\mathcal J$-\textit{admissible}.
\end{definition}
\begin{example}\label{exOfHorn}
	The following illustrates a few examples of admissible and not admissible horn inclusions.
	\begin{figure}[H]
		\centering
		\includegraphics[width=120mm]{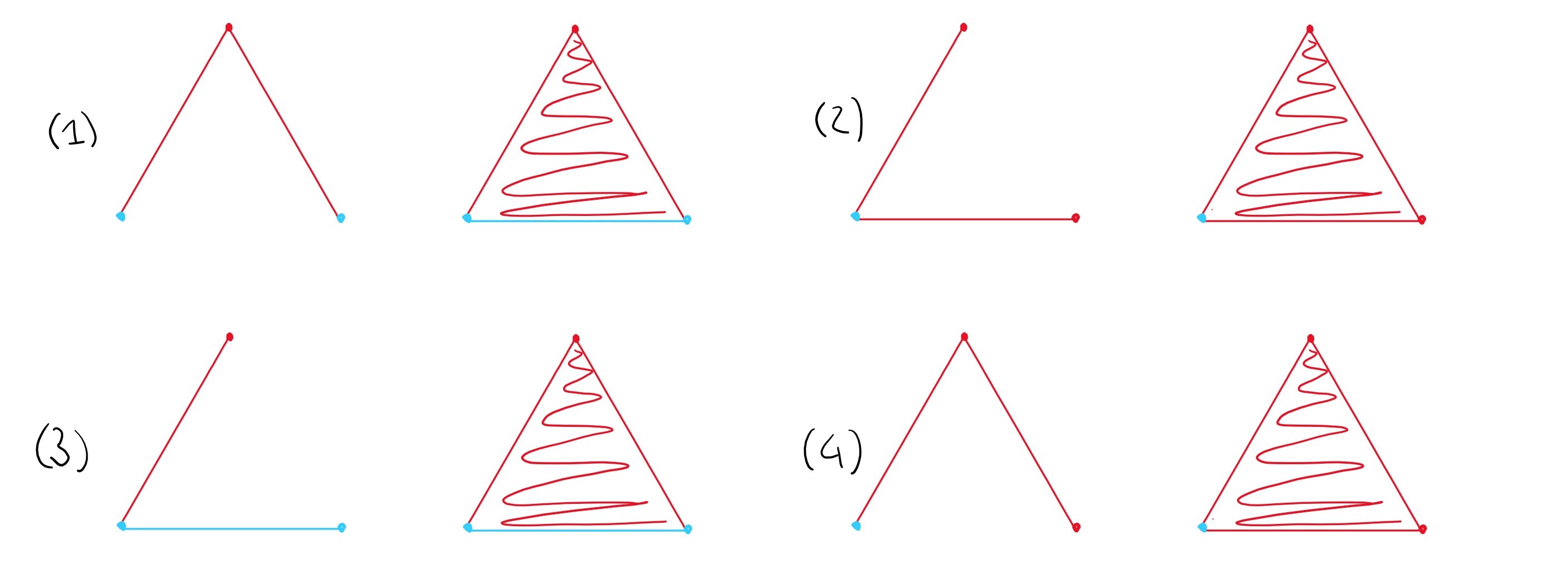}
		\caption{Examples of horn inclusions of $2$-simplices, filtered over $P ={0,1}$}
		\label{fig:exOfHornA}
	\end{figure}
	(1) and (3) show horn inclusions for $\mathcal J = (0 \leq 0 \leq 1)$ and $k= 2$ and $0$ respectively. (2) and (4) show horn inclusions for $\mathcal J = (0 \leq 1 \leq 1)$ and $k= 0$ and $2$ respectively. Out of these, only (3) and (4) are admissible.
\end{example}
Later on, in our investigation of stratified simple homotopy theory, admissible horn inclusions will be the elementary building blocks for simple equivalences. For now, the following homotopical characterization should make the definition of admissible horn inclusion a little more motivated.
\begin{lemma}\label{lemCharHornInc}
	A horn inclusion $i_k:\Delta^\mathcal J_k \hookrightarrow \Delta^\mathcal J$ is admissible if and only if $|i_k|_{P}$ is a stratum preserving homotopy equivalence.
\end{lemma}
\begin{proof} Let $\mathcal J = (p_0 \leq ... \leq p_n)$.
For the only if part, let without loss of generality $p_k = p_{k+1}$. Now, consider the deformation retraction of $|\Delta^n|$ onto its $({k+1})$-th face given by \begin{align*}
	R:I \otimes |\Delta^\mathcal J|_P &\longrightarrow |\Delta^\mathcal J|_P\\
	(s, \sum_{i \in [n]} t_i|p_i|) &\longmapsto \sum_{i \in [n]\setminus\{k,k+1\}} t_i |p_i| + (st_{k+1} + t_k)|p_k| + (1-s)t_{k+1}|p_{k+1}|. 
\end{align*}
This is easily checked to be stratum preserving (elementarily, as a quick consequence of \Cref{lemLineSeqStrat}, or by using an underlying simplicial stratum preserving homotopy as in \cite[Prop 1.13.]{douSimp}). Furthermore, it restricts to a deformation retraction of $|\Delta^{\mathcal J}_k|_P$ onto the $(k+1)$-th face. Now, for the converse, assume that $|i_k|_P: |\Lambda_k^{\mathcal J}|_P \hookrightarrow |\Delta^\mathcal{J}|_P$ is a stratum preserving homotopy equivalence and $p_k$ is not repeated in $\mathcal J$. Consider the restriction of $|i_k|_P$ to $Q=P \setminus \{p_k\} \subset P$. 
\begin{figure}[H]
	\centering
	\includegraphics[width=80mm]{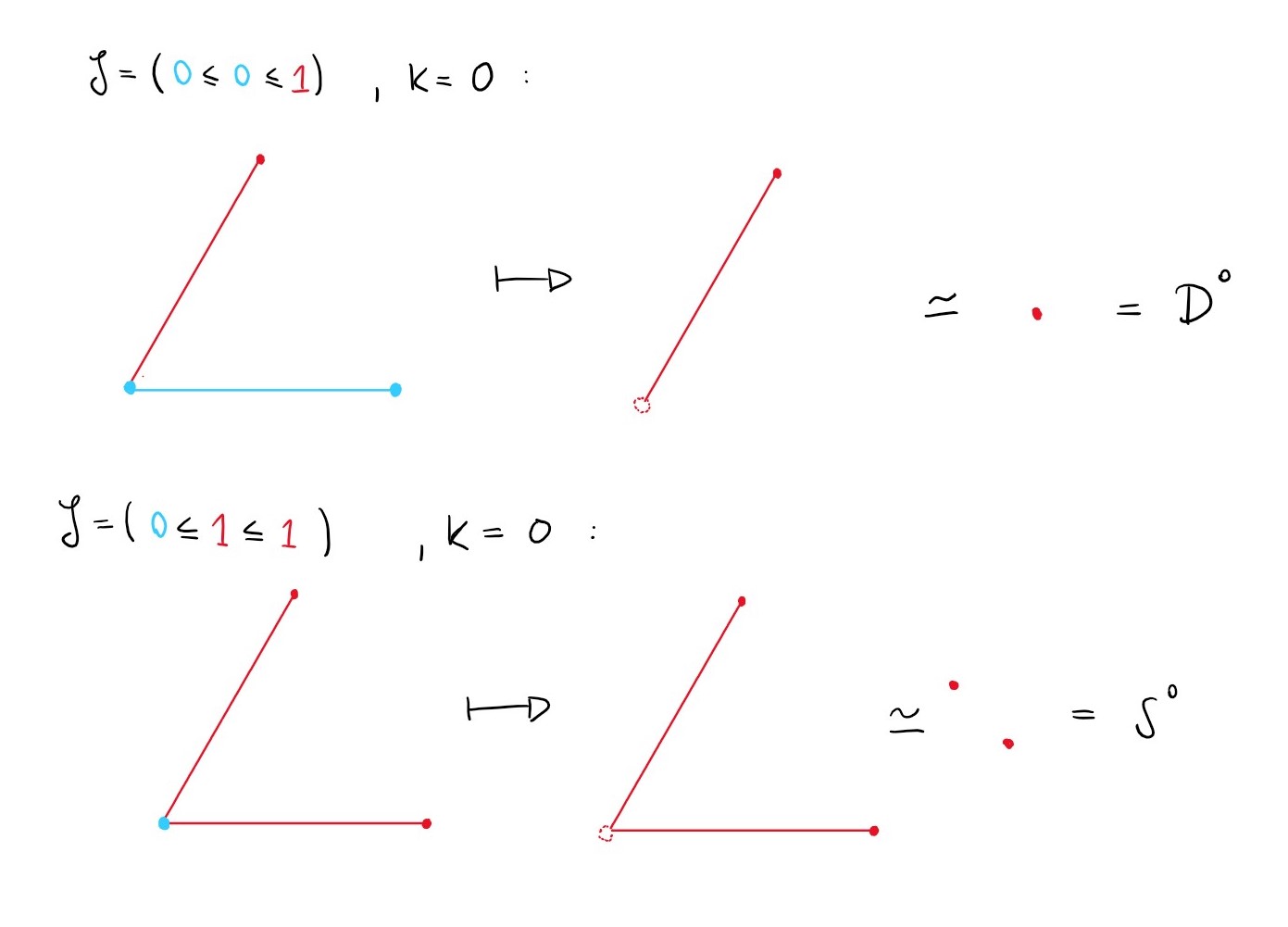}
	\caption{An illustration of $(|\Lambda^{\mathcal J}_k|)_Q$, for $k=0$ and $\mathcal J = (0 \leq 0 \leq 1)$ and $\mathcal J = (0 \leq1 \leq1 )$.}
	\label{fig:proofChaAdm}
\end{figure} This is still a stratum preserving homotopy equivalence (as the $\textnormal{\textbf{Top}}$ copower structures on $\textnormal{\textbf{Top}}_{P}$ and $\textnormal{\textbf{Top}}_{Q}$ are clearly compatible under the base change functor). In particular, it induces a homotopy equivalence, on the underlying topological spaces. We then have: 
\begin{align*}
\big(|\Delta^{\mathcal{J}}|_P \big )_Q = \{ \sum_{i \in [n]} t_i|p_i| \in |\Delta^\mathcal J|_P \mid t_i = 0 \textnormal{ for $i < k$} \implies t_k = 0 \}
\end{align*}
This set is clearly convex, hence contractible. For the horn we obtain $\big (|\Lambda^{\mathcal J}_k|_P \big)_Q = ..$.
\begin{align*}	
\Big \{ \sum_{i \in [n]} t_i|p_i| \in |\Delta^\mathcal J|_P \mid \big ( t_i = 0 \textnormal{, for $i < k$} \implies t_k = 0 \big ) \textnormal{ and } \big (t_i =0 \textnormal{ for some }i \in [n], i\neq k \big )\Big \}
\end{align*}
which contracts linearly onto the boundary of the $k$-th face via the linear homotopy between the identity and $$\sum_{i \in [n]} t_i|p_i| \longmapsto \sum_{i \in [n] \setminus \{k\}} \frac{t_i}{1-t_k}|p_i|.$$ In particular, the underlying space of $(\Lambda^{\mathcal J}_k)_Q$ is homotopic to $S^{n-2}$ (hence empty for $n =1$), in contradiction to the contractibility of the one underlying $(\Delta^\mathcal J)_Q$. 
\end{proof}
If one now takes the perspective that realizations of simplicial subsets should give cofibrations (in the sense of model categories) in $\textnormal{\textbf{Top}}_{P}$, then in light of \Cref{lemCharHornInc} the following definition should not be surprising.
\begin{definition}\label{defFStrar}
	A $P$-filtered space $X \in \textnormal{\textbf{Top}}_{P}$ is called an \textit{f-stratified} space (where the f stands for fibrant) if it has the right lifting property with respect to all realizations of admissible horns inclusions; that is, if for each $d$-flag $\mathcal J$ in $P$, and $k$ $\mathcal J$-admissible and each solid diagram in $\textnormal{\textbf{Top}}_{P}$ as below, a dashed arrow making the diagram commute exists.
	\begin{center}
		\begin{tikzcd}
			{|\Lambda^\mathcal{J}_k|_P} \arrow[d, hook] \arrow[r] &X \\
			{|\Delta^\mathcal{J}|_P} \arrow[ru, dashed] &
		\end{tikzcd}
	\end{center}
\end{definition}
\begin{remark}
	These spaces are often called fibrant spaces as their $\operatorname{Sing}_P(-)$ is fibrant with respect to a certain model structure, defined later on in \Cref{subsecPsset}. We find this nomenclature a bit confusing, as it conflicts with the model structure defined on $\textnormal{\textbf{Top}}_{P}$ we will be using. It seems to be an open question whether there exists a model structure on $\textnormal{\textbf{Top}}_{P}$ with respect to which these spaces are actually the fibrant objects. We thus decided to call them f-stratified.
\end{remark}
\begin{example}\label{exOfFstrat}
	Most examples of stratified spaces encountered in practice are f-stratified. More explicitly all homotopically stratified metric spaces with finite stratification (see \cite[Prop. 8.1.2.6.]{nand2019simplicial}) and all conically stratified spaces (see \cite[Prop. 4.212]{douSimp}) are f-stratified. Hence, the class of all f-stratified spaces also includes various of the classical definitions of stratified space, such as Whitney stratified, Thom-Mather stratified and topologically stratified spaces (see for example \cite[Rem. 5.1.0.14]{nand2019simplicial}).
\end{example}
Douteau has shown the following result. (It is formulated slightly differently there, but easily seen to be equivalent under the various inner hom adjunctions, analogously to the proof of \cite[Cor. 5.12]{douSimp}).
\begin{theorem}\cite{douSimp}[Thm. 4.23]\label{thrmWhitehead}
	Let $g: X \to Y$ be a stratum preserving map of f-stratified spaces over $P$ that are stratum preserving homeomorphic to the realizations of $P$-filtered simplicials sets. Then $g$ is a stratum preserving homotopy equivalence if and only if, for each flag $\mathcal J$ of $P$, $\textnormal{Hol}_P(\mathcal J, g)$ is a weak homotopy equivalence. 
\end{theorem}
\begin{remark}
	If one strengthens the requirements of \Cref{thrmWhitehead} to $X$ and $Y$ being conically stratified, then it in fact suffices to check for flags of length $\leq 2$, i.e. for homotopy links and strata. (See \cite[Cor. 5.12]{douSimp}). 
\end{remark}
The approach to stratified homotopy theory we follow here is essentially to take the equivalent characterization of stratum preserving homotopy equivalences in \Cref{thrmWhitehead} for a definition. For most stratified spaces classically encountered, the new notion of equivalence just coincides with stratum preserving (stratified) homotopy equivalences. However, for filtered spaces farther away from what is typically called stratified, this provides extra degrees of freedom allowing for the construction of a combinatorial simple homotopy theory in \Cref{subsecEckSiebAppHolds} (see also \Cref{subsecElemExp} for reasons why this notion is preferable). In fact, this notion of weak equivalence fits into a (simplicial) model structure on $\textnormal{\textbf{Top}}_{P}$. This was originally shown in \cite{douteauEnTop}, and is formulated in the form we are going to use in \cite[1.3.7]{haine2018homotopy}. Essentially what one does is post-compose the $\textnormal{Hol}$ functor with $\operatorname{Sing}$. One then right-transfers the projective model structure on $\textnormal{s\textbf{Set}}^{\textnormal{sd}(P)}$ over to $\textnormal{\textbf{Top}}_{P}$ along this functor (see \cite{nlab:combinatorial_model_category} for a definition of combinatorial model categories).
\begin{theorem}\label{thrmDouMod}
	There exists a combinatorial model structure on $\textnormal{\textbf{Top}}_{P}$, which is explicitly described as follows.
	\begin{itemize}
		\item $f: X \to Y$ is a weak equivalence if and only if $\textnormal{Hol}_P(\mathcal J, f)$ is a weak equivalence for each flag $\mathcal J$ in $\textnormal{sd}(P)$.
		\item $f: X \to Y$ is a fibration if and only if $\textnormal{Hol}_P(\mathcal J, f)$ is a serre fibration for each flag $\mathcal J$ in $\textnormal{sd}(P)$.
		\item The acyclic cofibrations are generated by $$ \big \{ |\Lambda_k^n| \otimes |\Delta^\mathcal J|_P \hookrightarrow |\Delta^n| \otimes |\Delta^\mathcal{J}|_P \mid \mathcal J \in \textnormal{sd}(P), n\geq 0, k \in [n] \big \}.$$
		\item The cofibrations are generated by $$ \big \{ |\partial \Delta^n| \otimes |\Delta^\mathcal J|_P \hookrightarrow |\Delta^n| \otimes |\Delta^\mathcal{J}|_P \mid \mathcal J \in \textnormal{sd}(P), n\geq 0 \big \}.$$
	\end{itemize}
\end{theorem}
In particular, by a\textit{ weak equivalence of $P$-stratified spaces} we mean a $P$-stratified map as in the first bullet point of \Cref{thrmDouMod}. We denote by $\mathcal H\textnormal{\textbf{Top}}_{P}$ the homotopy category, obtained by localizing the weak equivalences in $\textnormal{\textbf{Top}}_{P}$. We call this model structure the Henrique-Douteau model structure on $\textnormal{\textbf{Top}}_{P}$, following the nomenclature in \cite{haine2018homotopy}. The obvious question arises how the category $\mathcal H \textnormal{\textbf{Top}}_{P}$ differs from the one obtained by only localizing stratum preserving homotopy equivalences. Essentially the difference comes down to allowing certain thickenings of strata.
\begin{example}\label{exWeakNotStr}
Consider the stratum preserving map shown in \Cref{fig:ExNotStratEq}. We have already seen that it is not a stratum preserving homotopy equivalence. However, it is a weak equivalence of filtered spaces. If $X$ denotes the space on the left, and $Y$ the space on the right, then \Cref{propQuinHol} shows that (up to homotopy equivalences) the maps on homotopy links are given as in the following figure.
\begin{figure}[H]
	\centering
	\includegraphics[width=80mm]{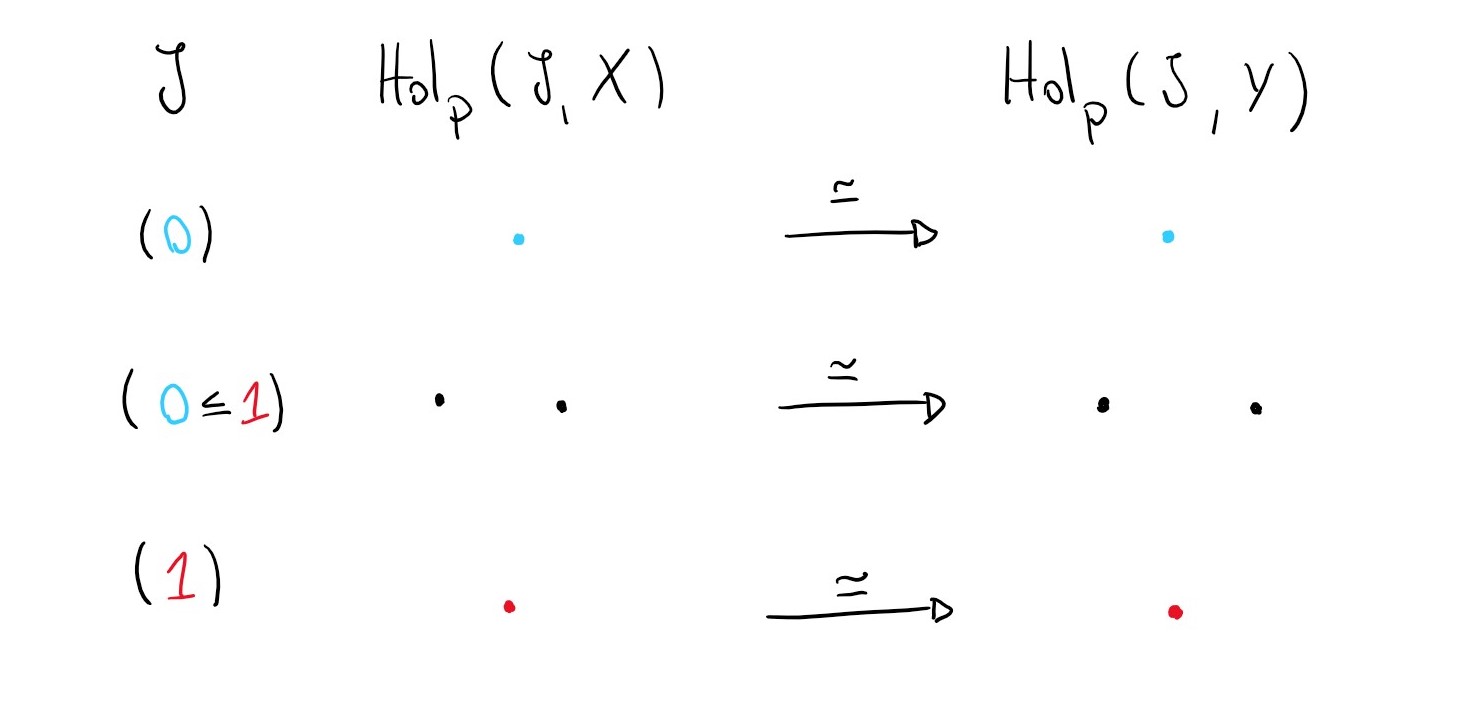}
	\caption{Illustration of the induced maps on generalized homotopy links of \Cref{fig:ExNotStratEq}, up to homotopy equivalence.}
	\label{fig:exWeakNotStr}
\end{figure}
This of course is not in contradiction to \Cref{thrmWhitehead} as $X$ is not f-stratified. 
\end{example}
This strata thickening interpretation is made more precise by \Cref{propSDPisCof} later on. However, as long as one is interested in most classical examples of stratified spaces (and willing to allow for some compactness restrictions), these two categories actually agree. We have shown this in \Cref{corCatEqu}.
\begin{remark}
The choice of calling such maps weak equivalences is not purely due to consistency with model theoretic language. In fact, one can define filtered homotopy groups, which essentially come down to being the homotopy groups of the generalized homotopy links. Then a weak equivalence is precisely one that defines isomorphisms on all of these homotopy groups (for appropriate choice of basepoints). For details see \cite{douSimp}.
\end{remark}
\begin{remark}\label{remWeirdStruct}
	The astute reader might be somewhat surprised by the choice of fibrations and cofibrations in \Cref{thrmDouMod}. Especially, after the definition of f-stratified space, one would expect the acyclic cofibrations to be generated by the realizations of admissible horn inclusions, and the fibrations to be defined accordingly. In particular, this would make the f-stratified spaces the fibrant objects and justify this nomenclature. In fact, it seems highly likely that such a model structure does exist, and can be obtained by right-transfer from one defined on the category $\textnormal{s\textbf{Set}}_P$ (see \Cref{thrmDouModSS}). We have however, so far not been able to obtain a complete proof of this transfer. For our intends and purposes, that is, for the formulation of a simple homotopy theory, the model structure in this form, will suffice. Nevertheless, we were able to obtain many of the results such a transfer of model structure would entail through other means. In particular, this applies to our results \Cref{thrmWeakEquSustain}, \Cref{thrmFullyFaithful} and \Cref{thrmHoClassofFSpace}.
\end{remark}
Finally, we need the fact that the model structure on $\textnormal{\textbf{Top}}_{P}$ is actually a simplicial one (see \cite[Ch.2, Sec. 2]{goerss2012simplicial} for a definition). The additional simplicial structure, is constructed as follows.
\begin{definitionconstruction}
	The enrichment and copower structure of $\textnormal{\textbf{Top}}_{P}$ over $\textnormal{\textbf{Top}}$ induces one over $\textnormal{s\textbf{Set}}$ as follows. Let $X,Y \in \textnormal{\textbf{Top}}_{P}$ and $S \in \textnormal{s\textbf{Set}}$.
	\begin{itemize}
		\item The copowering is given by $$S \otimes X = |S| \otimes X.$$
		\item The simplicial enrichment is then induced, by defining 
\begin{align*}
	Map(X,Y): \Delta^{op} \hookrightarrow \textnormal{s\textbf{Set}} \xrightarrow{- \otimes X} \textnormal{\textbf{Top}}_{P} \xrightarrow{\textnormal{Hom}_{\textnormal{\textbf{Top}}_{P}}(-,Y)} \textnormal{\textbf{Set}};
\end{align*}
that is, $Map(X,Y)([n]) = \textnormal{Hom}_{\textnormal{\textbf{Top}}_P}(\Delta^n \otimes X, Y)$.
\item Finally, the power structure $X^S$ is induced by the inclusion 
	$$\bigsqcup_{p \in P} \textnormal{Hom}_{\textnormal{\textbf{Top}}}(|S|, X_p) \hookrightarrow C^0(|S|, X);$$ that is, one equips the set of maps $|S| \to X$ whose image only intersects one stratum in $X$ with the ($\Delta$-generated) subspace topology, with respect to this inclusion.
	\end{itemize}
All of these are functorial in both arguments and checked to define a simplicial structure in \cite[5.1.3]{douteauFren}.
\end{definitionconstruction}
Then we have:
\begin{proposition}\cite{haine2018homotopy}\label{propModTopSim}
The model structure defined in \Cref{thrmDouMod} is combinatorial and simplicial.
\end{proposition}
This ends our introduction into the category $\textnormal{\textbf{Top}}_{P}$. 
\subsection{The category of $P$-filtered simplicial sets}\label{subsecPsset}
In some sense, the study of simple homotopy theory can be understood as asking the question: Can a specified "weak equivalence" be witnessed in a purely combinatorial way and if yes, how can we tell? In particular, to even make this question well-defined, one needs a combinatorial model for the homotopy setting that one is working in. Originally, Whitehead used simplicial complexes (\cite{whitehead1939simplicial}) as such a model. However, the category of simplicial complexes is a rather inconvenient one in many aspects. For example, pushouts, even along inclusions of simplicial complexes, will usually not agree with what one would geometrically expect them to look like (see also \Cref{remPushoutInSim}). This reflects for example in the fact that the PL-category does not seem to be equipable with a functorial mapping cylinder (see for example \cite[Remark 4.3.2]{waldhausen2000spaces}), even though non functorial ones exist. It was probably due to difficulties such as this one that Whitehead decided to move on to (and introduce) the setting of CW-complexes to formulate his simple homotopy theory later on in \cite{whitehead1950simple}. Since then, CW-complexes have been recognized as formidable objects for the study of algebraic topology. An area somewhat less well explored, however, seems to be their role in stratified topology. For example, to the author there seems to be no known stratified analogue to CW-complexes that fulfill a filtered version of the CW-approximation theorem. Furthermore, for our purposes CW-complexes have the disadvantage of only being partially combinatorial in nature as they still contain the data of continuous attaching maps.\\
\\
In the spirit of modern homotopy theory, much of the recent results about stratified homotopy theory have been formulated in the language of simplicial sets (see for example \cite{douSimp}, \cite{haine2018homotopy}, \cite{nand2019simplicial}). We are thus going to also take the approach of defining our simple stratified homotopy theory in this setting. This allows us to work in an entirely combinatorial but still categorially convenient world (appropriate for the use of model categories), fulfilling many of the advantages of simplicial and CW-complexes. Furthermore, we are going to show that in the trivially filtered case our theory agrees with the classical CW-formulation of simple homotopy theory at the end of this work (see \Cref{thrmSSvSC}). \\
\\
This subsection serves as a quick introduction to the category of simplicial sets filtered over a poset $P$, which we defined in \Cref{exFilteredObj}. For a more detailed introduction including proofs we refer to \cite{douSimp}, \cite{douteauFren} or \cite{haine2018homotopy}. We should first make the notational remark that we denote the set of $n$ simplices of a (filtered) simplicial set $X$ by $X([n])$ and not by $X_n$, as is usually done, to avoid any possible confusion with the strata. \\
\\
To begin with, it can be helpful to obtain a somewhat more explicit and geometrical description of the category of $P$-filtered simplicial sets.
\begin{remark}\label{remDesofSSP}
	First note that $N(P) \in \textnormal{s\textbf{Set}}$ is not just any simplicial set; It is one that comes from an ordered simplicial complex. We analyze this relationship in more detail in \Cref{subsecOrdered}, but for now it suffices to realize that a $n$-simplex in $N(P)$ is entirely determined by its family of vertices, that is, its images under the face maps induced by the inclusions $\Delta^0 \hookrightarrow \Delta^n$. We denote these vertices for $\sigma \in S([n])$ by $x_{i,\sigma}$, for $i \in [n]$. In particular, a simplicial map from an arbitrary simplicial set $f:S \to N(P)$ is entirely determined by $f[0]: S([0]) \to N(P)([0]) = P$. Conversely, a map on the $0$-skeletons $g: S([0]) \to P$ extends to a simplicial map $f: S \to N(P)$ if and only if for each simplex $\sigma \in S([n])$, $$g(x_{0,\sigma}) \leq ... \leq g(x_{n,\sigma}).$$ In particular, it already suffices to check this for all $2$-simplices. Hence, a $P$-filtered simplicial set can simply be thought of as a simplicial set $X$ together with a map $p_X:X([0]) \to P$, such that whenever two vertices $x_0$ and $x_1$ are connected by a $1$-simplex from $x_0$ to $x_1$ in $X$, then $p_X(x_0) \leq p_X(x_1)$. In this formulation, morphism $f: X \to Y$ is simply a map of the underlying simplicial sets such that $$p_X(x) = p_Y(f(x)),$$ for all vertices $x$ of $X$. Again, this justifies the term \textit{stratum preserving.}
\end{remark}
\begin{example}
	The following figure shows two simplicial sets whose vertices have been assigned values in $P = \{0,1\}$. However, only the first one of the two is a filtered simplicial set in the sense of \Cref{remDesofSSP}. In the second one, the upper left $1$-simplex is attached with the wrong orientation. 
	\begin{figure}[H]
		\centering
		\includegraphics[width=120mm]{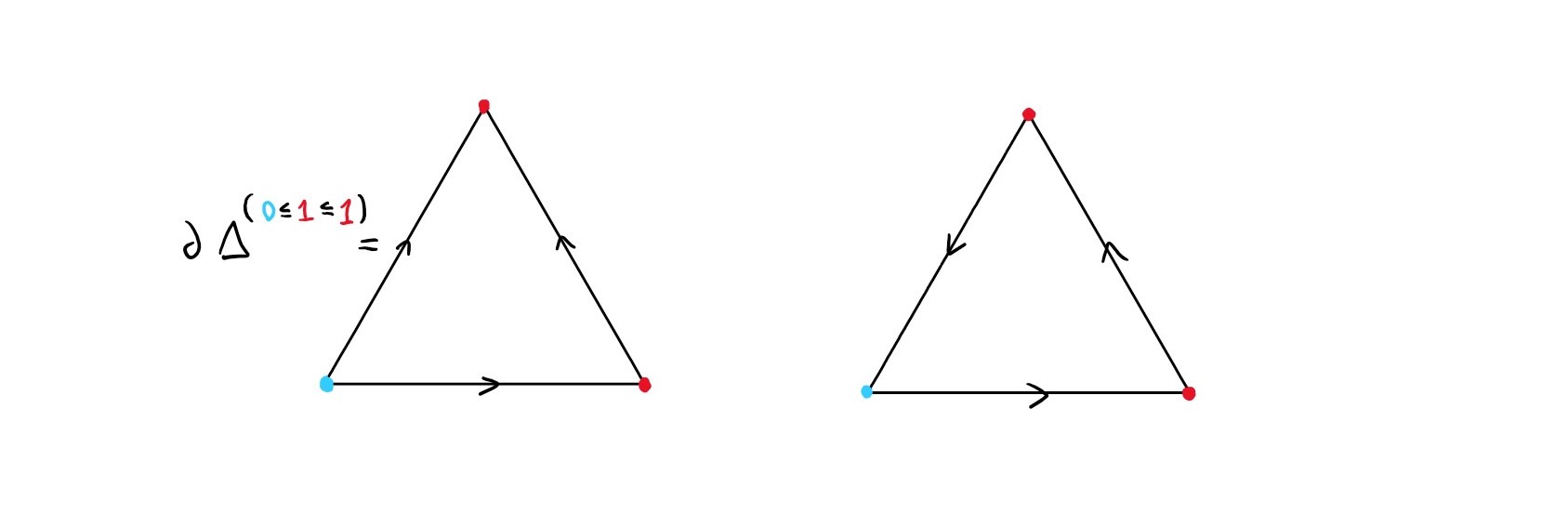}
		\caption{Examples of two simplicial sets, whose vertices have been assigned values in $P = \{0,1\}$.}
		\label{fig:exNotSimSet}
	\end{figure}
\end{example} 
While this perspective is very useful for a visual understanding and for verification that a map stratum preserving, there is another perspective that turns out to be very powerful from a model theoretic point of view. 
\begin{definitionconstruction}\label{conPreShPers}
	We denote by $\Delta(P)$, the category of simplices of $N(P)$; that is the subcategory of $\textnormal{s\textbf{Set}}_P$ of filtered simplicial sets with underlying simplicial set a standard simplex. 
	Less abstractly, by \Cref{remDesofSSP}, $\Delta(P)$ is equivalently given by the set of d-flags of $P$, with a morphism of d-flags $(p_0 \leq ... \leq p_n) \to (q_0 \leq ... \leq q_m)$ given by a map of posets $f: [n] \to [m]$ such that $p_{f(i)} = q_i$. Geometrically this is just the category of filtered simplices with stratum preserving face inclusions and degeneracy maps. Now, consider the Yoneda-Style functor 
	\begin{align*}
		\textnormal{s\textbf{Set}}_P &\hookrightarrow \textnormal{\textbf{Set}}^{\Delta(P)^{op}} \\
		X &\mapsto \big \{ \Delta(P) \hookrightarrow \textnormal{s\textbf{Set}}_P \xrightarrow{\textnormal{Hom}_{\textnormal{s\textbf{Set}}_P}(-,X)} \textnormal{\textbf{Set}} \big \},
	\end{align*}
	doing the obvious thing on morphisms. 
	This defines an equivalence of categories. In fact, the analogous statement is true for any category of presheaves, not just for $\textnormal{s\textbf{Set}}$. An inverse is explicitly given by sending $X \in \textnormal{\textbf{Set}}^{\Delta(P)^{op}}$ to $\tilde X$ defined by
	$$\tilde X([n]) = \bigsqcup_{\mathcal J \in \Delta(P), \# \mathcal J=n} X( \mathcal J),$$ with the obvious induced face and degeneracy maps and functoriality induced by functoriality of $\bigsqcup$. This is naturally filtered by just sending $\sigma \in \tilde X$ to the d-flag $\mathcal J$, denoting its component in the disjoint union. (For more details see \cite[Prop. 1.3]{douSimp})
\end{definitionconstruction}
This presheaf perspective is be particularly useful as it allows for the usage of work of Cisinski on model structures of presheaves \cite{cisinski2006prefaisceaux}.\\
\\
The category $\textnormal{s\textbf{Set}}_P$ is a simplicial category (see \cite[Ch.2, Sec. 2]{goerss2012simplicial}). The simplicial structure is constructed as follows:
\begin{definitionconstruction}\label{conSimSSP}
Let $S \in \textnormal{s\textbf{Set}}$ and $X,Y\in \textnormal{s\textbf{Set}}_P$. 
\begin{itemize}
	\item	We define their outer product $S \otimes X$, as the composition $$ S \times X \xrightarrow{\pi_X} X \xrightarrow{p_X} N(P).$$ This is clearly functorial in both arguments and associative (up to natural isomorphism) with respect to products, in both arguments. This defines the copowering. 
	\item The simplicial enrichment on $\textnormal{s\textbf{Set}}_P$ is then induced by defining 
	\begin{align*}
		Map(X,Y): \Delta^{op} \hookrightarrow \textnormal{s\textbf{Set}} \xrightarrow{- \otimes X} \textnormal{s\textbf{Set}}_P \xrightarrow{\textnormal{Hom}_{\textnormal{s\textbf{Set}}_P}(-,Y)} \textnormal{\textbf{Set}}.
	\end{align*}
	That is, $Map(X,Y)([n]) = \textnormal{Hom}_{\textnormal{s\textbf{Set}}_P}(\Delta^n \otimes X, Y)$. 
	\item Finally, the powering is given (in the presheaf perspective of \Cref{conPreShPers}) by \begin{align*}
		X^S: \Delta(P)^{op} \hookrightarrow \textnormal{s\textbf{Set}}_P \xrightarrow{S \otimes -} \textnormal{s\textbf{Set}}_P \xrightarrow{\textnormal{Hom}_{\textnormal{s\textbf{Set}}_P}(-,X)} \textnormal{\textbf{Set}}.
	\end{align*}
	That is, $$ X^S(\mathcal J) = \textnormal{Hom}_{\textnormal{s\textbf{Set}}_P}(S \otimes \Delta^\mathcal J, X) .$$
\end{itemize} 
All of these are clearly functorial, in all arguments involved, in the obvious way. They define a simplicial structure on $\textnormal{s\textbf{Set}}_P$ (see \cite[Sec. 3]{douSimp}).
\end{definitionconstruction}
In particular, we have a notion of (strictly) simplicial homotopy on $\textnormal{s\textbf{Set}}_P$ and the resulting notions of simplicial homotopy equivalence, simplicial homotopy class etc., induced by the cylinder functor given by $ \Delta^1 \otimes - $ (see \cite[Ch. 9.5]{hirschhornModel} for definitions.). We will sometimes switch up the order of arguments in this functor, in particular when we use it to define homotopies. However, it should always be clear from context which one the space inducing the filtration is.
%\begin{definition} Let $X,Y \in \textnormal{s\textbf{Set}}_P$. \begin{itemize}
%\item 	Two maps $P$-filtered simplicial maps $f_0,f_1: X \to Y$ are said to be \textit{elementary $P$-filtered homotopic} or \textit{elementary homotopic} (for short) if they are simplicially homotopic with respect to the copowering. That is if there exists an $P$-filtered simplicial map $H: \Delta^1 \otimes X \to Y$ fitting into a commutative diagram if there exists a $P$-filtered maps $H: I \otimes X \to Y$ such that the diagram
%\begin{center}
%	\begin{tikzcd}[column sep = large]			X \sqcup X \cong (\partial \Delta^1) \otimes X \arrow[r, "f_0 \sqcup f_1"] \arrow[d, hook, "i \otimes 1"]& Y\\
%		\Delta^1 \otimes X \arrow[ru, dashed, "H", swap] 
%	\end{tikzcd}
%\end{center}
%commutes. $i: \Delta^0 \cup \Delta^0 = \partial \Delta^1 \hookrightarrow \Delta^1$ denotes the boundary inclusion. 
%\item We use the standard language for such cylinder based homotopies, and speak of \textit{$P$-filtered homotopy class, $P$-filtered simplicial homotopy equivalence etc.}
%\item 
We denote by $[X,Y]_P$ the set of simplicial homotopy classes, that is the quotient of $\textnormal{Hom}_{\textnormal{s\textbf{Set}}_P}(X,Y)$ by the equivalence relation generated by elementary simplicial homotopies, also called strictly simplicial homotopies.\\
\\
%\end{itemize}
It is of course already clear for the case $P = \star$ and the classical theory simplicial homotopy theory (see for ex. \cite{goerss2012simplicial}) that strictly simplicial homotopy does not define an equivalence relation for arbitrary targets. For this to work, one needs to pass to a class of objects that fulfill certain horn filler conditions. More generally, for simplicial homotopy classes in a simplicial model category to properly describe the homotopy category, one needs the source object to be cofibrant and the target to be fibrant. This is made more precise in the following standard proposition of model category theory (see for example \cite[Prop. 5.11.]{dwyer1995homotopy} together with \cite[Prop. 9.5.24.]{hirschhornModel}).
\begin{proposition}\label{propHoClasses}
	Let $\mathcal M$ be a simplicial model category. Let $X,Y \in \mathcal M$. If $X$ is cofibrant and $Y$ is fibrant, then the natural map $$[X,Y] \to \textnormal{Hom}_{\mathcal H \mathcal M}(X,Y)$$ is a bijection.
\end{proposition}
\subsection{The Douteau model structure on $P$-filtered simplicial sets}
In \Cref{conPreShPers} we noted that $\textnormal{s\textbf{Set}}_P$ can be thought of as the category of presheaves $\textnormal{\textbf{Set}}^{\Delta(P)^{op}}$, where $\Delta(P)$ is the category of $d$-flags in $P$. In particular, this opens $\textnormal{s\textbf{Set}}_P$ up to the usage of work of Cisinski on model structures on categories of presheaves (\cite{cisinski2006prefaisceaux}). In \cite{douSimp}, Douteau uses this perspective to define a model structure on $\textnormal{s\textbf{Set}}_P$ that makes the simplicial analogue to $f$-stratified spaces the fibrant objects.
\begin{theorem}\cite[Thm. 2.14 and Thm. 3.4]{douSimp}\label{thrmDouModSS}
	With respect to the simplicial structure defined in \Cref{conSimSSP}, $\textnormal{s\textbf{Set}}_P$ can be given the structure of a combinatorial, simplicial model category such that:
	\begin{itemize}
		\item The cofibrations are the monomorphisms, that is, the inclusions of $P$-filtered simplicial sets $X \hookrightarrow Y$. They are generated by the boundary inclusions $\partial \Delta ^{\mathcal J} \hookrightarrow \Delta^{\mathcal J}$, for $d$-flags $\mathcal J$ in $P$.
		\item The acyclic cofibrations are generated by the class of admissible horn inclusions $\Lambda_k^\mathcal J \hookrightarrow \Delta^{\mathcal J}$, for $d$-flags $\mathcal J$ in $P$ and $k$ $\mathcal J$-admissible. (Furthermore, they are also the saturated class (see \Cref{defSat}) generated by these inclusions).
		\item The weak equivalences are the morphisms $f: X \to Y$ that induce bijections on simplicial homotopy classes $$f^*: [Y,Z]_P \to [X,Z]_P,$$ for all fibrant $Z$ (in the sense of the definition of the previous item). 
	\end{itemize}
\end{theorem}
We denote the induced homotopy category by $\mathcal H \textnormal{s\textbf{Set}}_P$ and call this model structure the Douteau model structure on $\textnormal{s\textbf{Set}}_P$. One should directly note the following connection to the topological setting, which is immediate from the adjunction $|-|_P \dashv \operatorname{Sing}_P$.
\begin{lemma}\cite{douSimp}
$X \in \textnormal{\textbf{Top}}_{P}$ is a f-stratified space if and only if $\operatorname{Sing}_P(X) \in \textnormal{s\textbf{Set}}_P$ is fibrant, with respect to the Douteau model structure.
\end{lemma}
While the definition of fibration is clearly motivated by the topological setting (see \Cref{defFStrar}), this characterization of weak equivalence is of course completely abstract. In fact, the characterization given here is simply the one holding in any simplicial, left proper model category, and is completely unspecific (see \cite[Thm 7.8.6.]{hirschhornModel} together with \cite[Prop. 9.5.24.]{hirschhornModel}). If one wants to get a better understanding of the weak equivalences and the filtered homotopy category, it can be very helpful to have an explicit description of fibrant replacement functors. Note that in the case where $P = \star$ this is of course already well understood as then the model structure described in \Cref{thrmDouModSS} is just the classical Kan-Quillen model structure on $\textnormal{s\textbf{Set}}_\star = \textnormal{s\textbf{Set}}_P$. Douteau explicitly constructs such a fibrant replacement functor in \cite{douSimp}. We try to motivate its description a little bit here. Recall first the classical setting of the Kan-Quillen model structure on the category of simplicial sets. One obtains fibrant replacements as follows (see \cite[Ch III. Sec. 4]{goerss2012simplicial}, for a detailed analysis of $\operatorname{Ex}^\infty$).
\begin{remark}\label{remClaFRep}
	Fibrant replacements in the category of simplicial sets, with respect to the Kan-Quillen model structure, are in a sense "geometric up to adjoint". To be more precise, the situation is the following. The subdivision functor $\textnormal{s\textbf{Set}} \to \textnormal{s\textbf{Set}}$ admits a (simplicial) right adjoint $\operatorname{Ex}(-)$. Explicitly it is defined by the composition
	$$\operatorname{Ex}(X): \Delta^{op} \hookrightarrow \textnormal{s\textbf{Set}} \xrightarrow{\textnormal{sd}} \textnormal{s\textbf{Set}}\xrightarrow{\textnormal{Hom}_{\textnormal{s\textbf{Set}}}(-,X)} \textnormal{\textbf{Set}},$$
	 i.e. $\operatorname{Ex}(X)([n]) = \textnormal{Hom}_{\textnormal{s\textbf{Set}}}(\textnormal{sd}(\Delta^n), X)$, with functoriality induced in the obvious way. Pulling simplices back with the last vertex map, induces a monomorphism $X \hookrightarrow \operatorname{Ex}(X)$. A fibrant replacement in the Kan-Quillen model structure is then given by the structure map into the following limit.
	 \begin{align*}
	 	1_{\textnormal{s\textbf{Set}}} \hookrightarrow \operatorname{Ex} \hookrightarrow \operatorname{Ex}^2 \hookrightarrow ... \hookrightarrow \varprojlim \operatorname{Ex}^n =: \operatorname{Ex}^\infty
	 \end{align*}
	 Geometrically, we should think of $\operatorname{Ex}^\infty $ as being the right adjoint to "infinite barycentric subdivision". As every object in the Kan-Quillen model structure is cofibrant, we then have $$\textnormal{Hom}_{\mathcal H\textnormal{s\textbf{Set}}}(X,Y) \cong [X, \operatorname{Ex}^\infty(Y)],$$ induced by the inclusion $Y \hookrightarrow \operatorname{Ex}^\infty(Y)$, for $X,Y \in \textnormal{s\textbf{Set}}$. In the case where $X$ is finite however, a compactness argument shows that every arrow in $\mathcal H \textnormal{s\textbf{Set}}$ is already witnessed by some finite $n$, and the same holds for every identification of arrows happening when one passes to the homotopy category. Using the induced adjunction $\textnormal{sd}^n \dashv \operatorname{Ex}^n$, this essentially means that every homotopy class as well as every homotopy is witnessed by some degree of barycentric subdivision. (We do this argument in detail in the filtered setting in \Cref{subsecRepRes}). Hence, from this perspective (i.e. the finite one), using that $[X,Y] \cong [|X|,|Y|]$, we may understand the fact that $Y \hookrightarrow \operatorname{Ex}^{\infty}(Y)$ gives a fibrant replacement as a simplicial approximation theorem. This is also the perspective we take for the proof of one of our main results (\Cref{thrmFullyFaithful}).
\end{remark}
Douteau takes the analogous approach for his fibrant replacement. 
\begin{definitionconstruction}\label{conSDp}
Consider the subdivision functor $\textnormal{sd}: \textnormal{s\textbf{Set}}_P \to \textnormal{s\textbf{Set}}_P$ induced by subdivision on $\textnormal{s\textbf{Set}}$, as in \Cref{exFiltFun}. This is not be the subdivision functor we use. Instead, define $\textnormal{sd}_P(N(P))$ as the $P$-filtered simplicial subset of $\textnormal{sd}(N(P)) \otimes N(P)$ (filtered via the second component) defined by
 $$
 \textnormal{sd}_P(N(P))([k]) := 
 \Big \{ \big ( (\sigma_0 \subset ... \subset \sigma_k ), \tau \big ) \mid 
 \hat \tau \subset \hat \sigma_0 \Big \} 
 \subset \big (\textnormal{sd}(N(P)) \otimes N(P) \big )([k]),
 $$ 
 where by the hat we denote the unique non-degenerate flags that the $d$-flags $\sigma$ and $\tau$ degenerate from. Note that this actually defines a simplicial set, i.e. that the defining property for simplices is closed under face and degeneracy maps of $\textnormal{sd}(N(P)) \otimes N(P)$. The projection to the first component, defines a (non stratum preserving) simplicial map $\textnormal{sd}_P(N(P)) \to \textnormal{sd}(N(P))$. Now, 
for some $X \in \textnormal{s\textbf{Set}}_P$ consider the following pullback diagram.
\begin{center}
	\begin{tikzcd}
		X' \arrow[d] \arrow[r] & \textnormal{sd}(X) \arrow[d,"\textnormal{sd}(p_X)"] \\
		{\textnormal{sd}_P(N(P))} \arrow[r] &{\textnormal{sd}(N(P))}
	\end{tikzcd}
\end{center}
$X'$ is naturally filtered via the composition $X' \to \textnormal{sd}_P(N(P)) \xrightarrow{p_{\textnormal{sd}_P (N(P)) } }N(P)$. 
We set $\textnormal{sd}_P(X)$ to this filtered space. This defines a functor
 $$\textnormal{sd}_P: \textnormal{s\textbf{Set}}_P \longrightarrow \textnormal{s\textbf{Set}}_P,$$ through functoriality of all the constructions involved. We call this functor the \textit{($P$-)filtered subdivision functor.} The subscript $P$ is there to distinguish it from the subdivision defined in \Cref{exFiltFun}. To the latter we refer to as the \textit{naive subdivision functor.}
 For details see \cite[Subsec. 2.1]{douSimp}. 
\end{definitionconstruction}
\begin{lemma}\label{lemSDPColim}
	$\textnormal{sd}_P$ preserves all small colimits. 
\end{lemma}
\begin{proof}
 Later on, it will alternatively follow directly from the fact that $\textnormal{sd}_P$ admits a right adjoint. Alternatively, it follow from the fact that $\textnormal{sd}$ preserves colimits, together with the fact that base change in $\textbf{Set}$ and hence in any category of presheaves does.
\end{proof}
\begin{example}\label{exSDPforSim}
	As a consequence of \Cref{lemSDPColim} it really suffices to understand what $\textnormal{sd}_P$ looks like when it is applied to filtered simplices $\Delta^\mathcal J$. In other words, we could have alternatively constructed $\textnormal{sd}_P$ by left Kan-extension of a functor defined only for filtered simplices. For these it can explicitly computed as:
	$$\textnormal{sd}_P(\Delta^\mathcal J)[k] \cong \Big \{ \big ((\mathcal \sigma_0, p_0) \leq ... \leq (\mathcal \sigma_k, p_k) \big )\mid p_i \in \sigma_0 \textnormal{, }\forall i \Big \} \subset {\textnormal{sd}(\Delta^\mathcal J \otimes N(P))},$$ where $\leq$ is with respect to the product order induced by $\subset$ and $\leq_P$. The face and degeneracy maps are induced by leaving out and repeating entries (\cite[Subsec. 2.1]{douSimp}). Below, we have illustrated two such examples of subdivisions of filtered simplices. Geometrically, this means that $\textnormal{sd}_P$ can be understood as thickening each stratum into the higher ones. This can also be seen in the following illustration of $\textnormal{sd}_P$ for two filtered simplices.
	\begin{figure}[H]
		\centering
		\includegraphics[width=120mm]{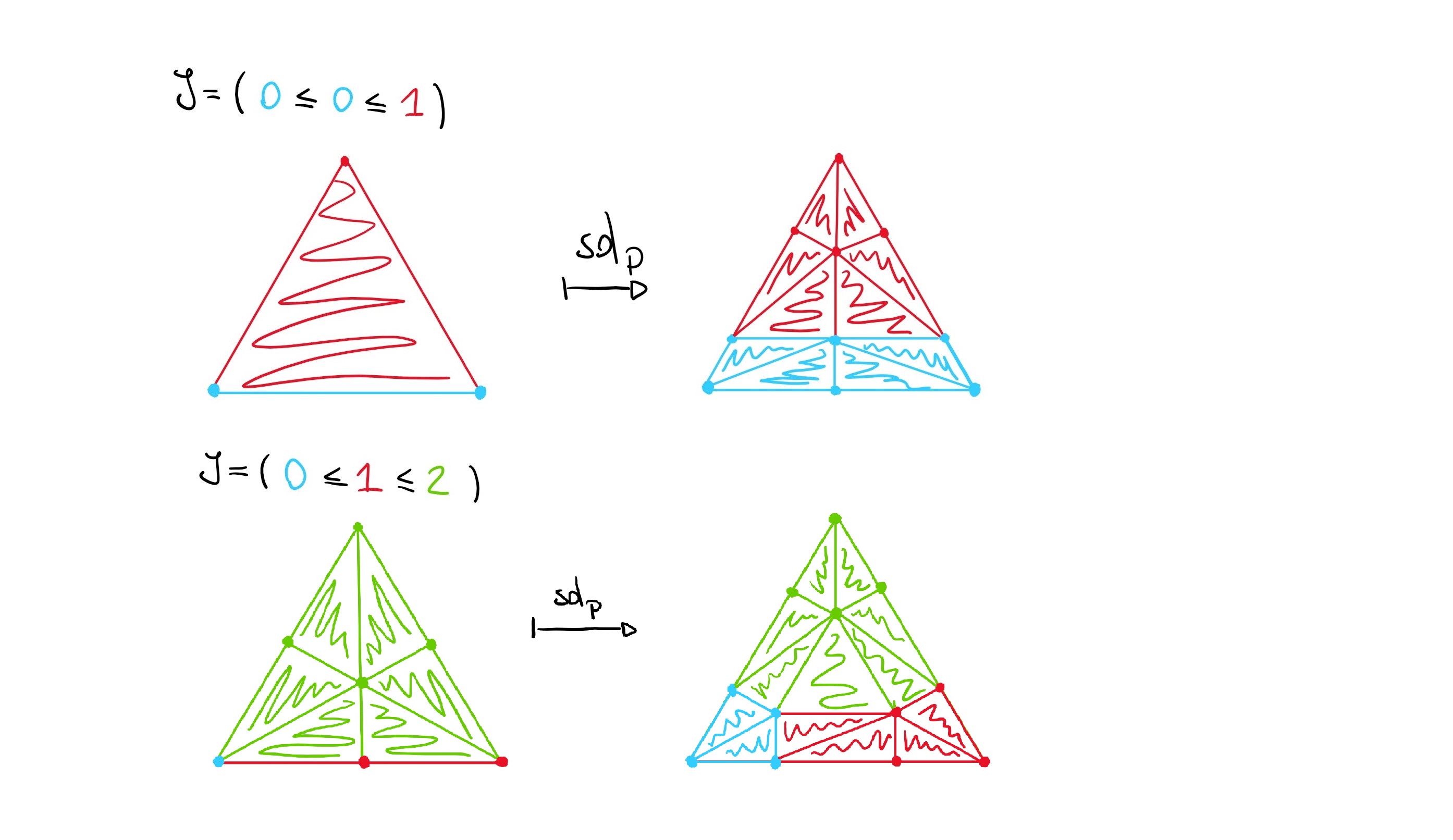}
		\caption{Illustration of $\textnormal{sd}_P(\Delta^{\mathcal J})$ for $P = \{0,1,2\}$.}
		\label{fig:exSDPforSim}
	\end{figure}
\end{example}
As a direct consequence of \Cref{lemSDPColim}, we can construct a last vertex map for $\textnormal{sd}_P$.
\begin{definitionconstruction}\label{conLVP}
By the left Kan extension perspective taken in \Cref{exSDPforSim}, to construct a last vertex map $\textnormal{sd}_P \to 1_{\textnormal{s\textbf{Set}}_P}$ it suffices to construct such a map for the filtered simplices, and show that it is compatible with face and degeneracy maps. Let $\mathcal J$ be a d-flag in $P$. By the definition of $\textnormal{sd}(\Delta^\mathcal J)$ and the fact that the underlying simplicial set of $\Delta^\mathcal J$ is $\Delta^n= N([n])$, with $n=\# \mathcal J$, it suffices to construct a map on the vertices and show that it extends. Such a map is defined by $$\big((x_0 \leq ... \leq x_i),p \big) \longmapsto \max{ \{i \in [k] \mid p_{\Delta^\mathcal J}(x_i) = p\}}.$$ I.e. instead of taking the last vertex of a simplex, one takes the last vertex lying over the correct $p$, thus making the map $P$-filtered. This, induces a natural transformation $$\textnormal{l.v.}_P: \textnormal{sd}_P \longrightarrow 1_{\textnormal{s\textbf{Set}}_P}.$$ For details see \cite[Def. 3.3.4]{douSimp}. We denote the $n$-th iteration of this functor by $\textnormal{sd}_P^n$. Further, we denote by $\textnormal{l.v.}_P^n$ the natural transformation $\textnormal{sd}^n_P \to 1_{\textnormal{s\textbf{Set}}_P}$ induced by $n$-times composition of $\textnormal{l.v.}_P$.
\end{definitionconstruction}
Then, just as in the classical setting, one constructs a right adjoint to $sd_P$. \begin{definitionconstruction}\cite[Def 2.4]{douSimp}
	Let $\operatorname{Ex}_P: \textnormal{s\textbf{Set}}_P \to \textnormal{s\textbf{Set}}_P$ be the functor defined just as in \Cref{remClaFRep}, replacing $\Delta$ by $\Delta(P)$, $\textnormal{s\textbf{Set}}$ by $\textnormal{s\textbf{Set}}_P$ (thought of the category of presheaves over $\Delta(P)$) and $\textnormal{sd}$ by $\textnormal{sd}_P$. 
	This defines a right adjoint to $\textnormal{sd}_P$. As in \Cref{remClaFRep}, through pulling back simplices with $\textnormal{l.v.}_P$ one obtains a natural inclusion $1_{\textnormal{s\textbf{Set}}_P} \hookrightarrow \operatorname{Ex}_P(-)$. 
	Passing to the limit we obtain a functor $\operatorname{Ex}^\infty_P: \textnormal{s\textbf{Set}}_P \hookrightarrow \operatorname{Ex}^\infty_P$ together with a natural inclusion $1_{\textnormal{s\textbf{Set}}_P} \hookrightarrow \operatorname{Ex}^\infty_P$. 
\end{definitionconstruction}
This construction defines a fibrant replacement functor in the Douteau model structure on $\textnormal{s\textbf{Set}}_P$.
\begin{proposition}\cite[Lem. 3.3.19 + App. B]{douSimp}
	The natural transformation $1_{\textnormal{s\textbf{Set}}_P} \hookrightarrow \operatorname{Ex}^\infty_P$ defines a cofibrant fibrant replacement (approximation). 
\end{proposition}
We will be making extensive use of this somewhat explicit description of fibrant replacements. We start this process in \Cref{subsecRepRes}.
One should probably lose a few words on why $\textnormal{sd}_P$ is the "correct" choice of subdivision here, and not the more straightforward $\textnormal{sd}$. First, note the following result, which probably makes the classical subdivision seem like a quite intuitive choice at first. It is of course well known in the non filtered setting. It uses \Cref{propBarSd}, which essentially states that $\textnormal{sd}$ gives a subdivision, in the sense of simplicial complexes, in a filtered way. However, as we only use this for heuristic purposes up to \Cref{subsecOrdered}, there is no risk of circularity here.
\begin{proposition}\label{propHomoFilteredSubd}
	For $X \in \textnormal{s\textbf{Set}}_P$ there is an isomorphism $$|\textnormal{sd}(X)|_P \cong |X|_P.$$Furthermore, this isomorphism is stratum preserving homotopic to the realization of $\textnormal{l.v.}:\textnormal{sd} \to 1_{\textnormal{s\textbf{Set}}_P}$. If $X$ is a filtered ordered simplicial complex (see \Cref{defFOSaSS}), it can be taken to be defined by the barycentric subdivision isomorphism (see \cref{propBarSd}).
\end{proposition}
\begin{proof}
	First note, that in the case of a general filtered simplicial set the respective isomorphism is not induced by barycentric subdivision, as the latter is not natural with respect to degeneracy maps. In fact, the homeomorphisms above can only be made natural in the homotopy category.
	The proof is pretty much identical to the (non-trivial) proof found for example in \cite[Thrm. 4.6.4]{fritsch1990cellular} if one checks that all maps and homotopies involved are stratum preserving. However, as all maps there are defined on a simplex level and all homotopies are given by straight line homotopies this is readily verified.
%	So we give a brief sketch. First note that as all functors involved commute with colimits, it really suffices to construct the isomorphism (compatibly with face and degeneracy maps) and a (compatible) homotopy for simplices. It is easy to check that the barycenter map from \Cref{propBarSd} then gives such an isomorphism, compatible with face and degeneracy maps. For the homotopy, just use the straight line homotopy. The only new thing to check in the filtered setting is that this is stratum preserving. As the strata of a filtered simplex are always convex (verify this elementary or use \Cref{lemLineSeqStrat}), this gives a stratum preserving homotopy. 
\end{proof}
\begin{remark}\label{remSDCor}
Our goal is to use the model structure on $\textnormal{s\textbf{Set}}_P$ for a study of $\mathcal H \textnormal{\textbf{Top}}_{P}$. If, motivated by the filtered Whitehead theorem (\Cref{thrmWhitehead}), we agree that the notion of fibrancy given by \Cref{thrmDouModSS} is the correct one, then the correct notion of subdivision should be one that allows an adjoint inducing a fibrant replacement as in \Cref{remClaFRep}. We will see in a bit that this works out for $\textnormal{sd}_P$. $\textnormal{sd}$ however, while it does admit a right adjoint, constructed just as in \Cref{remClaFRep}, does not fulfill this criterion (see \cite[Rem. 2.5]{douSimp}). Another, very heuristic way to think about this is the following. In \Cref{remClaFRep}, we have taken the perspective that, at least in a finite setting, the subdivision functor should witness all morphisms between realizations in the topological homotopy category. We think of this as the filtered simplicial approximation theorem. That the analogous argument holds for $\textnormal{sd}_P$ is a consequence of our result \Cref{thrmFullyFaithful} together with the results in \Cref{subsecRepRes}. $\textnormal{sd}$ however, can not witness all morphisms in $\mathcal H\textnormal{\textbf{Top}}_{P}$. By \Cref{propHomoFilteredSubd}, $\textnormal{sd}$ witnesses at most the morphisms in the category obtained by taking stratified homotopy classes in $\textnormal{\textbf{Top}}_{P}$. That it witnesses all of these (between realizations of filtered simplicial complexes) is precisely the content of \Cref{thrmSimplicialApproximationB}. The category obtained in this fashion, however, is not $\mathcal H\textnormal{\textbf{Top}}_{P}$. We have already seen that there are countless examples of weak homotopy equivalences in $\textnormal{\textbf{Top}}_{P}$ that are not stratum preserving homotopy equivalences (see \Cref{exWeakNotStr}). The thickenings of strata, described in \Cref{exSDPforSim}, allow precisely for the additional degree of freedom that is obtained when localizing all weak equivalences of $P$-filtered topological spaces.
\end{remark}
The fact that $\textnormal{sd}_P$ witnesses the passage to $\mathcal H\textnormal{s\textbf{Set}}_P$ in the sense of \Cref{remSDCor} is also reflected in the following statement which is integral to the proof of \Cref{thrmFullyFaithful}.
\begin{proposition}\cite[Prop. 8.1.1]{douteauFren}\label{propSDPisCof}
Let $Y \in \textnormal{s\textbf{Set}}_P$. Then $|\textnormal{sd}_P(Y)|_P \in \textnormal{\textbf{Top}}_{P}$ is a cofibrant object in the Henrique-Douteau model structure. 
\end{proposition}
\begin{proof}
The content of \cite[Prop. 8.1.1]{douteauFren} actually is a slightly stronger statement, using the strongly filtered category $\textnormal{\textbf{Top}}_{N(P)}$ (see \Cref{subsecStrongFilt}) but the forgetful functor $\textnormal{\textbf{Top}}_{N(P)} \to \textnormal{\textbf{Top}}_{P}$ is a left Quillen functor, and hence preserves cofibrant objects (\cite[Subsec 2.4]{douteauEnTop}).
\end{proof}
As a consequence of this, together with the fact that every object in $\textnormal{\textbf{Top}}_{P}$ is fibrant and \Cref{propHoClasses}, we obtain that $\textnormal{sd}_P$ sends morphisms that realize to weak equivalences, to such morphisms that realize to stratum preserving homotopy equivalences. 
\begin{example}\label{exSdWtoS}
	Consider the following simplicial set model for the weak equivalence in \Cref{exWeakNotStr}. It is not hard to see that the realization of $\textnormal{sd}_P$ applied to it is actually a stratum preserving homotopy equivalence.
	 \begin{figure}[H]
		\centering
		\includegraphics[width=120mm]{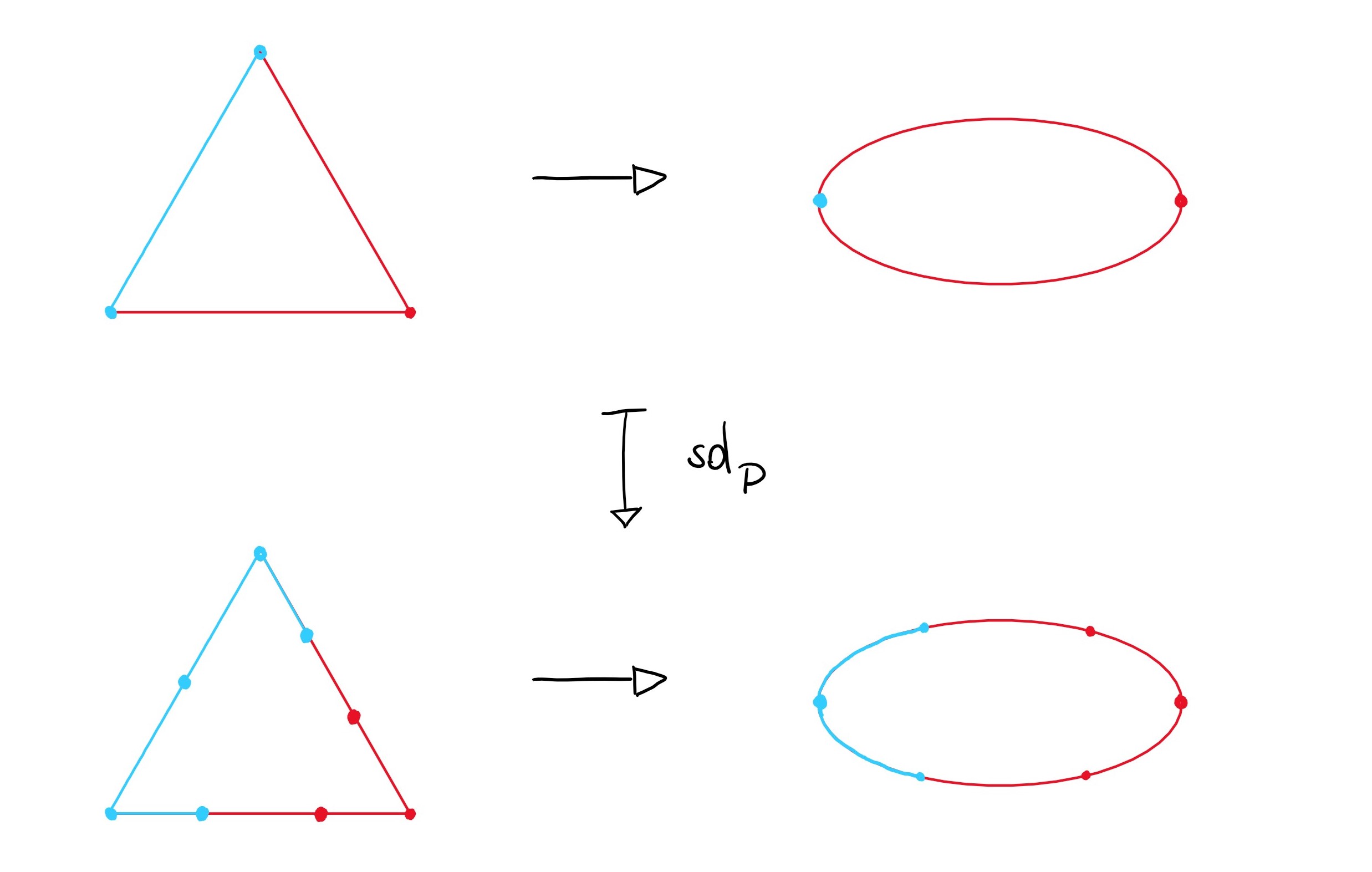}
		\caption{Illustration of a simplicial model of the map in \Cref{fig:ExNotStratEq}. The first map is given by collapsing the simplex in the $0$-stratum (in blue) to a point.}
		\label{fig:exSdWtoS}
	\end{figure}
\end{example}
This finishes our summary of the work of Douteau published in \cite{douSimp}, \cite{douteauEnTop} and \cite{douteauFren}. We finalize this section, by making a little bit more precise the heuristic in \Cref{remSDCor} that $\textnormal{sd}_P$ witnesses things happening in the homotopy category $\mathcal H \textnormal{s\textbf{Set}}_P$, at least as long as one restricts to simplicial sets with finitely many non-degenerate simplices.
\subsection{A representation result for homotopy classes}\label{subsecRepRes}
So far, morphisms in $\mathcal{H}\textnormal{s\textbf{Set}}_P$ are rather abstract things, at least within a finite framework. That is due to the fact that fibrant replacements in $\textnormal{s\textbf{Set}}_P$ are usually highly complicated infinite filtered simplicial sets, even if $X$ is finite. In fact, $\textnormal{Ex}^{\infty}_P(X)$ is clearly infinite whenever $X$ is not a point or the empty simplicial set. However, if we restrict to filtered simplicial sets with finitely many non-degenerate simplices, we can use the adjunction $\textnormal{sd}_P \dashv \operatorname{Ex}_P$ to give a more explicit description of morphisms in the homotopy category.
\begin{definition}
	We call a $P$-filtered simplicial set $X$ \textit{finite} if the underlying simplicial set has only finitely many non-degenerate vertices. We denote by $\textnormal{s\textbf{Set}}_P^{fin}$ and $\mathcal H \textnormal{s\textbf{Set}}_P^{fin}$ the respective full subcategories of $\textnormal{s\textbf{Set}}_P$ and $\mathcal H \textnormal{s\textbf{Set}}_P$ given by taking finite filtered simplicial sets as objects.
\end{definition}
\begin{lemma}\label{lemCharComp}
	The finite $P$-filtered simplicial sets are precisely the compact (small) objects in $\textnormal{s\textbf{Set}}_P$ (see \cite{nlab:compact_object} for a definition).
\end{lemma}
\begin{proof}
	First off, as colimits are computed on the underlying simplicial sets (\Cref{remLimandColim}) one easily reduces this statement to the respective one about simplicial sets. Now, let $S \in \textnormal{s\textbf{Set}}_P$ be compact. Then, consider the filtered diagram given by the finite simplicial subsets of $S$ and inclusions of the latter, $\mathcal S$. We have $S = \varinjlim \mathcal S$ as every simplex is contained in some finite sub simplicial set. By the definition of compactness, the identity $1_S: S \to S$ then factors through some finite $S_i$ in the diagram. Hence, $S$ is a sub simplicial set of a finite simplicial set and hence finite. Conversely, let $S$ be a finite simplicial set, and $\mathcal S'$ be a filtered colimit diagram in $\textnormal{s\textbf{Set}}$. We show that the natural morphism $$\varinjlim_{S'_i \in \mathcal S'}\textnormal{Hom}_{\textnormal{s\textbf{Set}}}(S,S'_i) \longrightarrow \textnormal{Hom}_{\textnormal{s\textbf{Set}}}(S, \varinjlim \mathcal S)$$ is a bijection. Denote by $\mathcal S$ the diagram given by the non-degenerate simplices of $S$. Then we have natural bijections:
	\begin{align*}
	\varinjlim_{S'_i \in \mathcal S'}\textnormal{Hom}_{\textnormal{s\textbf{Set}}}(S,S'_i) &\cong \varinjlim_{S'_i \in \mathcal S'}\textnormal{Hom}_{\textnormal{s\textbf{Set}}}(\varinjlim \mathcal S,S'_i) \\
	&\cong \varinjlim_{S'_i \in \mathcal S'} \big( \varprojlim_{\Delta^{n_j} \in \mathcal S}\textnormal{Hom}_{\textnormal{s\textbf{Set}}}(\Delta^{n_j},S'_i) \big ) \\ 
	& \cong \varprojlim_{\Delta^{n_i} \in \mathcal S} \big(\varinjlim_{S'_i \in \mathcal S'} \textnormal{Hom}_{\textnormal{s\textbf{Set}}}(\Delta^{n_j},S'_i)\big) \\
	& \cong \varprojlim_{\Delta^{n_j} \in \mathcal S} \big(\varinjlim_{S'_i \in \mathcal S'} S'_i([n_j])\big) \\
	& \cong \varprojlim_{\Delta^{n_j} \in \mathcal S} \big( (\varinjlim_{S'_i \in \mathcal S'} S'_i)([n_j])\big) \\
	&\cong \varprojlim_{\Delta^{n_j} \in \mathcal S} \big( \textnormal{Hom}_\textnormal{s\textbf{Set}}(\Delta^{n_j}, \varinjlim S')\big) \\
	& \cong \textnormal{Hom}_\textnormal{s\textbf{Set}}(S, \varinjlim \mathcal S')
	\end{align*}
	composing to the bijection we needed to verifiy. Note that while all of the other bijections involved hold for arbitrary $S$, we needed the fact that $\mathcal S$ is a finite diagram in the third bijection to commute finite limits with filtered colimits in the category $\textnormal{\textbf{Set}}$.
\end{proof}
As the objects of simple homotopy theory are classicaly finite in some combinatorial sense, these are the types of filtered simplicial sets we mostly are concerned with. In the finite setting, $\textnormal{sd}_P$ can serve as a more geometric placeholder for $\textnormal{Ex}_P^\infty$. We use this section to elaborate on this. Most of what follows now is certainly well known in the unfiltered setting, but as we needed a proof for the filtered one anyway, we are just going to go ahead and give them here.
 We start by illuminating the homotopy relationship between $\textnormal{Ex}_{P}$ and $\textnormal{sd}_{P}$ a bit. 
\begin{lemma}
	Consider the diagram in $\textnormal{s\textbf{Set}}_P$
	\begin{center}
		\begin{tikzcd}
		\textnormal{sd}_{P}(X) \arrow[r, "\textnormal{l.v.}_P"] \arrow[hook, rd] & X \arrow[d, "\eta"]\\
		& \textnormal{Ex}_{P}(\textnormal{sd}_{P}(X))
		\end{tikzcd},
	\end{center}
	where $\eta$ is the unit of the adjunction $\textnormal{sd}_{P} \dashv \textnormal{Ex}_{P}$. It commutes in $\mathcal{H}\textnormal{s\textbf{Set}}_P$.
\end{lemma}
\begin{proof}
	First, note that $\textnormal{sd}_{P}(X) \xrightarrow{\textnormal{l.v.}_P} X$ is a weak equivalence (see \cite[A.3]{douSimp}). As $ X \hookrightarrow \textnormal{Ex}_{P}(X)$ is an anodyne extension \cite[Appendix 2]{douteauFren}, $\textnormal{Ex}_{P}$ preserves weak equivalences. In particular, $$\textnormal{Ex}_{P}(\textnormal{sd}_{P}(X)) \xrightarrow{\textnormal{Ex}_{P}(\textnormal{l.v.}_P)} \textnormal{Ex}_{P}(X)$$ is a weak equivalence. Hence, it suffices to show that the diagram commutes in $\textnormal{s\textbf{Set}}_P$ after composing both $\eta$ and $\textnormal{sd}_{P}(X) \hookrightarrow \textnormal{Ex}_{P}(\textnormal{sd}_{P}(X))$ with this weak equivalence. We show this elementarily. Let $\Delta^{\mathcal J} \xrightarrow{\sigma} \textnormal{sd}_P(X)$ be a (filtered) simplex of $\textnormal{sd}(X)$. Consider the following commutative diagram
	\begin{center}
		\begin{tikzcd}
		\textnormal{sd}_{P}(\Delta^{\mathcal J}) \arrow[r,"\textnormal{sd}_{P}(\sigma)" ] \arrow[d, "\textnormal{l.v.}_P"] & \textnormal{sd}^2_{P}(X) \arrow[r, "\textnormal{sd}_{P}(\textnormal{l.v.}_P)"] \arrow[d, "\textnormal{l.v.}_P"]& \textnormal{sd}_{P}(X) \arrow[d, "\textnormal{l.v.}_P"]\\
		\Delta^{\mathcal J} \arrow[r, "\sigma"] &\textnormal{sd}_{P}(X) \arrow[r, "\textnormal{l.v.}_P"] & X
		\end{tikzcd}.
	\end{center}
	Now, the image of $\sigma$ under the composition containing $\eta$ is given by first postcomposing with the bottom right arrow, then applying $\textnormal{sd}_{P}$ (which corresponds to $\eta$), and finally postcomposing with the right vertical. Hence, by funtoriality of $\textnormal{sd}_{P}(X)$ this is given by the composition starting at the upper left, then going all the way to the right horizontally and then down. On the other hand, the image of $\sigma$ under the other composition is given by first pulling $\sigma$ back with the left vertical and then postcomposing with the bottom right horizontal. By commutativity of the diagram, this agrees with the first construction.
\end{proof}
Just as above (or by a simple inductive argument) one obtains the more general version:
\begin{lemma}\label{lemHoComSdEx}
	Let $n \geq 1$.
	Consider the diagram in $\textnormal{s\textbf{Set}}_P$
	\begin{center}
		\begin{tikzcd}
		\textnormal{sd}^n_{P}(X) \arrow[r, "\textnormal{l.v.}^n_P"] \arrow[hook, rd] & X \arrow[d, "\eta"]\\
		& \textnormal{Ex}^n_{P}(\textnormal{sd}^n_{P}(X))
		\end{tikzcd},
	\end{center}
	where $\eta$ is the unit of the adjunction $\textnormal{sd}^n_{P} \dashv \textnormal{Ex}^n_{P}$. It commutes in $\mathcal{H}\textnormal{s\textbf{Set}}_P$.
\end{lemma}
As a consequence we obtain the following.
\begin{lemma}\label{lemRepExbySd}
	For $n \in \mathbb N$, consider a diagram in $\textnormal{s\textbf{Set}}_P$ \begin{center}
		\begin{tikzcd}
		& \textnormal{sd}^n_{P}(X) \arrow[ld, "\textnormal{l.v.}^n_P", swap] \arrow[rd, "\hat f"]&\\
		X \arrow[rd ,"f", swap]& &Y \arrow[ld, hook']\\
		& \textnormal{Ex}^n_{P}(Y)&
		\end{tikzcd}
	\end{center}
	where $\hat f$ is the adjoint to $f$ under $\textnormal{sd}^n_{P} \dashv \textnormal{Ex}^n_{P}$. Then this diagram commutes in $\mathcal{H}\textnormal{s\textbf{Set}}_P$.
\end{lemma}
\begin{proof}
	By the unit counit form of an adjunction, $f$ is given by $$ X \xrightarrow{\eta} \textnormal{Ex}^n_{P}(\textnormal{sd}^n_{P}(X)) \xrightarrow{\textnormal{Ex}^n_{P}(\hat f)} \textnormal{Ex}^n_{P}(Y).$$ Hence, the diagram in the statement refines to a diagram:
	\begin{center}
		\begin{tikzcd}[column sep= huge, row sep= huge]
		& \textnormal{sd}^n_{P}(X) \arrow[ld, "\textnormal{l.v.}^n_P", swap] \arrow[rd, "\hat f"] \arrow[d, hook]&\\
		X \arrow[rd ,"f", swap] \arrow[r,"\eta"] & \textnormal{Ex}^n_{P}(\textnormal{sd}^n_{P}(X)) \arrow[d,"\textnormal{Ex}^n_{P}(\hat f)"] &Y \arrow[ld, hook']\\
		& \textnormal{Ex}^n_{P}(Y)&
		\end{tikzcd}.
	\end{center}
	The large right triangle commutes by naturality of the inclusion into $\textnormal{Ex}^n_{P}$. The lower left triangle commutes by assumption. The upper right triangle commutes in $\mathcal{H}\textnormal{s\textbf{Set}}_P$, by \Cref{lemHoComSdEx}. Hence, the outer square also commutes in $\mathcal{H}\textnormal{s\textbf{Set}}_P$. 
\end{proof}
As a direct consequence of this result, we obtain the following useful representation result for morphisms in $\mathcal H\textnormal{s\textbf{Set}}^{fin}_P$.
\begin{proposition}\label{propRepHoClasses}
	Let $X \xrightarrow{\alpha} Y$ be a morphism in $\mathcal H\textnormal{s\textbf{Set}}^{fin}_P$. Then there is an $n \in \mathbb{N}$ and a morphism $\textnormal{sd}^n_{P}(X) \xrightarrow{f} Y$ such that $\alpha$ fits into a commutative diagram
	\begin{center}
		\begin{tikzcd}
		& \textnormal{sd}^n_{P}(X) \arrow[ld, "{[\textnormal{l.v.}^n_P]}", swap] \arrow[rd, "{[f]}"]&\\
		X \arrow[rr, "\alpha"]& &Y 
		\end{tikzcd}.
	\end{center}
In particular, $$ \alpha = [f] \circ [\textnormal{l.v.}^n_P]^{-1}.$$
\end{proposition}
\begin{proof}
	As $ Y \hookrightarrow \textnormal{Ex}^{\infty}_{P}(Y)$ is a fibrant replacement of $Y$, $\alpha$ fits into a commutative diagram \begin{center}
		\begin{tikzcd}
		X \arrow[rr, "\alpha"] \arrow[rd, "{[f_0]}", swap] && Y \arrow[ld, hook]\\
		&\textnormal{Ex}^\infty_{P}(Y) &
		\end{tikzcd},
	\end{center}
for some morphisms $f_0$ in $\textnormal{s\textbf{Set}}_P$. As $X$ is finite (hence compact), and $\operatorname{Ex}^n_P(Y) \hookrightarrow \operatorname{Ex}^\infty_P(Y)$ is a weak equivalence, this diagram splits off a commutative diagram 
\begin{center}
	\begin{tikzcd}
	X \arrow[rr, "\alpha"] \arrow[rd, "{[f_1]}", swap] && Y \arrow[ld, hook]\\
	&\textnormal{Ex}^n_{P}(Y) &
	\end{tikzcd}
\end{center}
for some sufficiently large $n$ in $\mathbb N$ and some morphism $f_1: X \to \textnormal{Ex}^n_{P}(Y)$. Now, the result follows by \Cref{lemRepExbySd}. 
\end{proof}
\section{Realizations of weak equivalences}\label{secRealofWeak}
Now, that we have a notion of weak equivalence on $\textnormal{s\textbf{Set}}_P$ and one on $\textnormal{\textbf{Top}}_{P}$, the obvious question arises of what the relationship between the two is. In particular, one would like to understand the relationship with respect to the realization functor
$$|-|_P: \textnormal{s\textbf{Set}}_P \longrightarrow \textnormal{\textbf{Top}}_P.$$
In the case where $P$ is a singleton, it is of course a well known classical result that $|-|$ preserves weak equivalences and induces an equivalence of homotopy categories (where on the right hand side one uses the Quillen-Model structure \cite[Thm. 11.4]{goerss2012simplicial}). \\
\\
With the model structures given by \Cref{thrmDouModSS} and \Cref{thrmDouMod} however, $|-|_P$ is not the left adjoint of a Quillen equivalence. In fact, its right adjoint $\textit{Sing}_P$ does not even preserve fibrant objects (see \cite[Sec. 8.1]{douteauFren} for an example) and $|-|_P$ does not preserve cofibrant objects. This might of course not be due to the choice of weak equivalences on $\textnormal{\textbf{Top}}_P$, but due to the choice of fibrations (or cofibrations respectively.) So, it might very well still be true that for some other choice of model structures, inducing the same homotopy theory, the two form a Quillen equivalence (see also \Cref{remWeirdStruct}). \\
\\
We will not take the path of trying to define different model structures here. However, we are going to show that the realization functor does preserve weak equivalences. A first step to showing this result was taken in \cite[Thm. 8.3.7]{douteauFren}. There, it is shown that the realization functor to the more rigid model category of topological spaces over $|N(P)|$, $\textnormal{\textbf{Top}}_{N(P)}$, called the category of strongly filtered topological spaces, preserves weak equivalences. The latter is Quillen equivalent to $\textnormal{\textbf{Top}}_{P}$ \cite[Thm. 215]{douteauEnTop}. The right adjoint functor of this Quillen equivalence fits into a commutative diagram:
$$ \begin{tikzcd}
\textnormal{s\textbf{Set}}_P \arrow[r, "{|-|_{N(P)}}"] \arrow[rr, bend left= 60, "{|-|_P}"]& \textnormal{\textbf{Top}}_{N(P)} \arrow[r] & \textnormal{\textbf{Top}}_{P}
\end{tikzcd}.$$
So, one might hope that it preserves weak equivalences making the the same also hold for the composition. We show however in \Cref{exFDoesntRetain} that this is fact not the case. Nevertheless, we show that it holds for all weak equivalences between objects of the form $|X|_{N(P)}$ for $X \in \textnormal{s\textbf{Set}}_P$. This is the content of \Cref{thrmWeakEquSustain}. It is an immediate consequence of \Cref{thrmHolWeakEq}, which essentially says that for realizations of $P$-filtered simplicial sets the (generalized) homotopy links in the $N(P)$ and the $P$ setting are weakly homotopy equivalent.
\subsection{Strongly filtered topological spaces}\label{subsecStrongFilt}
We start by very roughly summarizing some results and definition of \cite[Sec. 1]{douteauEnTop}. The category $\textnormal{\textbf{Top}}_{N(P)}$ of strongly ($P$-)filtered topological spaces is obtained by taking the composition $$\textnormal{\textbf{Poset}} \xhookrightarrow{N} \textnormal{s\textbf{Set}} \xrightarrow{|-|} \textnormal{\textbf{Top}}$$ instead of the Alexandroff space functor, in \Cref{exFilteredObj}. Objects and morphisms in $\textnormal{\textbf{Top}}_{N(P)}$ are called \textit{strongly ($P$-)filtered spaces and strongly stratum preserving maps} respectively. Strongly stratum preserving maps are called \textit{strongly filtered maps} in \cite{douteauEnTop}.
%The category of strongly filtered topological spaces over $P$ is the over category in $\textnormal{\textbf{Top}}$ of $|N(P)|$. In other words, a \textit{morphism of strongly filtered topological spaces} $X,Y$ (we will also say strongly filtered morphism) is a continuous map $X \xrightarrow{f} Y$ making the diagram $$ \begin{tikzcd}
%X \arrow[rd, "p_X", swap]\arrow[rr,"f"]& &Y \arrow[ld, "p_Y"] \\ & {|N(P)|}
%\end{tikzcd}$$
%commute. We will denote this category by $\textnormal{\textbf{Top}}_{N(P)}$.
Analogously to the case of $\textnormal{\textbf{Top}}_{P}$ (just replace $P$ by $|N(P)|$ in \Cref{conEnTop}) this category is enriched over $\textnormal{\textbf{Top}}$.
We denote the hom space by $C^0_{N(P)}(X,Y)$, for $X,Y \in \textnormal{\textbf{Top}}_{N(P)}$. 
Also, just as in the $P$ case we obtain a cotensoring with the precisely analogous construction of $T \otimes X$, for $T \in \textnormal{\textbf{Top}}$ and $X \in \textnormal{\textbf{Top}}_{N(P)}$. Again tensoring with $I$ induces a notion of \textit{strongly stratified homotopy and strongly stratum preserving homotopy equivalence}. Then $C^0_{N(P)}(-,-)$ is compatible with this notion of homotopy. Further, one also has a realization functor $$|-|_{N(P)}: \textnormal{s\textbf{Set}}_P \to \textnormal{\textbf{Top}}_{N(P)}$$ together with a right adjoint \begin{align*}
\textnormal{Sing}_{N(P)}: \textnormal{\textbf{Top}}_{N(P)} \longrightarrow \textnormal{s\textbf{Set}}_P.
\end{align*} They are constructed via \Cref{conFiltFun}, by taking: $L = |-|$, $R= Sing$, $f=1_{\textnormal{\textbf{Top}}}$ and finally $g$ as the unit of the adjunction $N(P) \to Sing \circ |N(P)|$. 
Using, the realization $|-|_{N(P)}$ we can then define a (strong) homotopy link functor.
\begin{align*}
\textnormal{Hol}_{N(P)}: \textnormal{\textbf{Top}}_{N(P)} &\longrightarrow \textnormal{\textbf{Top}}^{\textnormal{sd}(P)^{op}}\\
X &\longmapsto C^0_{N(P)}(|-|_{N(P)}, X),
\end{align*} 
as in \Cref{defHomotopyLink}.
We call a morphism $X \xrightarrow{f} Y$ in $\textnormal{\textbf{Top}}_{N(P)}$ a (\textit{weak equivalence of strongly ($P$-)filtered spaces}) if for each flag $\mathcal J \in \textnormal{sd}(P)$ the induced map $\textnormal{Hol}_{N(P)}(\mathcal{J} , f )$ is a weak homotopy equivalence of topological spaces. As int he $P$-filtered setting one shows that every strongly stratum preserving homotopy equivalence is also a weak equivalence of strongly filtered spaces.
Finally, as for $\textnormal{\textbf{Top}}_{P}$, one obtains from \cite{douteauEnTop}[Thm. 2.8] the analogue to \Cref{thrmDouMod}:
\begin{theorem}\label{thrmDouModNP}
	There exists a cofibrantly generated model structure on $\textnormal{\textbf{Top}}_{N(P)}$ such that a morphism $f:X\to Y$ is \begin{itemize}
		\item a fibration if $\textnormal{Hol}_{N(P)}(\mathcal{J} , f)$ is a Serre fibration, for each flag $\mathcal J \in \textnormal{sd}(P)$;
		\item a weak equivalence if $f$ is a weak equivalence of strongly filtered topological spaces over $P$, that is if for each flag $\mathcal J \in \textnormal{sd}(P)$, $\textnormal{Hol}_{N(P)}(\mathcal{J} , f)$ is a weak homotopy equivalence;
		\item a cofibration if and only if it has the left lifting property against all acyclic fibrations. 
	\end{itemize}
\end{theorem}
With respect to this model structure and the Douteau-Henrique model structure on $\textnormal{s\textbf{Set}}_P$ one obtains:
\begin{theorem}\cite[Thm. 8.3.7]{douteauFren}\label{thrmDouteauNPrelRet}
	Both functors of the adjunction $$ |-|_{N(P)}: \textnormal{s\textbf{Set}}_P \longleftrightarrow \textnormal{\textbf{Top}}_{N(P)}: \textnormal{Sing}_{N(P)}$$
	preserve weak equivalences.
\end{theorem}
The Quillen equivalence between $\textnormal{\textbf{Top}}_{N(P)}$ and $\textnormal{\textbf{Top}}_{P}$ described in \cite[Subsec. 2.4]{douteauEnTop} is then given by the following construction.
There is a forgetful functor into the category $\textnormal{\textbf{Top}}_{P}$ given by postcomposing with the filtration of $|N(P)|_P$, $|N(P)| \to P$. We denote this functor by $$F: \textnormal{\textbf{Top}}_{N(P)} \longrightarrow \textnormal{\textbf{Top}}_{P}.$$ $F$ has a (enriched) right adjoint $$G: \textnormal{\textbf{Top}}_{P} \longrightarrow \textnormal{\textbf{Top}}_{N(P)},$$ given by pulling back along $|N(P)| \to P$. By the construction of $|-|_P$, $F$ fits into an on the nose commutative diagram: 
$$ \begin{tikzcd}
\textnormal{s\textbf{Set}}_P \arrow[r, "{|-|_{N(P)}}"] \arrow[rr, bend left= 60, "{|-|_P}"]& \textnormal{\textbf{Top}}_{N(P)} \arrow[r,"F"] & \textnormal{\textbf{Top}}_{P}
\end{tikzcd}.$$
Due to this, one might expect that the composition of $|-|_{N(P)}$ and $F$, $|-|_{P}$ also preserves weak equivalences. One might hope to prove this, by showing that $F$ preserves weak equivalences.
In fact, it turns out that the pair $(F,G)$ induces a Quillen equivalence, between the two model structures given in \Cref{thrmDouMod} and \Cref{thrmDouModNP} (see \cite[Thm. 2.15]{douteauEnTop}). However, this of course does not mean that $F$ preserves all weak equivalences, only that it preserves weak equivalences between cofibrant objects. For general weak equivalences, $F$ does indeed not preserve them:
\begin{example}\label{exFDoesntRetain}
	Consider the strongly stratum preserving map shown in \Cref{fig:ExNotEq}. The strong filtration is given by the height map, defined by projection to the vertical. The map is defined by identifying the upper two and lower two line segments respectively.
	\begin{figure}[H]
		\centering
		\includegraphics[width=80mm]{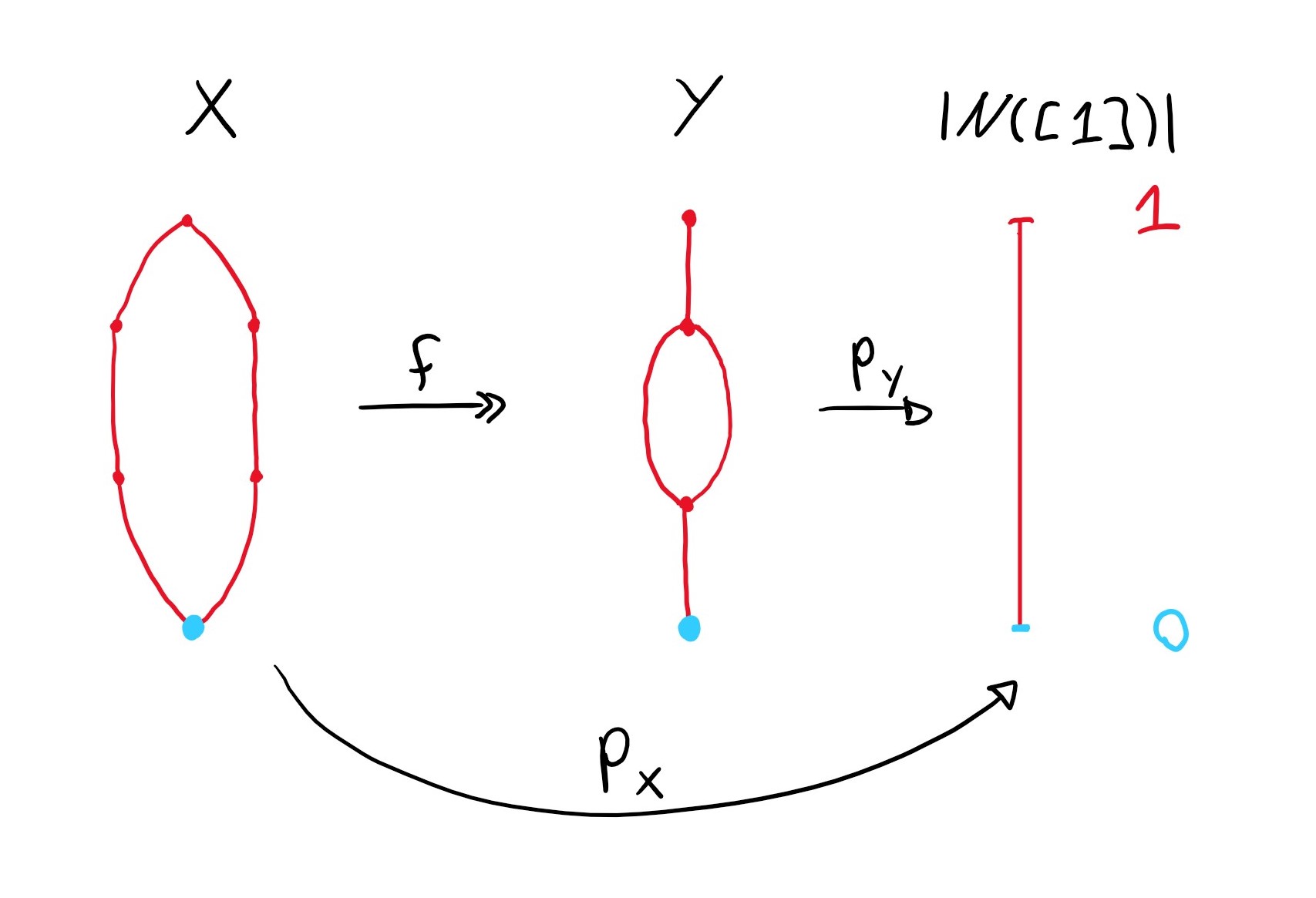}
		\caption{A strongly stratum preserving map over $P =[1]$}.
		\label{fig:ExNotEq}
	\end{figure}
	Figure \Cref{fig:ExNotEq2} shows the induced maps of \Cref{fig:ExNotEq} on generalized homotopy links, where the latter are depicted up to homotopy equivalence, on the right hand side. Homotopy links were computed using \Cref{propQuinHol}.
	\begin{figure}[H]
		\includegraphics[width=\linewidth]{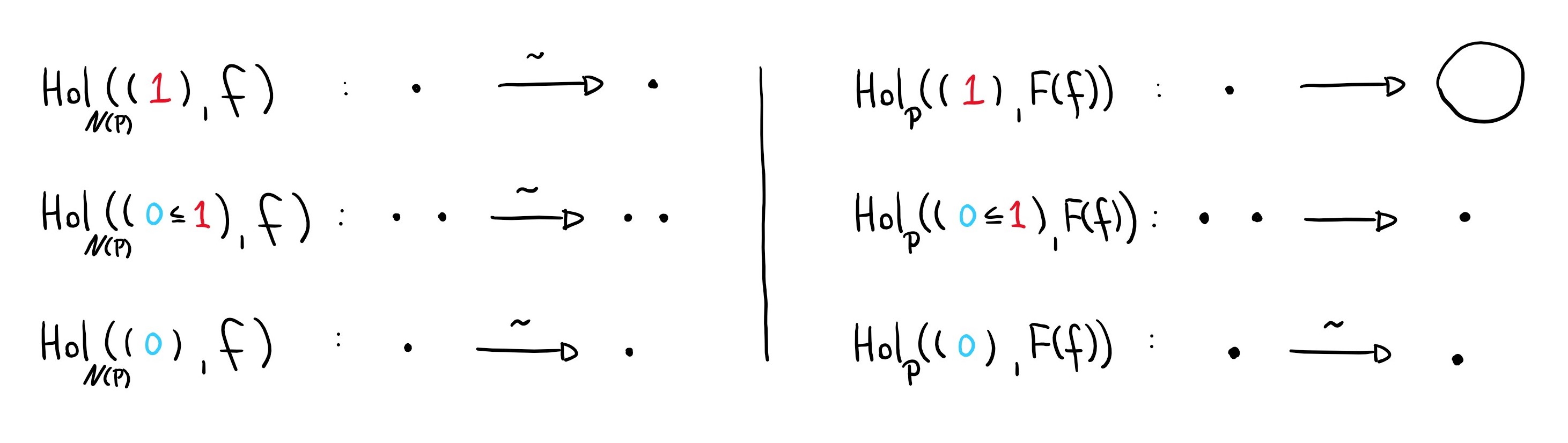}
		\caption{The induced maps on the generalized homotopy links of \Cref{fig:ExNotEq}. Homotopy links are shown up to homotopy equivalence. }
		\label{fig:ExNotEq2}
	\end{figure}
	While on the left hand side all of the maps are easily seen to be homeomorphisms even, this is clearly not the case on the right. In particular, $F$ does not sustain the property of being a weak equivalence here.
\end{example}
The failure of $F$ to preserve weak equivalences here comes from the fact that $F$ did not preserve the weak homotopy type of the strongly filtered spaces. As $F$ is an enriched functor, it induces a natural inclusion of homotopy links:
$$\textnormal{Hol}_{N(P)}(\mathcal{J} , X ) = C^0_{N(P)}(|\Delta^{\mathcal{J}}|_{N(P)}, X) \xhookrightarrow{F} C^0_{P}(F(|\Delta^{\mathcal{J}}|_{N(P)}), F(X)) = \textnormal{Hol}_{P}(\mathcal{J} , F(X)),$$
for a flag $\mathcal J\in \textnormal{sd}(P)$ and $X \in \textnormal{\textbf{Top}}_{N(P)}$. This induces a natural transformation:
$$\alpha: \textnormal{Hol}_{N(P)} \hookrightarrow \textnormal{Hol}_{P} \circ F$$
By naturality and the two out of three property for weak equivalences, one obtains:
\begin{lemma}\label{lemCondForFretain}
	Let $f:X \to Y$ in $\textnormal{\textbf{Top}}_{N(P)}$ be such that both $\alpha_X$ and $\alpha_Y$ are weak homotopy equivalences, at each $\mathcal J \in \textnormal{sd}(P)$. Then $f$ is a weak equivalence of strongly $P$-filtered spaces if and only if $F(f)$ is a weak equivalence of $P$-filtered spaces.
\end{lemma}
It is the content of the next subsection to show that the requirements of $F$ are fulfilled if $X$ and $Y$ are isomorphic (or even weaker, strongly stratum preserving homotopy equivalent) to realizations of filtered simplicial sets.
\subsection{Weak equivalence of homotopy links}
The content of this subsection is the proof of the following theorem:
\begin{theorem}\label{thrmHolWeakEq}
	Let $X \in \textnormal{s\textbf{Set}}_P$. Then, for each flag $\mathcal{J} \in \textnormal{sd}(P)$, the inclusion of generalized homotopy links: $$ \alpha_{|X|_{N(P)},\mathcal J}: \textnormal{Hol}_{N(P)}(\mathcal{J} , |X|_{N(P)} ) \hookrightarrow \textnormal{Hol}_{P}(\mathcal{J} , |X|_P)$$ is a weak equivalence of topological spaces.
\end{theorem}
We are first going to reduce to the case where $X$ is a finite filtered simplicial set. This allows us to assume for the homotopy links to be metrizable and in particular paracompact. Recall the following result on CW-complexes.
\begin{proposition}\cite{hatcherAlgTop}[Proposition A.1]\label{propCompactSub}
	A compact subspace of a CW-complex is contained in a finite subcomplex.
\end{proposition}
It is in an immediate consequence of this result that homotopy groups of a CW-complex $K$ can be computed as the colimit over the homotopy groups of finite subcomplexes of $K$. We obtain an analogous result in the filtered case:
\begin{proposition}\label{propColimPiHol}
	$X \in \textnormal{s\textbf{Set}}_P$. Let $X^i$ be the filtered diagram given by the finite filtered simplicial subsets of $X$. Then, for each $\mathcal J \in \textnormal{sd}(P)$, $n \geq 1$ and $x$ in some sufficiently large $|X^i|$ the morphisms:
	\begin{align*}
	\varinjlim(\pi_{0}(\textnormal{Hol}_{N(P)}(\mathcal{J} , |X^i|_{N(P)} ))) &\longrightarrow \pi_{0}(\textnormal{Hol}_{N(P)}(\mathcal{J} , |X|_{N(P)} )),\\
	\varinjlim(\pi_{n}(\textnormal{Hol}_{N(P)}(\mathcal{J} , |X^i|_{N(P)}),x)) &\longrightarrow \pi_{n}(\textnormal{Hol}_{N(P)}(\mathcal{J} , |X|_{N(P)} ),x)\\ 
	\end{align*} 
	and 
	\begin{align*}
	\varinjlim(\pi_{0}(\textnormal{Hol}_{P}(\mathcal{J} , |X^i|_{P} ))) &\longrightarrow \pi_{0}(\textnormal{Hol}_{P}(\mathcal{J} , |X|_{P} )),\\
	\varinjlim(\pi_{n}(\textnormal{Hol}_{P}(\mathcal{J} , |X^i|_{P}),x)) &\longrightarrow \pi_{n}(\textnormal{Hol}_{P}(\mathcal{J} , |X|_{P} ),x),\\ 
	\end{align*} 
	are isomorphisms (where for $n \geq 1$ the colimit is of course only taken over the final subdiagram given by $X^i$ such that $x \in |X^i|$).
\end{proposition}
\begin{proof}We are going to prove the result for $n \geq 1$ and $P$.
	The proof of the $N(P)$ case is completely analogous. The same can be said about the $0$ vs. the $n \geq 1$ case. Note that, for $Y \in \textnormal{\textbf{Top}}_{P}$, by the copower structure on $\textnormal{\textbf{Top}}_{P}$ , $\pi_{n}(\textnormal{Hol}_{P}(\mathcal{J} , Y),x)$ can alternatively described as the subquotient of $\textnormal{Hom}_{\textnormal{\textbf{Top}}_{P}}(S^n \otimes |\Delta^{\mathcal{J}}|_P, Y)$ given by such maps that restrict to $x$ on $\{\star\} \times |\Delta^{\mathcal{J}}|_P$ modulo stratified homotopies relative to this subspace, where $\star$ denotes the basepoint of $S^n$. The pointed filtered space $S^n \otimes |\Delta ^{\mathcal{J}}|$ is clearly compact and the realization of a simplicial set is a CW-complex. Hence, by \Cref{propCompactSub}, the map coming from the colimit is onto. Applying the same argument to homotopies gives injectivity.
\end{proof}
We are prove the following version of \Cref{thrmHolWeakEq} which by \Cref{propColimPiHol} implies the former.
\begin{theorem}\label{thrmHolWeakEqB}
	Let $X \in \textnormal{s\textbf{Set}}_P$ be locally finite (i.e. such that $|X|$ is locally compact). Then, for each flag $\mathcal{J} \in \textnormal{sd}(P)$, the inclusion of generalized homotopy links: $$ \alpha_{|X|_{N(P)}\mathcal J}: \textnormal{Hol}_{N(P)}(\mathcal{J} , |X|_{N(P)} ) \hookrightarrow \textnormal{Hol}_{P}(\mathcal{J} , |X|_P)$$ is a homotopy equivalence.
\end{theorem}
So, from here on out we assume $X$ to be locally finite. Further, we drop the indices from $\alpha$.
Next, we should make a quick remark on the topology of the mapping spaces.
\begin{remark}
	Recall that we take $\textnormal{\textbf{Top}}$ to be the category of $\Delta$-generated spaces; i.e. such spaces that have the final topology with respect to simplices (for details see \cite{duggerDelta}). Recall further that the topology on $C^0(T,T')$, for $T, T' \in \textnormal{\textbf{Top}}$, is the "$\Delta$-fication", denoted $k_{\Delta}$, of the compact open topology. That is, one takes the final topology with respect to all maps from simplices into the mapping space with the compact open topology. However, for questions of homotopy equivalence, this distinguishment is really not that important. This is so, as $k_{\Delta}$ preserves products where one of the factors is locally compact (see \cite{duggerDelta}) and hence homotopy equivalences. So, we instead take the compact open topology on $C^0$, for the remainder of this section. This has the advantage of making $\textnormal{Hol}_{P}(\mathcal{J}, Y)$ a metrizable space, when $Y$ is metrizable. A metric is then given by $d(f,g) = \sup_{x \in |\Delta^{\mathcal J}|_P}(f(x),g(x))$. The same construction works for $N(P)$. In particular, this is the case when $Y = |X|_P$ for $X$ a finite filtered simplicial set (as compact CW-complexes are metrizable \cite{fritsch1993cw}[Thm. A]). We make use of this fact later on. Note that it would of course suffice to assume that the filtered simplicial sets are locally finite so that the resulting spaces are locally compact.
\end{remark} 
As a first step in the proof of \Cref{thrmHolWeakEqB} we now reduce to the case, where $\mathcal J = P$. In particular, we can just assume $P$ to be a finite linearly ordered set $P = [q]$, $q \in \mathbb N$. 
\begin{definitionconstruction}\label{conStringRed}
	Recall the restriction functors described in \Cref{defStrata}. Then, by construction of the generalized homotopy links, one immediately obtains natural homeomorphisms (given by the postcomposition): \begin{align*}
	\textnormal{Hol}_{N(\mathcal J)}(\mathcal{J},Y_{\mathcal J}) &\cong \textnormal{Hol}_{N(P)}(\mathcal{J},Y), \\
	\textnormal{Hol}_{\mathcal J}(\mathcal{J},Y_{\mathcal J}) &\cong \textnormal{Hol}_{P}(\mathcal{J},Y), 
	\end{align*} 
	for $Y$ either in $\textnormal{\textbf{Top}}_{N(P)}$ or in $\textnormal{\textbf{Top}}_{P}$. By construction, we have a natural inclusion $$F(Y_{\mathcal{J}}) \hookrightarrow F(Y)_{\mathcal{J}}$$ (induced by the universal property of the pullback). In general however, this inclusion will not be a filtered homeomorphism. This is illustrated in \Cref{fig:ExSubspaces} for example.
	\begin{figure}[H]
		\includegraphics[width=\linewidth]{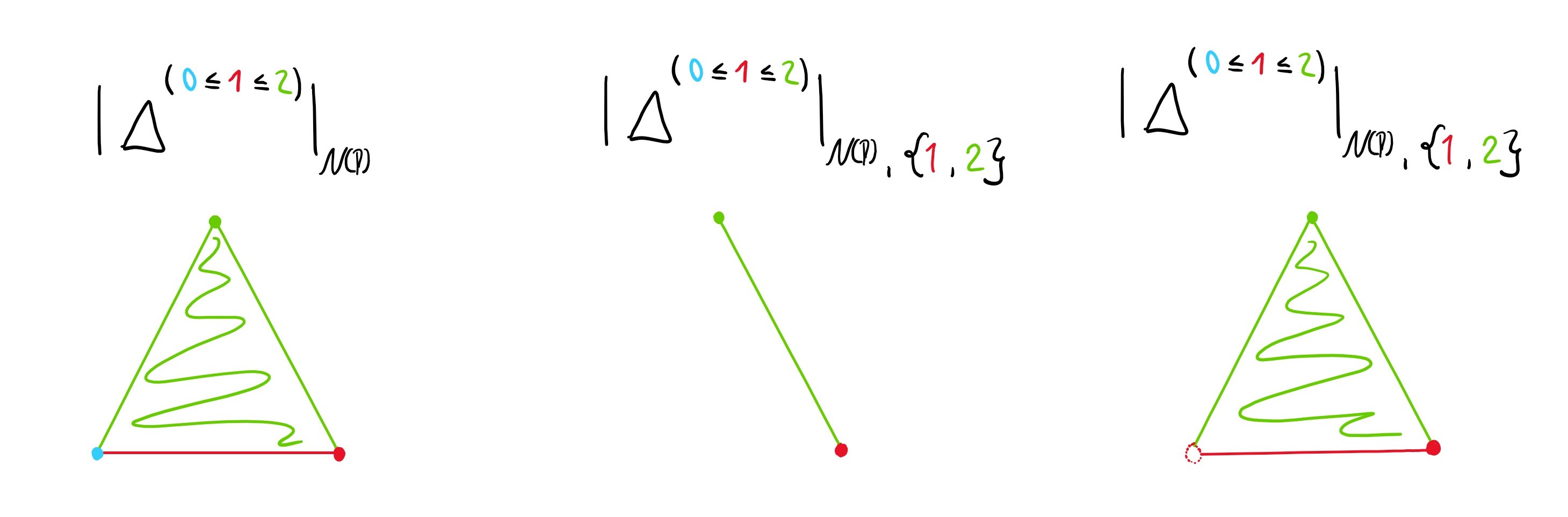}
		\caption{Comparison of $F(Y_{\mathcal J})$ and $F(Y)_{\mathcal J}$ in the case, where $P = [2]$, $\mathcal{J} = (0 \leq 1)$ and $Y = |\Delta^{\mathcal J}|_{N(P)}$. The right hand triangle is missing only the point in the $0$-stratum.}
		\label{fig:ExSubspaces}
	\end{figure}
\end{definitionconstruction}
In the case where $Y$ is the realization of a locally filtered simplicial set, one can give a more explicit construction of $Y_{\mathcal{J}}$. Before we do so, we introduce some notation that comes in handy during the technical parts of this section. 
\begin{definitionconstruction}
	Given a d-flag $\mathcal I $ in $P$, and a point $\xi$ in $|\Delta^{\mathcal I}|_P$ or in $|\Delta^{P}|$ we denote by $\xi_{i}$ its $p_i$-th coordinate in the standard realization in $\mathbb R^{\mathcal I}$, where $p_i$ is the $i$-th entry of $\mathcal I$. Even more, for a subset $\mathcal D \subset P$ we denote by $\xi_{\mathcal D}$ the vector given by such $\xi_i$ where $p_i \in \mathcal D$. In case $p_i \in \mathcal D$ holds for no $p_i$ we take this to be $0$. By a slight abuse of notation, we will sometimes think of these vectors as embedded into $\mathbb R^{\mathcal I}$ so that we can write expressions like $\xi_{\mathcal D} + \xi_{\mathcal D'}$. For any such vector, we denote by $|| - ||$ the sum of its values. Further, note that the vector $(||\xi_{\{p\}}||)_{p \in P}$ is precisely the value of $\xi$ under $|\Delta^{\mathcal I}|_P \to |N(P)|$, where we think of the right hand side as realized in $\mathbb R^{P}$. The points in the $p$-th stratum of $|\Delta^{\mathcal J}|_P$, for $p \in P$, are then precisely those points $\xi$, with $\xi_{\{p\}}\neq 0$ and $\xi_{\{p'\}}= 0$ for $p' \in \mathcal I$ with $p' > p$.
\end{definitionconstruction}
\begin{remark}\label{remExplDesc}
	Let $Y \in \textnormal{s\textbf{Set}}_P$ and $\mathcal J \in \textnormal{sd}(P)$. 
	One can think of $Y_{\mathcal J}$ as the full simplicial subset spanned by such vertices that map into $\mathcal J$ under $Y \to N(P)$. Then $(|Y|_{N(P)})_{\mathcal J} \cong (|Y_\mathcal{J}|)_{N(P)}$. This follows immediately from the fact that the realization functor into $\textnormal{\textbf{Top}}$ from simplicial sets, sustains finite limits (see \cite{nlab:geometric_realization}).\\
	However, $(|Y|_{P})_{\mathcal J}$ is a strictly larger space than this (see \Cref{fig:ExSubspaces}). It is given by the union of the strata with index in $\mathcal J$. We will be making use of another explicit description quite frequently. Let $Y$ be locally finite. Denote by $\Delta_i$ the diagram given by the non-degenerate simplices of $Y$. By \Cref{AppPropPullback}, $(|Y|_P)_{\mathcal J}$ is given by $\varinjlim \big( (|\Delta_i|_P)_{\mathcal J} \big)$. Hence, for most intends and purposes, it suffices to know what $(|\Delta|_P)_{\mathcal J}$ looks like for a filtered simplex. Let $\mathcal I = (p_0 \leq ... \leq p_k)$ be a d-flag in $P$. We identify the vertices of $\Delta^{\mathcal I}$ with the corresponding unit vectors in $\mathbb R^{\mathcal I}$ and $|\Delta^{\mathcal I}|$ with their convex hull. Then 
	%\begin{align*}
	%	(|\Delta^\mathcal{J}|_P)_{\mathcal I} = \Big\lbrace\sum_{p_i \in \mathcal I} t_{p_i} p_i \in |\Delta^{\mathcal I}|_P \mid \sum_{p_i \in \mathcal D} t_{p_i} = 1 \textnormal{ and }\sum_{\substack{p_i \in \mathcal D \\\textnormal{ s.t. } p_i \in \mathcal J}} t_{p_i} > 0 \Big \rbrace
	%\end{align*}
	\begin{align*}
	(|\Delta^\mathcal{I}|_P)_{\mathcal J} = \Big\lbrace \xi \in |\Delta^{\mathcal I}|_P \ \big | \xi \textnormal{ fulfills Condition \ref{equRemExplDesc}} \Big \rbrace.
	\end{align*}
	with
	\begin{equation}\label{equRemExplDesc}
	\max(\{p_i \mid ||\xi_{\{p_i\}}|| > 0 \}) \in \mathcal J
	\end{equation}
	Under this condition, we clearly have $||\xi_{\mathcal J}|| > 0$.
\end{remark}
Using this explicit description, we obtain:
\begin{proposition}
	Let $Y = |\hat Y|_{N(P)}$ where $\hat Y$ is a locally finite filtered simplicial set. Then the inclusion $F(Y_\mathcal{J}) \hookrightarrow F(Y)_\mathcal{J}$, from \Cref{conStringRed}, is a stratum preserving homotopy equivalence. 
\end{proposition}
\begin{proof}
	By \Cref{AppPropPullback}, we can construct a homotopy inverse on the simplex level, and check that the map as well as the homotopies are compatible with degeneracy and face maps, hence glue to a global one. For a filtered simplex $\Delta^{\mathcal I}$, with notation as in \Cref{remExplDesc}, the inclusion is 
	%\begin{align*}
	%			&\Big\lbrace \sum_{p_i \in \mathcal I} t_{p_i} p_i \in |\Delta^{\mathcal I}|_P \quad \Big | \quad \sum_{p_i \in \mathcal D \textnormal{ s.t. } p_i \in \mathcal J} t_{p_i} = 1 \Big\rbrace \\
	% \mbox{\large$\subset$}\quad &\Big\lbrace\sum_{p_i \in \mathcal I} t_{p_i} p_i \in |\Delta^{\mathcal I}|_P \quad \Big | \quad \sum_{p_i \in \mathcal D} t_{p_i} = 1 \textnormal{ and }\sum_{\substack{p_i \in \mathcal D \\\textnormal{ s.t. } p_i \in \mathcal J}} t_{p_i} > 0 \Big \rbrace.
	%\end{align*} 
	\begin{align*}
	&\Big\lbrace \xi \in |\Delta^{\mathcal I}|_P \ \Big | \ ||\xi_{\mathcal J}|| = 1 \Big\rbrace \\
	\mbox{\large$\subset$}\quad &\Big\lbrace \xi \in |\Delta^{\mathcal I}|_P \ \Big |\textnormal{ fulfills Condition \ref{equRemExplDesc}} \Big \rbrace.
	\end{align*} 
Under Condition \ref{equRemExplDesc}, we clearly have $||\xi_{\mathcal J}|| > 0$.
	%A filtered homotopy inverse to this is given as follows. For an element $\sum_{p_i \in \mathcal I} t_{p_i} p_i \in (|\Delta^{\mathcal I}|_P)_{\mathcal J}$ set $t_\mathcal{J}:= \sum_{\substack{p_i \in \mathcal D \textnormal{ s.t. } p_i \in \mathcal J}} t_{p_i} $. 
	Define a retract, $r$, of this inclusion, $i$, via
	%$$\sum_{p_i \in \mathcal I} t_{p_i} p_i \mapsto \frac{1}{t_{\mathcal J}} \sum_{\substack{p_i \in \mathcal D \\\textnormal{ s.t. } p_i \in \mathcal J}} t_{p_i}p_i.$$ 
	$$ \xi \mapsto \frac{1}{||\xi_\mathcal{J}||} \xi_{\mathcal J}.$$
	By Condition \ref{equRemExplDesc}, this does not change the maximal indices for which $\xi$ does not disappear and hence is stratum preserving.
	Clearly, $$ r \circ i = 1_{(|\Delta^{I}|_N(P)|)_{\mathcal J}}.$$ Conversely, $ i \circ r $ is stratified homotopic to the identity on $(|\Delta^\mathcal{I}|_P)_\mathcal{J}$ by the straight line homotopy, which is easily checked to be well-defined. It is not hard to see that the straight line homotopy of $i \circ r$ and the latter identity is compatible with degeneracy and face maps, hence induces a global stratum preserving deformation retraction of $F(Y_\mathcal{J}) \hookrightarrow F(Y)_\mathcal{J}$.
\end{proof}
We have now reduced the proof of \Cref{thrmHolWeakEqB} to the case where $\mathcal J = P = [q]$, for some $q \in \mathbb N$. We will still not write $|N(P)|$ at every possible location, just because it makes the notation look horribly convoluted. By using $|N(P)|$ instead of $|\Delta^{P}| = |\Delta^{\mathcal J}|$, we usually indicate what role in the proof the space is taking at this moment. One should however keep in mind that on as $P$-filtered spaces we have $|\Delta^{\mathcal J}| = |\Delta^\mathcal{P}|=|N(P)|$. We now illustrate the proof of \Cref{thrmHolWeakEqB} on a $\pi_0$ level, i.e. pointwise and for $q=1$, before we give a rigorous proof in more generality. 
\begin{example}\label{exEasyVerOfProof}
	It can be illustrative to see how every point in $\gamma \in \textnormal{Hol}_{P}(\mathcal J, |X|_{P})$ lies in the path component of one in $\textnormal{Hol}_{N(P)}(\mathcal J, |X|_{N(P)})$, for $\mathcal J = [1]$ the linear set with $2$ elements. Then $\textnormal{Hol}_{P}(\mathcal J, |X|_{P})$ is the space of paths, starting in the $0$-stratum and immediately leaving it. The proof of this statement serves as a model for the more general case. 
	Essentially, the idea is the followingUnder this condition, we clearly have $||\xi_{\mathcal J}|| > 0$.. Denote by $\mathcal X^{red}$ the subspace of $|X|_{N(\mathcal J)}$ obtained by by pulling back along $[0,1) \hookrightarrow I \cong |N([1])|$. By \Cref{AppPropPullback} this space is glued together from subspaces of simplices of the form $$|\Delta^{\mathcal I}|_{N(P)}^{red} = \Big\lbrace \xi \in |\Delta^{\mathcal I}|_P \ \Big | \ || \xi_{\{0\}}|| > 0 \Big \rbrace.$$ We have illustrated this space in \Cref{fig:exEasyVerOfProof1}.
	\begin{figure}[H]
		\centering
		\includegraphics[width=120mm]{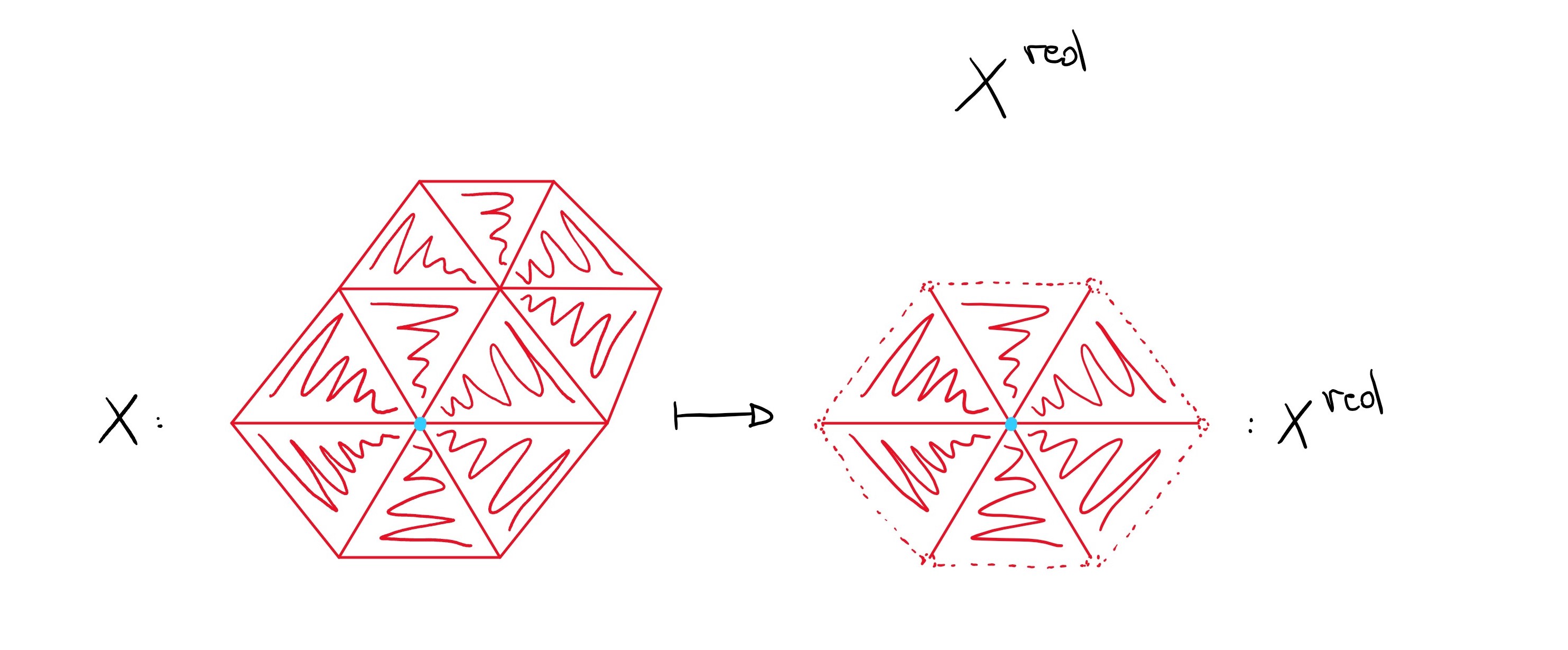}
		\caption{Illustration of an example of $\mathcal X^{red}$.}
		\label{fig:exEasyVerOfProof1}
	\end{figure}
	Now, $\mathcal X ^{red}$ admits a rescaling map: $$\rho: \mathcal X ^{red} \times_{P} |N(P)| \longrightarrow |X|_{N(P)}$$ defined on a simplex level by $$ (\xi, (t_0,t_1) )\mapsto \frac{t_0}{||\xi_{\{0\}}||}\xi_{\{0\}} + \frac{t_1}{||\xi_{\{1\}}||}\xi_{\{1\}}.$$
	It is not hard to see that this also makes sense for $||\xi_{\{j\}}||= 0$ as then also $t_j = 0$, and one can just set the corresponding summand to $0$. This construction is compatible with degeneracy and face maps, and hence extends to all of $\mathcal{X}^{red}$. Note that $\rho$ fits into a commutative diagram 
	$$
	\begin{tikzcd}
	{\mathcal X ^{red} \times_{P} |N(P)|} \arrow[rd, "{\pi_{|N(P)|}}", swap] \arrow[rr,"\rho"] & &{|X|_{N(P)}} \arrow[ld] \\
	&{|N(P)|}&
	\end{tikzcd}; 
	$$
	that is, it gives a strongly stratum preserving map (with the filtration on the left hand side given by the second component).
	Furthermore, $\rho$ is stratified homotopic to $$\mathcal{X} ^{red} \times_P |N(P)| \xrightarrow{\pi_{\mathcal{X}^{red}}} \mathcal{X}^{red} \hookrightarrow |X|_{P},$$ by a straight line homotopy $$R: \big (\mathcal{X} ^{red} \times_P |N(P)| \big) \times \Delta^1\to |X|_{P}$$ constructed simplex-wise (as a morphism in $\textnormal{\textbf{Top}}_{P}$).
	Now, $\mathcal{X}^{red}$ is an open neighbourhood of $(|X|_{P})_0$ in $|X|_{P}$. Thus, if we restrict $\gamma$ to a sufficiently small neighbourhood of $0$, it has image in $\mathcal{X}^{red}$. So, by continuously scaling down the domain of definition, $\gamma$ lies in the same path component as some (stratified) path $\gamma'$ with image in $\mathcal{X}^{red}$. Next, one uses $R$ to continuously rescale $\gamma'$ to a map over $|N(P)|$. This is done by defining \begin{align*}
	\hat \gamma: |\Delta^{\mathcal J}|_{P} &\longrightarrow \mathcal{X} ^{red} \times_{P} |N(P)|\\
	(t_0,t_1) &\longmapsto \big(\gamma'(t_0, t_1), (t_0 , t_1) \big)
	\end{align*} 
	and then applying $R$ to obtain a stratified homotopy from $\gamma'$ to a strongly stratum preserving map. The whole process is illustrated in the following picture.
	\begin{figure}[H]
		\centering
		\includegraphics[width=\textwidth]{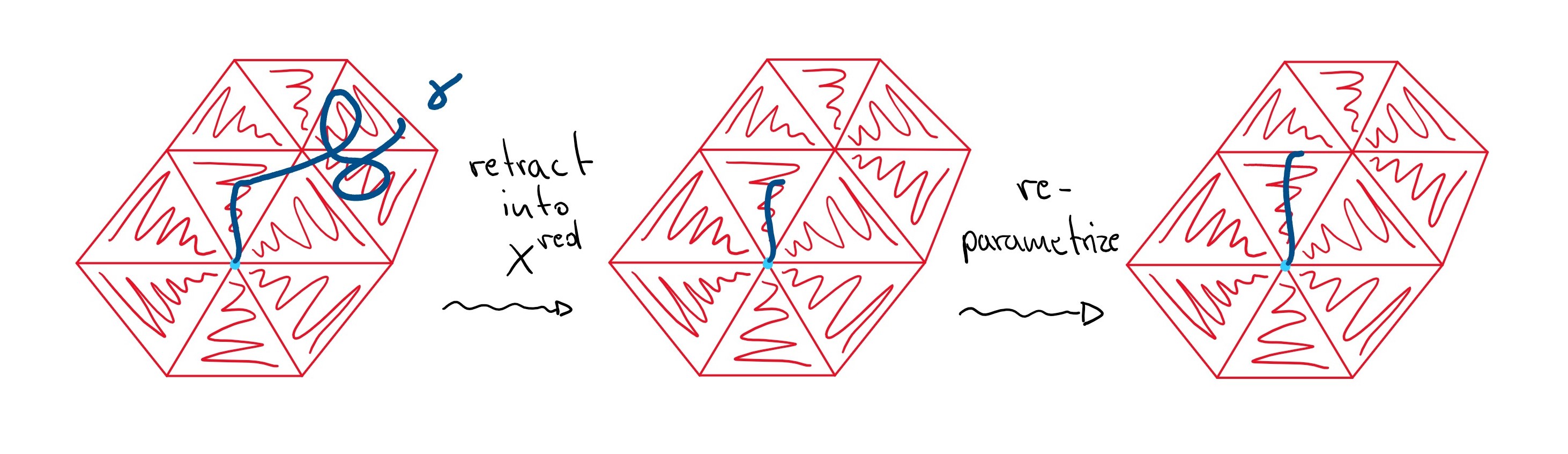}
		\caption{Illustration of the process described in \Cref{exEasyVerOfProof}.}
		\label{fig:exEasyVerOfProof2}
	\end{figure}
\end{example}
We start the proof of the general case by replicating the ``retracting-$\gamma$-part''

 of \Cref{exEasyVerOfProof}. To make notation a little bit more concise, we omit the index from the realization functors for simplices whenever it is clear from context what is meant.
\begin{definitionconstruction}\label{conFactorization}
	For a d-flag $\mathcal I$ in $P=[q]$ denote by $|\Delta^\mathcal{I}|^{red}$ the filtered subspace of $|\Delta^{\mathcal J}|$ given by
	$$|\Delta^{\mathcal I}|^{red} := \Big \lbrace \xi \in |\Delta^\mathcal I| \ \Big | \ \xi_{\{p\}} = 0 \implies \xi_{\{p'\}} = 0, \text{ for }p' \geq p \in \mathcal I \Big \rbrace.$$
	Further, denote by $\mathcal X^{red}$ the space filtered over $P$ obtained by pulling back $|X|_{N(P)}$ along $|\Delta^{P}|^{red} \hookrightarrow |\Delta^{P}| = |N(P)|.$ 
	$\mathcal X^{red}$ is essentially defined as the subspace of points $\xi \in |X|_{N(P)}$ fulfilling the implication $$\big ( p_{|X|_{N(P)}}(\xi) \big )_{\{p\}} = 0 \implies p_{|X|_P}(\xi) < p,$$ for all $p \in P$. \\
	Now, consider the following shrinking map $$\lambda:|\Delta ^{P}| \longrightarrow |\Delta^{P}|$$ induced by affinely extending 
	$$p \mapsto \textnormal{bar}(|\Delta^{[p]}|)\in |\Delta^{[p]}|\subset |\Delta^{P}|,$$ which sends a vertex to the barycenter of the simplex spanned by all vertices smaller than it in the linear order on $P$. It is clearly stratum preserving.
	We denote these barycenters by $\textnormal{bar}([p])$ for short. This map is stratified homotopic to the identity through a straight line homotopy. In particular, the inclusion $\textnormal{Hol}_{N(P)}(P, |X|_{N(P)}) \hookrightarrow \textnormal{Hol}_{P}(P , |X|_{P} )$ is homotopic to the map given by
	$$ \sigma \mapsto \sigma \circ \lambda.$$ 
	Furthermore, $\lambda$ has image in $|\Delta^{P}|^{red}$. 
	To see this, just note that any point in the image $\xi$ is of the shape $$ \xi = \sum_{p \in P} t_p \textnormal{bar}([p]) = \sum_{p \in P} \frac{t_p}{p+1}(|0| + ... +|p|),$$ for $t_p\geq 0, \sum_{p \in P}t_p=1 $ (using again the $|-|$ notation to denote the respective unit vectors). Hence, if $$\xi_{\{p\}} = \sum_{p' \geq p} \frac{t_{p'}}{p'+1} = 0,$$ so are all $t_p'$ with $p' \geq p$ and hence also $\xi_{\{p'\}} = 0$, for $p' \geq p$. Thus, for any $\sigma \in \textnormal{Hol}_{N(P)}(P , |X|_{N(P)} )$ we have $$ \sigma(\lambda (\xi)) \in \mathcal{X}^{red}.$$
\end{definitionconstruction}
As an immediate corollary of \Cref{conFactorization} we obtain:
\begin{corollary}\label{corFac}
	Let $X \in \textnormal{s\textbf{Set}}_P$ be locally finite. We use the notation from \Cref{conFactorization}. Then up to homotopy the inclusion $$\alpha:\textnormal{Hol}_{N(P)}(P, |X|_{N(P)}) \hookrightarrow \textnormal{Hol}_{P}(P , |X|_{P} )$$ factors as $$ \begin{tikzcd}
	{\textnormal{Hol}_{N(P)}(P, |X|_{N(P)})} \arrow[rr , hook] \arrow[rd, dashed]& & {\textnormal{Hol}_{P}(P, |X|_{P})} \\
	& \textnormal{Hol}_{P}(P, \mathcal X^{red}) \arrow[ru, hook]&
	\end{tikzcd},$$
	with the dashed arrow given by $$ \sigma \mapsto \Big \lbrace \xi \mapsto \sigma ( \lambda(\xi)) \Big \rbrace.$$
\end{corollary}
Thus, to prove \Cref{thrmHolWeakEqB}, it suffices to show that both arrows in the factorization given by \Cref{corFac} are homotopy equivalences. We start with the dashed one, by replicating the straightening map from \Cref{exEasyVerOfProof}. 
\begin{definitionconstruction}\label{conReparam}
	Note that, by the natural isomorphism \begin{align*}
		\mathcal{X}^{red} \times_{P} |N(P)| &= \big (|X|_{N(P)} \times_{|N(P)|} |\Delta^{P}|^{red} \big ) \times_{P} |N(P)| \\
		&\cong \big |X|_{N(P)} \times_{|N(P)|} \big (|\Delta^{P}|^{red} \times_{P} |N(P)| \big),
	\end{align*} the left hand side is a pullback in the sense of \Cref{AppPropPullback}. Hence, it is given by a colimit over the diagram $$|\Delta^{\mathcal I_i}|^{red} \times_{P} |N(P)|,$$ induced by pulling back the realizations of the non-degenerate simplices of $X$. The projection to the second component makes this a colimit over $|N(P)|$, i.e. one in $\textnormal{\textbf{Top}}_{N(P)}$. For a d-flag $\mathcal I \in N(P)$, define the straightening map 
	\begin{align*}
	\rho_{\mathcal I}: |\Delta^{\mathcal I}|^{red} \times_{P} |N(P)| &\longrightarrow |\Delta^{\mathcal I}| \\
	(\xi, (t_0,...,t_q)) &\longmapsto \sum_{p \in P}\frac{t_p}{||\xi_{\{p\}}||}\xi_{\{p\}},
	\end{align*}
	where the summands are taken to be $0$, when $||\xi_{\{p\}}|| = 0$. We should make a few remarks on why this is well-defined and continuous.
	First off, note that for right hand side vector in $\mathbb R^{\mathcal I}$ to actually lie in $|\Delta^\mathcal{J}|$, we have to have $$\xi_{\{p\}}=0 \implies t_p=0.$$ Otherwise, the vector does not have $1$ as the sum over its entries. To see this is the case, let $p \in P$ and $\xi_{\{p\}} = 0$. Note that as $\xi \in |\Delta^{\mathcal I}|^{red}$ we have that $p_{|\Delta^{\mathcal J}|_P}(\xi) < p$. Hence, as we have taken the fiber product over $P$ on the left hand side, this also means $t_{p'} = 0$, for $p' \geq p$.\\
	We can replicate the same argument, but replacing equality with convergence, to show that the map defined in this fashion is also continuous. It is clearly strongly stratum preserving, using the projection to the second component on the left hand side and $ |\Delta^{\mathcal I}| \to |\Delta^{P}| = |N(P)|$ on the right hand side. Furthermore, note that it fits into a commutative diagram
	\begin{equation*}
	\begin{tikzcd}
	{|\Delta^{\mathcal I}|^{red} } \arrow[bend right = 60, hook, rr] \arrow[r, "d"] & {|\Delta^{\mathcal I}|^{red} \times_P |N(P)| } \arrow[r, "\rho_{\mathcal I}"] & {|\Delta^{\mathcal I}| } 
	\end{tikzcd},
	\end{equation*} where $d$ denotes the map induced by $1_{|\Delta^{\mathcal I}|^{red} }$ and the composition with the strong filtration of $|\Delta^{\mathcal I}|^{red} $. Denote by $$R_{\mathcal I}: \Big (|\Delta^{\mathcal I}|^{red} \times_{P} |N(P)| \Big ) \times I\longrightarrow |\Delta^{\mathcal I}| $$ the straight line homotopy, between $\rho_{\mathcal I}$ and $$|\Delta^{\mathcal I}|^{red} \times_{P} |N(P)| \xrightarrow{\pi_{|\Delta^{\mathcal I}|^{red}}} |\Delta^{\mathcal I}|^{red} \hookrightarrow |\Delta^{\mathcal I}| $$ As both maps are stratum preserving, so is the straight line homotopy (as the strata of a filtered simplex are convex, see also \Cref{lemLineSeqStrat}). One can easily check that $R$ is compatible with the face and degeneracy maps of $X$. Hence, this induces a stratified homotopy $$R:\Big ( \mathcal{X}^{red} \times_{P} |N(P)| \Big ) \otimes \Delta^1\to |X|_{P}$$ between a strongly stratum preserving map $\rho$ and the stratum preserving map $$\mathcal X^{red} \times_{P} |N(P)| \xrightarrow{\pi_{\mathcal X ^{red}}} \mathcal X^{red} \hookrightarrow |X|_{P}.$$ Finally, note that $R$ maps into $\mathcal X^{red}$ at every point in time, but $t=1$. 
\end{definitionconstruction}
\begin{remark}\label{remRhoHat}
	$\rho$ can be slightly extended. Instead of pulling back to $|\Delta^{\mathcal P}|^{red} \times_{P} |N(P)| $, we can pull back to $$\Big(|\Delta^{P}|^{red} \times_{P} |N(P)| \Big) \cup \Big (|\Delta^{P}| \times_{|N(P)|} |N(P)| \Big )\subset |\Delta^{P}| \times |N(P)| = |\Delta^{P}|^2.$$ We essentially add points on the diagonal. Denote this space by $ \hat \Delta$. 
	One can easily check that the formula for the simplexwise definition of $\rho_{\mathcal{I}}$ in \Cref{conReparam} extends to $$|\Delta^{\mathcal I}| \times_{|N(P)|} \hat \Delta= \Big ( |\Delta^{\mathcal I}|^{red} \times_{P} |N({P})| \Big ) \cup \Big (|\Delta^{\mathcal I}|\times_{|N({P})|} |N(P)| \Big ),$$ where the union on the right hand side is to be understood in $|\Delta^{\mathcal I}| \times |N(P)|$. On $\Big (|\Delta^{\mathcal I}|\times_{|N(P)|} |N(P)| \Big )$ this is just given by the the projection to the first component. By the same argument as in \Cref{conFactorization}, we obtain an extension of $\rho$ to \begin{equation}\label{eqXtimesHat}
	|X|_{N(P)} \times_{|N(P)|} \hat \Delta = \Big ( \mathcal{X}^{red} \times_{P} |N(P)| \Big ) \cup \Big ( |X|_{N(P)} \times_{|N({P})|} |N(P)| \Big ) \subset |X|_{N(P)} \times_{P} |N(P)|.
	\end{equation} 
	$\hat \rho$ fits into a commutative diagram
	\begin{equation}\label{equXtimesHat2}
	\begin{tikzcd}
	{\mathcal |X|_{N(P)}} \arrow[bend right = 30, "1", rr] \arrow[r, "d", hook] & {|X|_{N(P)} \times_{|N(P)| } \hat \Delta }\arrow[r, "\hat \rho"] & {|X|_{N(P)}} 
	\end{tikzcd},
	\end{equation}
	where $d$ is just the identification with the second component of the union in \eqref{eqXtimesHat}.
\end{remark}
\begin{proposition}
	The dashed map from \Cref{corFac} is a homotopy equivalence.
\end{proposition}
\begin{proof}
	First, note that under the enriched the adjunction $F: \textnormal{\textbf{Top}}_{N(P)} \longleftrightarrow \textnormal{\textbf{Top}}_{P}:G$ we have a natural homeomorphism: $$\textnormal{Hol}_{P}(P, \mathcal X^{red}) \cong \textnormal{Hol}_{N(P)}(P, \mathcal X^{red}\times_{P} |N(P)| ),$$ where again the argument in the right hand Hol is strongly filtered by the projection to $|\Delta^{P}| = |N(P)|$. Under this natural homeomorphism the dashed map corresponds to $$ \sigma \mapsto \Big \lbrace \xi \mapsto (\sigma \circ \lambda (\xi), (\xi_0, ..., \xi_q))\Big \rbrace.$$ Denote this map by $\eta$. Denote by $\rho_*$ the map $$\textnormal{Hol}_{N(P)}(P, \mathcal X^{red}\times_{P} |N(P)| ) \longrightarrow \textnormal{Hol}_{N(P)}(P, |X|_{N(P)})$$ induced by $\rho$. We claim that $\rho_*$ is a homotopy inverse to $\eta$.\\
	\\
	First, consider the homotopy between the two maps $$	{\textnormal{Hol}_{N(P)}(P, |X|_{N(P)})} \longrightarrow \textnormal{Hol}_{P}(P, |X|_{P})$$ given 
	by $\sigma \mapsto \sigma \circ \lambda$ and the inclusion that is induced by the straight line homotopy from $\lambda$ to the indentity. Up to the last point in time, this maps into $\textnormal{Hol}_{P}(P, \mathcal{X}^{red})$, and at $t=1$ it maps into $\textnormal{Hol}_{N(P)}(P, |X|_{N(P)})$.
	 Hence, if we post-compose this homotopy with $$\textnormal{Hol}_{P}(P, |X|_{P}) \cong \textnormal{Hol}_{N(P)}(P, |X|_{P} \times_{P} |N(P)|),$$ (using\eqref{eqXtimesHat}) it factors through $$\textnormal{Hol}_{N(P)}(P, |X|_{N(P)} \times_{|N(P)|} \hat \Delta ) \hookrightarrow \textnormal{Hol}_{N(P)}(P, |X|_{P} \times_{P} |N(P)|).$$ That is, we obtain a new homotopy: $$H: \textnormal{Hol}_{N(P)}(P,|X|_{N(P)}) \times \Delta^1\longrightarrow \textnormal{Hol}_{N(P)}(P, |X|_{N(P)} \times_{N(P)} \hat \Delta ).$$ 
	$H$ is given by 
	\begin{align}
	\sigma &\longmapsto \Big \lbrace \xi \mapsto (\sigma \circ \lambda(\xi),(\xi_0,...,\xi_1)) \Big \rbrace \textnormal{, for $t=0$;} \label{proofH0} \\
	\sigma &\longmapsto \Big \lbrace \xi \mapsto (\sigma (\xi),(\xi_0,...,\xi_1)) \Big \rbrace \textnormal{, for $t=1$.} \label{proofH1} 
	\end{align} 
	Now, consider the composition
	$$
	\begin{tikzcd}
	{\textnormal{Hol}_{N(P)}(P,|X|_{N(P)})} 
	\arrow[rr] \arrow[rd, "H_t", swap] 
	&& 
	{\textnormal{Hol}_{N(P)}(P, |X|_{N(P)})} \\
	& {\textnormal{Hol}_{N(P)}(P, |X|_{N(P)} \times_{|N(P)|} \hat \Delta )} \arrow[ru,"{\textnormal{Hol}_{N(P)}(P, \hat \rho)}", swap]&
	\end{tikzcd}.$$
	By \eqref{proofH0} for $t=0$ the horizontal is $\rho_* \circ \eta$. By\eqref{proofH1} and the fact that $\sigma$ is strongly stratum preserving, $H_1$ maps into $\textnormal{Hol}_{N(P)}(P, |X|_{N(P)} \times_{|N(P)|} |N(P)|)$. Therefore, by \eqref{equXtimesHat2}, which essentially states that $\hat \rho$ is the projection to the first component on $|X|_{N(P)} \times_{|N(P)|} |N(P)|$ and again, by \eqref{proofH1}, we obtain that, for $t=1$, the horizontal is the identity. In particular, $\rho_* \circ \eta \simeq 1$.\\ \\
	%
	%Now, compose this, with the map induces on generalized homotopy links by $\hat \rho$. Then, this new homotopy, $H_2$, is given by: 
	%\begin{align*}
	%	\sigma &\mapsto \Big \lbrace \xi \mapsto (\sigma \circ \lambda(\xi),(\xi_0,...,\xi_1)) \Big \rbrace &\mapsto \Big \lbrace \xi \mapsto \rho((\sigma \circ \lambda(\xi),(\xi_0,...,\xi_1))) \Big \rbrace \textnormal{ for $t=0$,} \\
	%	\sigma &\mapsto \Big \lbrace \xi \mapsto (\sigma (\xi),(\xi_0,...,\xi_1)) \Big \rbrace &\mapsto \Big \lbrace \xi \mapsto \hat \rho (\sigma (\xi),(\xi_0,...,\xi_1)) = \sigma(\xi) \Big \rbrace \textnormal{ for $t=1$.}
	%\end{align*}
	%where the last identity comes from the fact, that $\sigma$ is strongly filtered, and hence $(\sigma(\xi), (\xi_0,...,\xi_q)) \in |X|_{N(P)} \times_{|N(P)|} |N(P)|$ on which $\hat \rho$ restricts to the projection to the first component. At $t=0$ this is $\rho_* \circ \eta$. At $t=1$ it is clearly the identity. \\
	%\\
	It remains to be shown that $\eta \circ \rho_*$ is also homotopic to the identity. We first post-compose $\eta \circ \rho_*$ with the enriched adjunction homeomorphism $$\textnormal{Hol}_{N(P)}(P, \mathcal X^{red} \times_{P} |N(P)|) \xrightarrow[\sim]{\psi} \textnormal{Hol}_{P}(P, \mathcal X^{red})$$ and then with the inclusion into $\textnormal{Hol}_{P}(P, |X|_{P})$, $i_*$. Note that the adjunction homeomorphism $\psi$ is given by first including into $\textnormal{Hol}_{P}(P, F(\mathcal X^{red}\times_{P} |N(P)|))$ via $\alpha$ and then pushing forward with the projection to $\mathcal X^{red}$, which we denote $\pi_*$.
	We then have a commutative diagram
	$$ \begin{tikzcd}
	{\textnormal{Hol}_{N(P)}(P, \mathcal X^{red} \times_{P} |N(P)|)}\arrow[rr, "i_* \circ \psi \circ \eta \circ \rho_*"] \arrow[rd, "\tilde \lambda", swap]&& {\textnormal{Hol}_{P}(P, |X|_{P})}\\ & {\textnormal{Hol}_{P}(P, F(\mathcal X^{red} \times_{P} |N(P)|))} \arrow[ru, "{\textnormal{Hol}_{P}(P, F( \rho))}", swap] & 
	\end{tikzcd},
	$$
	where $\tilde \lambda$ is given by first including into $\textnormal{Hol}_{P}(P, F(\mathcal X^{red} \times_{P} |N(P)|))$, via $\alpha$, and then pulling back with $\lambda$.
	Using the straight line homotopy between $\lambda$ and the identity, $\tilde \lambda$ is homotopic to $\alpha$. Denote this homotopy by $L$. $ R$ induces a homotopy $R_*$ from
	${\textnormal{Hol}_{P}(P, F(\rho))}$ to $i_* \circ \pi_*$. Now, consider the diagonally composed homotopy of $L$ and $ R_*$, i.e. $$D_t:= R_{*,t} \circ L_t.$$ At $t=1$, this is $i_* \circ \psi$. At $t=0$ it is $i_* \circ \psi \circ \eta \circ \rho_*$, by the commutativity of the last diagram. For $t \in (0,1)$, $R$ maps into $\mathcal X^{red}$ and hence, $D_t$ maps into $\textnormal{Hol}_{P}(P, \mathcal X ^{red})$. In particular, $D$ factors through $ \textnormal{Hol}_{P}(P, \mathcal X ^{red})$. That is, there is a unique $\tilde D$ with $i_* \circ \tilde D= D$. Then at $t=0$, $\tilde D$ is $\psi \circ \eta \circ \rho_*$. As $D$ is $i_* \circ \psi$ at $t=1$, $\tilde D$ is $\psi$ at $t=1$. As $\psi$ is a homeomorphism, in particular we also obtain $\eta \circ \rho_* \simeq 1$.
\end{proof}
For a proof of \Cref{thrmHolWeakEqB} we are left with showing the following.
\begin{proposition}\label{propSecondMapEq}
	In the setting of \Cref{corFac}, $$\textnormal{Hol}_{P}(P, \mathcal X ^{red}) \hookrightarrow \textnormal{Hol}_{P}(P, |X|_{P})$$ is a homotopy equivalence.
\end{proposition}
One would hope that one can make a similar Argument as in \Cref{exEasyVerOfProof}. However, note that the argument there involved the fact that for $q=1$, a neighbordhood of $|0| \in |\Delta^{P}|$ is always mapped into $\mathcal X^{red}$. This can not immediately be replicated in the general case. The argument needs a little bit of refinement. 
\begin{definitionconstruction}\label{conRedFiltr}
	Let $ k \in P$. Define $$|\Delta^{ P}|^{red,k}:= \{ \xi \in |\Delta^{P}| \mid \xi_p = 0 \implies \xi_{p'} = 0 \textnormal{, for all $p < k, p \leq p' $} \}.$$ Define the filtered space $\mathcal X^{k}$ by pulling back $|X|_{N(P)}$ along $|\Delta^{P}|^{red,k} \hookrightarrow |\Delta^{\mathcal P}|$. Note, how in case where $X = \Delta^{\mathcal I}$ for a d-flag $\mathcal{I}$ in $P$ this is equivalently described by
	$$|\Delta^{\mathcal I}|^{red,k}= \{ \xi \in |\Delta^{\mathcal I}| \mid \xi_{\{p \}} = 0 \implies \xi_{\{p'\}} = 0 \textnormal{, for all $p < k, p \leq p' $} \}.$$
	As these spaces are easily seen to be $\Delta$-generated (in fact they can be triangulated), the pullback here, is actually the pullback in the naive topological category (by \Cref{AppPropPullback}). Thus, we can actually think of $\mathcal{X}^{k}$ as actual topological subspaces of $|X|_{P}$ and not just in the $\Delta$-generated subspace sense (i.e. as the $\Delta$-fication of a subspace). They are then alternatively described as the inverse image under $p_{|X|_{N(P)}}$ of $|\Delta^{P}|^{red,k}$.
	Clearly, we then have: 
	$$\mathcal{X}^{red} = \mathcal{X}^{q} \subset \mathcal{X}^{q-1} \subset ... \subset \mathcal{X}^{0} = |X|_{P}.$$
	Furthermore, $\mathcal X^{k+1}$ is a neighbourhood of the $k$-th stratum of $\mathcal X^{k}$ in $\mathcal X^{k}$. To see this, note first that, by the pullback construction, we only need to show this for $|\Delta^{P}|$. The $k$-th stratum is given by such $\xi \in |\Delta^{P}|^{red,k}$ where $\xi_{k} >0$ and $\xi_{p'} = 0$, for $p' \geq k+1$. Hence, the defining property of $|\Delta^{P}|^{red,k+1}$ is clearly fulfilled in a small open ball around such a point. 
\end{definitionconstruction}
Instead of directly showing that the inclusion in \Cref{propSecondMapEq} is a homotopy equivalence, we show the analogous statement for the inclusions $\mathcal{X}^{k+1} \hookrightarrow \mathcal{X}^{k}$. \Cref{propSecondMapEq} is then an immediate consequence of this. 
\begin{lemma}
	In the setting of \Cref{conRedFiltr}, for $k \in P$, the map $$ \textnormal{Hol}_{P}(P, \mathcal{X}^{k+1}) \hookrightarrow \textnormal{Hol}_{P}(P, \mathcal{X}^{k})$$ induced by $\mathcal{X}^{k+1} \hookrightarrow \mathcal{X}^{k}$ admits a deformation retraction.
\end{lemma}
\begin{proof}
	First note that, by our local finiteness assumption, all spaces involved are metrizable and the sources in all mapping spaces involved are compact. In particular, the compact open topology on all mapping spaces involved can be thought of as coming from the supremum metric. This makes most continuity verifications very easy and they are mostly omitted.\\
	Consider the map 
	\begin{align*}
	S: &|\Delta^{P}| \times [0,1] \longrightarrow |\Delta^{P}|;\\
	&(\xi,s) \longmapsto (\xi_{[k-1]}, \xi_{k} + (1-s)||\xi_{\{k+1,...,q\}}||, s\xi_{\{k+1,...,q\}}).
	\end{align*}
	Note that this is the identity at $s=1$, stratum preserving outside of $s=0$ and that it restricts to the identity on $|\Delta^{[k]}| \subset |\Delta^{P}|$.
	We now claim the following.
	\begin{claim}\label{proofLemmaClaim}
		There exists a map of topological spaces: \begin{align*}
		t: \textnormal{Hol}_{P}(P , \mathcal{X}^{k}) \times \Delta^1\longrightarrow I
		\end{align*}
		such that 
		\begin{enumerate}
			\item $t(\sigma, s) = 0$ if and only if $s=0$. \label{claim1}
			\item $\sigma \Big (S \big (\xi, t\big (\sigma,||\xi_{\{k+1,...,q\}}||\big) \big ) \Big) \in \mathcal{X}^{k+1}$, for $\sigma \in \textnormal{Hol}_{P}(P , \mathcal{X}^{k})$ and $\xi \in |\Delta^{P}|$. \label{claim2}
		\end{enumerate}
	\end{claim} 
	We are going to prove this claim at the end of the proof. We show first how it implies the statement of the lemma. Consider the map 
	\begin{align*}
	\mathcal S: \textnormal{Hol}_{P}(P , \mathcal{X}^{k}) \times \Delta^1&\longrightarrow \textnormal{Hol}_{P}(P , \mathcal{X}^{k} )\\
	(\sigma,s) &\longmapsto \Big \{\xi \mapsto \sigma \big (S \big (\xi, (1-s) + st \big (\sigma, ||\xi_{\{k+1,...,q\}}||) \big ) \Big )\Big\}.
	\end{align*} 
	We need to verify that this is well-defined in the sense that the map described on the right hand side is actually stratum preserving. As $S$ is stratum preserving at each time outside of $0$, we only need to worry about the case where the second argument of $S$ i.e. $(1-s) +st(\sigma, ||\xi_{\{k+1,...,q\}})||$ is $0$. Then, in particular $s =1$, hence $t(\sigma, ||\xi_{\{k+1,...,q\}}||) = 0$ and thereby, using \ref{claim1} we obtain that $|| \xi_{\{k+1,...,q\}} || = 0$. Thus, $\xi \in |\Delta^{[k]}|$ which by the definition of $S$ implies $S(\xi,s) = \xi$ and thereby preservation strata is also verified in this case. Now, note furthermore that $H$ maps $\textnormal{Hol}_{P}(P , \mathcal{X}^{k+1} )$ into itself. At $s = 0$, it is given by pulling back with $S(-,1)$ which is the identity. At $s = 1$, by \ref{claim2}, it maps into $\textnormal{Hol}_{P}(P , \mathcal X^{k+1} )$. Thus, we have verified that indeed $\textnormal{Hol}_{P}(P , \mathcal{X}^{k+1}) \hookrightarrow \textnormal{Hol}_{P}(P , \mathcal{X}^{k}) $ admits a deformation retraction. \\
	\\
	It remains to show \Cref{proofLemmaClaim}.
	We denote by $d$ the metric on $|\Delta^{P}|$ induced by the supremum norm on $\mathbb R^{P}$.
	Let $\varepsilon \in (0,1]$. Denote \begin{align*}
	\Delta_{\varepsilon}&:= \{\xi \in |\Delta^{P}| \mid ||\xi_{\{k,...,q\}}|| \geq \varepsilon \},\\\Delta_{k,\varepsilon}&:=|\Delta^{[k]}| \cap \Delta_{\varepsilon} \subset |\Delta^{P}|.
	\end{align*} Alternatively, the latter is given by such $\xi \in |\Delta^{[k]}|$ with $\xi_{k} \geq \varepsilon$. Clearly, $\Delta_{k,\varepsilon}$ is a compact subset of the $k$-th stratum of $|\Delta^{P}|$ and $\Delta^{\varepsilon}$ is compact. Furthermore, note that \begin{equation}\label{equProofLemmaConv}
	\underset{\substack{\xi \in S(\Delta_{\varepsilon} \times [0,t])\\\xi' \in \Delta_{k,\varepsilon}}} {\sup \inf} d(\xi,\xi') \xrightarrow{t \to 0} 0.
	\end{equation}
	Now, let $\sigma_i \in \textnormal{Hol}_{P}(P, \mathcal X^{k})$. As $\sigma_i$ is stratum preserving, it maps $\Delta_{k,\varepsilon}$ into the $k$-th stratum of $\mathcal X^{k}$, $\mathcal X^k_k$. Hence, as we have seen in \Cref{conRedFiltr} that $\mathcal X^{k+1}$ is a neighbourhood of $\mathcal X^k_k$, a neighbourhood of $\Delta_{k,\varepsilon}$ maps into $\mathcal X^{k+1}$ under $\sigma_i$. Hence, by \eqref{equProofLemmaConv}, we also get \begin{equation}\label{equProofLemmaCont}
	\sigma(S(\Delta_{\varepsilon} \times [0,t_{i, \varepsilon}]) \subset \mathcal{X}^{k+1}
	\end{equation}
	for a sufficiently small $t_{i, \varepsilon} > 0$. 
	By the definition of the compact open topology, this also holds in a neighbourhood of $\sigma_i$, $U_i \subset \textnormal{Hol}_{P}(P, \mathcal X^{k})$. Now, cover $\textnormal{Hol}_{P}(P, \mathcal X^{k})$ by such neighbourhoods to obtain a covering $(U_i)$ together with $t_{i,\varepsilon}$ as above. 
	As $\textnormal{Hol}_{P}(P, \mathcal X^{red})$ is metrizable, in particular, it is paracompact. Let $\varphi_i$ be a partition of unity subordinate to $(U_i)$. Next, define
	\begin{align*}
	t_{\varepsilon}: \textnormal{Hol}_{P}(P, \mathcal X^{k}) &\longrightarrow (0,1]\\
	\sigma &\longmapsto \sum \varphi_i(\sigma) t_{i,\varepsilon}.
	\end{align*}
	Then, for each $\sigma \in \textnormal{Hol}_{P}(P, \mathcal X^{k})$, by \eqref{equProofLemmaCont}, we obtain 
	\begin{equation}\label{HUngrraaa}
	\sigma\Big (S \big (\Delta_{\varepsilon} \times [0,t_{\varepsilon}(\sigma)] \big ) \Big) \subset \mathcal{X}^{k+1}.	
	\end{equation}
	Suppose we have done this construction for $\varepsilon = \frac{1}{2^n}$. 
	%\begin{align*}
	%	t_{n}: \textnormal{Hol}_{P}(P, \mathcal X^{k}) &\longrightarrow (0,1]\\
	%	\sigma &\longmapsto \min_{n' \leq n}\big( t_{2^{-n'}}(\sigma)\big).
	%\end{align*} 
	Consider the covering of $(0,2)$ given by the open intervals $I_n:=(2^{-(n+1)}, 2^{-(n-2)})$ for $n \geq 1.$ Take a family of bump functions $\psi_n$ on $[0,2]$ such that $\psi_n$ has support in $(I_n)$ and is $1$ on $[2^{-n}, 2^{-(n-1)}]$. Then, define:
	\begin{align*}
	t:\textnormal{Hol}_{P}(P, \mathcal{X}^{k}) \times \Delta^1&\longrightarrow [0,1]\\
	(\sigma, s) &\longmapsto s\prod_{n \geq 1} \Big( 1- \psi_n(s)\big(1-t_{2^{-n}}(\sigma)\big ) \Big ).
	\end{align*}
	Note that the product on the right hand side is always bounded by $1$, so the expression actually makes sense even at $s=1$.
	Further, note that locally in the $s$ coordinate, for $s > 0$, this is actually just a finite product. Furthermore, as $\prod ...$ is always smaller than one, this is easily seen to be continuous in $s=1$ also. We need to see that this construction fulfills \Cref{claim1,claim2} of \Cref{proofLemmaClaim}. The first is immediate from the facts that clearly $t(\sigma, 0)= 0$, that the product at any other time is finite, and that $t_{2^{-n}}(\sigma) > 0$, for all $\sigma \in \textnormal{Hol}_{P}(P , \mathcal{X}^{k})$. For the second, note that for $s \in [2^{-n},2^{-(n-1)}]$ and $\sigma \in \textnormal{Hol}_{P}(P , \mathcal{X}^{k})$ we have:
	\begin{equation}\label{HUNGRYY}
	t(\sigma,s) \leq t_{2^{-n}}(\sigma)
	\end{equation}
	Now, let $\xi \in |\Delta^{P}|$. If $||\xi_{k+1,...,q}|| = 0$, then the expression in \ref{claim2} is just $\sigma(\xi)$ and thus $p_{|\Delta^P|}(\xi) = p_{|X|_P}(\sigma(\xi)) \leq k$. But for those $p$ we have $\mathcal X^{k}_{p} = \mathcal X^{k+1}_{p}$. If $||\xi_{\{k+1,...,q\}}|| > 0$, take $n \geq 1$ such that $||\xi_{k+1,...,q}|| \in [2^{-n}, 2^{-(n-1)}]$. Then $\xi \in \Delta_{2^{-n}}$ and by \eqref{HUNGRYY} $t(\sigma, ||\xi_{\{k+1,...,q\}}||) \leq t_{2^{-n}}(\sigma)$. Hence, by \eqref{HUngrraaa}, the result follows.
\end{proof}
This finished the proof of \Cref{thrmHolWeakEqB} and hence of \Cref{thrmHolWeakEq}. 
\subsection{The realization theorem}
The following realization theorem is now an immediate consequence of \Cref{thrmHolWeakEq}. \begin{theorem}\label{thrmWeakEquSustain}
	The realization functor: 
\begin{align*}
	|-|_P: \textnormal{s\textbf{Set}}_P \longrightarrow \textnormal{\textbf{Top}}_{P}
\end{align*}
preserves weak equivalences.
\end{theorem}
\begin{proof}
	We have already seen that $|-|_P$ factors as $F \circ |-|_{N(P)}$. By \Cref{thrmDouteauNPrelRet} $|-|_{N(P)}$ preserves weak equivalences. Furthermore, by \Cref{thrmHolWeakEq}, every strongly filtered space of the shape $|X|_{N(P)}$ fulfills the requirements of \Cref{lemCondForFretain} for $F$ to preserve weak equivalences. Hence, the composition, that is $|-|_P$, preserves weak equivalences.
\end{proof}
This result, even despite the lack of a Quillen equivalence between the simplicial and the topological model categories, gives strong justification, for the former being a good candidate to ``model'' the latter. By the universal property of the localization, $|-|_P$ induces a functor 
\begin{equation*}
	|-|_P: \mathcal{H}\textnormal{s\textbf{Set}}_P \longrightarrow \mathcal{H}\textnormal{\textbf{Top}}_{P}, 
\end{equation*}
denoted the same by abuse of notation.
We conjecture that this is in fact an equivalence of categories. While we do not have a proof of this statement, we will at least see later on (in \Cref{thrmFullyFaithful}) that this functor is fully faithful if one restricts to finite filtered simplicial sets. In particular, one obtains an equivalence of categories between the homotopy category of finite filtered simplicial sets and the category of filtered topological spaces that are weakly equivalent to the realization of a finite filtered simplicial set.
%%%%%%%%%%%%%%%%%%%%%
\section{The filtered simplicial approximation theorem}\label{secSimApp}
To understand the connection between the simplicial and the topological categories of filtered objects, we take a detour via the category of simplicial complexes. This approach might seem a bit unusual as the usage of simplicial sets and model categories tends to circumvent any appeals to the more classical and rigid simplicial complexes. However, as we are mainly concerned with understanding the underlying homotopy categories, the usage of a simplicial approximation theorem allows for a (admittedly somewhat bruteforce) connection of the two settings. This finally reflects in the fully faithful embedding described in \Cref{thrmFullyFaithful}. However, despite the path we are following maybe not being the most elegant from a model category perspective, it has the nice side-effect of shedding a lot more light on the piecewise linear world of filtered spaces. Much of the content of this section is a priori independent from the world of simplicial sets and can be understood as an investigation of the purely piecewise linear setting.\\
\\
The simplicial approximation theorem (\cite[Ch. I, Sec. 4., Theorem 8]{spanier1989algebraic}) is probably one of the most powerful and most used theorems in classical algebraic topology. It states that for any map continuous map $\phi: |K| \to |L|$ between the realizations of finite simplicial complexes $K$ and $L$ there exists a subdivision $K'$ of $K$ and a simplicial map $f: K' \to L$, such that $|K'| \xrightarrow{\sim} |K| \xrightarrow{\phi} |L|$ is homotopic to $|f|$ (for most of the standard language on simplicial complexes used here we refer to \cite[Ch. I]{spanier1989algebraic}). Thus, it allows one to reduce many questions on maps of polyhedra to the setting of purely combinatorial maps of simplicial complexes (see for example the proof of the Lefschetz fixed-point theorem in \cite[Ch.4, Sec. 23]{bredon2013topology}).\\
\\ Furthermore, there is also a relative version of this theorem, allowing one to keep the homotopy constant on a subcomplex $L \subset K$ such that $\phi$ on $|L|$ is already simplicial (see \cite{zeeman1964relative}). In particular, applying this relative version to homotopies, one obtains that every continuous map of compact polyhedra is homotopic to a piecewise linear map, and that two piecewise linear maps are homotopic if and only if they are homotopic through a piecewise linear homotopy. In other words, one obtains an equivalence of categories between the p.l. homotopy category of compact polyhedra and the full subcategory of the homotopy category of topological spaces given by compact polyhedra. Thus, at least as long as one is interested mostly in spaces homotopy equivalent to the latter, many questions of homotopy theory can be reduced to the piecewise linear setting.\\
\\
This of course begs the question whether a similar statement can be made in the filtered (stratified setting). Such a result was given by C.H. Schwartz in \cite{schwartz1971}. However, the proof there has several flaws. (To our best understanding condition $\bar \gamma $ of \cite[Theorem 2]{schwartz1971} is only fulfilled if the map is locally constant on the subcomplex $\tilde L$. In addition to that, the proof of the existence of the homotopy, as it is of now, seems to be based on a mistaken assumption in \cite[equation 1.51]{schwartz1971} which leads to the linear interpolation in the homotopy not being well-defined.)
Furthermore, the volume of Mathematica (Cluj) it is published in seems to be rather hard to access as of this moment. In the English language, to the best of our knowledge, there is no correct proof of such a theorem available at all.\\
\\
The proof of the simplicial approximation theorem for the filtered (stratified) setting is loosely based on the one in \cite{schwartz1971}. However, as we restrict our-self to the finite setting, our proofs are somewhat more concise. Furthermore, we give an explicit description of the subdivision used. Lastly, we give a new proof of the existence of the homotopy, circumventing the difficulties mentioned above.
\subsection{Standard constructions for filtered simplicial complexes}
Recall that in \Cref{exFilteredObj} we constructed a category of $P$-filtered simplicial complexes namely $\textnormal{\textbf{sCplx}}_P$. Similarly to the setting of $P$-filtered simplicial sets this category admits the following more explicit description.
\begin{remark}\label{remDeSC}
	Note that, by the definition of maps of simplicial complexes, a filtered simplicial complex can be equivalently characterized as a simplicial complex together with a map $p: K^{(0)} \to P$ such that whenever $\{x_0,...,x_k\}$ is a simplex of $K$, then $\{p(x_0), ..., p(x_k)\}$ is a flag in $P$. Furthermore, a morphism of $P$-filtered simplicial complexes $f: K \to L$ is equivalently a simplicial map of the underlying simplicial complexes such that $p_L(f(x)) = p_K(x))$, for $x \in K^{(0)}$. This justifies calling such morphisms \textit{stratum preserving} simplicial maps.
\end{remark}
Recall that we equipped $\textnormal{\textbf{sCplx}}_P$ with the realization functor $|-|_P$ into $\textnormal{\textbf{Top}}_{P}$ in \Cref{exFiltFun}. For readabilities sake, we mostly omit the $P$ index, for the remainder of this section. As we do not make use of other realization functors here, this should not lead to confusion.
\begin{remark}\label{remDesReal}
	The realization functor is naturally isomorphic to the functors given by the following construction. For a filtered simplicial complex $K$, consider $\mathbb R_{\geq 0}^{K^{(0)}}$ (where in the infinite case we take the weak topology induced by finite dimensional subspaces).
	We identify $x \in K^{(0)}$ with the vector in $\mathbb R_{\geq 0}^{K^{(0)}}$ that has $1$ at $x$ and $0$ for all other entries, and denote this vector by $|x|$. We send $K$ to
	$$ \Big\{ \xi = \sum t_x |x| \in \mathbb R_{\geq 0}^{K^{(0)}} \mid \sum t_x = 1, 
	\{x \mid t_x \neq 0\} \in K \Big\} \subset \mathbb R_{\geq 0}^{K^{(0)}},$$
	filtered by $$\xi = \sum t_x |x| \longmapsto p_K(x_m),$$ where the latter is \textnormal{max}imal with respect to $t_{x_m} \neq 0$.
	For $\sigma \in K$, we denote by $\mathring{\sigma}$ the open simplex in $|K|_P$ corresponding to $\sigma$. By construction, $\mathring{\sigma}$ lies in the stratum of $|K|_P$ corresponding to the \textnormal{max}imum of the flag $p_K(\sigma)$, i.e. $$ \mathring{\sigma} \subset |K|_{\textnormal{max}(p_K(\sigma))}.$$
	We usually think of filtered simplicial complexes as being embedded in $\mathbb R_{\geq 0}^{K^{(0)}}$ in this fashion, which makes the description of linear homotopies somewhat cleaner.
\end{remark}
\begin{definitionconstruction}
	Given a $A$ subcomplex of a $P$-filtered simplicial complex $A$, the filtration of $K$ naturally restricts to a filtration of $A$. Hence, we think of any full subcomplex of a $P$-filtered complex as also being $P$-filtered. For such a pair $A \subset K$ of filtered complexes, we denote by $K-A$ the ($P$-filtered) subcomplex of $K$ spanned by all vertices not in $A$. We use the "$-$" to distinguish this from the set of simplices not in $A$, denoted $K \setminus A$.
\end{definitionconstruction}
%
%We denote by $\textnormal{\textbf{sCplx}}_P$ the category with objects simplicial complexes filtered over $P$ and morphisms maps of filtered simplicial complexes filtered over $P$.\\
%\\
%As in the classical setting, there is a realization functor
%\begin{align*}
%|-|_P: \textnormal{\textbf{sCplx}}_P & \longrightarrow \textnormal{\textbf{Top}}_P
%\end{align*} given by first applying the usual realization functor to $K \xrightarrow{p_K} \textnormal{sd}(P)$ and then composing with the map $|\textnormal{sd}(P)| \to P$, given on $|[p_0, ... , p_k]|$ by $\xi = \sum t_i p_i \longmapsto p_m$, where $p_m$ is \textnormal{max}imal with respect to $t_i \neq 0$. The latter map is continuous with respect to the Alexandrov Topology on $P$. Hence, the composition gives a filtered topological space. The filtration map will be denoted by $p_{|K|}$. This notation will turn out to be compatible, with the one used for filtered simplicial sets later on, so there is really no harm, in naming them the same. $$ 
%
%If $K$ is an filtered simplicial complex, we denote by $K_p$ the full filtered subcomplex spanned by all such $x \in K^{(0)}$, such that $p_K(x) \leq x$. For $p \leq q$ there is the obvious containment $K_p \subset K_q$. The \textit{$p$-th stratum of $K$}, is the full subcomplex spanned by such vertices $x$, where $p_K(x) = p$. Note, that clearly the realization of the $p$-th stratum is strictly contained in the $p$-th stratum of the realization. However, the inclusion is a filtered homotopy equivalence.\\
%\\
We now quickly recap some standard constructions from p.l. topology. For more details see for example \cite{rourke2012introduction}. 
\begin{definitionconstruction}\label{conJoin}
	Recall that for two simplicial complexes $K_0$ and $K_1$ the join of the two, $K_0 \star K_1$, is given by the simplicial complex with vertices $K_0^{(0)} \sqcup K_1^{(0)}$ and simplices gives by subsets $\sigma$ of $K_0^{(0)} \sqcup K_1^{(0)}$ such that $K_0^{(0)} \cap \sigma \in K$ and $K_1^{(0)} \cap \sigma \in K_1$. The simplex in $K_0 \star K_1$ given by $\sigma \sqcup \tau$ for $\sigma \in K_0$ and $\tau \in K_1$ is denoted by $\sigma \star \tau$. Note that if $K$ and $L$ are filtered simplicial complexes over $P$, the join is in general not naturally filtered. There might be vertices $x$ and $y$ connected by a $1$-vertex in the join such that neither $p_K(x) \leq p_L(y) $ nor $p_L(y) \leq p_K(x)$. Thus, one has to be a little bit careful with using joins here and if we talk about the join of filtered complexes, we usually mean the underlying simplicial complex.
	Further, recall that for two simplices $\sigma $ and $\tau$ in $K_0$, $\sigma \star_{K_0} \tau \subset K_0^{(0)}$ denotes the union of $\sigma$ and $\tau$.
	In general of course, $\sigma \star_{K_0} \tau$ will not be a simplex of $K_0$. We use the same notation for filtered simplicial complexes. 
\end{definitionconstruction}
As expected from the proof of a simplicial approximation theorem, we are be making extensive use of the notion of stars.
\begin{definitionconstruction}
	Recall (see \cite[Ch. 3]{rourke2012introduction}) that the closed star of a simplex $\sigma$ in a simplicial complex $K$, denoted $\overline{\textnormal{star}}_K(\sigma)$, is the subcomplex given by all simplices that are contained in a simplex containing $\sigma$. Recall further that the open star of a simplex $\sigma$ in $K$, denoted $\textnormal{star}_K(\sigma)$, is the set of all simplices containing $\sigma$. 
	%We denote by $\partial{\overline{\textnormal{star}}_K(\sigma)}$ the simplicial subcomplex of %
	We denote by $|\textnormal{star}_K(\sigma)|$ the union of the open simplices in $\textnormal{star}_K(\sigma)$, $$|\textnormal{star}_K(\sigma)|:=\bigcup_{\tau \in \textnormal{star}_K(\sigma)} \mathring{\tau} \subset |K|.$$ The choice of notation should not mislead one to think that $\textnormal{star}_K(\sigma)$ is actually a simplicial complex. We use the same notation in the filtered setting.
\end{definitionconstruction}
The following properties of the star constructions are immediate from the non-filtered case.
\begin{remark}\label{remFactsOnStars} Let $K$ be a filtered simplicial complex and $\sigma = \{x_0,...,x_k\} \in K$. Then the following hold.
	\begin{enumerate}
		\item $\textnormal{star}_K(\sigma) = \bigcap_{x_i \in \sigma } \textnormal{star}_K(x_i)$ \label{remFactsOnStars1}
		\item $\overline{\textnormal{star}}_K(\sigma) \subset \bigcap_{x_i \in \sigma } \overline{\textnormal{star}}_K(x_i)$ \label{remFactsOnStars2}
		\item $|\textnormal{star}_K(\sigma)| = \{\sum_{x \in K^{(0)}}t_x x \in |K| \mid t_x > 0 \textnormal{ for all } x \in \sigma \} $ \label{remFactsOnStars3}
		\item $|\textnormal{star}_K(\sigma)|$ is an open neighbourhood of $\mathring{\sigma}$ in $|K|$\label{remFactsOnStars4}.
		\item $\overline{|\textnormal{star}_K(\sigma)|} = |\overline{\textnormal{star}}_K(\sigma)|$ \label{remFactsOnStars5}
	\end{enumerate}
	Now, let $\tau \in K$ be another simplex. Then the following are equivalent: 
	\begin{itemize}
		\item $\sigma \star_K \tau$ in $K$ .
		\item $\textnormal{star}_K(\sigma) \cap \textnormal{star}_K(\tau) \neq \emptyset$
		\item $|\textnormal{star}_K(\sigma)| \cap |\textnormal{star}_K(\tau)| \neq \emptyset$
		\item $\textnormal{star}_K(\sigma) \cap \overline{\textnormal{star}}_K(\tau) \neq \emptyset$
		\item $|\textnormal{star}_K(\sigma)| \cap |\overline{\textnormal{star}}_K(\tau)| \neq \emptyset$
	\end{itemize}
	In particular, one obtains a geometrical condition for when a set of vertices $\{x'_0, ...,x'_k\} \subset K^{(0)}$ forms a simplex in $K$; i.e. that the intersection of their (realized) open stars is non-empty. Now, if $f: K \to L$ is a map of simplicial complexes, then: 
	\begin{enumerate}[label = (\alph*)]
		\item $|f|(|\textnormal{star}_K(\sigma)|) \subset |\textnormal{star}_L(f(\sigma))|$
		\item $|f|(|\overline{\textnormal{star}}_K(\sigma)|) \subset |\overline{\textnormal{star}}_L(f(\sigma))|$
	\end{enumerate}
\end{remark}
As an immediate consequence of this remark together with \Cref{remEquCharFilt}, we obtain an equivalent characterization of when a map on the vertices of filtered complexes induces one of filtered complexes.
\begin{lemma}\label{lemEquCharMap}
	Let $K, L \in \textnormal{\textbf{sCplx}}_P$ and $f^0 : K^{(0)} \to L^{(0)}$ be a map such that $f^0(p_K(x)) = p_L(f(x))$. Then $f$ extends to a morphism in $\textnormal{\textbf{sCplx}}_P$ if and only if for every $\sigma \in K$, $\bigcap_{x \in \sigma } \textnormal{star}_L(f(x))$ or equivalently $\bigcap_{x \in \sigma } |\textnormal{star}_L(f(x))|$ is non-empty.
\end{lemma}
Even when the filtration condition of \Cref{lemEquCharMap} is not fulfilled, the open stars can be used to obtain an upper boundary for the stratum of an open simplex.
\begin{lemma}\label{lemBoundOnStrata}
	Let $K \in \textnormal{\textbf{sCplx}}_P$ and $x \in K^{0}$. Then for any $\xi \in |\textnormal{star}_K(x)|$, $$p_K(x) \leq p_{|K|}(\xi).$$ 
\end{lemma}
\begin{proof}
	Let $\tau$ be a simplex such that $\xi \in \mathring \tau$ and $x \in \tau$. Then $$p_K(x) \leq \textnormal{max}(p_K(\tau)) = p_{|K|}(\xi).$$
\end{proof}
We also need a few remarks about line segments.
\begin{definition}
	Let $\xi$ and $\eta$ in $|K|$ be two points that lie in a common simplex $|\sigma|$. Denote by
	$$[\xi,\eta] := \{(1-t)\xi + t\eta \mid t \in [0,1] \} \subset |\sigma| \subset |K| \subset \mathbb R^{K^{(0)}}$$ the affine line segment between them. 
	We call this the \textit{closed line segment between $\xi$ and $\eta$}. We further denote by $$(\xi, \eta) := [\xi, \eta] \setminus \{\xi, \eta\}$$ 
	and call this \textit{the open line segment} between $\xi$ and $\eta$. Furthermore, we define half open line segments in the obvious way.
\end{definition}
 By construction, every open line segment lies in a unique open simplex of $|K|$, corresponding to the join of $\tau,\tau' \subset \sigma$ with $\xi \in \mathring{\tau}$ and $\eta \in \mathring{\tau}'$. As a direct consequence there is the following lemma, which will be quite useful later on for the construction of homotopies.
\begin{lemma}\label{lemLineSegmentsInStars}
	Let $K \in \textnormal{\textbf{sCplx}}_P$ and $\sigma \in K$. Then, for any half open line segment $[\xi,\eta)$ with $\xi \in |\textnormal{star}_K(\sigma)|$, we also have $[\xi,\eta) \subset |\textnormal{star}_K(\sigma)|$. 
\end{lemma} 
\begin{proof}
	As $|\textnormal{star}_K(\sigma)|$ is open and $\xi \in \textnormal{star}_K(\sigma)$, we have $(\xi, \eta)\cap |\textnormal{star}_K(\sigma)| \neq \emptyset.$ Hence, as the open simplices of $|K|$ are pairwise disjoint, and $|\textnormal{star}_K(\sigma)|$ is the union of such simplices, we get $\tau \in \textnormal{star}_K(\sigma)$, where $\tau$ is the unique simplex with $(\xi,\eta) \subset \mathring{\tau}$. In particular, $$[\xi,\eta) = \{\xi\} \cup (\xi,\eta) \subset |\textnormal{star}_K(\sigma)|.$$
\end{proof}
Furthermore, line segments also interact rather nicely with the stratification.
\begin{lemma}\label{lemLineSeqStrat}
	Let $|\sigma|$ be a simplex of a filtered simplicial complex $|K|$. Let $\xi, \eta \in |\sigma|$ and $p_{|K|}(\xi) \leq p_{|K|}(\eta)$. Then $(\xi,\eta]$ lies in the $p_{|K|}(\eta)$-th stratum of $|K|$.
\end{lemma}
\begin{proof}
	This is immediate from the characterization of $|K|_p$ as the set of elements \begin{align*}
\Big \{\sum_{x \in K^{(0)}}t_x |x| \in |K| \mid p=\textnormal{max}\left(p_K \left(\{x \mid t_x > 0\} \right) \right) \Big \}\subset \mathbb{R}^{K^{(0)}}.
	\end{align*}
\end{proof}
Finally, we need to make use of regular neighbourhoods (see \cite[Ch. 3]{rourke2012introduction}).
\begin{definition}
	Recall that the simplicial neighbourhood of a subcomplex $L \subset K$, $N(K,L)$, is the subcomplex of $K$ given by $$N(K,L):=\bigcup_{\sigma \in L}\overline{\textnormal{star}}_K(\sigma).$$ Denote by $\partial N(K,L)$ the intersection of $K-L$ and $N(K,L)$. Then one has $$K = (K-L) \cup N(K,L).$$ Furthermore, we denote by $\mathring{N}(K,L)$ the set of simplices given by $N(K,L) \setminus \partial N(K,L)$. Just as for stars, denote by $|\mathring{N}(K,L)|$ the union of the open simplices in $\mathring{N}(K,L)$ in $|K|$. An open line segment $(\xi,\eta) \subset |N(K,L)|$ with $\xi \in |L|$ and $\eta \in |K-L|$ is called a \textit{ray of $N(K,L)$}. We use the same notation in the filtered setting.
\end{definition} Similarly to the setting of stars one then has: 
\begin{remark}\label{remPropOfN} Let $K \in \textnormal{\textbf{sCplx}}_P$ and $A$ be a full subcomplex of $K$. Then the following hold:
	\begin{enumerate}[label = (\roman*), ref = (\roman*)]
		\item $N(K,A)= \{\sigma \star_K \tau | \sigma \in A, \tau \in \partial N(K,A) \textnormal{ s.t. } \sigma \star \tau \in K\}$ \label{remPropOfN1}
		\item $|\mathring{N}(K,A)|$ is an open (regular in the topological sense) neighbourhood of $|A|$ in $|K|$. \label{remPropofN2}
		\item $|\mathring{N}(K,A)|\setminus|A|$ is the disjoint union of the rays of $N(K,A)$. \label{remPropofN3}
		\item $|\mathring{N}(K,A)|\setminus|A|$ is covered by sets of the shape $|\textnormal{star}_K(\{a,x\})|$, for $a \in A^{(0)}$ and $x \in \partial N(K,A)^{(0)}$, $\{a,x\} \in K$. \label{remPropOfNCover}
	\end{enumerate}
\end{remark}
We will be making explicit use of the constructions that turn $|\mathring{N}(K,L)|$ into a regular neighbourhood later on so we now go a little bit more into detail, describing them here. We use the notation from \cref{remDesReal}.
\begin{definitionconstruction}\label{conProjectionsOfSimNbhd}
	Let $K \in \textnormal{\textbf{sCplx}}_P$ and $A$ a be full subcomplex of $K$. Consider the map \begin{align*}
	s_A:|K| &\longrightarrow [0,1] \\
	\sum_{x \in (K-A)^{(0)}} t_x|x| + \sum_{a \in A^{(0)}} t_a|a| &\longmapsto \sum_{x \in (K-A)^{(0)}} t_x.
	\end{align*}
	Then $s_A^{-1}(\{1\})= |K-A|$, $s_A^{-1}(\{0\})= |A|$ and $s_A^{-1}((0,1))= |\mathring{N}(K,A)|\setminus|A|.$ $s_A$ restricted to $|N(K,A)|$ is called the \textit{radial parameter}. Furthermore, we have the two maps: 
	\begin{align*}
	\alpha_A: |\mathring{N}(K,A)| &\longrightarrow |A|\\
	\sum_{x \in \partial N(K,A)^{(0)}} t_x|x| + \sum_{a \in A^{(0)}} t_a|a| & \longmapsto \sum_{a \in A^{(0)}}\frac{t_a}{\sum_{a \in A^{(0)}} t_a}|a|,\\
	&
	\\
	\beta_A: |N(K,A)|\setminus|A| &\longrightarrow |\partial N(K,A)|\\
	\sum_{x \in \partial N(K,A)^{(0)}} t_x|x| + \sum_{a \in A^{(0)}} t_a|a| & \longmapsto \sum_{x \in \partial N(K,A)^{(0)}}\frac{t_x}{\sum_{x \in \partial N(K,A)^{(0)}} t_x}|x|.
	\end{align*}
	For $\xi \in |\mathring{N}(K,A)|\setminus|A|$, $(\alpha_A(\xi),\beta_A(\xi))$ is then the unique ray of $N(K,A)$ containing $\xi$. Further, the three maps are then related via:
	\begin{equation}\label{conProjectionsOfSimNbhdEqSum}
	\xi = (1-s_A(\xi))\alpha_A(\xi) + s_A(\xi)\beta_A(\xi).
	\end{equation} 
	Furthermore, we obtain a deformation retraction of $|\mathring{N}(K,A)|$ onto $|A|$ via:
	\begin{align*}
	H: |\mathring{N}(K,A)| \times \Delta^1&\longrightarrow |\mathring{N}(K,A)|\\
	(\xi,t) &\longmapsto (1-t)\xi + t \alpha_A(\xi).
	\end{align*} 
	Note that, by \Cref{lemLineSeqStrat}, this is a stratum preserving map up to $t=1$. Such a homotopy is called a nearly strict homotopy (see \cite{quinn1988homotopically}).
\end{definitionconstruction}
\subsection{Subdivisions and relative subdivisions}
Just as in the classical setting, the filtered simplicial setting has a notion of subdivision. This of course just comes down to subdividing the underlying simplicial complexes in a way compatible with the filtration.
\begin{definition}\label{defSubdivision}
	Let $K \in \textnormal{\textbf{sCplx}}_P$. A \textit{subdivision} $(K', s)$ of $K$, written $K' \vartriangleleft K$, is a filtered simplicial complex $K'$ over $P$, together with an inclusion $s:K'^{(0)} \hookrightarrow |K|$ fulfilling the following conditions:
	\begin{itemize}
		\item $s$ respects strata, that is $p_{|K|}(s(x')) = p_{K'}(x')$.
		\item For every $\sigma \in K'$ there exists a $\tau \in K$ such that $s(\sigma) \subset |\tau|.$
		\item The induced linear extension of $s$ $|K'| \to |K|$ with respect to $|K| \subset \mathbb R_{\geq 0}^{K^{(0)}}$ is a homeomorphism.
	\end{itemize}
\end{definition}
\begin{remark}\label{remEquCharFilt}
	Note that under the conditions of \Cref{defSubdivision} the induced map $|K'| \to |K|$ is automatically a stratum preserving homeomorphism. To see this, let $\sigma$ be a simplex in $K'$ and $\xi \in \mathring{\sigma}$. Let $\tau$ be the minimal simplex in $K$ such that $s(\sigma) \subset |\tau|$. Then $s(\xi) \subset \mathring{\tau}$ (compare \cite[Ch. I, Sec. 4]{spanier1989algebraic}). Now, let $x'_m$ be a \textnormal{max}imal vertex in $\sigma$ with respect to $p_{K'}$, i.e. one such that $\xi$ lies in the $p_{K'}(x'_m)$-th stratum. Then, by minimality of $\tau$, and the fact that $s$ respects strata, we also know that for a \textnormal{max}imal vertex $x_m$ of $\tau $ we get $p_K(x_m) = p_{K'}(x'_m)$. In particular, $\mathring{\tau } $ and hence also the image of $\xi$ lies in the $p_{K'}(x'_m)$-th stratum.\end{remark}
In the following, we will usually omit any mention of the map $s$ and just write $K' \vartriangleleft K$. \begin{remark}\label{remTransitive}
	Clearly, just as in the classical setting $\vartriangleleft$ behaves transitively i.e. subdivisions can be composed. That is, for filtered simplicial complexes $K'',K',K$ we have and subdivisions $K'' \vartriangleleft K', K' \vartriangleleft K$ one obtains a subdivision $K'' \vartriangleleft K$ by composing the map $K''^{(0)} \to |K'|$ with the induced stratum preserving homeomorphism $|K'|\to |K|$.
\end{remark}
Since we will frequently be subdividing simplicial complexes $K$ together with a full subcomplex $A$, the following lemma will find a lot of - albeit sometimes implicit - usage.
\begin{lemma}\label{lemFulnesusTained}
	Let $A \subset K$ be a full subcomplex of a (filtered) simplicial complex. Further, let $K' \vartriangleleft K$ be a subdivision of $K$ and $A' \vartriangleleft A$ the induced subdivision of $A$. Then $A' \subset K'$ is again a full subcomplex.
\end{lemma}
\begin{proof}
	Let $\sigma' \in K'$, such that all of its vertices lie in $A'$. Let $\sigma \in K$ be the minimal simplex supporting $\sigma'$. That is, $\sigma$ is minimal with respect to $|\sigma'| \subset |\sigma|$, where we have identitfied $|K'|$ and $|K|$ under the induced p.l. isomorphism. It suffices to show $\sigma \in A$. By minimality, either $|\sigma|=|\sigma'|$ or some vertex of $\sigma'$ is contained in the interior of $\sigma$. In the latter case, this directly implies $\sigma \in A$, and hence $\sigma'$ in $A'$. In the former, this means the vertices in $\sigma$ agree with those in $\sigma'$, and hence, by fullness of $A$, we obtain $\sigma \in A$.
\end{proof}
Recall that in \Cref{exFiltFun}, we constructed a barycentric subdivision functor $\textnormal{sd}: \textnormal{\textbf{sCplx}}_P \to \textnormal{\textbf{sCplx}}_P$.
As the name suggests the filtered barycentric subdivision actually gives a subdivision in the sense of \Cref{defSubdivision}.
\begin{proposition}\label{propBarSd}
	For $K \in \textnormal{\textbf{sCplx}}_P$, consider the map $s:\textnormal{sd}(K)^{(0)} \to |K|$ sending $\sigma $ to the barycenter of $|\sigma| \subset |K|$- We denote this by $\textnormal{bar}{(\sigma)}$. This induces a subdivision $\textnormal{sd}(K) \vartriangleleft K.$
\end{proposition}
\begin{proof}
	By the classical statement (see \cite[Ch I, Sec. 3]{spanier1989algebraic}) all one needs to show is that $s$ respects strata. This holds by definition of the filtration of $|K|$ as the barycenter of $|\sigma|$ always lies in $\mathring{\sigma}$, and thus in the stratum corresponding to $\textnormal{max}(p_K(\sigma)) = p_{\textnormal{sd}(K)}(\sigma)$.
\end{proof}
Denote by $\textnormal{sd}^n(K)$ the $n$-times application of $\textnormal{sd}$. Then as an immediate consequence of this proposition and \Cref{remTransitive} we have for $K \in \textnormal{\textbf{sCplx}}_P$ a subdivision $$\textnormal{sd}^n(K) \vartriangleleft K.$$ 
\begin{remark}\label{remSubDivGetSmall}
	As the barycentric subdivision of a filtered simplicial set is just given by putting a filtration on one of the underlying simplicial set, one obtains in particular the following classical result (Compare \cite[Ch I. Sec. 3]{spanier1989algebraic}). For a simplex $\sigma \in \textnormal{sd}^n(K) \vartriangleleft K$ denote by $d(\sigma)$ the diameter of the image of $|\sigma|$ in $|K|$, with respect to the metric induced by $|K| \subset \mathbb{R}^{K^{(0)}}$. 
	Denote by $D_n$ the supremum over the $d(\sigma)$ of $\textnormal{sd}^n(K)$. Then, if $K$ is finite, $$ D_n \xrightarrow{n \to \infty } 0.$$ In particular, by Lebesgues Lemma, for every open covering of $|K|$ there exists an $N$ such that for all $n \geq N$ in the pulled back covering on $|\textnormal{sd}^n(K)|$ every closed simplex is contained in some open set of the covering. 
\end{remark}
For the proof of the filtered simplicial approximation theorem, we will need to make use of relative barycentric subdivisions, akin to what Zeeman used in \cite{zeeman1964relative}. 
\begin{definitionconstruction}\label{constrRelSd}
	Let $K$ be a filtered simplicial complex with a full subcomplex $A$. Consider the simplicial complex given by $$ \{\sigma \star \{\tau_0 \subset ... \subset \tau_k\} \mid \sigma \star_K \tau_k \in K\} \subset A \star \textnormal{sd}(K-A).$$ For a simplex in the above complex, we have by definition that $p_K(\sigma) \cup p_K(\tau_k)$ is a flag in $P$. In particular, the same holds for $$p_K(\sigma)\cup \{\textnormal{max}(p_{K}(\tau_0)), ... , \textnormal{max}(p_K(\tau_k))\} \subset p_K(\sigma) \cup p_K(\tau_k).$$ Hence, by \Cref{remEquCharFilt} the above complex is filtered by the map induced by the disjoint union of $p_A$ and $p_{\textnormal{sd}(K-A)}$ on $A^{(0)} \sqcup \textnormal{sd}(K-A)$. We denote the induced filtered simplicial complex by $\textnormal{sd}(K\textnormal{ rel } A)$ and \textit{call it the barycentric subdivision of $K$ relative to $A$}. This comes with inclusions of filtered simplicial complexes \begin{align*}
	A &\hookrightarrow \textnormal{sd}(K\textnormal{ rel } A)\\
	\textnormal{sd}(K-A) &\hookrightarrow \textnormal{sd}(K\textnormal{ rel } A),
	\end{align*}
	and $$\textnormal{sd}(K\textnormal{ rel } A)^{(0)} = A^{(0)} \sqcup \textnormal{sd}(K-A)^{(0)}.$$ 
	\begin{figure}[H]
		\centering
		\includegraphics[width=120mm]{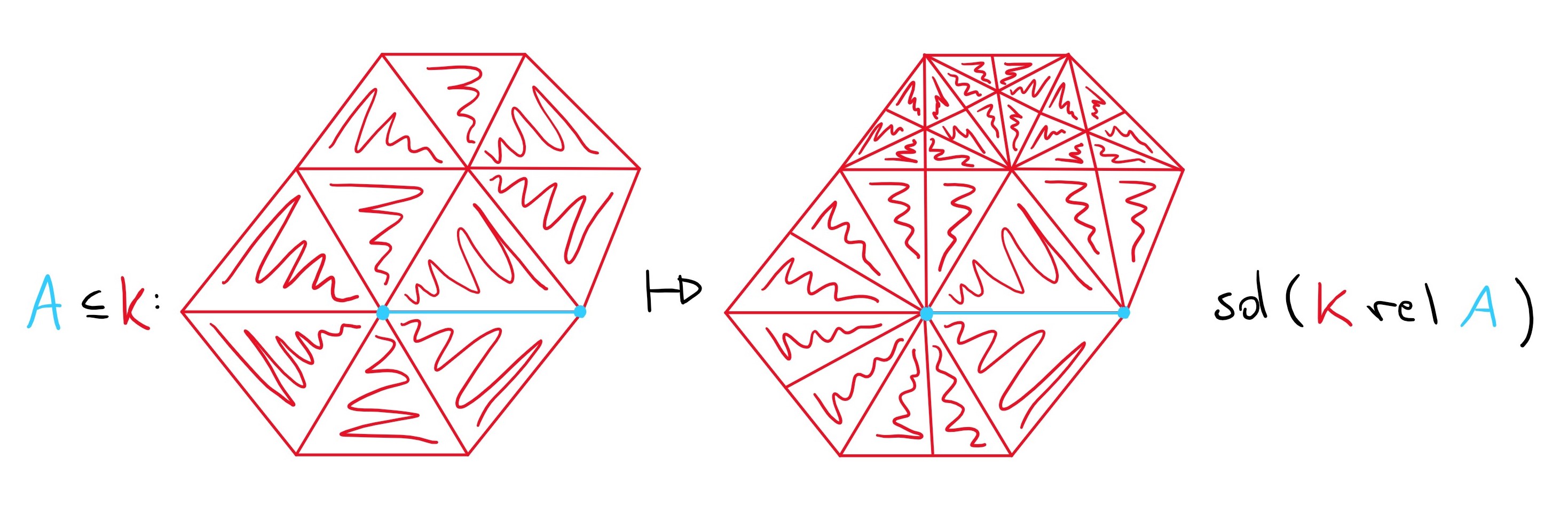}
		\caption{Illustration of a relative subdivision.}
		\label{fig:constrRelSd}
	\end{figure}
\end{definitionconstruction}
The next proposition shows that in the setting of \Cref{constrRelSd} $$\textnormal{sd}(K) \vartriangleleft \textnormal{sd}(K\textnormal{ rel } A) \vartriangleleft K.$$ 
\begin{proposition}\label{propBarRelSd}
	In the setting of \Cref{constrRelSd}, consider the maps
	\begin{align*}
	s_0:\textnormal{sd}(K)^{(0)} &\longrightarrow |\textnormal{sd}(K\textnormal{ rel } A)|\\
	\textnormal{sd}(K-A)^{(0)} \ni \tau &\longmapsto |\tau|,
	\\
	\textnormal{sd}(A)^{(0)} \ni \sigma &\longmapsto \textnormal{bar}(\sigma) \in |\sigma|,
	\\		
	\textnormal{sd}(N(K,A))^{(0)} \ni \sigma^n \star_K \tau^m &\longmapsto \frac{n+1}{m+n+2} \textnormal{bar}(\sigma) + \frac{m+1}{m+n+2}|\tau| \in |\sigma \star_K \tau|, 
	\\\intertext{for $\sigma$ and $\tau$ non-empty in the last case, and $n$ and $m$ indicating the dimensions, and} 	
	s_1: \textnormal{sd}(K\textnormal{ rel } A)^{(0)} &\longrightarrow |K|
	\\
	\textnormal{sd}(K-A)^{(0)} \ni \sigma &\longmapsto \textnormal{bar}(\sigma) \in |\sigma|,
	\\
	A^{(0)} \ni a &\longmapsto |a|.
	\end{align*}
	They induce filtered subdivisions $$\textnormal{sd}(K) \vartriangleleft \textnormal{sd}(K\textnormal{ rel } A) \vartriangleleft K.$$ The composition of these two subdivisions is the barycentric one, namely $\textnormal{sd}(K) \vartriangleleft K$ from \Cref{propBarSd}.
\end{proposition}
\begin{proof}
	The first two conditions of \Cref{defSubdivision} are easily verified. Next we show that the diagram of induced stratum preserving maps 
	$$\begin{tikzcd}
	{|\textnormal{sd}(K)|} \arrow[rd] \arrow["{\sim}", rr] & & {|K|}\\
	& {|\textnormal{sd}(K\textnormal{ rel } A)|} \arrow[ru] &
	\end{tikzcd}$$
	commutes. As all of them are given by piecewise linear extension, it suffices to show that this holds on the restriction to $|\textnormal{sd}(K)^{0}|.$ For $x$ in $\textnormal{sd}(K-A)^{(0)}$ and $\textnormal{sd}(A)^{(0)}$ this is obvious by definition. Now, in the case where $x = \sigma^n \star_K \tau^m$ for $\tau, \sigma \neq \emptyset$ we obtain: 
	\begin{align*}
	x & \longmapsto \frac{n+1}{m+n+2} \textnormal{bar}(\sigma) + \frac{m+1}{m+n+2}|\tau| \\ 
	&\longmapsto \frac{n+1}{m+n+2} \textnormal{bar}(\sigma) + \frac{m+1}{m+n+2}\textnormal{bar}(\tau)\\
	& = \textnormal{bar}(\sigma \star_K \tau) \\
	& = \textnormal{bar}(x)
	\end{align*}
	by (local) linearity of the maps. Hence, we have shown commutativity. To see that all maps involved are homeomorphisms, it suffices to restrict to the finite case as all spaces involved have the final topology with respect to finite subcomplexes. Then all the spaces involved are compact Hausdorff, hence it suffices to show bijectivity. We already know that the horizontal arrow is a bijection so it suffices to show that $|\textnormal{sd}(K \textnormal{ rel } A)| \to |K|$ is injective. This is immediate from the definition.
\end{proof}
\subsection{Proof of the filtered simplicial approximation theorem}\label{subsecProofofSimApp}
We now have all the necessary tools available for the statement and the proof of the filtered simplicial approximation theorem. We first state and prove the theorem in the case where $P = [0,...,q]$ is a finite linearly ordered set. We formally set $K_{-1} := \emptyset$. For the remainder of this section, let $K, L \in \textnormal{\textbf{sCplx}}_P$ be some fixed filtered simplicial complexes over $P$. \begin{definitionconstruction}\label{conIteratedSub}
	Let $\Sigma = ({\Sigma_p})_{p \in P}$ be a family of natural numbers indexed over $P$.
	Define inductively subdivisions of $K$ via $$\textnormal{sd}^{\Sigma_{\leq -1}}(K):= K$$ and 
	$$\textnormal{sd}^{\Sigma_{\leq p+1}}(K):= \textnormal{sd}^{\Sigma_{p}}(\textnormal{sd}^{\Sigma_{\leq p}}(K) \textnormal{ rel } \textnormal{sd}^{\Sigma_{\leq p}}(K_{\leq p})).$$
	Further, set $$\textnormal{sd}^{\Sigma}(K):= \textnormal{sd}^{\Sigma_{\leq q}}(K).$$
	Let $A$ be a full subcomplex of $K$ such that also, for all $p \in P$, $K_{\leq p} \cup A \subset K$ is a full subcomplex. Define analogously $$\textnormal{sd}^{\Sigma_{\leq -1}}(K \textnormal{ rel } A):= K$$
	$$\textnormal{sd}^{\Sigma_{\leq p+1}}(K \textnormal{ rel } A):= \textnormal{sd}^{\Sigma_{p+1}}(\textnormal{sd}^{\Sigma_{\leq p}}(K) \textnormal{ rel } \textnormal{sd}^{\Sigma_{\leq p}}(K_{\leq p} \cup A)).$$ 
	Then, set $$\textnormal{sd}^{\Sigma}(K \textnormal{ rel } A):= \textnormal{sd}^{\Sigma_{\leq q}}(K \textnormal{ rel } A).$$ Note that we have used \Cref{lemFulnesusTained} here to justify that in each step we again subdivide relative to a full subcomplex. By composability of subdivisions and \Cref{propBarRelSd} we get a sequence of subdivisions
	\begin{align*}
	\textnormal{sd}^{\Sigma}(K) &= \textnormal{sd}^{\Sigma_{\leq q}}(K) \vartriangleleft &...& &\vartriangleleft \textnormal{sd}^{\Sigma_{\leq -1}}(K) &= K;\\
	\textnormal{sd}^{\Sigma}(K \textnormal{ rel } A) &= \textnormal{sd}^{\Sigma_{\leq q}}(K \textnormal{ rel } A) \vartriangleleft &...& &\vartriangleleft \textnormal{sd}^{\Sigma_{\leq -1}}(K \textnormal{ rel } A) &= K,
	\end{align*}
	where in each step, $p+1$, there are no new vertices added to $|K|_{\leq p}$ (the union of the latter with $|A|$).
\end{definitionconstruction}
We then have:
\begin{theorem}[Filtered Simplicial Approximation A]\label{thrmSimplicialApproximation} 
	Let $K$ be a finite filtered simplicial complex over a finite linearly ordered set $P$. Further, let $L$ be another filtered simplicial complex over $P$ and $$\phi:|K|_P \longrightarrow |L|_P$$ be a stratum preserving map. Then there exists a $\Sigma$ as in \Cref{conIteratedSub} and a stratum preserving simplicial map $$f:\textnormal{sd}^{\Sigma}(K) \longrightarrow L$$ such that $$|f|_P \simeq_P \big (|\textnormal{sd}^{\Sigma}(K)|_P \xrightarrow{\sim} |K|_P \xrightarrow{\phi} |L|_P \big ).$$
\end{theorem}
We also have a relative version of this theorem. As we are really only interested in using it to approximate homotopies, we only state the existence part and not the homotopy part though. For a different version see \cite{schwartz1971}. Note however that one should beware of the caveats we mentioned in the beginning of this section when it comes to the version of the theorem in \cite{schwartz1971}. %First recall that a Vietoris-Rips complex, is a simplicial complex $K$, such that when ${x_0,...,x_k} \subset K^{(0)}$ is a set of vertices, where any two vertices are connected by an edge, then ${x_0,...,x_k}$ is a simplex of $K$.
%\begin{remark}
%Assuming that a complex is a Vietoris-Rips complex, is really not a big restriction, as the barycentric subdivision of any simplicial complex is a Vietoris-Rips complex. The advantage of Vietoris-Rips complexes is, that to define a simplicial map $K \to L$ into one, one only has to verify, that any two vertices connected by an edge in $K$ are connected by an edge in $L$.
%\end{remark} 
\begin{theorem}[Filtered Simplicial Extension]\label{thrmSimplicialExtension}
	Let $K$ be a finite filtered simplicial complex over a finite linearly ordered set $P$. Let $L$ be another filtered simplicial complex over $P$. % that is a Vietoris-Rips complex. 
	Let $A$ be a full subcomplex of $K$ such that for all $p \in P$, the complex $A \cup K_{\leq p} \subset K$ is a full subcomplex. Let $$\phi: |K|_P \longrightarrow |L|_P$$ be a stratum preserving map such that \begin{enumerate}[label = (\alph*)]
		\item $\phi$ restricted to $|A|$ is the realization of a stratum preserving simplicial map $f_A$.
		\item \label{thrmFAFAFAFAFA} $\phi$ is such that for $\sigma \in A$, we have $\phi(|\overline{\textnormal{star}}_K(\sigma)|) \subset |\overline{\textnormal{star}}_L(f_A(\sigma))|$.
	\end{enumerate}
	Then there is a $\Sigma $ as in \Cref{conIteratedSub} and a stratum preserving simplicial map $$\textnormal{sd}^\Sigma(K \textnormal{ rel }A) \xrightarrow{f} L$$ such that $f|_{A} = f_A$.
\end{theorem}
\begin{remark}\label{remOnbCond}
	First, it should be noted that the fullness conditions on the subcomplexes $A \cup X_p$ in \Cref{thrmSimplicialExtension} are really not too much of a restriction. They can always be obtained, by simply subdividing both $K$ and $L$ once barycentrically.\\
	It is interesting to illuminate condition \ref{thrmFAFAFAFAFA} of \Cref{thrmSimplicialExtension} a little bit. It arises, for example, if $\phi$ is not only a simplicial map on $A$ but also on the simplicial neighbourhood $N(K,A)$ of $A$. In the non-filtered case, (at least up to subdivision and homotopy of $\phi$ rel $|A|$) this can always be accomplished. Roughly speaking, to do this take a subdivision $K' \vartriangleleft K$ in a way that adds no new vertices to $|A|$, but vertices to the half way points (with respect to the radial parameter) to every open simplex in $|\mathring N(K,A)| \setminus |A|$. For a more detailed construction, see \cite{zeeman1964relative}. Then the identity on $|K|$ is homotopic relative to $|A|$ to a map that is given as follows. Take the identity on $|K-A|$ and then map the new vertices in $|N(K,A)|$ onto $|A|$ along a last vertex map (for some appropriate ordering). Then extend piecewise linearly. This map is given on $N(K',A)$ by a simplicial map $N(K',A) \to A$. In particular, $\phi$ is then homotopic to a map that is given on $|N(K',A)|$ by the simplicial map $$N(K',A) \longrightarrow A \xrightarrow{f_A} L.$$
	What we have done, is effectively constructing a particularly nice regular neighbourhood. In the filtered case however, these homotopies can in general not be chosen to be stratum preserving. Take for example $A = K_{\leq p}$ for some space with nontrivial filtration. Clearly, if $|K_{\leq p}|$ lies in the closure of $|K|\setminus|K_{\leq p}|$, there is no way to contract any neighbourhood of $|K_p|$ into $|K_p|$ in a stratum preserving way. Any such neighbourhood contains points outside of $|K_p|$. In other words, only very rarely are subspaces of filtered spaces NDRs in a filtered way. However, for the application we are most interested in, i.e. for approximating homotopies, the situation is a lot more favourable as we can see in the next example. 
\end{remark}
\begin{example}\label{exAppOfHo}
	Let $K$ be an ordered $P$-filtered simplicial complex and $L$ another $P$-filtered simplicial complex. For another ordered simplicial complex $M$, denote by $K \otimes M$ the filtered (ordered) simplicial complex obtained by taking the product of ordered simplicial complexes, and projecting to the first component, to obtain a filtration (see \Cref{subsecOrdered} for details). Let $H: |K \otimes \Delta^{1}| \cong |K| \otimes \Delta^1\to |L|$ be any stratum preserving homotopy, such that $H$ restricted to $|K| \sqcup |K|$ comes from a stratum preserving simplicial map $f \sqcup g: K \sqcup K \to L$. Then, consider $$| K \otimes \textnormal{sd}^2 \Delta^{1}| = |K| \otimes [0,4],$$ with the stratum preserving homeomorphism given by by thinking of $\textnormal{sd}^2 \Delta^{1}$ as being glued from $4$ $1$-simplices. The homotopy: $$\tilde H: | K \otimes \textnormal{sd}^2 \Delta^{1}| = |K| \otimes [0,4] \longrightarrow |L|$$ given by $H$ from $1$ to $3$ and the constant homotopies on the remainder, fulfills all of the requirements of \Cref{thrmSimplicialExtension}, but the fullness condition. To see this, just note that on $$|N(K \otimes \textnormal{sd}^2 \Delta^{1}, K \sqcup K)|= |K \otimes \Delta^{1} \sqcup K \otimes \Delta^{1}|$$ $\tilde H$ is given by the realization of $$K \otimes \Delta^{1} \sqcup K \otimes \Delta^{1} \longrightarrow K \sqcup K \xrightarrow{f \sqcup g} L,$$ where the left map is just the disjoint union of the projections onto $K$. Thus, by \Cref{remOnbCond}, it fulfills the requirements. We can then subdivide once barycentrically to obtain the situation of \Cref{thrmSimplicialExtension}.
\end{example}
We proceed with the proof of these two theorems in several steps. Notationally we stay in the setting of \Cref{thrmSimplicialExtension}. First we construct the simplicial map in \Cref{thrmSimplicialExtension} and show it has several additional properties. In fact, we construct a map on vertices in \Cref{propDetailedSimMapProp}, fulfilling a set of conditions that imply that is extends to a stratum preserving simplicial map as in \Cref{thrmSimplicialExtension}. Finally, we show that for $A = \emptyset$ its realization is stratum preserving homotopic to $\phi$, thus proving \Cref{thrmSimplicialApproximation}. \\
\\
For the remainder of this section, we denote$$B_{A,p} := \textnormal{sd}^\Sigma(K \textnormal{ rel }A) - \textnormal{sd}^\Sigma(K_{\leq p-1} \cup A \textnormal{ rel }A),$$ for a string $\Sigma$ as in \Cref{thrmSimplicialExtension} and $p \in P$. If $A = \emptyset$ we just write $B_p$ for this. Note that in this case $$B_p \cap \textnormal{sd}^\Sigma(K \textnormal{ rel }A)_{\leq p} = \textnormal{sd}^\Sigma (K \textnormal{ rel } A)_p.$$ Furthermore, we omit the subdivision homeomorphisms from the presentation, identifying $|K'|$ with $|K|$ to make notation a little more concise.
\begin{proposition}\label{propDetailedSimMapProp}
	There exists a $\Sigma$ and a map $$f^0: \textnormal{sd}^\Sigma(K \textnormal{ rel }A)^{(0)} \longrightarrow L^{(0)}$$ such that the following conditions hold. Let $K':=\textnormal{sd}^\Sigma(K \textnormal{ rel }A)$. Then: 
	\begin{enumerate}
		\item For $x \in K'^{(0)}$, we have $ p_{K'}(x) \leq p_L(f^0(x))$. \label{propDetailedSimMapProp1}
		%	\item For $x \in K'^{(0)}$, $\phi(x) \in |\textnormal{star}_L(f^0(x))|$.
		\item For $p \in P$ and $x \in K'^{(0)}_{\leq p} \setminus A^{(0)}$, we have $\phi(|\overline{\textnormal{star}}_{K'}(x) \cap B_{A,p}|) \subset |\textnormal{star}_{L}(f^0(x))|$. \label{propDetailedSimMapProp2}
		\item $f^0$ agrees with $f_A$ on $A^{(0)}$. \label{propDetailedSimMapProp3}
	\end{enumerate}
\end{proposition}
Before we prove this proposition, we show that it implies \Cref{thrmSimplicialExtension} to motivate the conditions a little more.
\begin{corollary}\label{corThmExtHolds}
	$f^0$ as in \Cref{propDetailedSimMapProp} extends to a morphism of filtered simplicial complexes $f: K' \to L$ fulfilling: \begin{enumerate}
		\item For $p \in P$ and $\sigma \in (K'-A)_{\leq p}$, $\phi(|\overline{\textnormal{star}}_{K'}(\sigma) \cap B_{A,p}|) \subset |\textnormal{star}_{L}(f(\sigma))|$. \label{corThmExtHolds1}
		\item $f|_A = f_A$.
	\end{enumerate}
	In particular, \Cref{thrmSimplicialExtension} holds.
\end{corollary}
\begin{proof}[Proof of \Cref{corThmExtHolds}]
	What we have to show is that $f^0$ extends to a map of simplicial complexes; that is, that for $\sigma \in K'$ $f^0(\sigma)$ is a simplex in $L$ and that furthermore (by \Cref{remEquCharFilt}) for each $x \in K'^{(0)}$ we also have $$ p_{K'}(x) \geq p_L{(f(x))}.$$ For the first condition, note that, by the equivalent characterization of when a join of simplices is a simplex in $K$ (\Cref{remFactsOnStars}), it suffices to show that certain intersections of stars in $L$ are non-empty. We start with the cases where $\sigma \in A$ or $\sigma \in K' - A$. For the former, this is obvious as we already know that $f^0$ extends to a simplicial map on $A$. For the latter case, let $p = \textnormal{max}(p_{K'}(\sigma))$. Then $\overline{\textnormal{star}}_{K'}(\sigma) \cap B_{A,p}$ is non-empty. 
	Hence, (using from \Cref{remFactsOnStars} \ref{remFactsOnStars1} and \ref{remFactsOnStars2}), together with \ref{propDetailedSimMapProp1} we also obtain that $\bigcap_{x_i \in \sigma}| \textnormal{star}(f^0(x_i))|$ is non-empty, hence by \Cref{remFactsOnStars} that $f^0(\sigma)$ is a simplex of $L$. It remains to show the case where the simplex lies in neither of the two. 
	Since $A$ is a full subcomplex of $K'$, such a simplex is of the shape $\sigma \star_{K'} \tau$, with $\tau \in K'-A$ and $\sigma \in A$. Now, let $x \in \tau $ be a vertex, maximal with respect to $p_{K'}$ and $p:=p_{K'}(x)$. 
	As we have just seen before, 
	$$\phi(|x|) \in |\textnormal{star}_{L}(f^0(\tau))|.$$ 
	At the same time $$|x| \in |\overline{\textnormal{star}}_{K'}(\sigma)| = |\overline{\textnormal{star}}_{K}(\sigma)|.$$ 
	Therefore, by the assumption \ref{thrmFAFAFAFAFA} of \Cref{thrmSimplicialExtension}, $$\phi(|x|) \in |\overline{\textnormal{star}}_L(f^0(\sigma))|.$$ 
	In particular, $|{\textnormal{star}}_L(f^0(\sigma))|$ and $|\overline{\textnormal{star}}_{L}(f^0(\tau))|$ have non-empty intersection. Hence, 
	$f^0(\sigma) \star_L f^0(\sigma)= f^0(\sigma \star_{K'} \tau )$ is a simplex of $L$ and thus $f^0$ extends to a map of simplicial complexes. \\
	\\
	Now, preservation of strata is an immediate consequence of \Cref{lemBoundOnStrata} and \ref{propDetailedSimMapProp2} of \Cref{propDetailedSimMapProp} and the fact that $\phi$ preserves strata.
\end{proof}
\begin{proof}[Proof of \Cref{propDetailedSimMapProp}]
	We define $\Sigma$ and $f$ inductively. As only values of $p'$ lesser or equal to $p$ occur in the definition of $\textnormal{sd}^{\Sigma}(-)$ it makes sense to talk about the latter, even if $\Sigma$ is only defined up to $p$. Assume inductively that we already have defined $\Sigma$ up to $p \in P$ and set $K^p:=\textnormal{sd}^{\Sigma_{\leq q}}(K \textnormal{ rel } A)$. Denote by $B^p_{r,A}$ the subcomplex of $K^p$ given by $K^p - (K^p_{\leq r} \cup A)$, for $r,p \in P$. Further, assume we have already defined $$g^{r}: K^{r,{(0)}} \longrightarrow L^{(0)}$$, for $0 \leq r \leq p$ such that $g^0$ agrees with $f_A$ on $A^{(0)}$, and that the following hold:
	\begin{enumerate}[label = (\roman*)$_r$, ref = (\roman*)$_r$]
		%\item $g^p|_{K^p_p \cap K^{p,(0)}} $ respects strata. 
		\item For $x \in \big (K_r^{r}\big )^{(0)}$, we have $p_{K^r}(x) \leq p_{L}(g^p(x)).$												\label{proofStatementsOnG1}
		%	\item For $x \in K^{p,(0)}$, $\phi(x) \in |\textnormal{star}_L(g^p(x))|$.
		\item For $x \in \big (K^{r}- A \big )^{(0)}$, we have \label{proofStatementsOnG2} $\phi(|\overline{\textnormal{star}}_{K^r}(x) \cap B^{r}_{A,r}|) \subset |\textnormal{star}_{L}(g^r(x))|.$
		\item $g^r$ agrees with $g^{r-1}$ on $K^{r,(0)}_{\leq r-1} \cup A^{(0)}= K^{r-1,(0)}_{\leq r-1} \cup A^{(0)}$, for $r \geq 1$.\label{proofStatementsOnG3}
	\end{enumerate} 
	We be drop the $(0)$ superscript for the $0$ skeleton from here on out as it will be clear when we mean a vertex or a positively dimensional simplex from usage of $x,y$ vs $\sigma, \tau$ respectively.
	For $p = 0$ consider the open covering of $|K-A|$ given by pulling back the covering of $|L|$ by the open stars $|\textnormal{star}_L(y)|$, $y \in L$. 
	By \Cref{remSubDivGetSmall}, for some sufficiently large $\Sigma_0$ we have that for each $x \in \textnormal{sd}^{\Sigma_0}(K-A)$ there exists a $y \in L$ such that $$\phi(|\overline{\textnormal{star}}_{\textnormal{sd}^{\Sigma_0}(K-A)}(x)|) \subset \textnormal{star}_L(y).$$
	Set $g^0(x)$ to such a $y$, for $ x \in \textnormal{sd}^{\Sigma_0}(K-A)$ and to $f_A(x)$ for $x \in A$. Then clearly the conditions \ref{proofStatementsOnG1}, \ref{proofStatementsOnG2} and \ref{proofStatementsOnG3}, for $r= 0$, are satisfied.\\ 
	\\
	For the following $p$, no new vertices are added to $|\mathring N(K,A)|$ and by \ref{proofStatementsOnG3} we keep $g^p = f_A$ on $A^{(0)}$. Hence, on $A$ the condition \ref{proofStatementsOnG1} is satisfied by assumption. As \ref{proofStatementsOnG2} does not depend on what is happening on $|A|$, we may as well replace $K$ by $K-A$ and hence assume $ A = \emptyset$ to make notation a little cleaner.\\
	\\ 
	For the inductive step $p$ to $p+1$, let $C$ be the subcomplex given by such simplices $\sigma \in K^{p}$ where $\phi(|\sigma|) \subset |\mathring{N}(L,L_{\leq p})| - |L_{\leq p}|$. Since $\phi$ is a stratum preserving map, $C \subset B^p_{p+1}$. Then, by \Cref{remPropOfN}, specifically \ref{remPropOfNCover} of the latter, $|C|$ is covered by the pulled back cover given by open sets of shape $|\textnormal{star}_L(\{y_p,y\})|$, for a $1$-simplex $\{y_p,y\} \in L$ with $y_p \in L_{\leq p}$ and $y \in \partial(N(L,L_{\leq p})$. Hence, for some sufficiently large $\Sigma_{p+1}$, we have that for every $x \in \textnormal{sd}^{\Sigma_{p+1}}(C) %\subset \textnormal{sd}^{\Sigma_{p+1}}(B^p_{p+1}) \subset \textnormal{sd}^{\Sigma_{p+1}}(K^p \textnormal{ rel } K^p_p)=:K^{p+1}
	$, there are $y_p \in L_p$ and a $y \in \partial N(L,L_p)$ such that $\{y_p,y\} \in L$ and 
	\begin{align}\label{proofOfDetPropCCond}
	\phi(|\overline{\textnormal{star}}_{K^{p+1}}(x) \cap \textnormal{sd}^{\Sigma_{p+1}}(C)|) \subset \textnormal{star}_L(\{y_p,y\}) \subset \textnormal{star}_L(y).
	\end{align}
	This defines $K^{p+1}$. Now, for $g^{p+1}$, consider the following cases: 
	\begin{enumerate}[label = (\alph*)]
		\item $x \in K^{p+1}_{\leq p} = K^{p}_{\leq p}$: \\
		$$g^{p+1}(x) := g^p(x).$$
		\item $x \in B^{p+1}_{p+1}$, $|x| \in \mathring{\sigma}$ for $\sigma \in K^p$, and $\sigma$ has a least one vertex, $x_0 \in \sigma $, such that $p_L \big(g^p(x_0) \big )> p $: \\
		$$g^{p+1}(x) := g^p(x_0),$$ for such a $x_0$.
		\item Neither are the case. Hence, $x \in B^{p+1}_{p+1}$. Take the $\sigma \in K^p$ such that $|x| \in \mathring{\sigma}$. Then $|\sigma| \subset |B^{p+1}_{p+1}|.$ and $g^p(\sigma)\subset L_{\leq p}$. As no new vertices were added to $|K^p|$ outside of $|B^{p+1}_{p+1}|$, we have that $|B^{p+1}_{p+1}| = |B^p_{p+1}|$. In particular, $\sigma \in B^p_{p+1} \subset B^p_p$. Hence, by \Cref{remFactsOnStars}, specifically \ref{remFactsOnStars1} and \ref{remFactsOnStars2} there together with \ref{proofStatementsOnG2} for $r=p$: $$\phi(|\sigma|) \subset \phi(|\overline{\textnormal{star}}_{K^p}(\sigma) \cap B^p_p|) \subset \bigcap_{x_i \in \sigma }|\textnormal{star}_L(g^p(x_i))| \subset |\mathring N(L, L_p)|.$$ Here, the last inclusion comes from $g^p(\sigma) \subset L_{\leq p}$. In other words, as $\phi$ preserves strata, $\sigma \in C$. Now, set $$g^{p+1}(x):=y$$ for a $y$ given by \eqref{proofOfDetPropCCond}.
	\end{enumerate}
	We need to see that this satisfies the inductive conditions (\ref{proofStatementsOnG1}, \ref{proofStatementsOnG2} and \ref{proofStatementsOnG3}), for $r=p+1$. The first condition is immediate by construction. We have not changed $g^{p+1}$ on vertices in $K^p_p$ and outside of it we have systematically set $g^{p+1}(x)$ to vertices in strata of index higher than $p$. The third condition is also obviously fulfilled. \\
	\\It remains to show \ref{proofStatementsOnG2}. For $x \in K^{p+1}_{\leq p}$, first note that as we have not added any new vertices to $|\mathring N(K^p,K^p_p)|$, $|\overline{\textnormal{star}}_{K^p}(x)| = |\overline{\textnormal{star}}_{K^{p+1}}(x)|$ and $|B^{p+1}_{p+1}| = |B^{p}_{p+1}|$. Hence, in this case \ref{proofStatementsOnG2} for $r=p$ gives the result. We are left with the case $x \in B^{p+1}_{p+1}$. So, $g^{p+1}(x)$ is either given as in (b) on in (c). \\
	\\
	In case (b), let $\sigma$ be the simplex with $|x| \in |\mathring{\sigma}|$ and $x_0 \in \sigma$ with $g^{p}(x_0) \notin L_{\leq p}$. Then $|\overline{\textnormal{star}}_{K^{p+1}}(x)| \subset |\overline{\textnormal{star}}_{K^{p}}(x_0)|$ and $|B^{p+1}_{p+1}| = |B^{p}_{p+1}| \subset |B^{p}_{p}|$. Hence, 
	$$\phi(|\overline{\textnormal{star}}_{K^{p+1}}(x) \cap B^{p+1}_{p+1}|) \subset \phi(|\overline{\textnormal{star}}_{K^{p}}(x_0) \cap B^{p}_{p}|) \subset |\textnormal{star}_{L}(g^p(x_0))| = |\textnormal{star}_{L}(g^{p+1}(x))|.$$ \\
	
	In case (c), consider $\sigma$ as in (c). For any vertex $x_i$ of $\sigma$ we have $g^{p}(x_i) \in L_{\leq p}$. In particular, by \ref{proofStatementsOnG2} , for $r=p$, and again \Cref{remFactsOnStars} we obtain: 
	$$\phi(|\overline{\textnormal{star}}_{K^{p}}(\sigma) \cap B^{p}_{p}|) \subset \bigcap_{x_i \in \sigma } \textnormal{star}_L(x_i) \subset |\mathring{N}(L,L_{\leq p})|.$$ Hence, as $\phi$ preserves strata, $|\overline{\textnormal{star}}_{K^{p}}(\sigma) \cap B^{p}_{p+1}|$ is mapped into $|\mathring{N}(L,L_{\leq p})| \setminus |L_{\leq p}|$. 
	In other words, $$\overline{\textnormal{star}}_{K^{p}}(\sigma) \cap B^{p}_{p+1} \subset C.$$ In particular, $$|\overline{\textnormal{star}}_{K^{p+1}}(x) \cap B^{p+1}_{p+1}| \subset |\overline{\textnormal{star}}_{K^{p}}(\sigma) \cap B^{p}_{p+1}| \subset |C|$$ and thereby $$|\overline{\textnormal{star}}_{K^{p+1}}(x) \cap B^{p+1}_{p+1}| = |\overline{\textnormal{star}}_{K^{p+1}}(x) \cap \textnormal{sd}^{\Sigma_{\leq p+1}}C|.$$ 
	Hence, by construction of $g^{p+1}(K)$, we have $$\phi(|\overline{\textnormal{star}}_{K^{p+1}}(x) \cap B^{p+1}_{p+1}|) \subset \textnormal{star}_L(\{y_p,y\}) \subset \textnormal{star}_L(y) = \textnormal{star}_L(g^{p+1}(x)),$$ 
	for $y$ and $y_p$ as in (c).\\
	\\
	Next we show that if we set $f^0 := g^q$ the conditions in the statement of the proposition are fulfilled. \ref{propDetailedSimMapProp3} is obvious by \ref{proofStatementsOnG3} and the start of the induction. For \ref{propDetailedSimMapProp1}, note that by construction for $p \leq q$ we have $$|B^{p}_{A,p}| = |B^{p+1}_{A,p}| = ... = |B^{q}_{A,p}| = |B_{A,p}|.$$ and for $x \in K'_{\leq p} = K^p_{\leq p}$: 
	$$|\overline{\textnormal{star}}_{K^p}(x)| = |\overline{\textnormal{star}}_{K^{p+1}}(x)| = ... = |\overline{\textnormal{star}}_{K^q}(x)| = |\overline{\textnormal{star}}_{K'}(x)|.$$
	Hence, \ref{propDetailedSimMapProp2} follows by \ref{proofStatementsOnG3} together with \ref{proofStatementsOnG2}. We are lacking that $f^0$ fulfills \ref{propDetailedSimMapProp1}. Using \ref{proofStatementsOnG3} and \ref{proofStatementsOnG1} we immediately obtain that, for each $x \in K'$, $$p_{K'}(x) \leq p_L(g^q(x)) = p_L(f^0(x)).$$ This finishes the proof.

	%%for $|\overline{\textnormal{star}}_{K^{p+1}}(x)| \cap |S^{p+1}_{p+1}| \subset |\overline{\textnormal{star}}_{K^{p}}(\sigma)| \cap |S^{p+1}_{p+1}|$ we have that $$\phi(|\overline{\textnormal{star}}_{K^{p+1}}(x)| \cap |S^{p+1}_{p+1}|) \subset \phi(|\overline{\textnormal{star}}_{K^{p}}(\sigma)| \cap |S^{p}_{p}|) \subset \phi(|\overline{\textnormal{star}}_{K^{p}}(x_0)| \cap |S^{p}_{p}|) \subset |\textnormal{star}|(g^p{(x_0)}) \subset N(L,L_{\leq p}).$$ As the left hand side does not map into $L_{\leq p}$ by filteredness of $\phi$, this means $\overline{\textnormal{star}}_{K^{p+1}}(x) \cap S^{p+1}_{p+1} \subset C$. Hence, $$|\overline{\textnormal{star}}_{K^{p+1}}(x)| \cap |
\end{proof}
To prove \Cref{thrmSimplicialApproximation} we are left with showing the following.
\begin{proposition}
	The stratum preserving simplicial map $K' \to L$ from \Cref{corThmExtHolds} is stratified homotopic to $\phi$.
\end{proposition}
\begin{proof}
		First, note that by \Cref{lemLineSeqStrat}, whenever the straight line homotopy between to stratum preserving maps can be constructed it is automatically stratum preserving. For the straight line homotopy between a continuous map between realizations of simplicial complexes $\psi: |\tilde K| \to |\tilde L|$ and the realization of a simplicial map $g: \tilde K \to \tilde L$ to exist it suffices that, for any simplex $\sigma \in \tilde K$, the condition $$\psi(\mathring \sigma) \subset |\overline{\textnormal{star}}_L(g(\sigma))|$$ is fulfilled. In particular, by \ref{corThmExtHolds1} of \Cref{corThmExtHolds}, we can use the straight line homotopy on $|K'_q \sqcup ... \sqcup K'_0|$. We now inductively extend this homotopy over $$\mathcal{K}^p:=|B_p \sqcup K'_{p-1} \sqcup K'_{p-2} \sqcup ... \sqcup K'_0|,$$ going from $p=q$ to $p=0$. We induce over the following assumption. There exists a stratum preserving homotopy $$F^p:\mathcal K^{p} \otimes \Delta^1\longrightarrow |L|$$ between the restrictions of $\phi$ and $|f|$, fulfilling additionally the following smallness condition. For $r \leq p$, $\sigma \in K'_r$ and $\xi \in |\overline{\textnormal{star}}_{K'}(\sigma) \cap B_r| \cap \mathcal{K}^p$ we also have \begin{equation}\label{proofHomotopyIndCond}
	F^p(\{\xi\} \times [0,1)) \subset |\textnormal{star}_L(f(\sigma)|.
	\end{equation}
	Note how, by \Cref{lemLineSegmentsInStars} and \ref{corThmExtHolds1} of \Cref{corThmExtHolds}, this condition always holds on $\mathcal{K}^q$ for any $p$ as then $$F^p(\{\xi\} \times [0,1)) = [\phi(\xi), |f|(\xi)),$$ which is a half open line segment starting in $|\textnormal{star}_L(f(\sigma)|$. This gives the start of the induction. Now, from $p+1$ to $p$, note that \begin{align*}
	\mathcal{K}^{p} \setminus \mathcal{K}^{p+1} &= (|\mathring N(K',K'_{\leq p})| \setminus |K'_{\leq p}|) \cap |B_{p}|\\
	&= |\mathring{N}(B_p,K'_p)| \setminus |K'_p|.
	\end{align*}
	In particular, by \Cref{remPropOfN}, $\mathcal{K}^{p} - \mathcal{K}^{p+1}$ is given by the disjoint union of rays $(\alpha, \beta)$ of $N(B_p,K'_p)$, with $\alpha \in |K'_p|$ and $\beta \in |B_p - K'_p| = |B_{p+1}|$. Now, consider the maps \begin{align*}
	\alpha_{K'_p}: |\mathring{N}(B_p,K'_p)| \longrightarrow |K'_p|;\\
	\beta_{K'_p}: |N(B_p,K'_p)| \setminus |K'_p| \longrightarrow |B_{p+1}|
	\end{align*}
	from \Cref{conProjectionsOfSimNbhd}, sending a $\xi \in \mathcal{K}^{p} - \mathcal{K}^{p+1}$ to the respective end points of the ray it is contained in. We omit the subscript $K'_p$ from here on out. Now, define: \begin{align*}
	\lambda: (\mathcal{K}^p \setminus \mathcal{K}^{p+1}) \times \left[0,2\right] &\longrightarrow |L|\\
	(\xi, t) &\longmapsto \phi((1-t) \alpha(\xi) + t\beta(\xi))\textnormal{ for $t\leq 1$},\\
	(\xi, t) &\longmapsto F^{p+1}(\beta(\xi), t-1) \textnormal{ for $t\geq 1$}.
	\end{align*}
	Note that for the radial parameter map $$s_{K'_p}: |\mathring{N}(B_p,K'_p)| \longrightarrow [0,1]$$ from \Cref{conProjectionsOfSimNbhd} (where we again omit the subscript from here on out) we have $$\lambda(\xi, s(\xi)) = \phi(\xi).$$
	Now, for a fixed $\xi \in \mathcal{K}^{p} \setminus \mathcal{K}^{p+1}$, let $r \leq p$ and $\sigma \in K'_p$ be such that $\xi \in |\overline{\textnormal{star}}_{K'}(\sigma) \cap B_r|.$ Then, as $\xi \in (\alpha(\xi), \beta(\xi))$, we also have that $$(\alpha(\xi), \beta(\xi)) \subset |\overline{\textnormal{star}}_{K'}(\sigma) \cap B_r|$$ and hence that $$[\alpha(\xi), \beta(\xi)] \subset |\overline{\textnormal{star}}_{K'}(\sigma) \cap B_r|.$$ Consequently, by \ref{corThmExtHolds1} of \Cref{corThmExtHolds}, we obtain $$\phi([\alpha(\xi),\beta(\xi)]) \subset \textnormal{star}_L(f(\sigma)).$$ By the induction hypothesis, we also have $$F^{p+1}(\{\beta(\xi)\} \times [0,1)) \subset |\textnormal{star}_L(f(\sigma))|.$$ In particular, we obtain \begin{equation}\label{proofHomotopyEqLambdaStar}
	\lambda(\{\xi\}\times [0,2)) \subset |\textnormal{star}_L(f(\sigma))|\end{equation} and 
	$$ \lambda(\{\xi\}\times [0,2]) \subset |\overline{\textnormal{star}}_L(f(\sigma))|.$$
	Such a $\sigma$ and $r$ always exist as $\xi \in |N(B_p,K'_p)|$. $\sigma$ can be taken to be the simplex in $K'_p$ such that $\alpha(\xi) \in \mathring{\sigma}$. In this case, we also have $$|f|(\alpha(\xi)) \in \mathring{f(\sigma)}.$$ In particular, the following map defined by linearly interpolating between $\lambda(\xi, t)$ and $|f|(\alpha(\xi))$ is well-defined.
	\begin{align*} 
	\Lambda: (\mathcal{K}^p \setminus \mathcal{K}^{p+1}) \times \left[0,2\right] \times [0,1] & \longrightarrow |L|\\
	(\xi, t, s) &\longmapsto (1-s)|f|(\alpha(\xi)) + s\lambda(\xi,t)
	\end{align*}
	Then, by \Cref{lemLineSegmentsInStars} and \eqref{proofHomotopyEqLambdaStar}, we obtain that, for a $\sigma$ as in \eqref{proofHomotopyIndCond}, we always have \begin{equation}\label{proofHomotopyEqLLambdaStar}
	\Lambda(\{\xi\} \times [0,2) \times (0,1] ) \subset |\textnormal{star}_L(f(\sigma))|.
	\end{equation}
	Furthermore, by \Cref{lemLineSeqStrat}, for $\xi \in \mathcal{K}^p \setminus \mathcal{K}^{p+1}$, $(\alpha(\xi), \beta(\xi)]$ lies in the same stratum as $\beta(\xi)$. Denote its index by $p'$. Then, again by \Cref{lemLineSeqStrat}, the fact that $\Lambda$ is constructed by linear interpolation and the assumption that $F^{p+1}$ is stratum preserving, we obtain that 
	\begin{equation}\label{proofHomotopyEqLLambdaFiltered}
	\Lambda(\{\xi\} \times (0,2] \times (0,1]) \subset |L|_{p'}
	\end{equation}
	Now, let $$\kappa:\sfrac{[0,2] \times [0,1]}{[0,2] \times \{0\}} \longrightarrow [0,1] \times [0,1],$$ such that $\kappa$ restricts to the linear homeomorphisms (where we indicate the orientation by the ordering of the interval):
	\begin{align*}
	[0,1] \times \{1\} &\longrightarrow \{0\} \times [0,1],\\
	[1,2] \times \{1\} &\longrightarrow [0,1] \times \{1\},\\
	\{2\} \times [1,0] &\longrightarrow \{1\} \times [1,0],\\
	\{0\} \times [0,1] &\longrightarrow [1,0] \times \{0\}.
	\end{align*}
	As $\Lambda$ is constant along $\{\xi\} \times [0,2] \times \{0\}$ and $(\mathcal{K}^p \setminus \mathcal{K}^{p+1})$ is locally compact (hence quotients commute with products with the latter) $\Lambda$ descends to a map $$\bar \Lambda: (\mathcal{K}^p \setminus \mathcal{K}^{p+1}) \times \big (\sfrac{[0,2] \times [0,1]}{[0,2] \times \{0\}} \big ) \longrightarrow |L|$$ and then by composing with an inverse to $1_{\mathcal{K}^p \setminus \mathcal{K}^{p+1}} \times \kappa$ induces a map: $$ \tilde \Lambda: (\mathcal{K}^p \setminus \mathcal{K}^{p+1}) \times I^2 \longrightarrow |L|$$ such that: 
	\begin{equation}\label{proofHomotopyEqTildeLambda}
	\tilde \Lambda(\xi,t,s) = 
	\begin{cases}
	\phi((1-s) \alpha(\xi) + s\beta(\xi)) & \text{, if } t = 0 \\
	F^{p+1}(\beta(\xi),t) & \text{, if } s= 1 \\
	(1-s)|f|(\alpha(\xi)) + s|f|(\beta(\xi)) & \text{, if } t= 1 \\
	F^{p+1}(\alpha(\xi),t) & \text{, if } s=0
	\end{cases}
	\end{equation}
	Now, set $F^{p}$ on $(\mathcal{K}^p \setminus \mathcal{K}^{p+1}) \times I$ to $$
	\begin{tikzcd}[row sep = 0]
	(\mathcal{K}^p \setminus \mathcal{K}^{p+1}) \times \Delta^1\arrow[r]& (\mathcal{K}^p \setminus \mathcal{K}^{p+1}) \times I^2 \arrow[r, "\tilde \Lambda"]&{|L|} \\
	(\xi, t) \arrow[r, mapsto]& (\xi, t, s(\xi)) \arrow[r, mapsto] &\tilde{\Lambda}(\xi, t, s(\xi)).
	\end{tikzcd}.$$
	By \eqref{conProjectionsOfSimNbhdEqSum} of \eqref{conProjectionsOfSimNbhd} and the piecewise linearity of $|f|$, we obtain that this in fact gives a homotopy between
	$\phi|_{\mathcal{K}^p \setminus \mathcal{K}^{p+1}}$ and
	$|f||_{\mathcal{K}^p \setminus \mathcal{K}^{p+1}}$. (For an illustration, see \Cref{fig:Lambda}). If $\xi$ converges to $\xi_0$ in the boundary of $\mathcal K ^p \setminus \mathcal K^{p+1}$, then $s(\xi)$ converges either to $1$ or to $0$ and $\xi$ to either $\beta(\xi_0)$ or $\alpha(\xi_0)$ respectively. Hence, again by \eqref{proofHomotopyEqTildeLambda}, we obtain a continuous extension of $F^{p+1}$ to $\mathcal{K}^p$. We need to check that this homotopy is in fact stratum preserving, and that the inductive condition \eqref{proofHomotopyIndCond} holds. Let $\xi \in \mathcal{K}^p \setminus \mathcal{K}^{p+1}$. Under the map $$(\mathcal{K}^p \setminus \mathcal{K}^{p+1}) \times \Delta^1\longrightarrow (\mathcal{K}^p \setminus \mathcal{K}^{p+1}) \times I^2$$ ${\xi} \times [0,1)$ maps into $ \{\xi\} \times [0,1) \times (0,1).$ Under the inverse of $1 \times \kappa$ the latter maps into $\{\xi\} \times (0,2) \times (0,1) \cup [0,1] \times \{1\}.$ Hence, the inductive smallness condition follows by \eqref{proofHomotopyEqLLambdaStar}. Analogously, $\{\xi\} \times [0,1]$ is mapped into $(0,1] \times [0,1]$ which in turn is mapped into $(0,2] \times (0,1]$. Hence, by \eqref{proofHomotopyEqLLambdaFiltered}, we conclude that $F^{p}$ is stratum preserving.
\end{proof}
\begin{figure}[H]
	\centering
	\includegraphics[width=120mm]{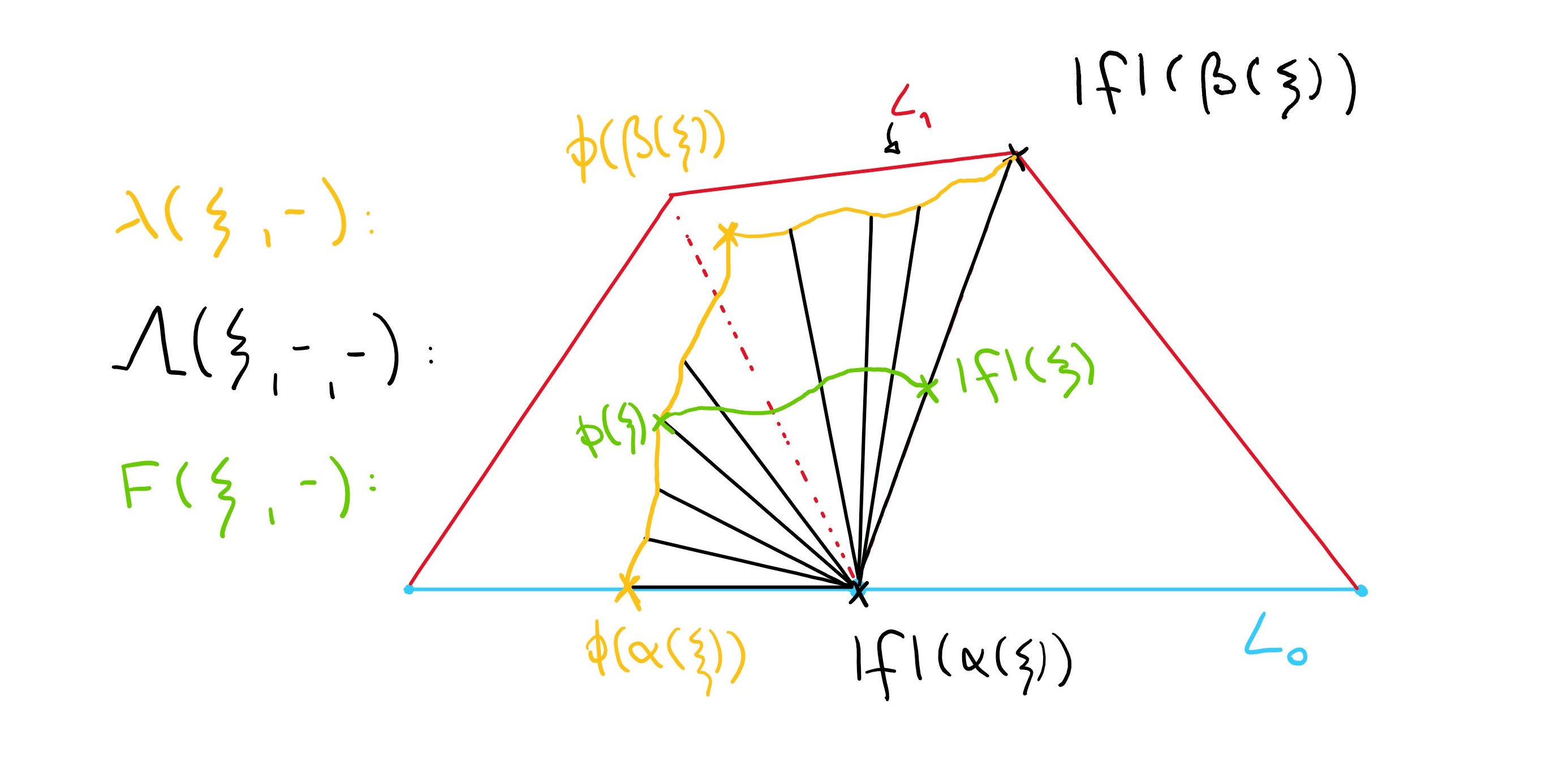}
	\caption{An illustration of $\lambda$, $\Lambda$ and $F$, marked respectively in different colours. }
	\label{fig:Lambda}
\end{figure}
\subsection{Filtered ordered simplicial complexes}\label{subsecOrdered}
Recall that an ordered simplicial complex is a simplicial complex $K$ together with a linear ordering on every simplex $\sigma \in K$, compatible in the sense that if $\tau \subset \sigma$, then the ordering on $\tau$ is the restriction of the one in $\sigma$. We usually write simplices in an ordered simplicial complex in the shape $\{x_0 \leq ... \leq x_k \}$. Further, recall that a map of ordered simplicial complexes is a map of simplicial complexes restricting to a monotonous map on each simplex. In \Cref{exFilteredObj} we defined the category of $P$-filtered ordered simplicial complexes, $\textnormal{\textbf{sCplx}}^{\operatorname{o}}_P$. As we did for the several other categories of filtered objects we introduced, we begin this subsection by giving some alternative descriptions of this category. 
\begin{remark}\label{remDeOSC}
	Similarly to \Cref{remDeSC}, a $P$-filtered ordered simplicial complex is alternatively described as a simplicial complex $K$ together with a map $p_K^{0}:K^{(0)} \to P$ fulfilling some condition. In the non-ordered case, this condition was that the image of every simplex is a flag in $P$. In the ordered case, we need to reflect the fact that $p_K: K \to \textnormal{sd}(P)$ is a map of ordered complexes. This is equivalently specified by: \begin{align*}
		(x < y) \in K \implies p^0_K(x) \leq p^0_K(y).
	\end{align*}
	In other words, this requires that the ordering on $K$ is such that vertices in lower strata always come before vertices in higher strata.
	A morphism of ordered filtered simplicial complexes $f:K \to L$ is equivalently a map of the underlying ordered simplicial complexes, fulfilling $p_K(x) = p_L\big(f(x) \big ).$ This justifies calling such simplicial maps \textit{stratum preserving}.
\end{remark}
%\\
%For a partially ordered set $P$, $\textnormal{sd}(P)$ is an ordered simplicial complex in the obvious way. Hence, we can define \textit{ordered filtered simplicial complexes} and \textit{maps of ordered filtered simplicial complexes} just as in the non-ordered (see \Cref{defFiltSimCx}) case, replacing every simplicial complex and every simplicial map by their ordered analogues. Denote the category of filtered ordered simplicial complexes over $P$ with morphisms given by maps of ordered filtered simplicial complexes (over $P$), by $\textnormal{\textbf{sCplx}}^{\operatorname{o}}_P$. This comes with the obvious forgetful functor $$\textnormal{\textbf{sCplx}}^{\operatorname{o}}_P \longrightarrow \textnormal{\textbf{sCplx}}_P.$$ 
The advantage of the ordered over the unordered category is that it embeds fully faithfully into $\textnormal{s\textbf{Set}}_P$. We use this to apply the simplicial approximation theorems \Cref{thrmSimplicialApproximation} and \Cref{thrmSimplicialExtension} in the setting of filtered simplicial sets.
%\begin{remark}
%$\textnormal{\textbf{sCplx}}^{\operatorname{o}}$ (the category of ordered complexes over a point, i.e. the non-filtered one) has all colimits. Coproducts are constructed by just taking the disjoint union of simplicial complexes and inheriting the order. Pushouts are given, as follows. For a diagram in $\textnormal{\textbf{sCplx}}^{\operatorname{o}}$ 
%\begin{center}
%	\begin{tikzcd}
%	K_0 \arrow[d, "f_1"]\arrow[r, "f_2"] & K_1\\
%	K_2 & {}
%	\end{tikzcd}
%\end{center}
%construct $K$ by setting $$K^{(0)}:= \sfrac{K_1^{(0)} \sqcup K_2^{(0)}}{\sim_K},$$ where $\sim_K$ is the equivalence relation generated by:
%	\begin{itemize}
%		\item For $x \in K_0^{(0)}:\\ f_1(x) \sim_K f_2(x)$.
%		\item For $x,y \in K_0^{(0)} \textnormal{ s.t. } f_1(x) \leq f_1(y), f_2(y) \leq f_2(x):\\ f_1(x) \sim_K f_1(y) \textnormal{ and } f_2(x) \sim_K f_2(y)$ 
%	\end{itemize}
%Now, equip this set with the smallest simplicial complex structure, such that both maps $K^{(0)}_i \to K^{(0)}$ induce a map of simplicial complexes. By the second relation we have modded out, the linear orders on the simplices of $K^0$ and $K^1$ then descend to compatible linear orders on the simplices of $K$. 
%\end{remark}
\begin{definitionconstruction}\label{conEmbedFOS}
Consider the functor 
\begin{align*}
	S^o:\textnormal{\textbf{sCplx}}^{\operatorname{o}} &\longrightarrow \textnormal{s\textbf{Set}}\\
	K &\longmapsto \big \{ [n] \mapsto \textnormal{Hom}_{{\textbf{sCplx}}^o}(\Delta^n,K) \big \}
\end{align*} 
with the obvious functoriality on morphisms and $\Delta^{n}$ thought of as the ordered simplicial complex given by flags of $[n]$.\\
Simplicial sets of the form $S^o(K)$ have the property that every simplex $\sigma \in X([k])$ is uniquely determined by its multiset of vertices, that is by $(x_i)_{i \in [k]}$ allowing for permutations. This follows immediately from the definition of a map of (ordered) simplicial complexes. It turns out that this property completely describes simplicial sets that are isomorphic to $S(K)$, for some $K \in \textnormal{\textbf{sCplx}}$.\\ To see this, denote by $\mathcal C$ the full subcategory of simplicial sets fulfilling the latter property. Consider the functor from $\textnormal{s\textbf{Set}}$ to $\textnormal{\textbf{sCplx}}$ constructed as follows. For $X\in \textnormal{s\textbf{Set}}$ define a simplicial complex with vertices $X([0])$, by taking the simplices to be such subsets $\{x_0,...,x_k\} \subset X_0$ that are the set of vertices of a simplex $\sigma \in X([k])$. This construction becomes functorial by checking that for a map of simplicial set $f:X \to Y$ the induced map $f([0]):X([0]) \to Y([0])$ induces a map of simplicial complexes. We obtain a functor \begin{align*}
	 C: \textnormal{s\textbf{Set}} \longrightarrow \textnormal{\textbf{sCplx}}.
\end{align*}
In general, the image of this functor does not naturally carry the structure of an ordered simplicial complex. This is due to the fact that two vertices in a simplicial set might be connected by two $1$-simplices with different orientation. However, in the case where $X \in \mathcal C$, the relation given on $X([0])$ by $$x \leq y \iff x= d_0(\sigma), y= d_1(\sigma)$$ for some $\sigma \in X([1])$, turns $C(X)$ into an ordered simplicial complex as then no two conflicting orders occur on any simplex. For $f: X \to Y \in \mathcal C$, $C(f)$ is then a map of ordered simplicial complexes, with respect to this additional structure.
Hence, we obtain a functor:
\begin{align*}
	C^o: \textnormal{s\textbf{Set}} \longrightarrow \textnormal{\textbf{sCplx}}^{\operatorname{o}}.
\end{align*}
We then have the following proposition.
\end{definitionconstruction}
\begin{proposition}\label{propCharFOS}
	In the setting of \Cref{conEmbedFOS}, the functor $S$ is a fully faithful embedding of categories. An inverse of its restriction in the image to $\mathcal C$ is given by $C^o$. 
\end{proposition}
\begin{proof}
	By a slight abuse of notation, we denote the factorization of $S$ into $\mathcal C$ also by $S$. It is a straightforward verification that $1_{\textnormal{\textbf{sCplx}}^{\operatorname{o}}} \cong C \circ S$. For the other composition, we obtain a natural transformation to the identity by the map $$\textnormal{Hom}_{\textnormal{\textbf{sCplx}}^{\operatorname{o}}}(\Delta^n, C^o(X)) \to X([n]),$$ sending a map $\sigma$ on the left to the unique simplex in $X$ given by the multiset given by $(\sigma(\{i\})_{i \in [n]}$. It suffices to show, this is a bijection. This is immediate, from the fact that $X \in \mathcal C$.
\end{proof}
Further, note that $S^o(\textnormal{sd}(P))$ is naturally isomorphic to $N(P)$. In particular, as $\textnormal{\textbf{sCplx}}^{\operatorname{o}}_P$ and $\textnormal{s\textbf{Set}}_P$ are given by over-categories of these respectively, we obtain:
\begin{corollary}\label{corEmbedFos}
$S^o$ from \Cref{conEmbedFOS} induces a fully faithful embedding $$\textnormal{\textbf{sCplx}}^{\operatorname{o}}_P \hookrightarrow \textnormal{s\textbf{Set}}_P.$$ Filtered simplicial sets isomorphic to spaces in the image of this embedding are precisely those $X \in \textnormal{s\textbf{Set}}_P$ where every $n$-dimensional simplex $\sigma: \Delta^n \to X$ is uniquely determined by the multiset of its vertices, that is by the family $(\sigma(\{i\}))_{i \in [n]}$ modulo permutation.
\end{corollary}
Using this fact, we will often think of filtered ordered simplicial complexes as filtered simplicial sets. In particular, the following definition makes sense:
\begin{definition}\label{defFOSaSS}
Let $X \in \textnormal{s\textbf{Set}}_P$ be such that every $n$-dimensional simplex $\sigma: \Delta^n \to X$ by its' multiset of vertices. Then, we call $X$ a $P$-\textit{filtered ordered simplicial complex}, or ($P$-)\textit{FOS-complex} for short. We nearly always omit the $P$.
\end{definition}
It is easy to see that this definition is also compatible with realization functors, i.e. that there are natural isomorphism between 
$$\mathcal C \xrightarrow{C^o} \textnormal{\textbf{sCplx}}^{\operatorname{o}}_P \to \textnormal{\textbf{sCplx}}_P \xrightarrow{|-|_P} \textnormal{\textbf{Top}}_{P}$$ and 
$$|-|_P: \textnormal{s\textbf{Set}}_P \to \textnormal{\textbf{Top}}_{P}.$$ Where the functor $\textnormal{\textbf{sCplx}}^{\operatorname{o}}_P \to \textnormal{\textbf{sCplx}}_P$ is the forgetful functor forgetting about the orderings on simplices. Hence, we do not distinguish between these two ways of realizing FOS-complexes, and denote them by $|-|_P$, or just $|-|$, for the sake of notational brevity.\\
\\ 
It can also be useful to have the following alternative characterization of an FOS-complex.
\begin{lemma}\label{lemCharFOS}
	A $P$-filtered simplicial set is an FOS-complex if and only if every non-degenerate simplex is uniquely determined by its set of vertices.
\end{lemma} 
\begin{proof}
	This has nothing to do with the filtrations, so let us assume they are trivial.
	First, note that an FOS-complex has no non-degenerate simplices with degenerate faces. In particular, every non-degenerate simplex of dimension $n$-has precisely $n$ different vertices. Thus, being an FOS-complex implies the other property. Now, let $X$ be a simplicial set satisfying the latter. Note that for such an $X$ every nondegenerate simplex of dimension $n$ has precisely $n$ different vertices. We can always take a maximal face of such a simplex without repetitions in the vertices. Such a face is non-degenerate and has the same underlying vertex set as the original simplex. Hence, it is the simplex. Next, let $\sigma$ and $\tau$ be two possibly degenerate simplices with the same multiset of vertices. Then, by assumption, the two simplices they degenerate from agree and have no repeating vertices. The respective degeneracy maps giving $\sigma$ and $\tau$ are already completely specified by the multiplicities in their corresponding multisets. Hence, as the latter agree, $\sigma$ and $\tau$ come from the same nondegenerate simplex and through the same degeneracy map, in other words $\sigma= \tau$.
\end{proof}
We finish this subsection with a type of ``dense up to homotopy equivalence'' result, for the embedding $S^o$. It is an immediate consequence of \Cref{thrmSSvSC}. The reader worried about circularity may be assured that we only use this result in \Cref{secDetHTop}. The largest part of the second chapter up to \Cref{secTopWh} is completely independent from this result. 
\begin{proposition}\label{propSSvSC}
	Every finite filtered simplicial set $X \in \textnormal{s\textbf{Set}}_P^{fin}$ is isomorphic in $\mathcal H\textnormal{s\textbf{Set}}_P$ to a finite FOS-complex of the same dimension.
\end{proposition} 
In particular, many of the homotopy theoretic questions on filtered simplicial sets can be reduced to questions on FOS-complexes.
\subsection{Last vertex maps of relative subdivisions}
Another advantage, of FOS-complexes over nonordered ones, is that their subdivisions admit last vertex maps. This provides a way to represent subdivision homeomorphisms up to stratum preserving homotopy. We have already seen such constructions for filtered simplicial sets in \Cref{subsecPsset}. 
The fully faithful inclusions of categories $$S^o: \textnormal{\textbf{sCplx}}^{\operatorname{o}}_P \hookrightarrow \textnormal{s\textbf{Set}}_P$$ is compatible with subdivisions in the following sense.
\begin{definitionconstruction}
	Let $K$ be any filtered simplicial complex. Then $\textnormal{sd}(K)$ is naturally an FOS-complex, with the order on vertices given by inclusion. Hence, we can also think of $\textnormal{sd}$ as an endofunctor of $\textnormal{\textbf{sCplx}}^{\operatorname{o}}_P$, as well as as a functor $$\textnormal{\textbf{sCplx}}_P \longrightarrow \textnormal{\textbf{sCplx}}^{\operatorname{o}}_P.$$ It is not hard to see that there is a natural isomorphism $$\textnormal{sd}(S^o(K)) \cong S^o(\textnormal{sd}(K)).$$ Indeed, first note that the two definitions are clearly compatible on simplices and then extend by the colimit definition of $\textnormal{sd}$ for simplicial sets. Hence, this definition is also compatible with simplicial set version of subdivision. Finally, consider the last vertex map $$\textnormal{sd}(K) \longrightarrow K $$ given on vertices by $$\{ x_0 \leq ...\leq x_k \}\mapsto x_k.$$ This induces a natural transformation: 
	$$\textnormal{sd} \longrightarrow 1_{\textnormal{\textbf{sCplx}}^{\operatorname{o}}_P}.$$
	Again, it is easy to see that this is compatible with the definition of the last vertex transformation $sd \longrightarrow 1_{\textnormal{s\textbf{Set}}_P}$ under the inclusion of categories given by $S^o$.
\end{definitionconstruction}
In the ordered setting, we can also equip the relative versions of subdivision, \Cref{constrRelSd}, with last vertex maps.
\begin{definitionconstruction}\label{conLVLRel}
Let $K$ be an FOS-complex and $A$ a full subcomplex such that whenever two vertices $x \in K-A$ and $a \in A$ lie in a common simplex we have $a < x$. Then the relative subdivision $\textnormal{sd}(K\textnormal{ rel }A)$ becomes an FOS-complex (over $P$), by taking the induced ordering on $\textnormal{sd}(K-A) \sqcup A$ and setting $a < \sigma$, whenever two vertices $a \in A^{(0)}$ and $\sigma \in \textnormal{sd}(K-A)^{(0)}$ lie in a common simplex. Now, the last vertex map $\textnormal{l.v.}: \textnormal{sd}(K) \to K$ factors into two maps of FOS-complexes, \begin{equation*}
	\textnormal{sd}(K) \xrightarrow{l_0} \textnormal{sd}(K\textnormal{ rel }A) \xrightarrow{l_1} K,
\end{equation*}
fitting into the commutative diagram:
\begin{center}
\begin{tikzcd}
	\textnormal{sd}(A) \arrow[d, hook', swap] \arrow[r, "\textnormal{l.v.}"]& A \arrow[r, "\textnormal{1}"] \arrow[d, hook', swap]& A \arrow[d, hook', swap]\\
	\textnormal{sd}(K) \arrow[r, "l_0"]& \textnormal{sd}(K\textnormal{ rel }A) \arrow[r, "l_1"]& K\\
	\textnormal{sd}(K-A) \arrow[u, hook] \arrow[r, "1"]& \textnormal{sd}(K-A) \arrow[r, "\textnormal{l.v.}"] \arrow[hook, u]& K-A \arrow[hook, u] 
\end{tikzcd}
\end{center}
The maps are constructed as follows.\\
$l_0$ is defined on vertices via \begin{align*}
	\textnormal{sd}(K-A)^{(0)} &\ni &\tau &\longmapsto \textnormal{ }\tau \\
	\big (\textnormal{sd}(N(X,A))-\textnormal{sd}(A) \big )^{(0)} & \ni &\sigma \star \tau &\longmapsto \textnormal{ }\tau \\
	\textnormal{sd}(A)^{(0)} &\ni &\sigma &\longmapsto \textnormal{ }\textnormal{l.v.}(\sigma).
	\end{align*}
	This maps a simplex $\{ \sigma_0 \subset ... \subset \sigma_k \subset \sigma_{k+1} \star \tau_{1} \subset ... \subset \sigma_{k+l} \star \tau_{l} \}$, with $\sigma_i \in A$ and $\tau_j \in K-A$, to $$\{ \textnormal{l.v.}(\sigma_0) \leq ... \leq \textnormal{l.v.}(\sigma_k) \leq \tau_1 \leq ... \leq \tau_l \}=\{\textnormal{l.v.}(\sigma_0) \leq ... \leq \textnormal{l.v.}(\sigma_k)\}\star\{\tau_1 \leq ... \leq\tau_l\}.$$ As $\{\textnormal{l.v.}(\sigma_0) \leq ... \leq\textnormal{l.v.}(\sigma_k)\} \subset \sigma_{k+l} \star \tau_k$ this actually defines an (ordered) simplex in $\textnormal{sd}(K\textnormal{ rel }A)$. It is clearly filtered, as we have assumed that vertices in $A$ precede those in $K$ in the ordering. $l_1$ is given on vertices by:
	\begin{align*}
	\textnormal{sd}(K-A)^{(0)} \ni \tau &\longmapsto \textnormal{ } \textnormal{l.v.}(\tau) \\
	A^{(0)} \ni \sigma &\longmapsto \sigma,
	\end{align*}
	which clearly preserves strata.
	It maps a simplex $\sigma \star \{\tau_0 \leq ... \leq \tau_k\}$ to $$\sigma \star \{\textnormal{l.v.}(\tau_0) \leq ... \leq \textnormal{l.v.}(\tau_k)\}.$$ Again, by construction of the FOS-complex $\textnormal{sd}(K\textnormal{ rel }A)$ and the ordering assumption on $A\subset K$, this defines an ordered simplex.
\end{definitionconstruction}
\begin{proposition}\label{propLareHoToSd}
	In the setting of \Cref{conLVLRel}, the realizations of the two arrows $l_0$ and $l_1$ are stratified homotopic to the filtered subdivision homeomorphisms from \Cref{propBarRelSd}.
\end{proposition}
\begin{proof}
	This is immediate through a use of straight line homotopies.
\end{proof}
We now want to iterate this construction. For this we need: 
\begin{lemma}\label{lemOrderedFullSustain}
	In the setting of \Cref{conLVLRel} let $A' \subset K$ be another full subcomplex fulfilling the ordering condition and containing $A$. Then the subcomplex $\textnormal{sd}(A'\textnormal{ rel }A) \subset \textnormal{sd}(K\textnormal{ rel }A)$ satisfies the requirements of \Cref{conLVLRel} as well.
\end{lemma}
\begin{proof}
	We have already seen in \Cref{lemFulnesusTained} that fullness is a property that is preserved under subdivision. Therefore, it suffices to check the ordering condition. Let $a' \in \textnormal{sd}(A'\textnormal{ rel }A)$ and $x \in \textnormal{sd}(K'\textnormal{ rel }A) - \textnormal{sd}(A'\textnormal{ rel }A)$ be two vertices contained in a common simplex. In case $a \in A$, the result is immediate from the construction of the ordering on $\textnormal{sd}(K\textnormal{ rel }A)$, as $A \subset \textnormal{sd}(A'\textnormal{ rel }A) .$ If not, then both lie in $\textnormal{sd}(K-A)$. But here, the ordering relation holds for any, subcomplex $\textnormal{sd}(A'')$, $A'' \subset K$. This follows, by the definition of the ordering on $\textnormal{sd}(K)$ via inclusions of simplices, as no simplex not contained in $A''$ can contain a simplex contained in $A''$.
\end{proof}
We can now use this to study an iterative version of relative subdivisions of FOS-complexes.
\begin{definitionconstruction}\label{conIteratedOrdRelSd}
	Let $K$ be an FOS-complex and $ \mathcal A := (A^1 \subset ... \subset A^{n})$ be a family of full subcomplexes of $K$ each satisfying the ordering condition of $\Cref{conLVLRel}.$ Now, define $\textnormal{sd}^{\mathcal A}(K)$ by induction over $n$ as follows. For $n=1$, just set $\textnormal{sd}^{\mathcal A}(K):= \textnormal{sd}(K\textnormal{ rel }A_1)$. For $n$ to $n+1$, let $\tilde{\mathcal A}$ be the subfamily of $\mathcal A$ where $A^{n+1}$ was removed. By another inductive use of \Cref{lemOrderedFullSustain} we set $\textnormal{sd}^\mathcal A (K)$ to the FOS-complex defined by $$\textnormal{sd}^\mathcal A (K):=\textnormal{sd}\Big (\textnormal{sd}^{\tilde{\mathcal A}} (K) \textnormal{ rel } \textnormal{sd}^{\tilde{\mathcal A}}(A^{n+1}) \Big).$$ Then, inductively by \Cref{propBarRelSd}, we have subdivisions $$\textnormal{sd}^n(K) \vartriangleleft \textnormal{sd}^{\mathcal A}(K) \vartriangleleft K,$$ where the latter is so that no new vertices are added to $|{\mathring N}(K,A_1)|$.
\end{definitionconstruction}
\begin{corollary}\label{corLVTOrdSd}
	Let $K$ be an FOS-complex together with a family $\mathcal A$ as in \Cref{conIteratedOrdRelSd} of length $n$. Then the ($n$-times iterated) last vertex map $\textnormal{l.v.}^n: \textnormal{sd}^n(K) \to K$ factors into two maps of FOS-complees $$\textnormal{sd}^n(K) \xrightarrow{l_0^{\mathcal A}} \textnormal{sd}^{\mathcal A}(K) \xrightarrow{l_1^{\mathcal A}} K$$ fitting into the commutative diagram: 
	\begin{center}
		\begin{tikzcd}
		\textnormal{sd}^n(A^1) \arrow[d, hook', swap] \arrow[r, "\textnormal{l.v.}^n"]& A^1 \arrow[r, "\textnormal{1}"] \arrow[d, hook', swap]& A_1 \arrow[d, hook', swap]\\
		\textnormal{sd}^\mathcal{A}(K) \arrow[r, "l_0^\mathcal{A}"]& \textnormal{sd}^\mathcal{A}(K) \arrow[r, "l_1^\mathcal{A}"]& K
		\end{tikzcd}.
	\end{center}
	Furthermore, $|l_0^\mathcal{A}|_P$ and $|l_1^{\mathcal A}|_P$ are respectively stratified homotopic to the subdivision homeomorphisms $|\textnormal{sd}^n(K)|_P \xrightarrow{\sim} |\textnormal{sd}^{\mathcal A}(K)|_P$ and $|\textnormal{sd}^{\mathcal A}(K)|_P \xrightarrow{\sim} |K|_P$.
\end{corollary}
\begin{proof}
	For $l_1^\mathcal{A}$, just take the iterative composition of the respective $l_1$ from \Cref{conLVLRel}. This clearly make the right hand side of the diagram commute and is stratified homotopic to the subdivision homeomorphism by an inductive use of \Cref{propLareHoToSd}.\\ $l_0^n$ is constructed inductively as follows. For $n=1$, take $l_0$ from \Cref{conLVLRel} for $A= A^1$. For the inductive step from $n$ to $n+1$, set $l_0^{n+1}$ to the composition:
	\begin{align*}
	\textnormal{sd}^{n+1}(K)=\textnormal{sd}\big (\textnormal{sd}^{n}(K) \big ) \xrightarrow{\textnormal{sd}(l_0^{\tilde{\mathcal A}})} \textnormal{sd}\big (\textnormal{sd}^{\tilde{\mathcal A}}(K)\big)
	&\xrightarrow{l_0}
	\textnormal{sd}\big(\textnormal{sd}^{\tilde{\mathcal A}}(K) \textnormal{ rel } \textnormal{sd}^{\tilde{\mathcal A}}(A^{n+1}) \big ) \\&= \textnormal{sd}^{\mathcal A}(K),
	\end{align*}
	 with $\tilde{\mathcal{A}}$ as in \Cref{conIteratedOrdRelSd}. 
	The right hand side $l_0$ is the one corresponding to the FOS-pair $\textnormal{sd}^{\tilde{\mathcal A}}(A^{n+1}) \subset \textnormal{sd}^{\tilde{\mathcal A}}(K)$. Now, consider the diagram: 
	\begin{center}
		\begin{tikzcd}
		\textnormal{sd}^{n+1}(A^1) \arrow[d, hook', swap] \arrow[r, "\textnormal{sd}(\textnormal{l.v.}^{n})"]& \textnormal{sd}(A) \arrow[r, "\textnormal{l.v.}"] \arrow[d, hook', swap]& A \arrow[d, hook', swap]\\
		\textnormal{sd}^{n+1}(K) \arrow[r, "\textnormal{sd}(l_0^{\tilde{\mathcal A}})"]& \textnormal{sd}\big (\textnormal{sd}^{\tilde{\mathcal A}}(K) \big ) \arrow[r, "l_0"]& \textnormal{sd}^{\mathcal A}(K)
		\end{tikzcd}.
	\end{center}
	This commutes by the induction hypothesis and functoriality of $\textnormal{sd}$. Furthermore, by naturality of $\textnormal{l.v.}$, the upper horizontal composition equals $ \textnormal{l.v.}^{n+1}:\textnormal{sd}^{n+1}(A) \to A$. This shows the commutativity assertion. We then have \begin{align*}
		l^{\mathcal A}_1 \circ l^{\mathcal A}_0 &= l^{\tilde{\mathcal A}}_1 \circ l_1 \circ l_0 \circ \textnormal{sd}(l_0^{\tilde{\mathcal{A}}})\\
		 &= l^{\tilde{\mathcal A}}_1 \circ \textnormal{l.v.} \circ \textnormal{sd}(l_0^{\tilde{\mathcal{A}}})\\
		 &= l^{\tilde{\mathcal A}}_1 \circ l_0^{\tilde{\mathcal{A}}} \circ \textnormal{l.v.}\\
		 &= \textnormal{l.v.}^{n+1}
	\end{align*} 	
	For the homotopy claim, note that it suffices to construct such a homotopy for the compositions with $|\textnormal{sd}^{\mathcal A}(K)|_P \xrightarrow{\sim} |K|_P$ as the latter is a stratum preserving homeomorphism. The composition $$|\textnormal{sd}^{n}(K)|_P \xrightarrow{\sim}|\textnormal{sd}^{\mathcal A}(K)|_P \xrightarrow{\sim} |K|_P$$ is the subdivision homeomorphism corresponding to $\textnormal{sd}^n(K) \vartriangleleft K$, by \Cref{propBarRelSd}. The latter is stratified homotopic to the realization of $\textnormal{l.v.}^{n}$ by \Cref{propHomoFilteredSubd}. On the other hand, we have already seen that $|\textnormal{sd}^{\mathcal A}(K)|_P \xrightarrow{\sim} |K|_P.$ is homotopic to $|l_1^{\mathcal A}|_P$. Hence, its composition with $|l_0^{\mathcal A}|_P$ is also homotopic to $|\textnormal{l.v.}^n|_P$. This finishes the proof.
\end{proof}
In the next subsection, we use these relative last vertex maps to connect stratified topological and stratified simplicial homotopy classes.
\subsection{Filtered simplicial approximation, part two}
In this subsection, we are going to prove the following alternative version of the filtered simplicial approximation theorem.
\begin{theorem}[Filtered simplicial approximation B]\label{thrmSimplicialApproximationB}
	Let $K,L$ be $P$-filtered ordered simplicial complexes over $P$, where $K$ is finite. Let $$|K|_P \xrightarrow{\phi} |L|_P$$ be a stratum preserving map. Then there exists an $n \in \mathbb N$ and a stratum preserving (ordered) simplicial map $\textnormal{sd}^n(K) \xrightarrow{f} L$ such that $$|f|_P \simeq_P \Big ( |\textnormal{sd}^n(K)|_P \xrightarrow{|\textnormal{l.v.}^n|_P} |K|_P\xrightarrow{\phi} |L|_P\Big ) \simeq_P \Big ( |\textnormal{sd}^n(K)|_P \xrightarrow{\sim} |K|_P\xrightarrow{\phi} |L|_P\Big ).$$
	Conversely, if two stratum preserving (ordered) simplicial maps $f_0,f_1: K \to L$ are such that $$|f_0|_P \simeq_P |f_1|_P,$$ then there exists an $n \in \mathbb N$ and a stratum preserving (ordered) simplicial map $$H:\textnormal{sd}^n(K \otimes \Delta^1) \to L$$ such that the following diagrams commutes.
	\begin{center}
		\begin{tikzcd}
		\textnormal{sd}^n(K) \arrow[r, "\textnormal{l.v.}^n"] \arrow[d, hook, "\textnormal{sd}^n(i_j)", swap]	& K \arrow[d, "f_j"]\\
		\textnormal{sd}^n(K \otimes \Delta^1) \arrow[r, "H"] & L
		\end{tikzcd}.
	\end{center} Here $i_j$ denotes inclusion of $K$ into $K \otimes \Delta^1$ at the respective endpoint, for $j =0$, $1$.
\end{theorem}
\begin{proof}
	We start by showing the results for a linearly ordered set $P =[q]$. Note that, for a sequence $\Sigma=(\Sigma_i)_{i\in [q]}$ in $\mathbb N$, $\textnormal{sd}^{\Sigma}(K \textnormal{ rel }A)$ from \Cref{subsecProofofSimApp} is a special case of the iterative relative subdivision defined in \Cref{conIteratedOrdRelSd}, where we take 
	$$\mathcal A := \{\underbrace{A \subset ... \subset A}_{\Sigma_0} \subset A \underbrace{\cup K_0 \subset ... \subset A \cup K_0}_{\Sigma_1} \subset A \cup K_1 \subset ...\subset \underbrace{A \cup K_{q-1} \subset ... \subset A \cup K_{q-1}}_{\Sigma_{q}}\}.$$
	 Let $\Sigma$ and $f':\textnormal{sd}^\Sigma(K) \to L$ be a filtered simplicial approximation of $\phi$ given by \Cref{thrmSimplicialApproximation}. Denote $|\Sigma|:=\sum_{p \in [q]} \Sigma_p$. Then note that as the filtration map is ordered, the subcomplexes $K_q \subset K$ fulfill the ordering condition of \Cref{corLVTOrdSd}. Hence, by the latter, we can pull back $f'$ with $l_0^{\mathcal A}$ (using the notation of \Cref{corLVTOrdSd}) to obtain a stratum preserving simplicial map $$f'': \textnormal{sd}^{|\Sigma|} \xrightarrow{ l_0^{\mathcal A}} \textnormal{sd}^{\Sigma}(K) \to L$$ such that
	 \begin{equation}\label{equProofTHrmSimpB1}
	 	|f''|_P \simeq_P \Big ( |\textnormal{sd}^{|\Sigma|}(K)|_P \xrightarrow{|\textnormal{l.v.}^n|_P} |K|_P\xrightarrow{\phi} |L|_P\Big ) \simeq_P \Big ( |\textnormal{sd}^{|\Sigma|}(K)|_P \xrightarrow{\sim} |K|_P\xrightarrow{\phi} |L|_P\Big )
	 \end{equation}
	 Note however that this is not necessarily an ordered map. Now, set $n= |\Sigma| +1$ and $$f: \textnormal{sd}^{n}(K) \xrightarrow{\textnormal{sd}(f'')} \textnormal{sd}{L} \xrightarrow{\textnormal{l.v.}}L.$$ Then all of the maps involved are maps of FOS-complexes and by naturality of $\textnormal{l.v.}$ and \eqref{equProofTHrmSimpB1} we obtain:
	 \begin{align*}
	 	|f|_P &=_{\textnormal{ }} |\textnormal{l.v.} \circ \textnormal{sd}(f'')|_P \\
	 		&=_\textnormal{ } |f'' \circ \textnormal{l.v.} |_P \\
	 		& \simeq_P \phi \circ |\textnormal{l.v.} \circ \textnormal{l.v.}^{|\Sigma|} |_P \\
	 		& \simeq_P \phi \circ |\textnormal{l.v.}^{n}|_P
	 \end{align*} 
	 finishing the first part of the proof.\\
	 \\
	 For the proof of the homotopy statement, let $f_0, f_1: K \to L$ be maps of FOS-complexes over $P$ and $H: |K \otimes \Delta^{1}|_P \cong |K|_P \otimes \Delta^1\to |L|_P$ be a stratum preserving homotopy between their realizations. Then by \Cref{exAppOfHo}, we can apply \Cref{thrmHolWeakEqB} to the induced homotopy constructed in this example $$\tilde H: |\textnormal{sd}\big ( K \otimes \textnormal{sd}^2(\Delta ^1) \big)|_P \longrightarrow |\textnormal{sd}(L)|_P$$ between $|\textnormal{sd}(f_0)|_P$ and $|\textnormal{sd}(f_1)|_P$. 
	 Then, arguing just as we did for the previous statement but using \Cref{thrmSimplicialExtension} instead, we obtain an $n' \in \mathbb N$ and a morphism of filtered simplicial complexes
	 $$H':\textnormal{sd}^{n'}\Big(\textnormal{sd}\big ( K \otimes \textnormal{sd}^2(\Delta^1)\big) \Big) \longrightarrow \textnormal{sd}(L).$$ 
	 By the commutative diagram in \Cref{corLVTOrdSd}, this is given by 
	 $\textnormal{sd}(f_j) \circ \textnormal{l.v.}^{n'}$ on the respective inclusions of $\textnormal{sd}^{n'+1}(K)$. Now, consider the map of FOS-complex $$l: \textnormal{sd}^2(K \otimes \Delta^1) \xrightarrow{\big(\textnormal{l.v.}^2 \circ \textnormal{ sd}^2(\pi_K), \textnormal{ sd}^2(\pi_{\Delta^1})\big)} K \otimes \textnormal{sd}^2(\Delta^1).$$ $l$ restricts to 
	 $\textnormal{l.v.}^2$ followed by the inclusions $i'_j: K \hookrightarrow K \otimes \textnormal{sd}^2(\Delta^1)$ ,at the respective endpoints, $j=0,1$. Now, set $n=n'+4$ and $H$ to the composition of maps of FOS-complexes
	 \begin{equation*}
	H: \textnormal{sd}^{n}(K \otimes \Delta^1) \xrightarrow{\textnormal{sd}^{n'+2}(l)} \textnormal{sd}\Big( \textnormal{sd}^{n'+1}\big(K \otimes \textnormal{sd}^2(\Delta^1) \big )\Big) \xrightarrow{sd(H')} \textnormal{sd}^2(L) \xrightarrow{\textnormal{l.v.}^2} L.
	 \end{equation*}
	 Pulling this back with $\textnormal{sd}^n(i_j)$ and using the restriction statements on $H'$ and $l$ we obtain the commutative diagram:
	 \begin{center}
	 	\begin{tikzcd}[column sep = huge, row sep = huge]
	 		\textnormal{sd}^{n}(K) \arrow[d, "\textnormal{sd}^n(i_j)"] \arrow[r, "\textnormal{sd}^{n'+2}(\textnormal{l.v.}^2)"] & \textnormal{sd}^{n'+2}(K) \arrow[d, "\textnormal{sd}^{n'+2}(i'_j)"] \arrow[r, "\textnormal{sd}(\textnormal{l.v.}^{n'})"] & \textnormal{sd}^2(K) \arrow[d, "\textnormal{sd}^2(f_j)"] \arrow[r, "\textnormal{l.v.}^2"]& K \arrow[d, "f_j"] \\
	 		\textnormal{sd}^n(K \otimes \Delta ^1) \arrow[r, "\textnormal{sd}^{n'+2}(l)"]& \textnormal{sd}^{n'+2}\big ( K \otimes \textnormal{sd}^2(\Delta ^1)\big ) \arrow[r, "\textnormal{sd}(H')"]& \textnormal{sd}^2(L) \arrow[r, "\textnormal{l.v.}^2"]& L
	 	\end{tikzcd}.
	 \end{center}
 	By definition, the bottom horizontal composition is $H$. Furthermore, by naturality of $\textnormal{l.v.}$, the top horizontal composition is $\textnormal{l.v.}^n$. Hence, this provides the diagram in the statement of the theorem.
%	 \begin{align*}
%	 H \circ \textnormal{sd}^n(i_j) &= \textnormal{l.v.}^2 \circ \textnormal{sd}(H') \circ \textnormal{sd}^{n'+2}(l) \circ \textnormal{sd}^n(i_j)\\
%	 	&= \textnormal{l.v.}^2 \circ \textnormal{sd}(H') \circ \textnormal{sd}^{n'+2}(l) \circ \textnormal{sd}^{n'+2}\big (\textnormal{sd}^2(i_j) \big )\\
%	 	&= \textnormal{l.v.}^2 \circ \textnormal{sd}(H')\circ \textnormal{sd}^{n'+2}(i'_j) \circ \textnormal{sd}^{n'+2}(\textnormal{l.v.}^2)\\
%	 	&= \textnormal{l.v.}^2 \circ \textnormal{sd}^2(f_j) \circ \textnormal{sd}(\textnormal{l.v.}^{n'}) \circ \textnormal{sd}^{n'+2}(\textnormal{l.v.}^2) \\
%	 	&= f_j \circ \textnormal{l.v.}^2 \circ \textnormal{sd}({\textnormal{l.v.}^{n'}}) \circ \textnormal{sd}^{n'+2}(\textnormal{l.v.}^2)\\
%	 	&= f_j \circ \textnormal{l.v.}^{n'+2} \circ \textnormal{sd}^{n'+2}(\textnormal{l.v.}^2)\\
%	 	&= f_j \circ \textnormal{l.v.}^n,
%	 \end{align*}
%	 for $j=0,1$, showing the commutativity of the diagram in the statement, and hence, finishing the proof.
 We are left with showing the statements in the case where $P$ is not a finite linear partially ordered set. Note that all of the simplicial complexes involved are finite. Thus, can without loss of generality restrict to the case where $P$ is finite. (Rigorously, this is done by pulling back to the finite subset (subcomplex spanned by) the elements of $P$ in the image of the filtration, approximating, and then composing with the inclusion into $P$ again.) The finite, not linearly ordered case now is immediately reduced to the linear order one by refining the order on $P$ to a linear one and using the following lemma.
\end{proof}
\begin{lemma}
	Let $P$ be a partially ordered set. Let $P'$ be another partially ordered set together with a bijective map of partially ordered set $\rho: P \to P'$. Denote by 
	\begin{align*}
	N(\rho)_*: \textnormal{s\textbf{Set}}_P \longrightarrow \textnormal{s\textbf{Set}}_{P'}\\
	\rho_*: \textnormal{\textbf{Top}}_{P} \longrightarrow \textnormal{\textbf{Top}}_{P'}	
	\end{align*}
	the functors obtained by postcomposing filtrations with $N(\rho)$ and (the map induced on the Alexandroff spaces by) $\rho$ respectively. Then these two functors are fully faithful, and fit into a diagram (commutative up to natural isomorphism):
	\begin{center}
		\begin{tikzcd}
			\textnormal{s\textbf{Set}}_P \arrow[d, "|-|_{P}"]\arrow[r, "N(\rho)_*"] & \textnormal{s\textbf{Set}}_{P'} \arrow[d, "|-|_{P'}"]\\
			\textnormal{\textbf{Top}}_{P} \arrow[r,"\rho_*"]& \textnormal{\textbf{Top}}_{P'}
		\end{tikzcd}.
	\end{center}
	Furthermore, $N(\rho)_*$ commutes with $\otimes (-)$, $\textnormal{sd}$ and $\textnormal{l.v.}$.
\end{lemma}
\begin{proof}
	The commutativity statements are all immediate from the definitions of the respective functors and natural transformations. It remains to show that $\rho_*$ and $N(\rho)$ are fully faithful. Both maps we post-compose with, $\rho$ and $N(\rho)$, are monomorphisms in their respective categories. Indeed, it is an elementary and easily verified fact on over-categories that postcomposing with a monomorphism is a fully faithful functor. 
\end{proof}
This finishes our investigation into filtered simplicial complexes. We now return to investigating the relationship between $\textnormal{\textbf{Top}}_{P}$ and $\textnormal{s\textbf{Set}}_P$.
%%%%%%%%%%%%%%%%%%%%%%%%
\section{A more detailed analysis of $\mathcal H \textnormal{\textbf{Top}}_{P}$}\label{secDetHTop}
We now have several tools available to start a detailed analysis of $\mathcal H \textnormal{\textbf{Top}}_{P}$. Our main tool is the filtered simplicial approximation theorem, \Cref{thrmSimplicialApproximationB}. However, by \Cref{propSSvSC}, this does not necessarily mean we need to restrict ourselfs only to FOS-complexes.
\subsection{$|-|_P:\mathcal H\textnormal{s\textbf{Set}}^{fin}_P \to \mathcal H \textnormal{\textbf{Top}}_{P}$ is fully faithful}
In the case where $P= \star$ is a one-point set, the Douteau model structure on $\textnormal{s\textbf{Set}}_P$ agrees with the classical Kan-Quillen one under the obvious isomorphism of categories $\textnormal{s\textbf{Set}}_P \cong \textnormal{s\textbf{Set}}$. In this setting, it is of course well known that the realization functor defines an equivalence of categories $$|-|: \mathcal{H}\textnormal{s\textbf{Set}} \xrightarrow{\sim} \mathcal{H}\textnormal{\textbf{Top}}.$$ By \Cref{thrmWeakEquSustain}, the filtered realization functor also induces a functor $|-|_P: \mathcal H \textnormal{s\textbf{Set}}_P \to \mathcal H \textnormal{\textbf{Top}}_{P}$ on homotopy categories. To the best of our knowledge it is not yet known whether this defines an equivalence of categories. However, we are going to show in this section 
(\Cref{corCatEqu}) that if one restricts to the finite setting this functor is fully faithful. This shows that at least for filtered topological spaces that are weakly equivalent to the realizations of finite filtered simplicial sets most question on homotopy theory can be reduced to the combinatorial setting. In particular, this allows us to define the Whitehead group and Whitehead torsion in the topological setting in \Cref{secTopWh}.
\\
\\ 
As the title of the subsection states, we are now going to prove \Cref{thrmFullyFaithful}. Our proof is somewhat unorthodox, at least from a model theoretical perspective, as we first to reduce to FOS-complexes and then use simplicial approximation theorems. This is certainly more of a classical than a modern approach to the problem. It comes at the price of having to restrict to the finite setting. We conjecture that this is just due to the insufficiency of our method of proof, and in fact $|-|_P: \mathcal H\textnormal{s\textbf{Set}}_P \longrightarrow \mathcal H\textnormal{\textbf{Top}}_{P}$ is an equivalence of categories (induced by a Quillen equivalence even, see also \Cref{remWeirdStruct}). However, as this work was written with simple homotopy theory in mind, some compactness assumptions are really not that big of a price to pay.
\begin{theorem}\label{thrmFullyFaithful}
The functor $$ |-|_P:\mathcal H\textnormal{s\textbf{Set}}_P^{fin} \longrightarrow \mathcal H\textnormal{\textbf{Top}}_{P}$$ induced by \Cref{thrmWeakEquSustain} is fully faithful.
\end{theorem}
\begin{proof}
	Let $X,Y \in \textnormal{s\textbf{Set}}_P^{fin}$. We want to show that $$\textnormal{Hom}_{\mathcal H\textnormal{s\textbf{Set}}_P}(X,Y) \xrightarrow{|-|_P} \textnormal{Hom}_{\mathcal H \textnormal{\textbf{Top}}_{P}}(|X|_P,|Y|_P)$$ is a bijection. First, note that, by \Cref{propSSvSC}, we may assume without loss of generality that $X$ and $Y$ are finite FOS-complexes. We now want to reduce to the case, where the right hand side is given by actual stratified homotopy classes. By \Cref{propHoClasses} and the fact that by definition every object in $\textnormal{\textbf{Top}}_{P}$ is fibrant, we have for cofibrant $T$ and $T'$ arbitrary: \begin{equation}\label{proofThmFullFaithHoclass}
	\textnormal{Hom}_{\mathcal{H}\textnormal{\textbf{Top}}_{P}}(T,T') \cong [T,T']_{P}.
	\end{equation} 
	Now, consider the last vertex map
	$$\textnormal{l.v.}_P:\textnormal{sd}_P(X) \longrightarrow X.$$ By \cite[A.3]{douteauFren} or alternatively \Cref{propLvtSim}, the latter is a weak equivalence. In particular, the induced map $$|\textnormal{l.v.}_P|_P: |\textnormal{sd}_P(X)|_P \longrightarrow |X|_P$$ induces an isomorphism in $\mathcal H \textnormal{\textbf{Top}}_{P}$. 
	By \Cref{propSDPisCof}, $|\textnormal{sd}_P(X)|_P$ is cofibrant. Furthermore, similarly to the case of $\textnormal{sd}$ 
	one easily checks that $\textnormal{sd}_P(X)$ is still an FOS-complex.
	Hence, we assume without loss of generality that both $X$ and $Y$ are finite FOS-complex, and that $|X|_P$ is cofibrant. 
	Thus, using \Cref{proofThmFullFaithHoclass}, we are left with showing that 
	\begin{equation*}
	|-|_P: \textnormal{Hom}_{\mathcal H\textnormal{s\textbf{Set}}_P}(X,Y) \longrightarrow 	\textnormal{Hom}_{\mathcal H\textnormal{\textbf{Top}}_{P}}(|X|_P,|Y|_P) \cong [|X|_P,|Y|_P]_P
	\end{equation*}
	is a bijection. Now, let $[\phi]$ be any arrow on the right hand side, represented stratum preserving map $\phi:|X|_P \to |Y|_P$. By the second version of the filtered simplicial approximation theorem (\Cref{thrmSimplicialApproximationB}), for some sufficiently large $n$ and some map of FOS-complexes $f: \textnormal{sd}^n(X) \xrightarrow{f} Y$ we have $$[\phi]\circ [|\textnormal{l.v.}|_P^n]= [|f|_P].$$ $[\textnormal{l.v.}^n]$ is an isomorphism in $\mathcal{H}\textnormal{s\textbf{Set}}_P$. In particular, $$[\phi] = |[f] \circ [\textnormal{l.v.}^n]^{-1}|_P$$ showing surjectivity. Now, conversely let $\varphi_0, \varphi_1$ be two arrows on the left hand side of the equation. By \Cref{propRepHoClasses}, both fit into a commutative diagram \begin{center}
		\begin{tikzcd}
		& \textnormal{sd}^n_{P}(X) \arrow[ld, "{[\textnormal{l.v.}^n_P]}", swap] \arrow[rd, "{[f_i]}"]&\\
		X \arrow[rr, "\varphi_i"]& &Y 
		\end{tikzcd}
	\end{center}
	For appropriate morphisms in $\textnormal{s\textbf{Set}}_P^{fin}$ $f_i$. Without loss of generality, we can assume $n_0 = n_1$. Furthermore, inverting $[\textnormal{l.v.}^{n_i}]$, we can then assume that $n_i=0$. Hence, we are now in the setting where $[f_i] = \varphi_i$ and $|f_0|_P \simeq_P |f_1|_P$. By the homotopy part of \Cref{thrmSimplicialApproximationB}, there is a commutative diagram in $\textnormal{s\textbf{Set}}_P$:
	\begin{center}
		\begin{tikzcd}
		\textnormal{sd}^n(X) \arrow[r, "\textnormal{l.v.}^n"] \arrow[d, hook, "\textnormal{sd}^n(i_j)", swap]	& X \arrow[d, "f_j"]\\
		\textnormal{sd}^n(X \otimes \Delta^1) \arrow[r, "H"] & Y
		\end{tikzcd}.
	\end{center}
	\Cref{corSDRetainWeakEq} here, states that $\textnormal{l.v.}:\textnormal{sd}(-)$ preserves weak equivalences (for finite arguments, but this is easily generalized). Again, we assure that there is no risk of circularity involved here. Thus, 
	$\textnormal{sd}^n(X) \sqcup \textnormal{sd}^n(X) \hookrightarrow \textnormal{sd}^n(X \otimes \Delta^1)$
	gives a cylinder object for $\textnormal{sd}^n(X)$ that witnesses a left homotopy between $f_0 \circ \textnormal{l.v.}^n$ and $f_1 \circ \textnormal{l.v.}^n$. Hence, $$[f_0 \circ \textnormal{l.v.}^n] = [f_1 \circ \textnormal{l.v.}^n]$$ and, by the invertibility of $[\textnormal{l.v.}]$ in $\mathcal{H}\textnormal{s\textbf{Set}}_P$, we obtain $$[f_0] = [f_1].$$
	% Now, by \cite[Prop. 8.1.1]{douteauFren} together with \cite[8.1.3]{douteauFren} the realization of the filtered last vertex map $$|\textnormal{l.v.}_P|_{N(P)}: |\textnormal{sd}_P(X)|_{N(P)} \longrightarrow |X|_{N(P)}$$ gives a cofibrant replacement for $|X|_{N(P)}$. As by \Cref{thrmHolWeakEqB} this weak equivalence is retained under the functor $$F: \textnormal{\textbf{Top}}_{N(P)} \longrightarrow \textnormal{\textbf{Top}}_{P}$$ and as $F$ is part of a quillen adjunction, hence retains cofibrancy, the filtered map $$|\textnormal{l.v.}_P|_P_|\textnormal{sd}_P(X)|_{P} \longrightarrow |X|_P$$ gives a cofibrant replacement of $|X|_P$ in $\textnormal{\textbf{Top}}_{P}$. As the 
\end{proof}
In particular, we obtain the following immediate corollary, using the fact that a morphism in the model category is a weak equivalence if and only if it descends to an isomorphism in the homotopy categoy.
\begin{corollary}\label{corRelRef}
A morphism $f:X \to Y$ in $\textnormal{s\textbf{Set}}_P^{fin}$ is a weak equivalence with respect to the Douteau model structure if and only if $|f|_P$ is a weak equivalence with respect to the Henrique-Douteau model structure.
\end{corollary}
Hence, for spaces that are (stratified homotopy equivalent) stratum preserving homeomorphic to the realization of a finite simplicial set, questions on homotopies can often be reduced to the combinatorial setting of filtered simplicial sets. It is useful to have a name for such spaces.
\begin{definition}
Let $T \in \textnormal{\textbf{Top}}_{P}$. \begin{itemize} \item A filtered simplicial set $X \in \textnormal{s\textbf{Set}}_P$ together with a stratum preserving homeomorphism $|X|_{P} \cong T$ is called a \textit{triangulation} of $T$. If $X$ is finite, we call it a \textit{finite triangulation}.
	\item A filtered topological space admitting such a (finite) triangulation is called \textit{(finitely) triangulable}
\end{itemize}
\end{definition}
In particular, we obtain:
\begin{corollary}\label{corCatEqu}
	Denote by $\mathcal H\textnormal{\textbf{Top}}^{t-fin}_{P}$ the full subcategory of $\mathcal{H}\textnormal{\textbf{Top}}_{P}$ given by such filtered spaces that are weakly equivalent to a triangulable filtered space.
	Then $$ |-|_P:\mathcal H\textnormal{s\textbf{Set}}_P^{fin} \longrightarrow \mathcal H\textnormal{\textbf{Top}}_{P}$$
	induces an equivalence of categories $$\mathcal{H}\textnormal{s\textbf{Set}}_P^{fin} \xrightarrow{\simeq} \mathcal H\textnormal{\textbf{Top}}^{t-fin}_{P}.$$
\end{corollary} 
\begin{corollary}
	A morphism of finite filtered simplicial sets $f:X \to Y$ in $\textnormal{s\textbf{Set}}_P$ is a weak equivalence, with respect to the Douteau model structure if and only if $|f|_P$ is a weak equivalence, with respect to the Henrique-Douteau model structure on $\textnormal{\textbf{Top}}_{P}$.
\end{corollary}
\subsection{Homotopy classes into fibrant, finitely triangulable filtered spaces}
As we have already shown in the proof of \Cref{thrmFullyFaithful}, the $\textnormal{Hom}$-sets in $\mathcal{H\textnormal{\textbf{Top}}_{P}}$, with a triangulable filtered space at the domain, admit fairly explicit descriptions. This is owed to the fact that cofibrant approximations are given by $\textnormal{sd}_P(-)$. More precisely, in this case we have a natural bijection:
$$\textnormal{Hom}_{\mathcal H\textnormal{\textbf{Top}}_{P}}(|X|_P,T) \cong \textnormal{Hom}_{\mathcal H\textnormal{\textbf{Top}}_{P}}(|\textnormal{sd}_P(X)|_P,T) = \big[ |\textnormal{sd}_P(X)|_P,T \big]_P.$$
It would be nice to have such a result, not only for spaces cofibrant in the Douteau-Henrique model structure, but also for spaces arising more naturally in the study of spaces with singularities. Such a result would make the homotopy category of filtered spaces a lot easier to understand. Recall (see \Cref{defFStrar}) that a f-stratified space (over $P$) is a filtered space $T$ that has the right lifting property with respect to admissible horn inclusions as in the diagram below. 
\begin{center}
\begin{tikzcd}
	{|\Lambda_k^{\mathcal J}|_P} \arrow[d, hook]\arrow[r] & T\\
	{|\Delta^\mathcal{J}|_P} \arrow[ru, dashed]
\end{tikzcd}
\end{center}
As we already mentioned in \Cref{exOfFstrat}, examples of such spaces include all conically stratified spaces (in particular all PL pseudomanifolds and PL stratified spaces in the original sense of Goresky and MacPherson) and all homotopically stratified spaces in the sense of Quinn. One then has: 
\begin{lemma}\label{lemLiftForFAE}
	Let $T \in \textnormal{\textbf{Top}}_{P}$ be f-stratified. Then $T$ has the right lifting property with respect to all realizations of (filtered) anodyne extensions $A \hookrightarrow B$ in $\textnormal{s\textbf{Set}}_P$.
\end{lemma}
\begin{proof}
	The class of morphisms having the left lifting property with respect to another class of morphisms, is always saturated. That is, it is closed under coproducts, pushouts, transfinite compositions (see \cite{nlab:transfinite_composition}, for a definition) and retracts and contains all isomorphisms. The class of anodyne extensions is the saturated class generated by all admissible horn inclusions. Hence, the result is immediate from the fact that $|-|_P$ preserves colimits, and hence maps the saturated class of the latter into the saturated class generated by realizations of horn inclusions. 
\end{proof}
In \Cref{remWeirdStruct}, we conjectured that f-stratified spaces can in fact be taken to be the fibrant objects in an appropriate model structure on $\textnormal{\textbf{Top}}_{P}$ that has the same weak equivalence as the Henrique-Douteau model structure. A strong hint at such a result is that it turns out that for f-stratified spaces, which are triangulable by finite FOS-complexes, the $\textnormal{Hom}$ sets in $\mathcal H \textnormal{\textbf{Top}}_{P}$ are actually given by sets of stratified homotopy classes.
\begin{theorem}\label{thrmHoClassofFSpace}
	Let $X,Y \in \textnormal{s\textbf{Set}}_P^{fin}$ such that $X$ is an FOS-complex and $Y$ such that $|Y|_P$ is an f-stratified space. Then the natural map $$[|X|_P,|Y|_P]_P \longrightarrow \textnormal{Hom}_{\mathcal H \textnormal{\textbf{Top}}_{P}}(|X|_P,|Y|_P)$$
	is a bijection.
\end{theorem}
\begin{proof}
Let $Y \xhookrightarrow{i} F(Y)$ be a cofibrant fibrant replacement for $Y$. Then, by \Cref{lemLiftForFAE} $|i|_P$ admits a retract $r$, which is a weak equivalence as it is a retract of a weak equivalence by \Cref{thrmWeakEquSustain}. Now, consider the following commutative diagram: 
\begin{equation*}
	\begin{tikzcd}
	{\big [X,F(Y) \big ]_P} \arrow[r, "\sim"] \arrow[d] & \textnormal{Hom}_{\mathcal H \textnormal{s\textbf{Set}}_P}(X,F(Y)) \arrow[d, "\sim"]\\
	{\big [|X|_P,|F(Y)|_P \big ]_P} \arrow[r] \arrow[d,"r_*"] & \textnormal{Hom}_{\mathcal H \textnormal{\textbf{Top}}_{P}}(|X|_P, |F(Y)|_P) \arrow[d, "\sim ", ]\\
	{\big [|X|_P, |Y|_P \big ]_P} \arrow[r] & {\textnormal{Hom}_{\mathcal H \textnormal{\textbf{Top}}_{P}}(|X|_P, |Y|_P) } 
	\end{tikzcd},
\end{equation*}
where the right lower vertical is the isomorphism induced by the weak equivalence $r$. The upper right vertical is an isomorphism by \Cref{thrmFullyFaithful}, which clearly extends to all $P$-filtered simplicial sets that are weakly equivalent to a finite one, i.e. also to $F(Y)$. In particular, by commutativity, the bottom right horizontal is also onto. It remains to show injectivity. Consider stratified homotopy classes of maps $\phi_0,\phi_1: |X|_P \to |Y|_P$ whose images in $\textnormal{Hom}_{\mathcal H \textnormal{\textbf{Top}}_{P}}(|X|_P,|Y|_P)$ agree. By pulling back with the subdivision homeomorphism $|\textnormal{sd}^n(X)|_P \xrightarrow{\sim} |X|_p$ and invoking \Cref{thrmSimplicialApproximationB}, we may without loss of generality assume that both come from homotopy classes of morphisms of filtered simplicial sets $f_0, f_1: X \to Y$. They are mapped to by $ [i \circ f_0],[i \circ f_1] \in [X,F(Y)]_P$ respectively. Thus, by commutativity of the diagram and the fact that the composition "right, down, down" is a bijection, this already implies $[\phi_0] = [\phi_1]$.
\end{proof}
 In particular, this result also holds true for filtered spaces $T_X,T_Y$ which are respectively stratified homotopy equivalent to realizations of such filtered simplicial sets. 
\begin{remark}
	The additional assumption that $X$ is an FOS-complex and not just a finite simplicial set should be omittable. To do this, one only needs to extend the filtered simplicial approximation theorem to all finite $P$-filtered simplicial sets. One way to approach this, might be to define a model structure on $\textnormal{s\textbf{Set}}_P$ that uses $\textnormal{sd}$ instead of $\textnormal{sd}_P$. However, one might just as well take the approach we illuminated in \Cref{remWeirdStruct}, which should give a more complete understanding of the situation.
\end{remark}
This result gives a rather clear view of what the morphisms in $\mathcal H\textnormal{\textbf{Top}}_{P}$ look like for many classical examples of stratified spaces. To be more precise, we obtain:
\begin{corollary}
	Let $\textnormal{\textbf{Top}}_{P}^{f, fin,PL}$ be the full subcategory of f-stratified, compact polyhedra in $\textnormal{\textbf{Top}}_{P}$, i.e. f-stratified spaces, triangulable by a finite FOS-complex. 
	Let $\mathcal {H}^{naive}\textnormal{\textbf{Top}}_{P}^{f, fin, PL}$ be the category obtained from it, by identifying stratified homotopic maps. Then $\textnormal{\textbf{Top}}_{P}^{f, fin,PL} \hookrightarrow \textnormal{\textbf{Top}}_{P}$ induces a fully faithful embedding $$\mathcal H^{naive}\textnormal{\textbf{Top}}_{P}^{f, fin,PL} \hookrightarrow \mathcal H \textnormal{\textbf{Top}}_{P}.$$ In particular, the analogous result holds, for the full subcategories of $\textnormal{\textbf{Top}}_{P}^{f, fin,PL}$ given by such realizations of finite FOS-complexes that are pseudo-varieties, cs-stratified, conically stratified or homotopically stratified spaces.
\end{corollary}
We should note that a similar statement has been conjectured by Douteau in \cite[8.1.4]{douteauFren}. However, there it is conjectured for the strongly filtered setting.
\chapter{Simple Stratified Homotopy Theory} \label{chII}
Roughly speaking, the origin of simple homotopy theory is the question of which homotopy equivalences $f: X\to Y$ between appropriately combinatorial objects, say CW-complexes, can be represented through a sequence of certain combinatorial moves (expansions and collapses). As there is already a lot of excellent introductory work into simple homotopy theory, we refrain from giving an overview of the classical perspective. The reader completely new to the subject should find motivation and analogies for our stratified perspective in \cite{cohenCourse}. In the classical setting of CW-complexes (or simplicial complexes) and elementary expansions, it turns out that the answer to this question is encoded by an element $\tau(f)$ - the Whitehead torsion - of a purely algebraic group $Wh(X)$ - the Whitehead group of $X$. In this chapter we ask and attempt to answer the analogous question but for filtered (stratified) spaces. Before one can even dare to attempt giving an answer, one of course has to be a little bit more specific as to what the correct analogy should be. In other words:\\
\\ What are the ``appropriately combinatorial objects''? \\What are the ``homotopy equivalences''? \\What are the ``combinatorial moves''? \\
\\The first of these three questions we have already answered in \Cref{subsecPsset} with the category $\textnormal{s\textbf{Set}}_P$ of $P$-filtered simplicial sets. It simply seems the most well explored category of filtered objects of combinatorial nature that is not too restrictive when it comes to pushout and mapping cylinder constructions. To answer the second question, we take the weak equivalences in the Douteau model structure on $\textnormal{s\textbf{Set}}_P$ and the Henrique-Douteau model structure on $\textnormal{\textbf{Top}}_{P}$ respectively. We have seen in various results of \Cref{chI} (in particular \Cref{thrmWeakEquSustain}, \Cref{thrmFullyFaithful} and \Cref{corRelRef}) that they interact rather nicely. In particular, $|-|_P$ preserves and reflects weak equivalences as long as the source and target are finite filtered simplicial sets. The reason we chose weak equivalences over the more rigid filtered homotopy equivalences partially lies in the answer to the third question. The combinatorial moves of our choice are the pushouts of admissible horn inclusions. We illustrate in detail in \Cref{subsecElemExp} why we think that they, together with the weak equivalences, make for a good candidate for a study of a stratified simple homotopy theory. Now, that at least the question is well posed, we can illuminate how we are set out to answer it. The goal is of course to construct a stratified analogue to the Whitehead group and the Whitehead torsion and show that these have similar properties as their classical counterparts. To do so, we make use of a result of Eckmann, Bolthausen and Siebenmann which allows for the construction of such objects in a very general category theoretical setting, given that a certain set of axioms is verified. Introducing this approach (and fixing a minor mistake in the formulation of the axioms) is the content of \Cref{subsecEckSiebApp}. In the following section (\ref{secElandFSAE}) we then define the analogue to expansions in the classical theory and do an in-depth study of their behavior. One of our main tools of study is a recent publication on strong anodyne extensions by Sean Moss (\cite{MossSae}). We use the results obtained in this section to show that the axioms of \Cref{subsecEckSiebApp} are verified (\Cref{thrmEckSiebAxAreTrue}). This proof marks the beginning of \Cref{subsecEckSiebAppHolds}. In particular, doing so we have then shown the existence of a Whitehead group and torsion for filtered and in particular stratified spaces. They answer the initial question at the beginning chapter. The remainder of \Cref{subsecEckSiebAppHolds} is spent on analyzing the behavior of this stratified Whitehead torsion and group in more detail, showing that it behaves much like the classical theory. For example, we show that with respect to the relevant notion of simple equivalence every finite $P$-filtered simplicial set has the simple homotopy type of a $P$-filtered ordered simplicial complex (\Cref{thrmSSvSC}). Finally in \Cref{secTopWh}, we use the results of \Cref{chI} to give a series of equivalent descriptions of the stratified Whitehead group and torsion, also generalizing the latter to continuous stratum preserving maps instead of purely simplicial ones (\Cref{propCharOfWh}). We use this to prove that the Whitehead group we defined agrees with the classical one in case where $P= \star$ is a one point set (\Cref{thrmOldNewAgree}).
\section{The Eckmann-Siebenmann approach}\label{subsecEckSiebApp}
Historically, several authors have been concerned with the question of how to generalize Whitehead's simple homotopy theory to other settings such as locally finite CW-complexes (see \cite{siebenmannInfinite}, \cite{eckmann2006}) or even more generally to more abstract homotopy theoretical settings, for example categories of chaincomplexes (see \cite{kampsPo}). Luckily, these approaches have all been formulated in a very general category theoretical setting, so they are easily transferred to our question at hand. Here, we present the approach that was independently taken by Siebenmann as well as Eckmann and his student Bolthausen (see \cite{siebenmannInfinite}, \cite{eckmann2006}). It will turn out, later on in \Cref{subsecEckSiebApp}, that this approach can be used to formulate a simple homotopy theory for filtered simplicial sets, and ultimately for finitely triangulated filtered spaces. However, we should note that the way the axioms for this approach were formulated in \cite{siebenmannInfinite} and \cite{eckmann2006} is slightly flawed. To be more precise, in the precise way they are phrased in it does not apply to any setting known to us and certainly not to the settings the authors were interested in (see \Cref{remAxAreWrong}). However, this is more of a technical difficulty, and easily fixed by a slight reformulation of the axioms. To make sure that all the results we need still hold under these new conditions, we give sketches of the proofs in \cite{eckmann2006}. We note that it might certainly be fruitful to translate the work below to a more model category theoretic setting. This has already been hinted at in \cite[Ch.VI.2]{kampsPo}. However, the following suffices for the applications we have in mind.\\
\\
For the remainder of this section, let $\mathcal C$ be some category embedded in a larger category $\hat {\mathcal C}$ that has the same objects as $\mathcal C$ but potentially more morphisms. Further, let $\Sigma $ be a class of morphisms in $\mathcal C$. Now, consider the localized category $\mathcal C(\Sigma^{-1})$ (see for example \cite[Ch. I]{gabrielZisCalc}) and denote by $Q$ the structure functor $Q:\mathcal C \to \mathcal C(\Sigma^{-1})$. 
\begin{definition}\label{defSimpleSetting}
	A morphism in $\mathcal C(\Sigma^{-1})$ is called \textit{simple} if it is given by a composition of morphisms of the shape $Q(s),Q(s)^{-1}$ for $s \in \Sigma$. 
	Two morphisms $\alpha, \beta$ in $\mathcal C^{-1}$ with the same source $X \in \mathcal C$ are said to have the same \textit{simple morphism class} if $ \beta = \gamma \circ \alpha$ for some $\gamma$ that is simple (clearly this construction gives an equivalence relation, and hence a well-defined notion of equivalence class). We denote the class of $\alpha$ by $\langle \alpha\rangle $. Denote by $A(X)$ the class of simple morphism classes of morphisms with source $X \in \mathcal C$. Denote by $E(X)$ the subclass of $A(X)$ given by classes $\langle \alpha\rangle $ where some (and hence every) $\alpha' \in \langle \alpha\rangle $ is an isomorphism in $\mathcal C(\Sigma^{-1})$. $E(X)$ and hence $A(X)$ has a distinguished element given by $\langle 1_X\rangle $.
\end{definition}
Now, suppose we are in a sufficiently nice setting. That is, one where $\mathcal C(\Sigma^{-1})$ is equivalent to some homotopy category we are interested in studying and one where $A(X)$ is sufficiently small, i.e. of set size. This is for example the case for $\hat {\mathcal C}$ the category of finite CW-complexes, $\mathcal C$ the given by inclusions of subcomplexes and $\Sigma$ the class of compositions of elementary expansions. We know that in these cases $E$ in fact defines a functor on $\mathcal C(\Sigma^{-1})$ into monoids and $E(X)$ defines a functor into abelian groups, the Whitehead group. (\cite{eckmann2006},\cite{siebenmannInfinite}). One can ask the general question of what requirements need to be fulfilled for this to be the case. This is answered in the following definition and theorem.
\begin{definition}\label{defAxEckSieb}
	We say that a triple $\mathcal C \subset \hat{\mathcal{C}}$, $\Sigma$ as above (and such that $A(X)$ is of set size for $X \in \mathcal{C}$) \textit{admits a Whitehead group} if the following axioms are fulfilled. 
	\begin{enumerate}[label = {(\normalfont{\textbf{A\arabic*}})}]
		\setcounter{enumi}{-1}
		\item \label{eckAx0}$\Sigma$ contains all isomorphisms in $\mathcal{\hat C}$ and is closed under composition. 
		\item \label{eckAx1}Let $f: X \to X'$ and $s:X \to Y$ be morphisms in $\mathcal C$. Then the the pushout diagram $$\begin{tikzcd}
		X \arrow[r, "f"] \arrow[d, "s"] & Y \arrow[d, "s'"]\\
		X' \arrow[r, "f'"] & Y'
		\end{tikzcd},$$
		in $\hat{ \mathcal C}$ exists and furthermore $f',s' \in \mathcal C$. If further $s \in \Sigma$, then so is $s'$.
		\item \label{eckAx2}Let $f,g: X \to Y$ be two morphisms in $\mathcal{C}$ such that $Q(f) = Q(g)$. Then there exists a commutative diagram in $\mathcal{C}$
		$$\begin{tikzcd}
		X \arrow[r, "f"] \arrow[d, "g"] & Y \arrow [d, "s"]\\
		Y \arrow[r, "t"] & Z
		\end{tikzcd},$$
		where $s,t \in \Sigma$.
	\end{enumerate}
\end{definition}
\begin{theorem}\label{thrmEckSieb}
	Let $\mathcal{C} \subset \hat{\mathcal{C}}, \Sigma$ admit a Whitehead group.\\
	\\
	Then every morphism in $\mathcal C(\Sigma^{-1})$ is of the form $Q(s)^{-1}Q(f)$ for some $s \in \Sigma $, $f\in \mathcal C$. In particular, for each $X \in \mathcal C$ and each $\langle \alpha\rangle \in A(X)$ there exists a $a \in \mathcal C$ such that $\langle \alpha\rangle = \langle Q(a)\rangle .$\\
	\\
	Let $X \in \mathcal C$. For $\langle \alpha\rangle $ and $\langle \beta\rangle $ in $A(X)$ define $\langle \alpha\rangle + \langle \beta\rangle $ by the simple isomorphism class of the diagonal in a pushout diagram 
	$$\begin{tikzcd}
	X \arrow[r, "a"] \arrow[d, "b"] & Y \arrow[d]\\
	Z \arrow[r] & W
	\end{tikzcd}$$
	in $\hat{\mathcal{C}}$, where $Q(a) \in \langle \alpha\rangle $ and $Q(a) \in \langle \beta\rangle $. Furthermore for $\langle \alpha\rangle $ in $A(X)$ and $f:X \to X' \in \mathcal C$ define $f_*\langle \alpha\rangle \in E(Y)$ as $\langle Q(a')\rangle $, where $a'$ is the pushout of a representative of $\alpha$ in $\mathcal C$ along $f$ in $\hat{\mathcal C}$. Both of these constructions are well-defined. \\
	\\Then,
	\begin{align*} 
	\mathcal C &\to \textnormal{\textbf{AbMon}}\\
	 X &\mapsto (A(X),+, \langle 1_X\rangle )\\
	\{X \xrightarrow{f} X'\} &\mapsto \{\langle \alpha\rangle \mapsto f_*\langle \alpha \rangle \}											
	\end{align*}
	defines a functor into the category of abelian monoids. Furthermore,
	\begin{align*} 
	\mathcal C &\to \textnormal{\textbf{Ab}}\\
	X &\mapsto (E(X),+, \langle 1_X\rangle )\\
	\{X \xrightarrow{f} X'\} &\mapsto \{\langle \alpha\rangle \mapsto f_*\langle \alpha \rangle \}											
	\end{align*}
	defines a functor into the category of abelian groups. Both send $s \in \Sigma$ to an isomorphism, i.e. induce functors: 
	\begin{align*}
		E: \mathcal C(\Sigma^{-1}) &\to \textnormal{\textbf{Ab}},\\
		A: \mathcal C(\Sigma^{-1}) &\to \textnormal{\textbf{AbMon}}.
	\end{align*}
\end{theorem}
\begin{remark}\label{remAxAreWrong}
Before we move on to a proof, it might be useful to shed a little bit of light on the usage of the two categories $\mathcal C$ and $\hat{\mathcal C}$ here. Note that both the functoriality as well as the group structure are defined by taking pushouts of representatives of equivalence classes. In general pushouts interact badly with homotopies and the induced equivalence relation (see for example \cite[Ch.V]{kampsPo}). For example, the pushout of a homotopy equivalence might not be a homotopy equivalence anymore, which would interfere with the functoriality of $E$ above. Similarly glueings of homotopy equivalences might not be homotopy equivalences, interfering with the definition of addition. However, there are of course classes of maps that interact more nicely, when it comes to pushouts: Cofibrations (see for example \cite[Ch. V Sec. 6, 7, Sec]{kampsPo}). Note for example that the pushout of an acyclic cofibration of topological spaces is again an acyclic cofibration (\cite[Ch. I, Cor. 6.14]{kampsPo}) and of course this also holds true in any model category \cite{hirschhornModel}. Usually, the way to circumvent these difficulties is by replacing maps $f: X \to Y$ by the mapping cylinder inclusions $X \hookrightarrow M_f$, i.e. passing to homotopy pushouts. This is for example how abstract simple homotopy theory is built purely cylinder based in \cite[Ch. VI, Sec. 3]{kampsPo}. Alternatively, one can restrict to whatever class of morphisms interact nicely with homotopy and pushouts in this setting. This is essentially the step of passing from $\hat{\mathcal C}$ to $\mathcal C$. In many cases, one can then use an appropriate factorization system to represent morphisms in $\hat{\mathcal C}$ by morphisms in $ \mathcal{C}$. Furthermore, one then hopes that $\mathcal C(\Sigma^{-1})$ is still equivalent as a category to the homotopy category one is interested in studying. This is for example also done implicitly in \cite{cohenCourse}, where the geometric Whitehead group is constructed from pairs of CW-complexes that are deformation retractions, as any inclusion of a subcomplex of CW-complexes is a cofibration.\\
\\
In the original sources for this chapter (\cite{eckmann2006},\cite{siebenmannInfinite}), however, no mention of a second category $\hat{\mathcal C}$ is made. In fact, the analogue to \ref{eckAx1} in \Cref{defAxEckSieb} is just formulated via pushouts in $\mathcal C$. While both authors are applying the construction precisely as we mentioned in the above paragraph, they seem to be under the assumption that in the category with objects CW-complexes and with morphisms inclusions of subcomplexes the pushout is given by the pushout in the category of CW-complexes (\cite[p.6]{eckmann2006}, \cite[Paragraph 2]{siebenmannInfinite}). This is false. Yes, the pushout diagram in CW-complexes does exist and all of its structure maps are still inclusions of subcomplexes, but it does not fulfil the existence part of the universal property in the subcategory of inclusions of subcomplexes. Take for example the pushout along the empty complex of two points, $\star$. Trivially, the constant map into another point makes the diagram (given by the maps from the empty complex into the points) commute. But there is no map $\star \sqcup \star \to \star$ that is an inclusion of a subcomplex. \\
\\
Nevertheless, this is neither really hurtful to the results nor to the theoretical impact of their contributions as neither of them really uses the universal property of the pushout anywhere in their proof. What the authors are really using the pushout for is to have a canonical way to produce commutative squares. To be more precise, they need a way to produce commutative squares, unique up to isomorphism of diagrams and they want this to be functorial in the sense that these diagrams have the composition properties of pushout squares. This is precisely what we verify by our usage of \ref{eckAx1}.
\end{remark}
We now begin the proof of \Cref{thrmEckSieb}. All the following statements are to be understood in this setting.
\begin{lemma}\label{lemCharOfMorCs}
	 Every morphism in $\mathcal C(\Sigma^{-1})$ is of the form $Q(s)^{-1}Q(f)$, for some $s \in \Sigma $, $f\in \mathcal C$.
\end{lemma}
\begin{proof}
	By definition of the localized category \cite[Ch. I]{gabrielZisCalc}, every morphism in $\mathcal C(\Sigma^{-1})$ from $X$ to $Y$ is given by a zigzag $$X = X_0\leftrightarrow X_1 \leftrightarrow ... \leftrightarrow X_n = Y,$$ where we allow all arrows in $\mathcal C$ in right direction and only arrows in $\Sigma $ in left direction. We need to show that each such zigzag can be brought into the shape $$ X \xrightarrow{f} Z \xleftarrow{s}Y, $$ for $s \in \Sigma$, under the operations inducing the morphism structure in $\mathcal C(\Sigma^{-1})$ \cite[Ch. I]{gabrielZisCalc}. Now, assume by induction that we have already proven this for zigzags up to length $n-1$. The case $n=1$ is obvious. Now, by the composition rules in a localized category, we can reduce a zigzag of length $n$ to the shape
	$$ X \xrightarrow{f} X_1 \xleftarrow{s} X_2 \xleftrightarrow{g} Y,$$ with $f,g \in \mathcal C$ and $s \in \Sigma$.
	If $g$ points to the left, it is necessarily in $\Sigma$. Hence, as $\Sigma$ is closed under composition by \ref{eckAx0}, we are done. So let us assume $g$ points to the right. Now, consider the diagram $$ \begin{tikzcd}
	& & X_2 \arrow[ld, "s", swap] \arrow[rd, "g"] &\\
	X \arrow[r, "f"]& X_1 \arrow[rd, "g'", swap] & & Y \arrow[ld, "s'"]\\
	& & X_2'&
	\end{tikzcd},$$ where the right hand square is a pushout diagram in $\hat{\mathcal C}$. By \ref{eckAx1}, this square again lies in $\mathcal C$ and $s' \in \Sigma$. By commutativity, the morphism in $\mathcal C(\Sigma^{-1})$ given by $$Q(s')^{-1}Q(g')Q(f) = Q(s')^{-1}Q(g' \circ f)$$ is then the same as the one induced by the zigzag we started with, concluding the proof.
\end{proof}
By the same argument as in the proof of \Cref{lemCharOfMorCs}, one obtains: 
\begin{lemma}\label{lemCharOfSimCs}
	Every simple morphism $\sigma$ in $\mathcal C(\Sigma^{-1})$ is of shape $Q(s)^{-1}Q(s')$ for some $s,s' \in \Sigma$.
\end{lemma}
We can now use this to obtain an alternative description of $A(X)$ and $E(X)$.
\begin{lemma}\label{lemEqConE}
	For every $\langle \alpha\rangle $ in $A(X)$ there exists an $a: X \to Y$ in $\mathcal{C}$ such that $\langle \alpha\rangle = \langle Q(a)\rangle $.\\
	\\ 
	Furthermore, for $a: X \to Y$, $b: X \to Z$ in $\mathcal C$:\\$\langle Q(a)\rangle = \langle Q(b)\rangle $ if and only if there exists a commutative diagram in $\mathcal{C}$:
	 $$\begin{tikzcd}
	& Y \arrow[rd, "s"] &\\
	X \arrow[rd, "a", swap] \arrow[ru, "b"] & & S\\
	 & Z \arrow[ru , "t", swap] &
	\end{tikzcd},
	$$
	with $t,s \in \Sigma$. 
 \end{lemma}
\begin{proof}
	The first statement is an immediate consequence of \Cref{lemCharOfSimCs}. For the second statement first note that, by \Cref{lemCharOfSimCs}, $\langle Q(a)\rangle = \langle Q(b)\rangle $ if and only if there exists a commutative diagram in $\mathcal C(\Sigma^{-1})$ $$\begin{tikzcd}
	& Y \arrow[rd, "Q(s)"] &\\
	X \arrow[rd, swap, "Q(b)"] \arrow[ru, "Q(a)"] & & S\\
	& Z \arrow[ru , swap, "Q(t)"] &
	\end{tikzcd},
	$$
	whith $s,t \in \Sigma$. Hence, we can apply \ref{eckAx2} to $s \circ a$ and $t \circ b$. Now, using that, by \ref{eckAx0}, $\Sigma$ is closed under composition, we obtain the result.
\end{proof}
In particular, we can think of $A(X)$ as the set of morphisms in $\mathcal C$ with source $X$ modulo composition with morphisms in $\Sigma$. Using this, we write $\langle a\rangle $ instead of $\langle Q(a)\rangle $ for $a \in \mathcal{C}$ from now on.
Note that this description of $A(X)$ (the induced one of $E(X)$ to be more precise) is the one used by Cohen in \cite{cohenCourse} if one takes $\hat C$ the category of finite CW-complexes with cellular maps, $\mathcal{C}$ the subcategory given by inclusions of subcomplexes and $\Sigma$ the class of finite compositions of elementary expansions. 
\begin{lemma}\label{lemAass}
	The construction $\mathcal C \to \textnormal{\textbf{MonAb}}$ in \Cref{thrmEckSieb} gives a well-defined functor into $\textnormal{\textbf{MonAb}}$.
\end{lemma}
\begin{proof}
Consider the following diagram of pushout squares in $\hat{\mathcal C}$: \begin{equation}\label{diagWhatAGrid}
	 \begin{tikzcd}
	X_{00} \arrow[r, "f_{00}"] \arrow[d, "g_{00}"] & X_{01} \arrow[r, "f_{01}"] \arrow[d, "g_{01}"] & X_{02} \arrow[d, "g_{02}"]\\
	X_{10} \arrow[r, "f_{10}"] \arrow[d, "g_{10}"] & X_{11} \arrow[r, "f_{11}"] \arrow[d, "g_{11}"] & X_{12} \arrow[d, "g_{12}"] \\
	X_{20} \arrow[r, "f_{20}"] & X_{11} \arrow[r, "f_{21}"] & X_{22} 
	\end{tikzcd},
\end{equation}
with $X_{00} = X$. By composability of pushout squares, if all the small squares are pushout squares, then so are all rectangles in the diagram. Further, note that for any operation involving pushout, the choice of pushout diagram is up to natural isomorphism. In particular, by \ref{eckAx0}, it is not relevant up to simple morphism class. \\
\\
We start by showing that ``$+$'' is well-defined on $A(X)$. So let $\langle a\rangle ,\langle b\rangle $ be simple morphism classes in $A(X)$, for $a,b \in \mathcal {C}$. By \Cref{lemEqConE}, we know that in fact any simple morphism class is of this shape and it suffices to show that the addition construction is invariant under the composition of $a$ and $b$ by morphisms in $\Sigma$. So in \eqref{diagWhatAGrid} take $f_{00} = a$, $g_{00} =b$ and $f_{01}, g_{10} \in \Sigma$. By \ref{eckAx1}, \eqref{diagWhatAGrid} then lies in $\mathcal{C}$, and all the arrows in the lower right square are in $\Sigma $. In particular, the diagonal of the upper left square differs from the diagonal of the large outer square by a composition with morphisms in $\Sigma$. As all squares involved are cocartesian, this shows both the inner upper left and the outer pushout diagonals belong to the same simple morphism class.\\
\\
If one takes $f_{00} = 1_{X_{00}}$, then the upper left square with $g_{01} = g_{00} \in \mathcal{C}$ is cocartesian in $\hat{\mathcal C}$, showing that $\langle 1_X\rangle $ in fact defines a neutral element with respect to addition. Addition clearly is commutative. To see associativity, consider a commutative cube \begin{equation}\label{diagComCube}
	\begin{tikzcd}[row sep=2.5em]
	X_{000} \arrow[rr,"f_{000}"] \arrow[dr,swap,"g_{000}"] \arrow[dd,swap,"h_{000}"] &&
	X_{100} \arrow[dd,swap] \arrow[dr] \\
	& X_{010} \arrow[rr,crossing over] && X_{110}
	 \arrow[dd] \\
	X_{001} \arrow[rr, near end] \arrow[dr,swap] && X_{101} \arrow[dr,swap] \\
	& X_{011} \arrow[rr] \arrow[uu,< -,crossing over]&& X_{111}
	\end{tikzcd},
\end{equation}in $\hat{\mathcal C}$ constructed by first taking the pushout of $f_{000}$ and $g_{000}$ and then pushing this diagram forward along $h_{000}$. A quick diagram chase and the standard properties of the pushout show that all squares in the diagram are pushout squares. In particular, the commutative squares running diagonally through the cube (given by the composition of two composable faces) are also cocartesian. Hence, by commutativity, the diagonal from $X_{000}$ to $X_{111}$ is obtained both by pushing the diagonal of the upper face square along $h_{000}$ and by pushing the diagonal of the left face square along $f_{000}$. If one takes $f, g,h \in \mathcal{C}$, then by \ref{eckAx1} the whole cube lies in $\mathcal{C}$ and the statement we have just shown is precisely the associativity of the addition on $A(X)$. Summarizing, we have shown that in fact $(A(X), +, \langle 1_X\rangle )$ defines an abelian monoid.\\
\\
We are now left with showing the functoriality of $A(-)$. Again consider \eqref{diagWhatAGrid}.
The argument for well-definedness of $f_*$ (see definition in \Cref{thrmEckSieb}), for $f: X \to X'$, works similarly to the well-definedness of ``+'', by taking $f = f_{00}$, $g$ to be $g_{00}$ and $g_{10}$ to be a morphism in $\Sigma$. One obtains ${f_1}_* \circ {f_0}_*= (f_1 \circ f_0)_{*}$ (wherever this is defined) by taking $f_{00} = f_0$, $f_{01} =f_1$ and $g_{00} = a$ a representative of a simple morphism class and again using composition of pushout diagrams. The proof that ${1_X}_* = 1_{A(X)}$ is pretty much identical to the one that $\langle 1_X\rangle $ gives a neutral element of addition and also the one that $f_*$ preserves the neutral element. To see that $f_*$ is compatible with addition, again refer to \eqref{diagComCube} and use the same argument as for associativity of "$+$".
\end{proof}
Now, that we have established that $A(-)$ induces a functor $\mathcal C \to \textbf{MonAb}$, we can also check that it descends to a functor on $\mathcal C(\Sigma^{-1})$. We start by proving the following useful lemma about simple morphism classes of compositions.
\begin{lemma}\label{lemPushAlongf}
	For $f: X \to Y$ in $\mathcal{C}$ and $g: Y \to Z$ in $\mathcal{C}$ the following equation holds:
	$$ f_*\langle g \circ f\rangle = f_*\langle f\rangle + \langle g\rangle. $$
\end{lemma}
\begin{proof}
	Just consider the following composition of pushout squares. $$ \begin{tikzcd}
	X \arrow[d, "f"] \arrow[r,"f"] & Y \arrow[r, "g"] \arrow[d] & Z \arrow[d]\\
	Y \arrow[r]& \arrow[r] Y' & Z'
	\end{tikzcd}.$$ The sum to the right corresponds to the diagonal of the right square. $f_*\langle f\rangle $ corresponds both to the middle vertical and the lower left horizontal. $f_*\langle g \circ f\rangle $ corresponds to the lower horizontal composition (by composition of pushouts). But as the lower left horizontal and the middle vertical agree, the lower horizontal composition is the diagonal of the right hand square, showing the equation.
\end{proof}
\begin{lemma}
	Let $s:X\to X'$ $\in \Sigma$. Then $s$ induces an isomorphism of monoids $s_*:A(X) \to A(X')$.
\end{lemma}
\begin{proof}
	We need to show that $s_*$ is a bijection. Let $\langle a_0\rangle ,\langle a_1\rangle \in A(X)$ given by some $f,g \in \mathcal{C}$ with source $X$. If $s_*\langle a_0\rangle = s_*\langle a_1\rangle $, then, by \Cref{lemEqConE}, we have two commutative diagrams ($i=0,1$) in $\mathcal{C}$ $$ \begin{tikzcd}
	X \arrow[d, "s"] \arrow[r, "a_i"] & Y_i \arrow[d, "s_i'"]\\
	X' \arrow[r, " a_i'"] & Y_i' \arrow[r, "t_i"]& S
	\end{tikzcd} $$
	with all $s$ and $t$ simple by \ref{eckAx1}, such that $$ t_0 \circ a_0' = t_1 \circ a_1'.$$ Applying $Q$ and chasing the diagrams, we obtain $$ Q(t_0 \circ s_0') \circ Q(a_0) = Q(t_1 \circ s_1') \circ Q(a_1).$$ As, by \ref{eckAx0}, $\Sigma $ is closed under compositions, we obtain $\langle a_0\rangle = \langle a_1\rangle $. Conversely, let $\langle b\rangle $ be in $A(Y)$. Then by \Cref{lemPushAlongf} we have an equation $$ s_*\langle b \circ s\rangle = s_*\langle s\rangle + \langle b\rangle = s_*0 + \langle b\rangle = \langle b\rangle .$$ using the fact that $s \in \Sigma $ and that $s_*$ is a morphism of monoids. In particular, we have shown surjectivity of $s_*$. 
\end{proof}
By the universal property of the localization we obtain:
\begin{corollary}\label{lemFfactors}
	$X \mapsto A(X)$ induces a functor $\mathcal C(\Sigma^{-1}) \to \textnormal{\textbf{MonAb}}$ as described in \Cref{thrmEckSieb}.
\end{corollary}
We are left with showing that the statements in \Cref{thrmEckSieb} on $E(-)$ hold. In fact, we are going to show that $E(X)$ is the group of invertible elements of $A(X)$. In particular, $E$ is just given by postcomposing $A: \mathcal C(\Sigma^{-1}) \to \textnormal{\textbf{MonAb}}$ with the functor that sends an abelian monoid to the subgroup of its invertible elements. Again we start with a series of lemmata.
\begin{lemma}\label{lemCharOfLeftInv}
Let $f: X \to X'$ be a morphism in $\mathcal{C}$. $Q(f)$ has a left inverse if and only if there exists a $\bar f: X' \to X''$ in $\mathcal{C}$ such that $\bar f \circ f \in \Sigma$.
\end{lemma}
\begin{proof}
This is an immediate consequence of \Cref{lemCharOfMorCs} and \ref{eckAx0} and \ref{eckAx2}. To be more precise, if $Q(f)$ has a left inverse, then, by \Cref{lemCharOfMorCs}, the inverse is of the shape $Q(s)^{-1} \circ Q(\bar f_0)$, for some $\bar f_0$ with source $X'$ and $s \in \Sigma$. Hence $$Q(\bar f_0 \circ f) = Q(s).$$ Now, apply \ref{eckAx2} to obtain $s'$ and $t$ in $\Sigma$ such that $$ t \circ \bar f_0\circ f = s' \circ s.$$ By \ref{eckAx0} the right hand side of this equation lies in $\Sigma$. Now, set $\bar f = t \circ \bar f_0$. This shows the only if part. The if part is obvious.
\end{proof}
As a consequence of this we have: 
\begin{lemma}
$\langle \alpha \rangle \in A(X)$ is invertible (w.r.t. ``+'') if and only if $\alpha$ is an isomorphism in $\mathcal C(\Sigma^{-1})$ i.e. $\langle \alpha \rangle \in E(X)$ . In particular, $E(X)$ is precisely the group of invertible elements of $A(X)$.
\end{lemma}
\begin{proof}
We start with the only if part. 
By \Cref{lemEqConE}, there exist $a:X \to Y$ and $b:X \to X'$ in $\mathcal{C}$ representing $\langle \alpha \rangle $ and its inverse. In particular, we have a commutative diagram in $\mathcal{C}$ $$ \begin{tikzcd}
 X \arrow[r, "a"] \arrow[d, "b"] & Y \arrow[d,"b'"] \\
 X' \arrow[r, "a'"] & Y' \arrow[r, "t"] & S
\end{tikzcd} $$
such that the composition of the diagonal of the square with $t$ lies in $\Sigma$, and that the square is cocartesian. Inverting $Q(t)$ we obtain that $Q(b')$ has a right inverse. By \Cref{lemCharOfLeftInv}, $Q(b)$ has a left inverse. However, by \ref{eckAx2} and \Cref{lemCharOfLeftInv}, the property of $Q(f)$ having a left inverse is stable under pushout. In particular, $Q(b')$ also has a left inverse, making it an isomorphism. But then $Q(a)$ is such that composition with the isomorphism $Q(t) \circ Q(b')$ makes it an isomorphism (as $t \circ b' \circ a \in \Sigma)$. In particular, $Q(a)$ and hence also $\alpha$ is an isomorphism.\\
\\
Conversely, let $\langle \alpha \rangle$ be such that $\alpha$ is invertible. Take a representative $a: X\to Y$ in $\mathcal{C}$, using \Cref{lemComHo}) of $\langle \alpha \rangle$. Then $Q(a)$ is invertible and in particular by \Cref{lemCharOfLeftInv} we find some $b: Y \to Z$ such that $b \circ a \in \Sigma$. Now, by \Cref{lemPushAlongf} we then obtain an equation $$a_*\langle b \circ a \rangle = a_*\langle a\rangle + \langle b \rangle.$$
 But as $b \circ a \in \Sigma$, the left hand side of this equation is $0$. In particular, $a_*\langle a \rangle$ has an inverse. Since $Q(a)$ is an isomorphism $a_*$ is an isomorphism of monoids, by \Cref{lemFfactors}. In particular, $\langle a \rangle $ also has an inverse.
\end{proof}
This finishes the proof of \Cref{thrmEckSieb}. We now have a general theory at our disposal that allows us to construct Whitehead groups for a wide range of settings. We apply this to filtered simplicial sets in \Cref{subsecEckSiebAppHolds}.
\label{axEckSieb}
\label{thrmEckmann}
\section{Filtered Strong Anodyne Extensions}\label{secElandFSAE}
Now, that we have a general machinery for the construction of Whitehead groups available, the next step is to specify what the combinatorial moves (expansions) - that is, the morphism in $\Sigma$ (following the notation of \Cref{subsecEckSiebApp}) - should be. Similarly to the classical theory, we take the perspective that these expansion should be generated in some sense by certain elementary moves. These moves are the (pushouts of) admissible horn inclusions (\Cref{defAdmHorn}). In \Cref{subsecElemExp}, we compare these to another possible choice of candidate. We then begin studying the resulting class of expansions (called finite filtered strong anodyne extensions, or finite FSAEs for short) in detail through methods we generalize from their non-filtered analogues found in \cite{MossSae}. We introduce these methods in \Cref{subsecFSAE} and use them to show that, in particular, many morphisms arising in the setting of the subdivision functors $\textnormal{sd}_P$ and $\textnormal{sd}$ are FSAEs (\Cref{propExLotsOffFSAEs}). We then show in \Cref{secSmall} that FSAEs interact nicely with the homotopy category $\mathcal H \textnormal{s\textbf{Set}}_P$ (\Cref{propHoCharFin}), as they can be interpreted as relative cell complexes in a small object argument. Finally, we prove that in most aspects (such as their interactions with mapping cylinders and products) they behave much like classical expansions, making them a good candidate for the construction of a simple homotopy theory. This is the content of \Cref{subsecSAEandProd}. All of these result are then be used in the next section, showing that the axioms of \Cref{defAxEckSieb} hold in this setting; that is, that we obtain a (combinatorial) stratified Whitehead group.
\subsection{Elementary expansions}\label{subsecElemExp}
%As we already illustrated in \Cref{Introduction} the basic idea of simple homotopy theory is, to study the objects of a category (equipped with some notion of weak equivalence) up to the equivalence relation generated by a subclass of these weak equivalences, called elementary expansions. Provided certain axioms are fulfilled \Cref{thrmEckmann}, one obtains a functor assigning to an objects its Whitehead group, which can essentially be thought of as homotopy classes of weak equivalences modulo these elementary operations. 
While, by \Cref{thrmEckSieb}, any class of morphism in $\textnormal{s\textbf{Set}}_P$ compatible with the axioms of \Cref{defAxEckSieb} serves for the construction of some sort of simple homotopy theory, the goal of this work is construct one that is - similarly to the classical setting - combinatorial in nature. That is, we want the class $\Sigma$ in \Cref{defAxEckSieb} to be given by composition of some type of elementary moves, which then ultimately define the whole simple homotopy theory. 
In this subsection, we try to motivate our particular choice of what an elementary expansion should be. Thus, this chapter is very example driven and not particularly result heavy. The reader (who is in a hurry) will be fine with just reading the definitions and passing on to the next chapter.\\
\\
In their recent paper \cite{banagl2020stratified}, the authors have proposed a notion of elementary expansion for filtered simplicial complexes over the poset $\{0 \leq 1\}$. They used these to reduce the size of certain filtered simplicial complexes obtained in the pursuit of applying intersection homology to a topological data analysis setting. They then went on to ask the question of whether there might be a larger class of expansion, allowing for more collapses. Translated to the language of filtered simplicial sets and generalized to arbitrary posets their definition essentially comes down to the following notion of ``strict admissibility''. For the sake of comparison, we also repeat the definition of an admissible horn inclusions (see \Cref{defStrictlyAdmissible}).
\begin{definition}\label{defStrictlyAdmissible}Let $\mathcal J = (p_0 \leq ... \leq p_n)$ be a $d$-flag in $P$.
\begin{itemize}
	\item A horn inclusion $\Lambda^\mathcal{J}_k \hookrightarrow \Delta^\mathcal J$, $0\leq k \leq n$, is called \textit{admissible} if the $k$-th vertex is repeated in $\Delta ^\mathcal J$, i.e. either $p_k = p_{k-1}$ or $p_k = p_{k+1}$. 
	\item A horn inclusion $\Lambda^\mathcal{J}_k \hookrightarrow \Delta^\mathcal J$, $0\leq k \leq n$, is called \textit{strictly admissible} if it is admissible and further $p_k$ is maximal in $\mathcal J$.
\end{itemize}
To be a little bit more concise, we sometimes just say $k$ is admissible (strictly admissible) when a flag $\mathcal J$ is specified, to refer to the above.
\end{definition}
\begin{remark}\label{RemElemAreEq}
One should note that the first condition is equivalent to the inclusion $\Lambda^\mathcal{J}_k \hookrightarrow \Delta^\mathcal J$ being a simplicial homotopy equivalence \cite[Proposition 1.13]{douSimp}, and also to its realization being stratified homotopy equivalence (see \Cref{lemCharHornInc}). However in general, only in the second case can this homotopy equivalence be made relative to the horn inclusion, i.e. only then is the inclusion a stratum preserving strong deformation retract. For a construction of such a relative homotopy, see the orthogonal deformation retraction in \cite{banagl2020stratified} and note that this works just as well for arbitrary $P$.
\begin{example}
	Consider the admissible horn inclusions from \Cref{exOfHorn}, shown again in \Cref{fig:ElemExp}.
	While in each case, checking whether the conditions of \Cref{defStrictlyAdmissible} are fulfilled is of course very easy, it can be interesting to check why they are not fulfilled from a more geometric perspective.\\
	(1) and (3) show the horn inclusions corresponding to the d-flag $\mathcal J = (0 \leq 0 \leq 1)$ for $k = 2$ and $k = 0$, respectively. While (3) is an admissible inclusion, (1) is not. A more geometric way to see the latter is that the $0$-stratum of the horn in (1) has two path components while the $0$-stratum of the simplex has one. In particular, the inclusion can not be a homotopy equivalence and thereby not admissible. (3) is not strictly admissible. Geometrically this is reflected in the fact that any map that collapses the simplex to the horn relative to the horn has to map some of the red part into the blue part. Hence, such a map is not stratum preserving.\\
	(2) and (4) show the horn inclusions corresponding to the d-flag $\mathcal J = (0 \leq 1 \leq 1)$ for $k = 0$ and $k = 2$ respectively. Geometrically speaking, (2) can not be admissible since in the horn the $1$-stratum has two connected components but in the simplex, it only has one. (4), however, even gives a strictly admissible inclusion. An inverse up to homotopy is given by projecting the lower face and the interior vertically upwards onto the horn and it is easy to check that this turns the horn inclusion into a stratum preserving strong deformation retract.
	\begin{figure}[H]
		\centering
		\includegraphics[width=\linewidth]{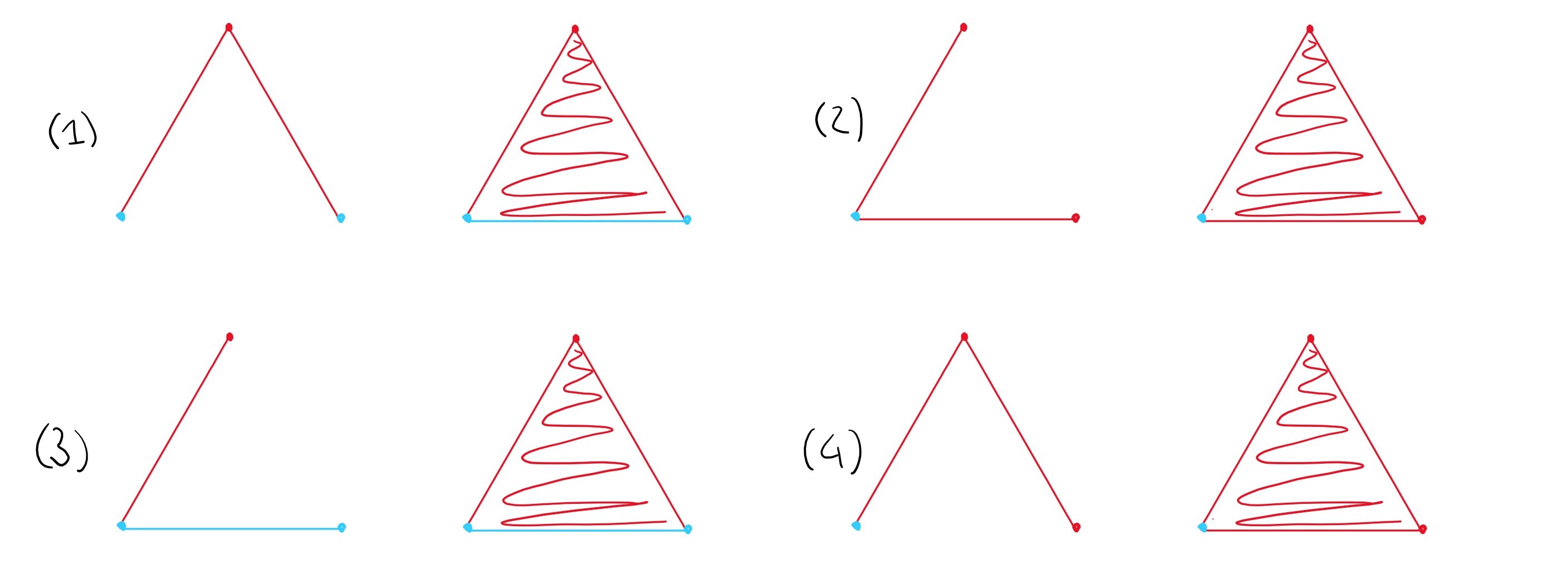}
		\caption{Examples of horn inclusions}
		\label{fig:ElemExp}
	\end{figure}
\end{example}
\end{remark}
\begin{definition}\label{defElem}
	An arrow $X \to Y$ in $\textnormal{s\textbf{Set}}_P$ is called a \textit{(strict) elementary expansion} if it fits into a pushout diagram
	$$\begin{tikzcd}
	\Lambda^\mathcal{J}_k \arrow[r, hook] \arrow[d] & \Delta^\mathcal J \arrow d\\
	X \arrow[r, hook ] & Y
	\end{tikzcd}$$
	where the upper horizontal is a (strictly) admissible horn inclusion. The formal inverse (and the inverse in the homotopy category by an abuse of notation) of such an inclusion is called a \textit{(strict) elementary collapse}.
\end{definition}
\begin{remark}
As, by construction of the latter, every admissible horn inclusion is a acyclic cofibration in the Douteau model structure, and acyclic cofibrations are closed pushouts (cobase change), every elementary expansion is a weak homotopy equivalence of $P$-filtered simplicial sets. In particular, by \Cref{corRelRef}, they realize to weak equivalences of filtered spaces. However, these will in general not be stratified homotopy equivalences (see example below). In case of a strict elementary expansion however, $|\Lambda^\mathcal{J}_k \hookrightarrow \Delta^\mathcal J |_P$ is not only a homotopy equivalence, but also a homotopy equivalence relative to $|\Lambda^\mathcal{J}_k|_P$ by \Cref{RemElemAreEq}. Hence, using the fact that $| - |_P$ preserves pushouts and the universal property of the latter, one obtains a homotopy inverse of $X \hookrightarrow Y$. Another way one could show this is by checking that in the case of strict horn inclusions the realization $|\Lambda^\mathcal{J}_k \hookrightarrow \Delta^\mathcal J |_P$ has the homotopy extension property with respect to stratum preserving homotopies. One can then use the standard argument found for example in \cite[Ch. 6]{kampsPo}
to show that stratified homotopy equivalences that have this property are stable under pushouts.
\end{remark}
\begin{example}\label{weakEqNotEq}
In \Cref{fig:SimpleEq} we see a zigzag of elementary expansions. This gives a weak equivalence in $\textnormal{s\textbf{Set}}_P$ between combinatorials of the models of the spaces in \Cref{fig:exWeakNotStr}. Notice how out of these only (3) is strict. This is reflected in the fact that the realization of the filtered simplicial set on the left hand side is actually not stratified homotopy equivalent to the one on the right hand side. We have already seen this in \Cref{exWeakNotStr}.
\begin{figure}[H]
	\centering
	\includegraphics[width=\linewidth]{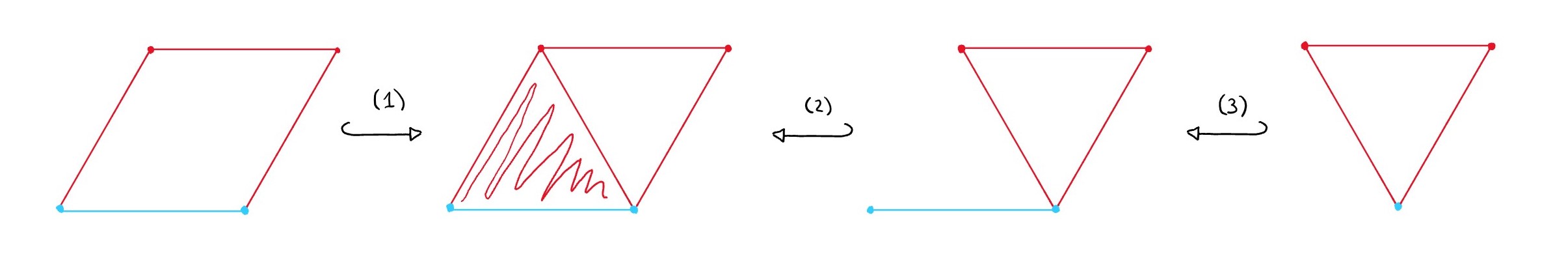}
	\caption{Example of a sequence of elementary expansions between combinatorial models of the spaces in \Cref{exWeakNotStr}.}
	\label{fig:SimpleEq}
\end{figure}
\end{example}
\begin{remark}\label{elemAreOk}
	At first glance \Cref{weakEqNotEq} might give the impression that for building a stratified simple homotopy theory "up to elementary expansion" induces too coarse of an equivalence relation and one should only consider strict expansions. However, one should note that in stratified topology one is usually mostly interested in the study of filtered spaces with particularly well behaved interactions between the strata, i.e. in increasing order of generality stratified pseudo manifolds, CS sets and homotopically stratified spaces. For the latter, we have already seen in \Cref{thrmWhitehead} that weak equivalences are also stratified homotopy equivalences. Hence, if two filtered simplicial sets have realizations that belong to one of the above classes and the two are connected by a zigzag of (realizations) of elementary expansions, then they are also stratified homotopy equivalent.
\end{remark} Another reason why working with general elementary expansions and not only with strict ones is preferable is that the latter have a series of shortcomings when it comes to building a simple homotopy theory from them. This is strongly reflected in the following example.
\begin{example}\label{exStrictBad}
Many geometric arguments in classical simple homotopy theory involve the mapping cylinder in some shape or form (see for example the classical book by Cohen \cite{cohenCourse} for a cellular or Whitehead's original notes \cite{whitehead1939simplicial} for a simplicial perspective). In particular, they often build upon the fact that for a (cellular) map $f:X \to Y$ in the classical setting the inclusion into the mapping cylinder $$Y \hookrightarrow Mf$$ is a composition of elementary expansions (or more generally a simple equivalence). However, in the filtered setting, this is hardly ever the case. Take for example $P = \{0 \leq 1\}$ and consider inclusion of $\Delta^\mathcal J$ with $\mathcal J = (0 \leq 1)$, into the cylinder $\Delta^1\otimes \Delta^\mathcal{J}$ at $0$ (see \Cref{fig:IncCyl}). \begin{figure}[H]
	\centering
	\includegraphics[width=80mm]{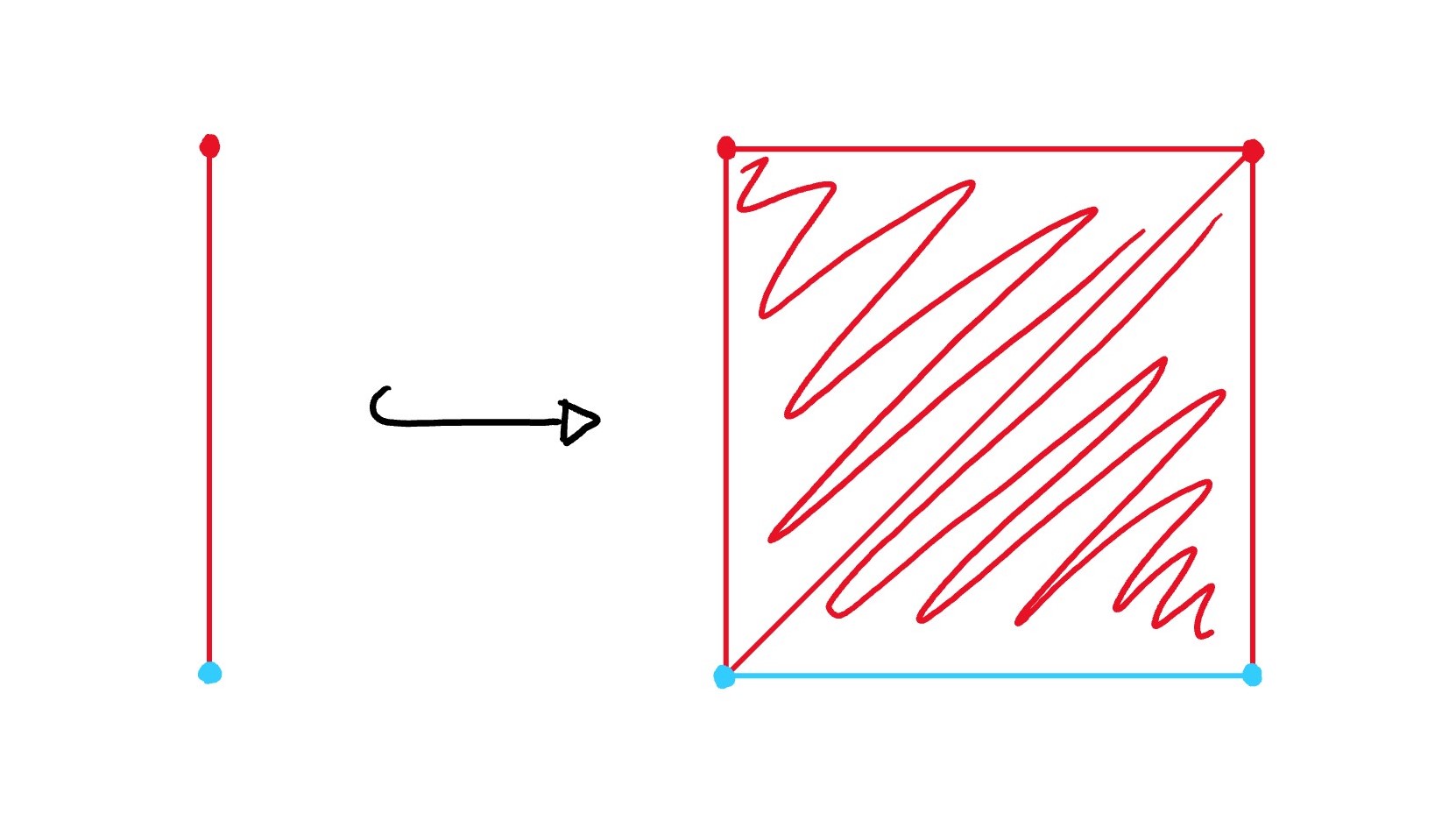}
	\caption{Inclusion at 0 into the cylinder of $\Delta^\mathcal J$, for $\mathcal J = (0 \leq 1).$}
	\label{fig:IncCyl}
\end{figure} 
In fact, these two filtered simplicial sets can not even be transformed into each other through strict elementary moves. To see this, note that the only way to remove the additional vertex in the $0$-stratum is by also removing the lower $1$-simplex of type $\mathcal J = (0 \leq 0)$ at the same or some earlier point in time. Denote this simplex by $\sigma$. In both cases, the simplex of with d-flag $(0 \leq 0 \leq 1)$ attached to $\sigma$ has to be removed first. The only way to remove such a simplex strictly is through a three simplex $\Delta^{\mathcal J'}$ with $\mathcal J'= (0 \leq 0 \leq 1 \leq 1)$ where the $(0 \leq 0)$ part corresponds to $\sigma$. But even after such a removal, $\sigma$ is again attached to a simplex with d-flag $(0 \leq 0 \leq 1)$, so we can repeat the argument.
%Clearly, there is no way to remove the upper $2$-simplex through a strict elementary collapse relative to $\Delta^{\mathcal J}$. But even if we allow for arbitrary zigzags of elementary expansions, there is no way to remove the $1$-simplex lying in the $0$-stratum, which we denote by $\sigma$.\\
%To see this, just note that such a $1$-simplex of type $(0 \leq 0)$ can not be removed as the face of a $2$-simplex that is not completely in the $0$-stratum. In particular, it can only be removed when it is no face of any such $2$-simplex (as it needs to be either free or a free face to be removed). Hence, one first needs to remove the upper $2$-simplex of type $(0 \leq 0 \leq 1)$ attached to $\sigma$. However, the only way to remove this simplex through a strict collapse is by removing a larger simplex of type $(0 \leq 0 \leq 1 \leq 1)$. But then, even after such a removal, $\sigma$ is still the face of a simplex of type $(0 \leq 0 \leq 1)$ and we can repeat the argument.
\end{example}
The above example illustrates that the equivalence relation induced by strict elementary expansions is probably too fine for many applications of interest. A fruitful workaround might be taking simple equivalences to be the class of morphisms generated by inclusions into cylinders under pushouts with cofibrations and two out of three. This is for example done in a general setting in \cite{kampsPo}. While this might turn out to be an interesting approach, it lacks (at least a priori) an obvious elementary combinatorial interpretation.\\
\\
We will see in \Cref{subsecSAEandProd} that general elementary expansions do not suffer from the shortcomings illustrated in \Cref{exStrictBad}. In fact, they are very well suited for the development of a simple homotopy theory, while interacting smoothly with the Douteau model structure on $\textnormal{s\textbf{Set}}_P$. At the same time ``up to elementary expansion'' is not too course of an invariant to be of geometric interest as it models weak equivalences in the Douteau-Model structure on $\textnormal{\textbf{Top}}_{P}$. By \Cref{thrmWhitehead}, this means it models stratified homotopy equivalences as long as source and target are stratified in some reasonable sense.
\subsection{Filtered strong anodyne extensions}\label{subsecFSAE}
One advantage of the definition of elementary expansions in \Cref{defElem} is that in the non-filtered case there is already a good amount of machinery available, when it comes to the study of these objects. Much of this machinery generalizes to the filtered setting. Recall, that in the classical Quillen model structure on simplicial sets the acyclic cofibrations are also called anodyne extensions. 
(see for example \cite[Ch. 3]{joyalNotes}). Recall further, that an equivalent characterization of this class of morphisms is given as follows.
\begin{definition}\label{defSat}
	Let $\mathcal{C}$ be a category that has small colimits. Let $\Sigma$ be a class of morphisms in $\mathcal{C}$. $\Sigma$ is called \textit{saturated} if it fulfills the following properties. \begin{enumerate}
		\item $\Sigma$ contains all isomorphisms. \label{defSatPropIso}
		\item $\Sigma$ is closed under pushouts. That is, if for a pushout square
		$$\begin{tikzcd}
		A \arrow[r, "s"] \arrow[d] & B \arrow[d]\\
		A' \arrow[r, "s'"] & B'
		\end{tikzcd}$$
		 $s \in \Sigma$ then $s' \in \Sigma$.
		\item $\Sigma$ is closed under arbitrary coproducts. That is, if $A_i \to B_i$ is a family of morphisms in $\Sigma$, then $$ \bigsqcup A_i \to \bigsqcup B_i \in \Sigma .$$
		\item $\Sigma$ is closed under $\omega$-composites. That is, if $A_i \to A_{i+1}$, $i \in \omega$, is a countable family of morphisms in $\Sigma$, then so is $$ A_0 \to \varinjlim A_i.$$ Herem $\omega$ denotes the first countable ordinal. \label{defSatPropComp}
		\item $\Sigma$ is closed under retracts. That is, if for a commutative diagram 
		$$\begin{tikzcd}
		A \arrow[r] \arrow[d, "s"] & A' \arrow[d, "s'"] \arrow[r] & A \arrow[d, "s"]\\
		B \arrow[r] & B' \arrow[r] & B
		\end{tikzcd}$$
		with both horizontal compositions the identity $s' \in \Sigma$, then $s \in \Sigma$. 
	\end{enumerate} 
\end{definition}
\begin{proposition}\cite[Cor. 3.3.1.]{joyalNotes}
	Let $\mathcal H$ be the class of horn inclusions in $\textnormal{s\textbf{Set}}$, $\Lambda^n_k \hookrightarrow \Delta^n$, for $k\leq n \in \mathbb N$. Let $\mathcal{A}$ be the smallest saturated class containing $\mathcal H$. Then $\mathcal A$ is the class of anodyne extensions which is the class of acyclic cofibrations.
\end{proposition}
The analogous statement holds for the Douteau model structure on $\textnormal{s\textbf{Set}}_P$.
\begin{proposition}\cite[Thm. 2.14.]{douSimp}\label{propACareAE}
	Let $\mathcal H$ be the class of admissible horn inclusions in $\textnormal{s\textbf{Set}}_P$. Let $\mathcal{A}$ be the smallest saturated class containing $\mathcal H$. Then $\mathcal A$ is the class of acyclic cofibrations, with respect to the Douteau model structure.
\end{proposition}
From here on out, when we refer to cofibrations we always mean cofibrations with respect to the Douteau model structure on $\textnormal{s\textbf{Set}}_P$ if not stated otherwise.
Out of the closure properties in \Cref{defSat} only the closure under retracts is non-constructive. It turns out, that a lot can still be said about the closure of the class of horn inclusions under property \ref{defSatPropIso}-\ref{defSatPropComp} in \Cref{defSat}. 
\begin{definition}\label{defSae}
	Let $\mathcal H$ be the class of admissible horn inclusions \\$\Lambda^\mathcal J_k \hookrightarrow \Delta^\mathcal J$ in $\textnormal{s\textbf{Set}}_P$. Let $\mathcal{SA}$ be the smallest class containing $\mathcal H$ that is closed under property \ref{defSatPropIso}-\ref{defSatPropComp} of \Cref{defSat}. An element of $\mathcal{SA}$ is called a \textit{filtered strong anodyne extension}, or \textit{FSAE} for short. If both the target, as well as the source of an FSAE are finite, it is called a \textit{finite FSAE.}
\end{definition}
By \Cref{propACareAE}, every FSAE is an acyclic cofibration.
In the context of small object arguments (\cite[Ch. 10.5]{hirschhornModel}) the class constructed in \Cref{defSae} is often referred to as the relative cell complexes (with respect to the class of horn inclusions). We take this perspective in \Cref{secSmall}.
The astute reader will immediately notice that any composition of elementary expansions \Cref{defElem} is an FSAE. It turns out that the class of FSAEs is in fact the closure of elementary expansions under transfinite composition (see \Cref{propEqCharSaeTot}). 
\\
We start with a series of examples of finite FSAEs. However, we will only show that they are in fact FSAEs at the end of this section.
\begin{example}\label{ExLotsOfFSAEs}
All of the following filtered simplicial sets involved are simplicial complexes. Thus, we only describe the morphisms on the vertices. 
\begin{enumerate}
	\item Let $\mathcal J = ( p_0 \leq ... \leq p_n )$ and $\mathcal J'$ be d-flags in $P$ such that there is a degeneracy morphism $\Delta^{\mathcal J} \to \Delta^{\mathcal J'}$. Then any section of this morphism $$\Delta ^{\mathcal J'} \hookrightarrow \Delta ^{\mathcal J}$$ is a finite FSAE. \label{ExLotsOfFSAE1}
		\item Let $\mathcal J = ( p_0 \leq ... \leq p_n )$ be a d-flag in $P$. Then the cofibration 
	\begin{align*}
		 \Delta^{\mathcal J} &\hookrightarrow \textnormal{sd}(\Delta^{\mathcal J}) \\
		 p_k &\mapsto ( p_0 \leq ... \leq p_k )
	\end{align*} is a finite FSAE. \label{ExLotsOfFSAE2}
	\item Let $\mathcal{J}$ be a flag in $P$. Then the cofibration 
	\begin{align*}
		\Delta^{\mathcal J} &\hookrightarrow \textnormal{sd}_{P}(\Delta^{\mathcal J})\\
		p &\mapsto (p, \mathcal J)
	\end{align*}
	is a finite FSAE. See \Cref{fig:exLotsOfFSAE3}. \label{ExLotsOfFSAE3}
	\item For any FSAE, $f: X \hookrightarrow Y$, the induced map $$\textnormal{sd}_P(X) \xhookrightarrow{\textnormal{sd}_P(f)} \textnormal{sd}_P(Y)$$ is also an FSAE. \label{ExLotsOfFSAEDout}
\end{enumerate} 
\begin{figure}[H]
	\centering
	\includegraphics[width=100mm]{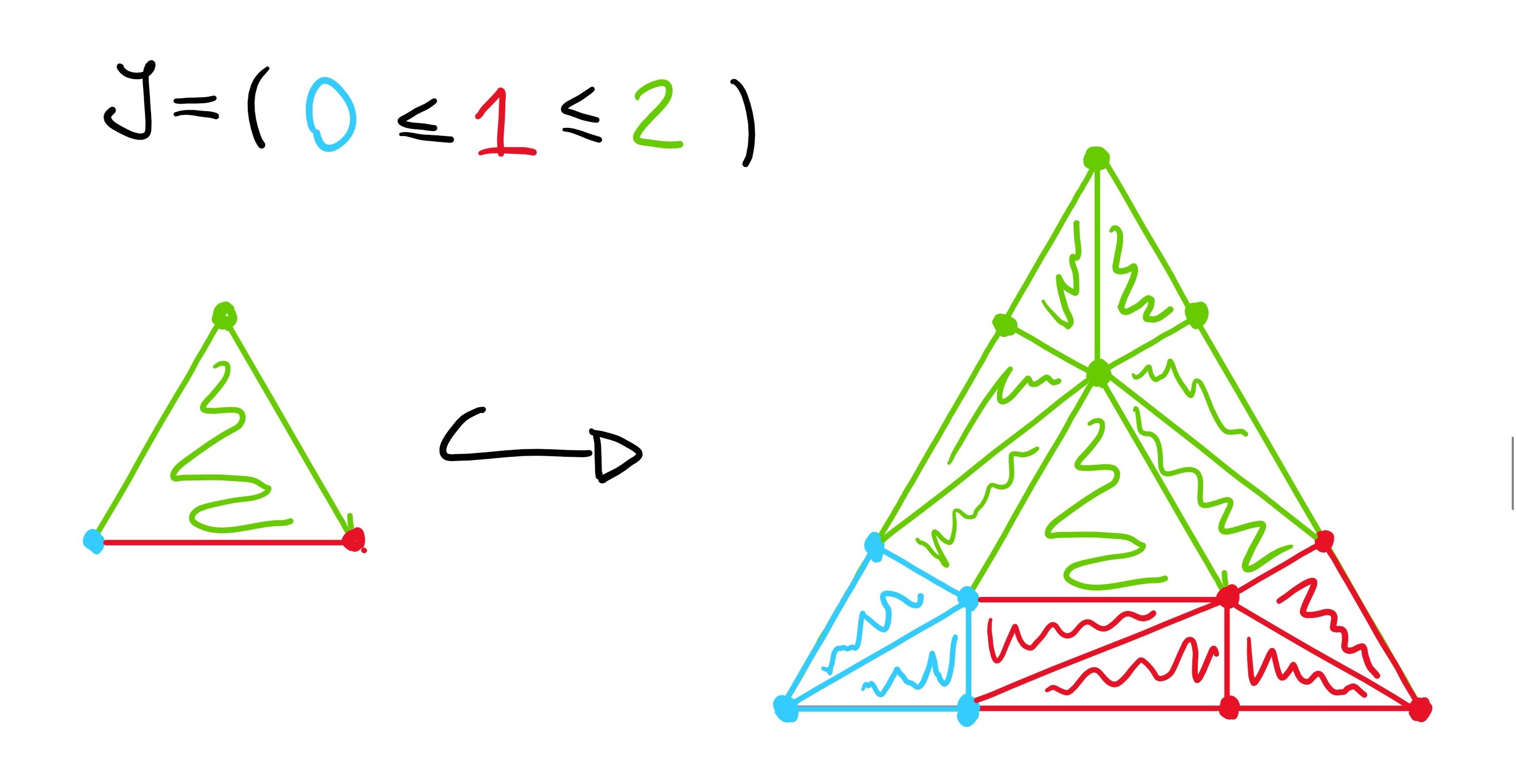}
	\caption{Illustration of the FSAE in \ref{ExLotsOfFSAE3}.}
	\label{fig:exLotsOfFSAE3}
\end{figure}
\end{example}
In the non-filtered case, strong anodyne extensions have recently been studied in \cite{MossSae}. There, the author develops an even more constructive perspective on the class of strong anodyne extentions. In particular, it is shown that the fibrant replacement inclusion $X \hookrightarrow \operatorname{Ex}(X)$ is a strong anodyne extension \cite[Thm. 22]{MossSae}. Douteau has transferred some of this work to the setting of filtered simplicial sets in \cite[Ch. 3.3.2]{douteauFren}. In this subsection, we expand on this generalization. We use this in \Cref{subsecSAEandProd} to prove that the filtered analogon of \cite[Thm. 22]{MossSae}, \Cref{propCylIncs}, still holds. This, together with the fact that FSAEs are just infinite compositions of elementary expansions, is one of the deciding steps in the development of a simple homotopy theory for filtered simplicial sets in \cref{subsecEckSiebAppHolds}. Finally, at the end of this subsection we show a useful lemma (\Cref{lemSav}) that essentially allows us to replace arbitrary FSAEs by finite compositions of elementary expansions in many situations.\\
\\
We start by translating some of the definition in \cite{MossSae} into the filtered setting and showing how they are related to the non-filtered on. This is pretty much a verbatim copy of what is being done there, replacing simplicial sets by filtered simplicial sets and adding the admissibility conditions. Hence, we do not go into too much detail and refer to \cite{MossSae}. By a slight abuse of notation, we treat cofibrations $A \hookrightarrow B$ as inclusions of sub-simplicial sets when it comes to notation such as $ B \setminus A$, translating all the inclusions to general monomorphisms would be as tedious as it is trivial. For a ($P$-filtered) simplicial set $B$, we denote by $B_{n.d.}$ the set of its non-degenerate simplices. For now, let 
$ A \xhookrightarrow{s} B$ be a cofibration in $\textnormal{s\textbf{Set}}_P$.
\begin{definition}
	An anodyne presentation of $s$ consists of an ordinal $\kappa$ and a $\kappa$-indexed increasing family of sub-simplicial sets $(A^{\alpha})_{\alpha \leq \kappa} $ of $B$ satisfying: 
	\begin{itemize}
		\item $A^{0} = A$ and $A^{\kappa} = B$,
		\item For every non-zero limit ordinal $\lambda < \kappa$, $\bigcup_{\alpha < \lambda} A^{\alpha} = A^{\lambda}$ holds.
		\item For every $\alpha < \kappa$ the inclusion $A^{\alpha} \hookrightarrow A^{\alpha +1}$ is a pushout of a coproduct of admissible horn inclusions.
	\end{itemize}
	In other words, $A \to B$ is the transfinite composition of pushouts of coproducts of admissible horn inclusions.
\end{definition}
Recall that a binary relation $\prec$ on a set $S$ is called \textit{well founded} if every non-empty $S' \subset S$ has a minimal element. That is, it exists an $x \in S'$ such that for no $y \in S'$ the relation $y \prec x$ holds.
\begin{definition} \text{ }
\begin{enumerate}
		\item A pairing $T$ on $s$ is a partition of $B_{n.d.} \setminus A_{n.d.}$ into disjoint sets $B_I$ and $B_{II}$ together with a bijection $T: B_{II} \to B_{I}$. Simplices in $B_{II}$ or $B_I$ are called of \textit{type} $II$ or \textit{type} $I$, respectively.
		\item A pairing is called \textit{proper} if for each $\sigma \in B_{II}$ the following holds: There exists a unique $k \leq \text{dim} (T(\sigma))$ such that $\sigma = d_k(T(\sigma))$ and $k$ is such that $\Lambda^\mathcal J_k \hookrightarrow \Delta^\mathcal J$ is an admissible horn inclusion, where $\mathcal J$ is the d-flag corresponding to $T(\sigma)$).
		\item A proper pairing $T$ on $s$ induces a relation on $B_{II}$ given by: $\sigma \prec \tau$ if and only if $\sigma \neq \tau$ and $\sigma$ is a face of $T(\tau)$. This is called the \textit{ancestral relation}.
		\item The smallest transitive and reflexive relation on $B_{n.d.} \setminus A_{n.d.}$ generated by $T(\sigma) \preceq \sigma$, for $\sigma \in B_{II}$, as well as $d_k(\sigma) \preceq \sigma$, for $\sigma \in B_{n.d.}$, is called the \textit{ancestral preorder} and denoted by $\preceq_T$. 
		\item A proper pairing $T$ is called \textit{regular} if its ancestral relation is well founded.
	\end{enumerate}
\end{definition}
It can be helpful to decode these definitions a bit.
\begin{remark}\label{remPresToPair}
 Essentially, the point is that regular pairings are what arises from an anodyne presentation one does the following. For every horn $ \Lambda^\mathcal J_k \hookrightarrow \Delta^\mathcal J \to B$ in $B$ that is filled at some point of the presentation, if $\tau$ is the non-degenerate simplex of $B$ corresponding to $\Delta^\mathcal{J} \to B$, then $\tau$ is assigned to $B_{II}$ and $\sigma:=d_k(\tau)$ to $B_{I}$. One then constructs $T$ via $T(\sigma):=\tau$. One easily verifies that this defines a proper pairing. The ancestral preorder for this pairing specifies in what order simplices appear in the anodyne presentation. The map \begin{align*}
	B_{II} &\to \kappa \\
	x &\mapsto \text{sup}\{\alpha \mid \sigma \notin A^\alpha\}
\end{align*}
is relation-preserving, with respect to the ancestral relation on the left. Any relation that maps relation preserving into an ordinal is well founded (this is an easy exercise in elementary set theory). Thus, this means our pairing is regular.\\
\\
We should thus think of the condition of properness as specifying that the pairing actually pairs non-degenerate simplices with faces that correspond to admissible horn inclusions. The condition of regularity on the other hand specifies that the adding of simplices happens in a well mannered order, i.e. that no simplex is added before all but one of its proper faces are present and that in the case of $B_{n.d.} \setminus A_{n.d}$ being infinite, this can be done in a transfinitely iterative fashion. In practice, there is a useful alternative characterization of regularity that is often easy to verify. The proof is identical to the non-filtered case in \cite{MossSae}.
\begin{lemma}\cite[Lem. 14]{MossSae}\label{lem14Moss}
	Let $T$ be a proper pairing on $s$. Then $T$ is regular if and only if
	there exists a function \begin{align*}
		\Phi: B_{II} & \longrightarrow \mathbb{N}
	\end{align*}such that, for each $n$ and for all type II simplices
	$\sigma$ and $\tau$ of dimension $n$, the implication \begin{align*}
	\sigma \prec \tau \implies \Phi(\sigma) < \Phi(\tau)
	\end{align*}
	holds.
\end{lemma}
\end{remark}
\begin{example}
\Cref{fig:Pairing} shows an anodyne presentation in $\textnormal{s\textbf{Set}}$ as well as the corresponding pairing and its ancestral relation. The paired simplices have been marked with the same color, which should not be confused with a filtration.
\begin{figure}[H]
\includegraphics[width=\linewidth]{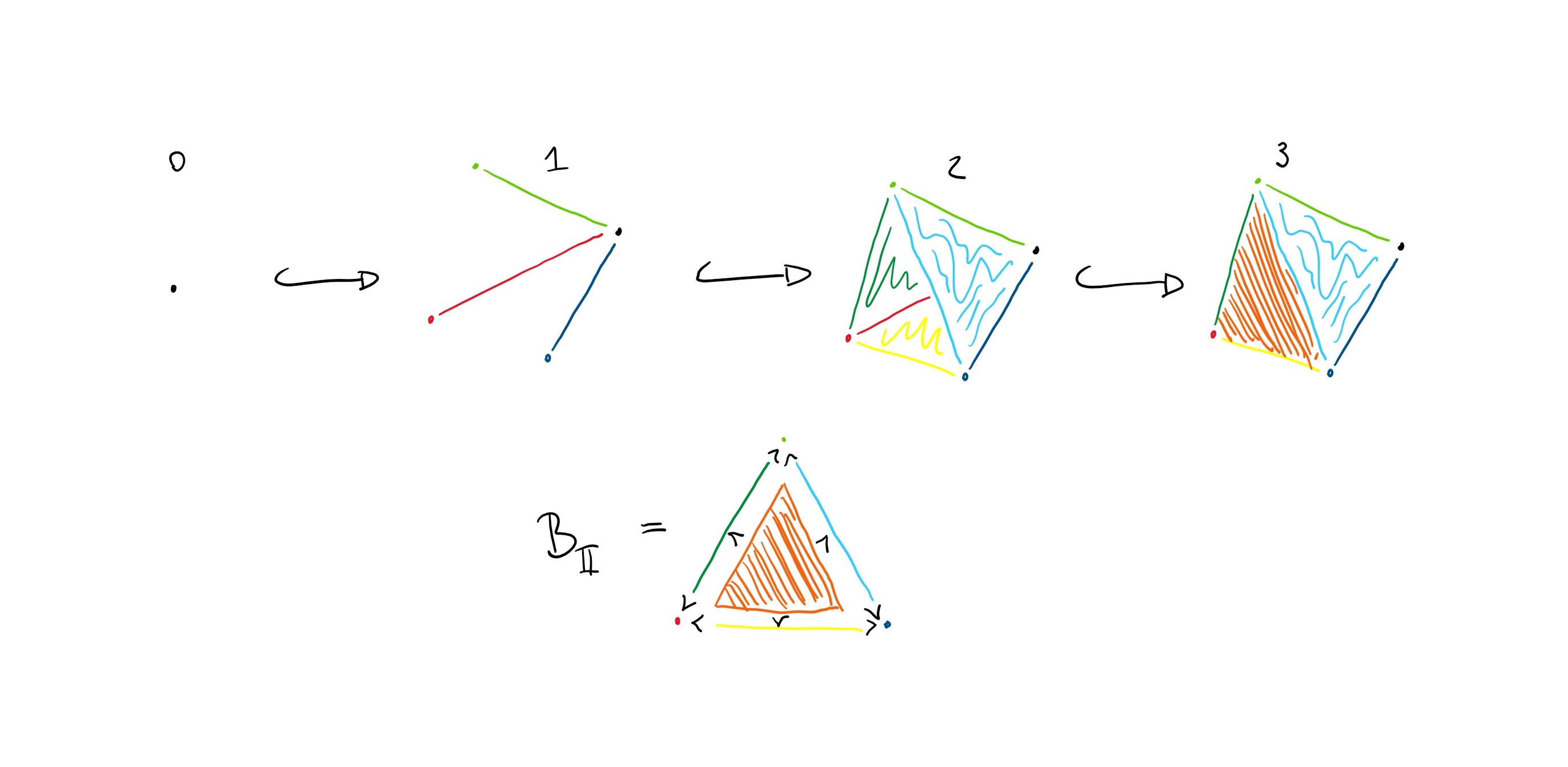}
\caption{An anodyne presentation of an inclusion $\Delta ^0 \hookrightarrow \Delta^3$.}
\label{fig:Pairing}
\end{figure}
\end{example}
As there has already been quite a bit of work on pairings in the non-filtered setting in \cite{MossSae}, it can be useful to know when a proper pairing on the underlying simplicial sets also specifies a pairing of filtered simplicial sets.
\begin{lemma}\label{lemSimtoFil}
	Let $T$ be a pairing on $s$. Let $s'$ be the image of $s$ under the forgetful functor $\textnormal{s\textbf{Set}}_P \to \textnormal{s\textbf{Set}}$. Let $T'$ be the pairing induced by $T$ on $s'$. Then the following hold. \begin{enumerate}
\item $T$ is proper if and only if $T'$ is proper and, for each $x \in B_{II}$ of type $\mathcal J$, the unique $k$ such that $d_k(T(x)) = x$ is such that $ \Lambda^{\mathcal J}_k \hookrightarrow \Delta^\mathcal J$ is admissible.
\item $T$ is regular if and only if $T'$ is regular.
	\end{enumerate}
\end{lemma}
\begin{proof}
	This is immediate from the definition of properness and the fact that the additional regularity condition is clearly independent from the filtration maps into $N(P)$.
\end{proof}
Next, we want to prove the converse of \Cref{remPresToPair}, i.e. that every regular pairing gives rise to an anodyne presentation. Before we can prove this, we need a useful technical lemma.
\begin{lemma}\label{lemTech}
	Let $A \hookrightarrow B$ be an FSAE with proper pairing $T$ and corresponding ancestral preorder $\preceq_T$.
	Let $\sigma \in B_{n.d.} \setminus A_{n.d.}$. Then the downset $S$ generated by $\sigma$ under $\preceq_T$ is still finite. %\\
	%\\
	%Further, $B'$ is such that $A \hookrightarrow B'$ and $B' \hookrightarrow B$ are both FSAEs.
\end{lemma}
\begin{proof}
First note that $\sigma' \succeq_T \tau$ is equivalent to there being a finite sequence $$\sigma'=\sigma_0 \succeq_T \sigma_1 \succeq_T ... \succeq_T \sigma_n = \tau$$ such that in every step either $\sigma_{i+1} = d_k(\sigma_{i})$ or $\sigma_{i+1} = T(\sigma_i)$. Hence, $S = \bigcup_{n \in \omega} S^n$ where $S^n$ is inductively defined as the closure of $S^{n-1}$, first under the pairing $T$ and then under the face relation, and we set $S^0 = \{\sigma\}$. In particular, as both closures preserve finiteness, each $S^n$ is finite. Now, assume $S$ is infinite. Then, by finiteness of the $S^n$, no $S^n$ can contain all of $S$. In particular, (by refining a little if necessary) we obtain an infinite sequence $$\sigma=\sigma_0 \succeq_T \sigma_1 \succeq_T ... \succeq_T \sigma_n ...$$ where in each step $d_k(\sigma_{i}) = \sigma_{i+1}$, for some $k$, or $T(\sigma_{i}) = \sigma_{i+1}$. Let $\sigma'_{i}$ be the subsequence given by such $\sigma_{i}$ where $T(\sigma_{i}) = \sigma_{i+1}$. By definition of the ancestral preorder, we then have $\sigma_i' \succ \sigma'_{i+1}$. In particular, by well foundedness, there are only finitely many $\sigma'_i$. But at the same time, these are precisely the elements of the sequence $\sigma_i$ where $\text{dim}(\sigma_{i+1}) > \text{dim}(\sigma_i)$. Everywhere else, $\text{dim}(\sigma_{i+1}) = \text{dim}(\sigma_i) -1$. In particular, the sequence $\sigma_i$ obtains negative dimension, which is a contradiction.
\end{proof}
The following remark corresponds to \cite[Prop. 12]{MossSae}.
\begin{remark}\label{remPairToPres}
Conversely to \Cref{remPresToPair}, every regular pairing $T$ gives rise to an anodyne presentation of length $\omega$ (i.e. of countable length) as follows. Define, using induction on the well foundedness of the ancestral relation: \begin{align*}
	F:B_{II} &\to \text{Ord}; \\
	\sigma&\mapsto \text{sup}\{1, F(\tau) + 1 \mid \tau \prec \sigma\},
\end{align*} where $\text{Ord}$ denotes well founded class of all ordinals.
As the ancestral relation is well founded and the set that we take the supremum over is necessarily finite by \Cref{lemTech}, one can inductively show that $F(\tau)$ is a natural number, i.e. the construction above maps into $\omega$.\\
\\
One now defines $A^n$, $n \in \omega$, as the filtered simplicial subset of $B$ with non-degenerate simplices $$A^n_{n.d.} = A_{n.d.} \cup \{\sigma, T(\sigma) \in B_{n.d.} \mid F(\sigma) \leq n\}$$ and checks that this fulfills the requirements of an anodyne presentation.
\end{remark}
These constructions are summarized in the following proposition:
\begin{proposition}\label{propEqCharSaeTot}
Let $A \xrightarrow{s} B$ be a morphism in $\textnormal{s\textbf{Set}}_P$. Then the following are equivalent.
\begin{enumerate}
	\item $s$ is an FSAE. \label{proEqCharSaeI1}
	\item $s$ is a cofibration and admits a regular pairing. \label{proEqCharSaeI2}
	\item $s$ is a cofibration and admits a countable anodyne presentation. \label{proEqCharSaeI3}
	\item $s$ is a cofibration and admits an anodyne presentation. \label{proEqCharSaeI4}
	\item $s$ is a transfinite composition of elementary expansions. (By a transfinite composition over $0$ we mean an isomorphism here.) \label{proEqCharSaeI5}
	\item $s$ lies in the smallest class of morphisms that contains all admissible horn inclusions and isomorphisms and is furthermore closed under pushout and transfinite composition. \label{proEqCharSaeI6}
\end{enumerate} 
\end{proposition}
\begin{proof}
	The equivalences \ref{proEqCharSaeI2} $\iff$ \ref{proEqCharSaeI3} $\iff$ \ref{proEqCharSaeI4} are the content of \Cref{remPairToPres} and \Cref{remPresToPair}. \\
	\\\ref{proEqCharSaeI5} implies \ref{proEqCharSaeI4} by definition. Conversely, let $(A^\alpha)_{\alpha \leq \kappa}$ be an anodyne presentation of $A \hookrightarrow B$. Each inclusion $A^\alpha \hookrightarrow A^{\alpha +1}$ is given by a pushout $$ 
	\begin{tikzcd}
	\bigsqcup_{i\in J} \Lambda^{\mathcal J_i}_{k_i} \arrow[r, hook] \arrow[d] & \bigsqcup_{i\in J} \Delta^{\mathcal J_i} \arrow[d]\\
	A^\alpha \arrow[r] & A^{\alpha+1}
	\end{tikzcd}
	$$ 
	with the vertical given by a disjoint union of admissible horn inclusions. Choose a well founded ordering of $J$, $J \xleftarrow[\psi]{\sim} \kappa_\alpha$, for some ordinal $\kappa_\alpha$. Via transfinite induction define $A^{\alpha,\beta} \hookrightarrow A^{\alpha,\beta +1}$ by the pushout:
	$$\begin{tikzcd}
	\Lambda^{\mathcal J_{\psi(\beta + 1)}}_{k_{\psi(\beta + 1)}}
	 \arrow[r, hook] \arrow[d]
	 & \Delta^{\mathcal J_{\psi(\beta + 1)}} \arrow[d]\\
	A^{\alpha, \beta} \arrow[r] & A^{\alpha, \beta +1}
	\end{tikzcd}$$
	and via $$ A^{\alpha, \beta} \hookrightarrow A^{\alpha,\lambda} = \varinjlim_{\beta' < \lambda} A^{\alpha,\beta'}$$ for limit ordinals $\lambda \leq \kappa_\alpha$. Then (up to natural isomorphism) $A^{\alpha} \hookrightarrow A^{\alpha +1}$ is given by the transfinite composition of the diagram given by $(A^{\alpha, \beta})_{\beta < \kappa_{\alpha}}$ All the maps to successor ordinals are given by elementary expansions. Thus, the statement to be shown follows by the fact that a transfinite composition of transfinite compositions is again a transfinite composition.\\
	\\Again, by definition, we have \ref{proEqCharSaeI1} $\implies$ \ref{proEqCharSaeI6} and \ref{proEqCharSaeI3} $\implies$ \ref{proEqCharSaeI1}. So, by the equivalences we have already shown, it suffices to show \ref{proEqCharSaeI6} $\implies$ \ref{proEqCharSaeI5}. We need to show that the class given by transfinite compositions of elementary expansions is closed under pushout and transfinite composition. For pushouts, this is clear by commutativity of cobase change and transfinite compositions (or any colimit, that is) and composition of pushout squares. For transfinite composition, closedness follows by composability of transfinite compositions.
\end{proof}
In particular, we obtain the following useful characterization of finite FSAEs.
\begin{corollary}\label{corSaeareExp}
	Let $A \xhookrightarrow{s} B$ be a cofibration such that $B_{n.d.} \setminus A_{n.d.}$ is finite. Then, the following are equivalent.
	\begin{enumerate}
		\item $s$ is a (finite) FSAE.
		\item $s$ admits a regular pairing. \label{corSaeareExpCHarReg}
		\item $s$ admits a finite anodyne presentation.
		\item $s$ is a composition, possibly empty, of elementary expansions. 
		\item $s$ lies in the smallest class of morphisms that contains all admissible horn inclusions and isomorphisms and is closed under pushout and composition. \label{corSaeareExpCharClass}
	\end{enumerate}
\end{corollary}
\Cref{corSaeareExp} can now be used to finally show the following:
\begin{proposition}\label{propExLotsOffFSAEs}
All cofibrations in \Cref{ExLotsOfFSAEs} are finite FSAEs.
\end{proposition}
\begin{proof}
For \ref{ExLotsOfFSAEDout} note that as $\textnormal{sd}_P$ is a left adjoint, it respects all colimits. Hence, by \ref{corSaeareExpCharClass} of \Cref{corSaeareExp}, it suffices to show that $\textnormal{sd}_P$ sends admissible horn inclusions into finite FSAEs. This in done in the proof of \cite[Prop. 2.9]{douSimp}. For the remaining three examples the proof is very similar and relies on \ref{corSaeareExpCHarReg} of \Cref{corSaeareExp} and \Cref{lem14Moss}.\\
\\
We begin with \ref{ExLotsOfFSAE2} of \Cref{ExLotsOfFSAEs} as the proof is by far the most involved of the remaining three. Note that we may assume that $\mathcal{J}$ is non-degenerate, as the requirements of a regular pairing are strictly stronger in the non-degenerate case. So, without loss of generality, let $\mathcal J = [q]$ for some $q \in \mathbb N$. For $p \in [q] \cup \{-1\} $ we define $D^p$ to be the full filtered subcomplex of $\textnormal{sd}(\Delta^{[q]})$ spanned by the vertices $\sigma$ such that for all $p' \in \sigma$ the implication
\begin{align*}
	r \leq \min(p',p) \implies r \in \sigma
\end{align*}
holds. Thinking of $\Delta^{[q]}$ as a subcomplex of $\textnormal{sd}(\Delta^{[q]})$ via the cofibration in question, we then have 
\begin{align*}
	\Delta^{[q]} = D^{q} \subset D^{q-1} \subset ... \subset D^{-1} = \textnormal{sd}(\Delta^{[q]}).
\end{align*}
Note that, due do a slight quirk in the indexing, which could certainly be fixed by a shift, we also have $D^q= D^{q-1}$. However, this makes the remainder of the proof a little more clean when it comes to indexing.
\begin{figure}[H]
	\centering
	\includegraphics[width=120mm]{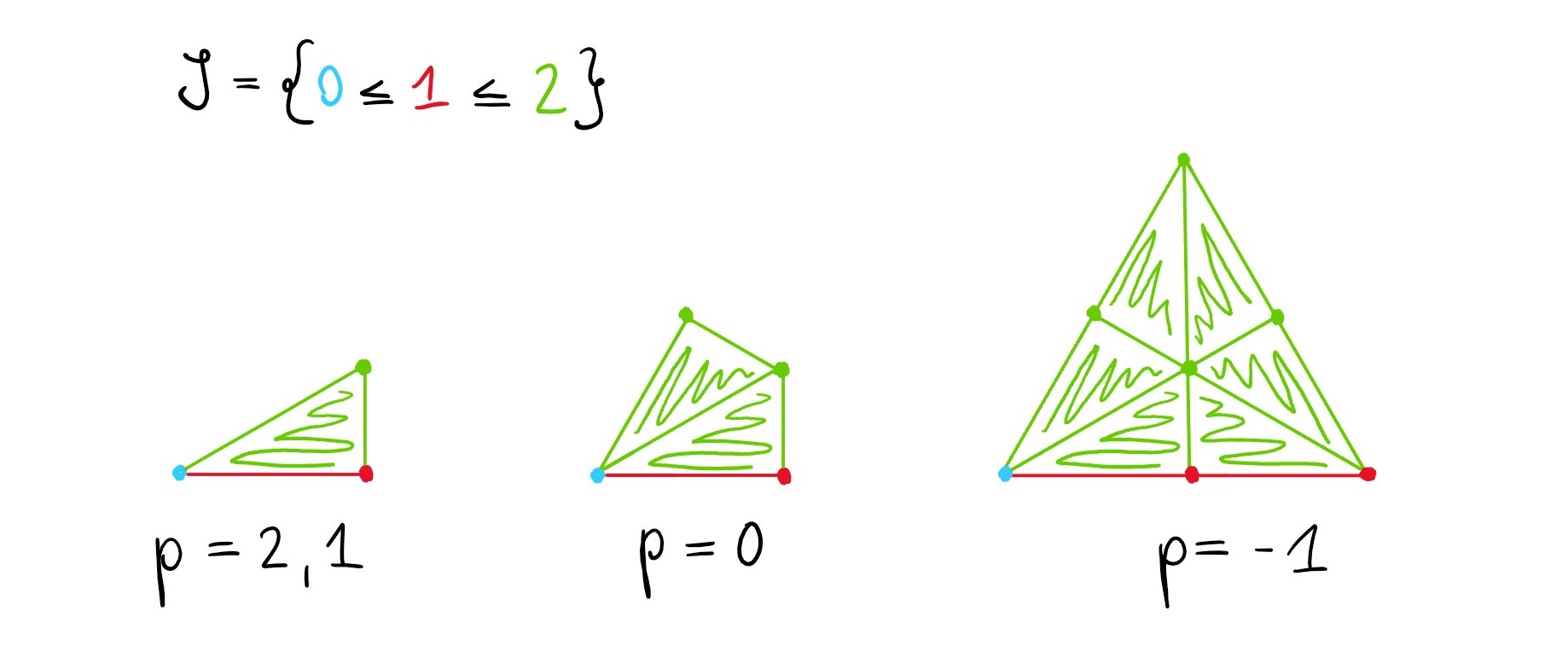}
	\caption{Illustration of the $D^p$ for $q=2$.}
	\label{fig:propExLotsOffFSAED}
\end{figure}
We construct a regular pairing for each of these inclusions (with the first one being trivial). Let $p \in [q]$ and $B:= D^{p-1}, A:= D^{p}$. For $(\sigma_0\leq ... \leq\sigma_k) \in D^{p-1}_{n.d.} \setminus D^{p}_{n.d.}$, by definition, there exists an $i \in \{0,...,k\}$ such that $\sigma_i \notin D^{p}$. One easily deduces, $p \notin \sigma_i$ and there exists a $p' \geq p$ such that $p' \in \sigma_i$. Let $m \in \{0,...,k\}$ be maximal with respect to $\sigma_m$ not containing $p$. Therefore, as $\sigma_i \subset \sigma_m$ we also have $p' \in \sigma_m$. Now, there are two cases: either $m < k$ and $\sigma_{m+1}=\sigma_m \cup \{p\}$, or not. We set $B_{I}$ to the set of simplices with the former property, and $B_{II}$ to the set of simplices with the latter property. Now, define 
\begin{align*}
	T: B_{II} &\longrightarrow B_{I}\\
	( \sigma_0\leq ...\leq \sigma_{k}) &\longmapsto (\sigma_0\leq ...\leq \sigma_m\leq \sigma_m \cup \{p\}\leq \sigma_{m+1}\leq...\leq\sigma_{k}).
\end{align*}
Note first that $\sigma_m \cup \{p\}$ trivially still lies in $D^{p}$. The map is clearly bijective. Furthermore, as $\sigma_m$ contains a $p' >p$, adding $p$ to $\sigma_m$ does not change the stratum $\sigma_m$ it lies in. Hence, the admissibility condition of properness is fulfilled, making this a proper pairing. Define \begin{align*}
	\Phi: B_{II} &\longrightarrow \mathbb{N}\\
	(\sigma_0\leq ...\leq \sigma_{k}) &\longmapsto m,
\end{align*}
where $m$ is maximal with respect to $\sigma_m$ not containing $p$. Now, if $(\sigma_0\leq ...\leq \sigma_{n}) \prec (\tau_0\leq ...\leq \tau_{n})$, then, by definition, $$(\sigma_0\leq ...\leq \sigma_{n}) \subset (\tau_0\leq ...\leq \tau_m\leq \tau_m \cup \{p\} \leq \tau_{m+1}\leq ...\leq \tau_{n}).$$ As $\sigma \neq \tau $ and both lie in $B_{II}$ this means $$(\sigma_0\leq ...\leq \sigma_{n}) = (\tau_0\leq ...\leq \tau_{m-1}\leq \tau_m \cup \{p\}\leq \tau_{m+1}\leq ...\leq \tau_{n}).$$ In particular, $$\Phi(\sigma) = m-1 < m = \Phi(\tau).$$ By \Cref{lem14Moss}, this shows the regularity of $T$. \begin{figure}[H]
	\centering
	\includegraphics[width=\textwidth]{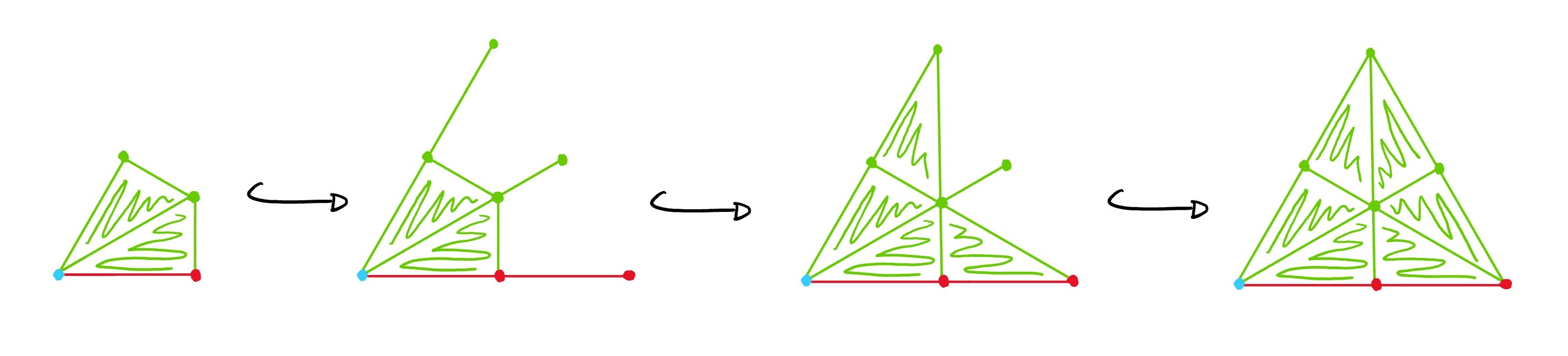}
	\caption{Illustration of the induced anodyne presentation of $D^{0} \hookrightarrow D^{-1}$, for $q=2$. The steps in this presentation are the ones corresponding to $\textnormal{dim}(\sigma)=0,1,1$ and $\Phi(\sigma)=0,0,1$ in this order.}
	\label{fig:propExLotsOffFSAEE}
\end{figure}
The proof of the remaining cases \ref{ExLotsOfFSAE1} and \ref{ExLotsOfFSAE3} is very similar to this. So we only specify $B_I$, $B_{II}$, $T$ and $\Phi$ and leave it to the reader to check that this actually defines a regular pairing. $B$ and $A$ here refer to the respective codomain and domain of the cofibration. For \ref{ExLotsOfFSAE1}, first note that ,by composability of FSAEs, we may restrict to the case $\mathcal J = (x_0\leq...\leq x_k\leq...\leq x_n)$ and 
$\mathcal J' = (x_0\leq...\leq x_k\leq x_k...\leq x_n)$. The section of the degeneracy map is the either given by the $k$-th or $(k+1)$-th face map. We do the $k$ case, the other one is analogous. Denote the vertices of $\mathcal J'$ by $y_0,...,y_{n+1}$. A simplex in $B_{n.d.} \setminus A_{n.d.}$ is given by a d-flag $(y_{i_0}\leq ...\leq y_{k} \leq ... \leq y_{i_{j}})$. We then set $B_{I}$ to those simplices where the vertex $y_{k}$ is followed by $y_{k+1}$, and $B_{II}$ to the set of those where it is not. Further, set: \begin{align*}
	T: B_{II} &\longrightarrow B_I\\
	(y_{i_0}\leq ...\leq y_{k}\leq ... y_{i_{j}}) &\longmapsto (y_{i_0}\leq ...\leq y_{k}\leq y_{k+1}\leq ... y_{i_{j}}).
\end{align*} 
Note that there is no relevant (in the sense of \Cref{lem14Moss}) ancestral relation to check for this pairing, so it is automatically regular. 
\\
For \ref{ExLotsOfFSAE3} a simplex in $B_{n.d.} \setminus A_{n.d.}$ is given by a flag $((p_0, \sigma_0)\leq...\leq(p_n, \mathcal \sigma_n))$ in $P \times \textnormal{sd}(\Delta^{\mathcal J})$, where $p_i \in p_{\Delta^{\mathcal J}}(\sigma_{0})$ and $\sigma_0 \neq \mathcal J$. Again, let $m$ be maximal with respect to $\sigma_m \neq \mathcal J$. Further, let $B_{I}$ be the set of those flags where $m < n$ and $p_m = p_{m+1}$, and $B_{II}$ to the set of those where this does not hold. Define \begin{align*}
	T: B_{II} &\longrightarrow B_I\\
	\big((p_0, \sigma_0)\leq...\leq(p_n, \mathcal \sigma_n) \big) &\longmapsto \big ((p_0, \sigma_0)\leq...\leq(p_m, \mathcal \sigma_m)\leq (p_m, \mathcal J)\leq...\leq(p_n, \mathcal J) \big)
\end{align*}
and 
\begin{align*}
	\Phi: B_{II} &\longrightarrow \mathbb{N}\\
		\big ((p_0, \sigma_0)\leq...\leq(p_n, \mathcal \sigma_n) \big ) &\longmapsto m,
\end{align*}
with $m$ as above. We have illustrated the induced anodyne presentation in \Cref{fig:AnIncSDp}.
\begin{figure}[H]
	\centering
	\includegraphics[width=\textwidth]{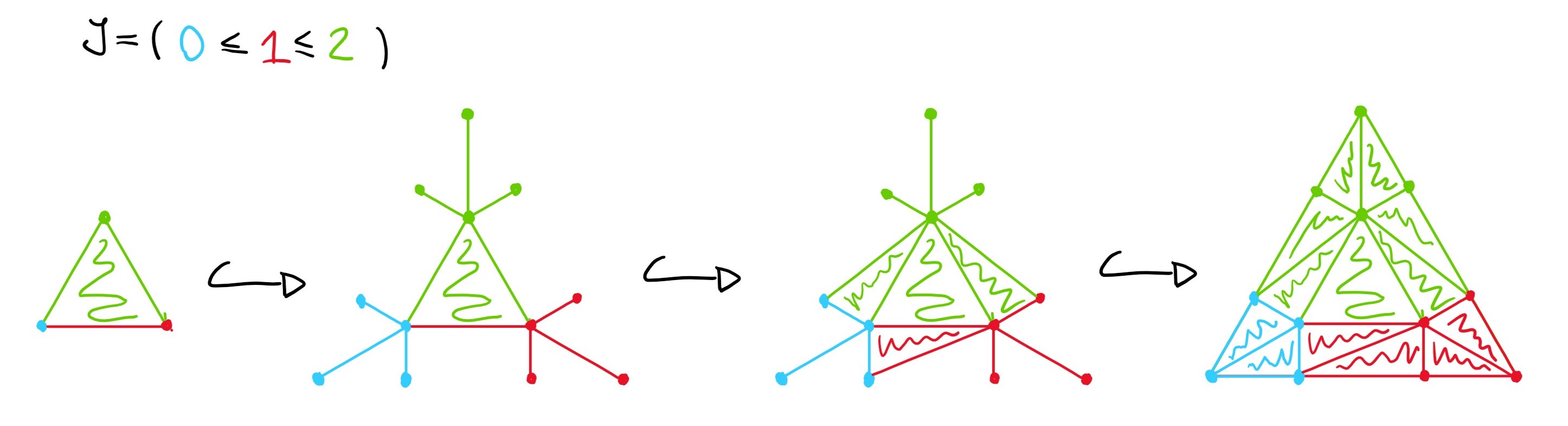}
	\caption{Illustration of the anodyne presentation induced by the construction for the proof of \ref{ExLotsOfFSAE3} above, in the case $\mathcal J = (0 \leq 1 \leq 2).$ The steps in this presentation are the ones corresponding to $\textnormal{dim}(\sigma)=0,1,1$ and $\Phi(\sigma)=0,0,1$ (in this order).}
	\label{fig:AnIncSDp}
\end{figure}
\end{proof}
To finish this section, we prove \Cref{lemSav}, which essentially allows us to replace fibrant replacements by finite FSAEs. It is at the heart of many arguments in the construction of a simple homotopy theory for filtered simplicial sets.
\begin{lemma}\label{lemSav}
Consider a diagram in $\textnormal{s\textbf{Set}}_P$ as below
$$\begin{tikzcd}
C \arrow[rd] & A \arrow[d, hook, "s"]\\
& B
\end{tikzcd},
$$
where $A$ and $C$ are finite simplicial sets and $s$ is an FSAE. Then there exists a finite filtered simplicial subset $B' \subset B$ and a factorization as below:
$$\begin{tikzcd}
C \arrow[rdd] \arrow[rd, dashed] & A \arrow[d, hook, dashed] \arrow[dd, hook, bend left = 60, "s"] 
\\ & B' \arrow[d , hook]
\\
& B
\end{tikzcd}
$$
such that both $A \hookrightarrow B'$ and $B' \hookrightarrow B$ are FSAEs.
\end{lemma}
\begin{proof}
	By mono-epi factorization in $\textnormal{sSet}$ we may without loss of generality assume that $C, A \subset B$. Now, choose a regular pairing for $s$, $T$. This exists by \Cref{propEqCharSaeTot}. Now, let $S \subset B$ be the downset generated by $C_{n.d.}\setminus A_{n.d.}$ in $B_{n.d.}\setminus A_{n.d.}$ with respect to $\preceq_T$. As $C_{n.d.}$ is finite, by \Cref{lemTech}, $S$ is finite. Further, $S \cup A_{n.d.}$ is closed under the face relation, by construction. In particular, there is a unique finite filtered sub-simplicial set $B' \subset B$ with $B'_{n.d.} = S \cup A_{n.d.}$. By definition of the ancestral preorder, $\preceq_T$, and the fact that $S = B'_{n.d.} \setminus A_{n.d.}$, we have $\sigma \in B'_{n.d.} \setminus A_{n.d.}$ if and only if $T(\sigma) \in B'_{n.d.} \setminus A_{n.d.}$. Hence, $T$ restricts to a pairing on $A \hookrightarrow B'$. This is clearly still regular and it is proper, as any subset of a well founded set is well founded. By the same argument, the same holds for the restriction of $T$ to $B' \hookrightarrow B$. In particular, we get a factorization diagram as in the claim.
\end{proof}
\subsection{A small object argument}\label{secSmall} The reader familiar with Quillens small object argument (see for ex. \cite[Ch. 10.5]{hirschhornModel}) will have noticed that it applies to the class of admissible horn inclusions and that the relative cell complexes in this setting are just the FSAEs. The reader unfamiliar with the argument is advised a quick skim of \cite[Ch. 10.5]{hirschhornModel} as we rely on some nomenclature from this source. 
\begin{definition}\cite[Def. 10.5.15]{hirschhornModel}\label{defSmallApp}
	Let $\mathcal C$ be a category and $\mathcal I$ a set of morphisms in $\mathcal C$. We say that $\mathcal I$ \textit{permits the small object argument} if the domains of the elements in $\mathcal I$ are small relative to $\mathcal I$ (see \cite[Def. 10.5.12]{hirschhornModel}).
\end{definition}
\begin{example}\label{exHornSmall}
	The set of admissible horn inclusions $\mathcal H$ admits the small object argument. This is immediate, as all objects involved are finite filtered simplicial sets. Hence, by \Cref{lemCharComp}, the sources of $\mathcal H$ are small with respect to any class of morphisms.
\end{example}
The small object argument then states:
\begin{proposition}\label{propSmallObj}\cite[Prop. 10.5.16]{hirschhornModel}
	If $\mathcal C$ is a cocomplete category and $\mathcal I$ is a set of maps in $\mathcal C$ that permits the small object argument, then there is a functorial factorization of every map in $\mathcal C$ into a relative $\mathcal I$-cell complex (see \cite[Def. 10.5.8]{hirschhornModel}) followed by a morphism that has the right lifting property with respect to all morphisms in $\mathcal I$.
\end{proposition}
\begin{corollary}\label{corFibReplace}
	There exists a functor $F: \textnormal{s\textbf{Set}}_P \to \textnormal{s\textbf{Set}}_P$ together with a natural transformation $1 \xrightarrow{i} F$ such that for each $X \in \textnormal{s\textbf{Set}}_P$: \begin{enumerate}
		\item $F(X)$ is fibrant.
		\item $i_X$ is an FSAE.
	\end{enumerate}
\end{corollary}
\begin{proof}
By \Cref{exHornSmall} we can apply \Cref{propSmallObj} to the class of admissible horn inclusions $\mathcal H$ and terminal maps $X \to N(P)$. We obtain a functorial splitting $$\begin{tikzcd}
X \arrow[rd] \arrow[d, dashed, hook] \\
F(X) \arrow[r] & N(P)
\end{tikzcd}.
$$ where the vertical morphism lies in $cell(\mathcal H)$, i.e. is an FSAE and the horizontal morphism has the right lifting property with respect to $\mathcal H$, making it a fibration.
\end{proof}
\Cref{corFibReplace} is particularly useful because it sheds light on the interaction of FSAEs and the homotopy category of $\textnormal{s\textbf{Set}}_P$, $\mathcal{H}\textnormal{s\textbf{Set}}_P$. This is summarized in the following proposition which we are using all throughout the remainder of this work. 
\begin{proposition}\label{propHoCharFin}
	Let $X,Y \in \textnormal{s\textbf{Set}}_P$.
	\begin{enumerate}
			\item Let $X \xrightarrow{\alpha} Y$ be a morphism in $\mathcal{H}\textnormal{s\textbf{Set}}_P$. Then there exists a morphism $f:X \xhookrightarrow Z$ in $\textnormal{s\textbf{Set}}_P$ and an FSAE $Y \xrightarrow{s} Z$ such that $$ \alpha = [s]^{-1}\circ [f].$$ If $X$ and $Y$ are finite, then $Z$ can be taken to be finite.
			\item Let $X \xrightarrow{f,g} Y$ be two morphisms in $\textnormal{s\textbf{Set}}_P$. Then $[f] = [g]$ if and only if there exists an FSAE $Y \xhookrightarrow{s} Z$ such that $s \circ f$ and $s \circ g$ are strictly simplicially (stratum preserving, elementary) homotopic. If $X$ and $Y$ are finite, then the same statement holds with $Z$ finite.
	\end{enumerate}
\end{proposition}
\begin{proof}
 Denote by $F, 1\xrightarrow{i}F$ the fibrant replacement functor from \Cref{corFibReplace}. We start with the proof of $(i)$.
We only do the proof in the finite case, as the infinite one is a strict simplification of it. Let $X,Y \in \textnormal{s\textbf{Set}}_P^{\textnormal{fin}}$ and $X \xrightarrow{\alpha} Y$ be a morphism in $\mathcal{H}\textnormal{s\textbf{Set}}_P$. Then, as $F(Y)$ is fibrant and $X$ cofibrant, $\alpha$ fits into a commutative diagram \begin{equation*}\label{diagFulA}
\begin{tikzcd}
X \arrow[r, "\alpha"] \arrow[rd, "{[f]}",swap] 
& 
Y \arrow[d, "{[i_Y]}"]\\
& F(Y)
\end{tikzcd},
\end{equation*}
where $f$ is a morphism $X \xrightarrow{f} F(Y)$ in $\textnormal{s\textbf{Set}}_P$ (see \Cref{propHoClasses}). Now, apply \Cref{lemSav} to get a factorization 
$$\begin{tikzcd}
X \arrow[r, "\alpha"] \arrow[rdd, "{[f]}", swap] \arrow[rd, dashed, "{[f']}"] & Y \arrow[d, hook, dashed, "{[i_Y']}"] \arrow[dd, hook, bend left = 60, "{[i_Y]}"] 
\\ & Y' \arrow[d , hook]
\\
& F(Y)
\end{tikzcd},
$$
where we know of the outer diagram and all of the triangle diagrams but the most upper one that they are commutative, and where $i'_Y$ is an FSAE and $Y'$ is finite. Using the fact that the morphism $Y' \to F(Y)$ is an FSAE and hence an isomorphism in the homotopy category and chasing the diagram shows that the most upper triangle also commutes. To summarize, we have shown $$[i'_Y] \circ \alpha = [f']$$ and hence $$\alpha = [i'_Y]^{-1}[f']$$ for the FSAE $i_Y'$ between finite filtered simplicial sets.\\
\\
The if part is immediate from the fact that every FSAE is a weak equivalence and again \Cref{propHoClasses}. For the only if statement, consider the fibrant replacement $Y \xhookrightarrow{i_Y} F(Y)$. Then $[i_Y \circ f] = [i_Y \circ g]$. As $X$ is cofibrant and $F(Y)$ is fibrant, we have a simplicial homotopy $H: X \otimes \Delta^1 \to F(Y)$ between $i_Y \circ f$ and $i_Y \circ g$ (see for ex. \cite{hirschhornModel}). In other words there is a commutative diagram
$$ \begin{tikzcd}
X \sqcup X \arrow[d, "i_0 \sqcup i_1"]\arrow[r, "{f \sqcup g}"] 					& Y \arrow[d, hook, "i_Y"]\\
X \otimes \Delta^1 \arrow[r, "H"]															& F(Y) 
\end{tikzcd}.$$
If we apply \Cref{lemSav} to this we obtain: 
$$ \begin{tikzcd}
X \sqcup X \arrow[dd, "i_0 \sqcup i_1"]\arrow[rr, "f \sqcup g"] &					& Y \arrow[dd, hook, "i_Y"] \arrow[ld, hook, dashed, "i'_Y"] \\
&	Y' \arrow[rd, hook] 	\\
X \otimes \Delta^1 \arrow[rr, "H"] \arrow[ru, "H'", dashed]			 &					& F(Y) 
\end{tikzcd}, $$
where $i_Y'$ is an FSAE of finite filtered simplicial sets.
A quick diagram chase together with the fact that $Y' \to F(Y)$ is a monomorphism shows that the upper left part of the diagram commutes. In particular, $i_Y' \circ f$ and $i_Y' \circ g$ are homotopic through a simplicial homotopy. 
\end{proof}
%\subsection{$X \hookrightarrow \operatorname{Ex}_P$ is a strong anodyne extension}
%Recall that for $X \in \textnormal{s\textbf{Set}}_P$ a fibrant approximation is constructed as follows. For an $\mathcal{J}=[p_0,...,p_n]$ simplex the by $$
%X \hookrightarrow \operatorname{Ex}_P(X); \\
%$$
\subsection{The interaction of FSAEs with mapping cylinders and products}\label{subsecSAEandProd}
Vital to many of the arguments we need for the construction of a Whitehead group for filtered simplicial sets are the various interactions of strong anodyne extensions with mapping cylinders. In the classical cellular setting, this can be found in \cite[Ch. 2, \S 5]{cohenCourse} and from a simplicial perspective in Whiteheads original publication in \cite[Sec. 6]{whitehead1939simplicial}. We have already seen that the class of strict elementary expansions is rather limiting in this sense, c.f. \Cref{exStrictBad}. In this subsection we show that the same is not the case for general elementary expansions.\\
\\
Recall that in a simplicial model category the simplicial mapping cylinder, $M_f$ of a morphism $f: X \to Y$ is defined via the pushout
\begin{center}
	\begin{tikzcd}
		X \cong X \otimes \Delta^0 \arrow[r, "f"] \arrow[d, "i_0"] & \arrow[d] Y\\
		 X \otimes \Delta^1 \arrow[r] & M_f 
	\end{tikzcd}.
\end{center}
Recall further that if $X$ is cofibrant, then $Y \to M_f$ is an acylic cofibration and $X \xhookrightarrow{i_1} X \otimes \Delta^1 \to M_f$ is a cofibration. This allows to factor any such $f$ into an acyclic cofibration followed by the induced weak equivalence $p_Y:M_f \to Y$. For the category of simplicial sets with the Kan-Quillen model structure, the statement that $Y \hookrightarrow M_f$ is a acyclic cofibration (i.e. an anodyne extension) can be slightly strengthened. This is a consequence of the following result.
\begin{proposition}\cite[Prop. 16]{MossSae}
	Let $A \xhookrightarrow{s} B$ be a strong anodyne extensions and $X \xhookrightarrow{i} Y$ be any cofibration (both in $\textnormal {sSet}$). Then $$ B \times X \cup_{A \times X} A \times Y \hookrightarrow B \times Y$$ is also a strong anodyne extension.
\end{proposition}
In particular, one obtains that the inclusions $X \hookrightarrow X \times \Delta^1$ at both endpoints is a strong anodyne extensions. Hence, for $X \xrightarrow{f} Y$ a morphism in $\textnormal{s\textbf{Set}}_P$ the inclusion of $Y \hookrightarrow M_f$, is a strong anodyne extension too, by stability under pushouts. This is of course the simplicial analogue to the respective result for mapping cylinders of cellular maps and cellular expansions (see for example \cite[Cor. 5.1A]{cohenCourse}). It turns out this result generalizes to the filtered setting. To obtain this in a fairly general form, we need the following outer product construction.
\begin{definition}
First, note that for two partially ordered sets $P$ and $P'$ there is a natural isomorphism $N(P) \times N(P') \cong N(P \times P')$. Under this identification, we obtain an outer product functor $$- \otimes -: \textnormal{s\textbf{Set}}_P \times \textnormal{s\textbf{Set}}_P \longrightarrow \textnormal{s\textbf{Set}}_{P\times P'},$$ by simply taking the product of arrows into $N(P)$ and $N(P')$ respectively. In case where $P'$ is a point, this corresponds to the simplicial copower on $\textnormal{s\textbf{Set}}_P$ under $\textnormal{s\textbf{Set}}_{P'} \cong \textnormal{s\textbf{Set}}$, justifying the overload of notation. From the presheaf perspective of \Cref{conPreShPers} this construction is given by: $$X \otimes Y: \Delta(P \times P')^{op} \hookrightarrow \Delta(P)^{op} \times \Delta(P')^{op} \xrightarrow{X \times Y} \textnormal{\textbf{Set}} \times \textbf{\textnormal{\textbf{Set}}} \xrightarrow{- \times -} \textbf{Set}.$$
\end{definition}
\begin{proposition}\label{propIncInProd}
	Let $A \xhookrightarrow{s} B$ be a strong anodyne extension in $\textnormal{s\textbf{Set}}_P$ and let $X \hookrightarrow Y$ be a cofibration in $\textnormal{sSet}_{P'}$. Then
	$$ B \otimes X \cup_{A \otimes X} A \otimes Y \hookrightarrow B \otimes Y$$
	is a strong anodyne extension in $\textnormal{sSet}_{P \times P'}$.
\end{proposition}
\begin{proof}
	This is really just a modification of the proof of \cite[Prop. 16]{MossSae}. The proof goes consists of two steps. First we show that the class of $X \xhookrightarrow{s} Y$ that have this property for some fixed $i$ is closed under: 
	\begin{enumerate}
		\item arbitrary coproducts,
		\item pushouts,
		\item transfinite composition,
	\end{enumerate}
	and that the analogous statement holds with rolls of $s$ and $i$ swapped. For the second step, note that the class of cofibrations in $\textnormal{sSet}_{P'}$ is generated under these operations by the boundary inclusions $ \partial \Delta ^{\mathcal J'}$ $\hookrightarrow \Delta^{\mathcal J'}$, for d-flags $\mathcal J'$ in $P'$ and the class of FSAEs is generated by the admissible horn inclusions $\Lambda^\mathcal J_k \hookrightarrow \Delta^\mathcal J$, for flags $\mathcal J$ in $P$. Hence, it suffices to show that, for such $\mathcal J$, $\mathcal J',$ $k$, 
	$$ \Delta ^\mathcal J \otimes \partial \Delta^{\mathcal J'} \cup_{\Lambda^{\mathcal J}_k
		 \otimes \partial \Delta ^{\mathcal J'}}\Lambda_k^\mathcal{J} \otimes \Delta ^{\mathcal J'} \hookrightarrow \Delta^\mathcal J \otimes \Delta ^{\mathcal J'} $$ 
	is an FSAE. \\
	\\
	The proof of the first step is analogous in both arguments, so we only do the one with fixed $i$. First note that the exterior product $\otimes$ does commute with colimits in both arguments (by the analogous statement in $\textbf{Set}$). From this (i) immediately follows. For (ii), note that for an arrow $A \to A'$ and $B'$ the pushout of $s$ along this arrow, the following diagram is a pushout diagram. $$ \begin{tikzcd}
	B \otimes X \cup_{A \otimes X} A \otimes Y \arrow[r] \arrow[d] & B \otimes Y \arrow d\\
	B' \otimes X \cup_{A' \otimes X} A' \otimes Y \arrow[r] & B' \otimes Y
	\end{tikzcd}$$
	An easy way to verify this statement is to note that by the alternative construction of $\otimes$ via the perspective of categories of $\textnormal{\textbf{Set}}$-valued presheaves (\Cref{conPreShPers}) we may just as well check this in $\textnormal{\textbf{Set}}$, replacing $\otimes$ by $\times$. This is an easy exercise of manipulations of pushouts in a closed monoidal category. It can be found in the appendix (\Cref{lemSetPerspective}). We are lacking (iii). So, let $A: \kappa \to \textnormal{sSet}_{P \times P'}$ be a transfinite composition diagram for some ordinal $\kappa$. By transfinite induction, let $\kappa$ be minimal such that we have not shown closedness under transfinite composition (the induction start is obvious). If $\kappa = \kappa' + 1$ is a successor ordinal then $A^\kappa = A^{\kappa'}$. If $\kappa'$ is a successor ordinal, then the transfinite composition is just the one given by the restriction of $A$ to $\kappa'$. So the result follows by the induction hypothesis. If $\kappa'=\kappa''+1$ is a successor ordinal, then consider the diagram below:
	\begin{equation}\label{diagSat}
	\begin{tikzcd}
	A^{\kappa''} \otimes X\cup_{A^0 \otimes X} A^0 \otimes Y \arrow[r, hook] \arrow[d, hook] & 	A^{\kappa''}\otimes Y \arrow[d, hook]\\
	A^{\kappa'} \otimes X\cup_{A^0 \otimes X} A^0 \otimes Y \arrow[r , hook] & A^{\kappa'} \otimes 
	X\cup_{A^{k''} \otimes X} A^{\kappa''} \otimes Y \arrow[r, hook] & A^{\kappa'} \otimes Y = A^{\kappa} \otimes Y
	\end{tikzcd}.
	\end{equation}
	Similarly to the previous diagram, one checks in \textnormal{\textbf{Set}} that the square in this diagram is a pushout diagram. By assumption, the upper vertical and the lower right morphism are FSAEs. By stability under pushouts and composition, the map on the bottom is an FSAE, but this is precisely the map, corresponding to the composition $A^0 \to A^{\kappa''} \to A^{\kappa'} = A^{\kappa}$ i.e. to the transfinite composition we considered. Now, let $\kappa$ be a limit ordinal. Just as before, consider the pushout square in \eqref{diagSat} but replace (formally) $\kappa'$ with $\kappa$, $\kappa''$ with $\kappa'$, and $0$ with $\kappa''$. We obtain a diagram $\kappa \to \textnormal{sSet}_{P\times P'}$ given by 
	$$ \kappa'' < \kappa' \mapsto \{ A^{\kappa} \otimes X\cup_{A^{\kappa''} \otimes X} A^{\kappa''} \otimes Y \hookrightarrow A^{\kappa} \otimes 
	X\cup_{A^{k'} \otimes X} A^{\kappa'} \otimes Y\}.$$ By the induction hypothesis and stability of FSAEs under pushout, all the arrows in this diagram are FSAEs. Further, by commutativity of colimits and the fact that $\otimes$ commutes with colimits, this fulfils the requirements of a transfinite composition diagram, for each limit ordinal $\kappa ' < \kappa$. The colimit of this diagram is $A^\kappa \otimes Y$ (again, this is easily checked in $\textbf{Set}$ by the same argument as before). In particular, $$ A^{\kappa} \otimes X\cup_{A^{0} \otimes X} A^{0} \otimes Y \hookrightarrow A^\kappa \otimes Y$$ is a transfinite composition of FSAEs, making it an FSAE also.\\
	\\
	For the second part of the proof, consider a d-flag $\mathcal J= (p_0 \leq ... \leq p_m)$ in $P$, a d-flag $\mathcal J'= (p'_0 \leq ... \leq p'_m) $ in $P'$ and some $k \leq n$ such that $\Lambda^{\mathcal{J}}_k \hookrightarrow \Delta ^\mathcal J$ is admissible. We use the pairing in the proof of \cite[Prop. 16]{MossSae} and \Cref{lemSimtoFil} to show that the latter still works in the filtered setting. We give a sketch of this construction. Let $A := \Delta ^\mathcal J \otimes \partial \Delta^{\mathcal J'} \cup_{\Lambda^{\mathcal J}_k
		\otimes \partial \Delta ^{\mathcal J'}}\Lambda_k^\mathcal{J} \otimes \Delta ^{\mathcal J'}$ and $B:= \Delta^{\mathcal J } \otimes \Delta ^{\mathcal J'}.$ \\
	\\
	We are now in the setting where all the simplicial sets involved come from simplicial complexes. On the underlying simplicial sets $\otimes$ is just given by the product of ordered simplicial complexes. Every $N$-simplex in $B$ is given by a finite sequence $$(\mu_0, \nu_0) \leq (\mu_1, \nu_1) \leq ... \leq (\mu_N, \nu_N)$$ in $[m] \times [n]$ with respect to the induced preorder on the product. The simplex is non-degenerate if at each step the $\mu$- or the $\nu$-entry strictly increases. The $i$-th face is obtained, by leaving out the $i$-th entry. Thus, $i$ is admissible with respect to the product filtration of this simplex if and only if $ (p_{\mu_i}, p'_{\nu_i}) = (p_{\mu_{i+1}}, p'_{\nu_{i+1}})$ or $ (p_{\mu_i}, p'_{\nu_i}) = (p_{\mu_{i-1}}, p'_{\nu_{i-1}}).$ A non-degenerate simplex in $B$ belongs to $B_{n.d.} \setminus A_{n.d.}$ if and only if in the corresponding sequence the the $\nu_i$ skip no value in $[n]$ and the $\mu_i$ skip no value in $[m]$ other than $k$. Pictorially, we may think of such a sequence as a walk in the $[m] \times [n]$ grid as in \Cref{fig:Path}.
	\begin{figure}[H]
		\includegraphics[width=\linewidth]{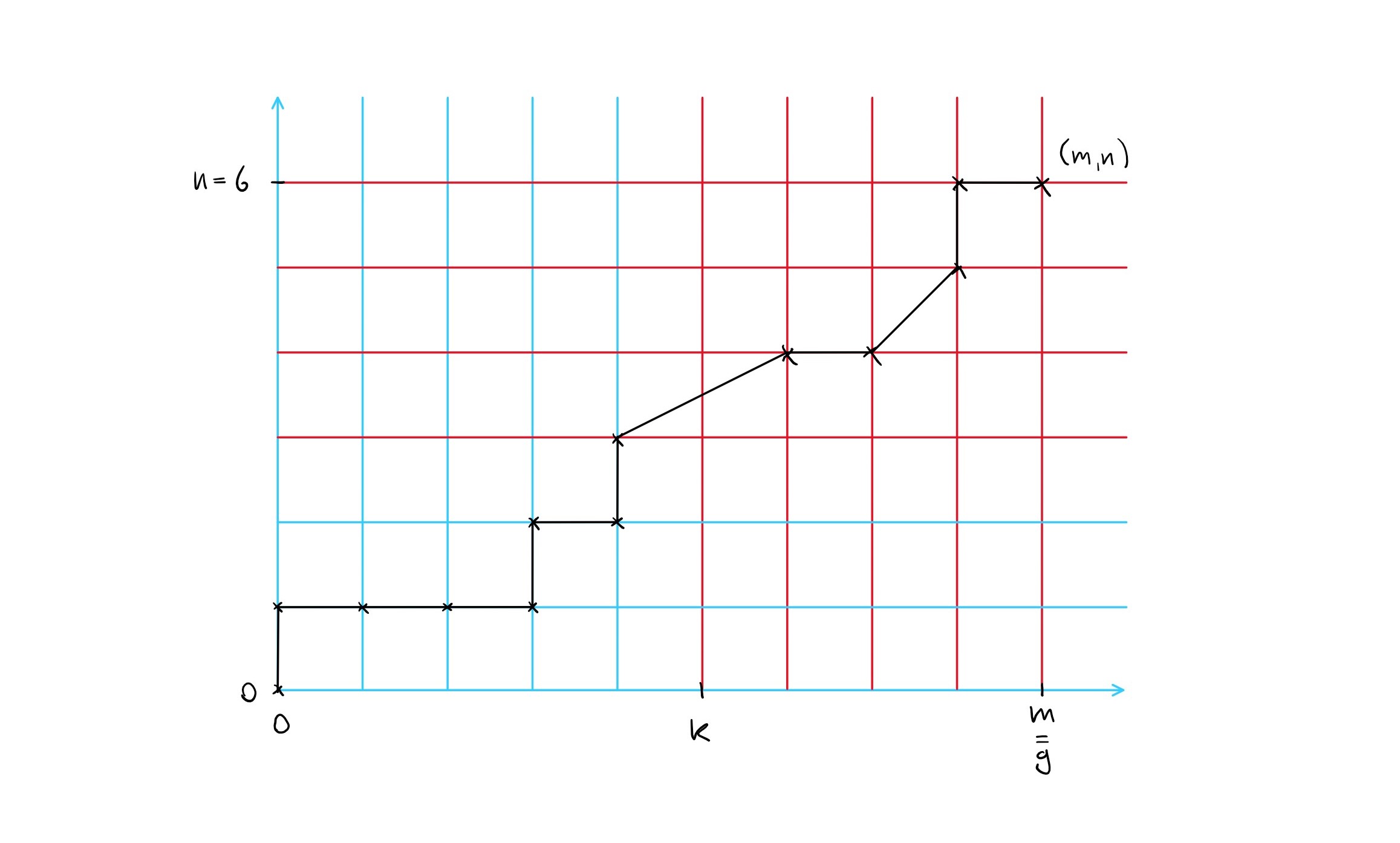}
		\caption{Walk in $[m] \times [n]$ where $P,P' = [1]$ and $\mathcal J = (0 \leq 0 \leq 0 \leq 0 \leq 0 \leq 1 \leq 1 \leq 1 \leq 1 \leq 1)$, $\mathcal J' = (0 \leq 0 \leq 0 \leq 1 \leq 1 \leq 1 \leq 1)$ indicated by blue and red.}
		\label{fig:Path}
	\end{figure}
	At each time $i$ in the walk, the next step is of type $(+1,+0),(+0,+1),(+1,+0)$, or in addition to this possibly of type $(+2,+0),(+2,+1)$ if the next column is the $k$-th and it is skipped. For $k < m$, \cite{MossSae} constructs a pairing on $A \hookrightarrow B$ by setting $B_{I}$ to consist of those simplices where the $k$-th column is not skipped and the step that leaves the column is of type $(+1,0)$. This is then paired with the simplex described by the path where the last point on the $k$-th column is removed (see \Cref{fig:prodPair} for an example). 
	\begin{figure}[H]
		\includegraphics[width=\linewidth]{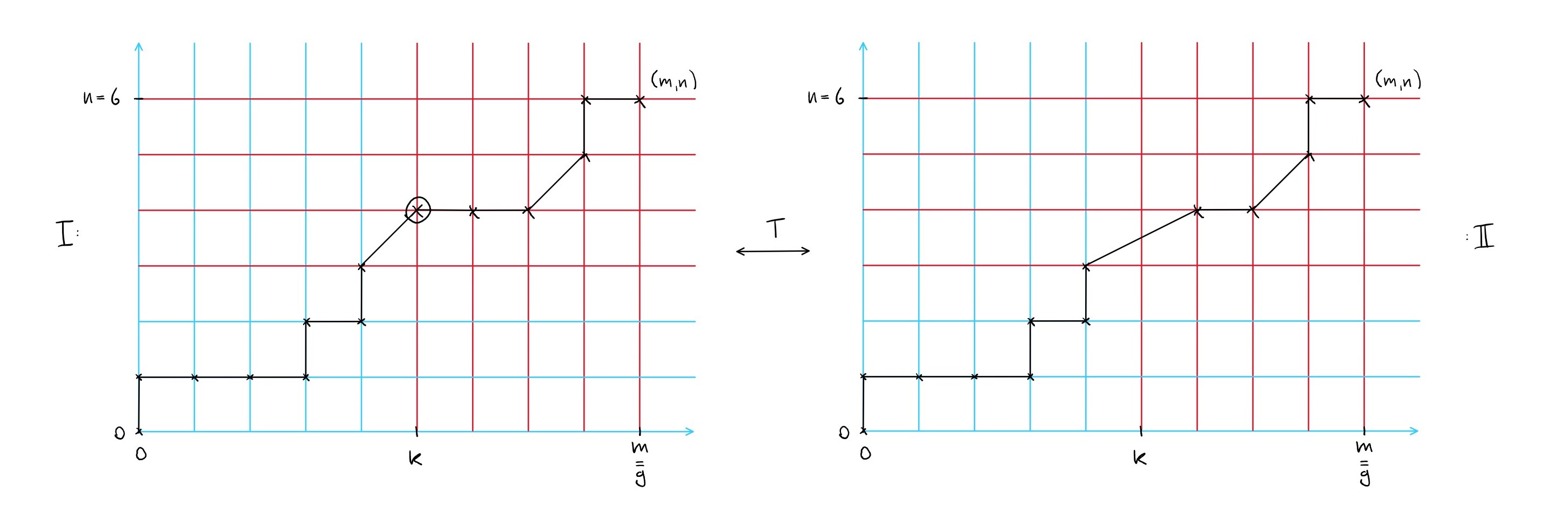}
		\caption{The walks corresponding to a pair of simplices in $B_{n.d.}\setminus A_{n.d.}$. The removed vertex is marked.}
		\label{fig:prodPair}
	\end{figure}
	For $0 < k$, one obtains a pairing by taking $B_{I}$ to consist of those simplices where the corresponding walk does not skip the $k$-th column and the prior step is of type $(+1,+0)$. One then pairs such a simplex with the one given removing the \textit{first} point on the $k$-th column. In both cases Moss shows in \cite{MossSae} that this gives a regular, proper pairing on the underlying simplicial sets. Now, if $p_k = p_{k+1}$ we take the prior pairing, if $p_k = p_{k-1}$ we take the latter pairing on the underlying simplicial map $A \hookrightarrow B$. To see this gives a proper, regular pairing on $A \hookrightarrow B$ we just need to verify that in both cases the requirements of \Cref{lemSimtoFil} are met. That is, we need to check that, for each pair given by $T$, if $i$ is the index of the removed point, then $i$ is admissible. If $p_k = p_{k+1}$, then the removed point corresponds to a vertex of stratum $(p_k,p_j')$, for some $0 \leq j \leq n$. As the next point is reached through a $(+1,+0)$, it corresponds to a $(p_{k+1},p'_j)= (p_{k},p'_j)$ vertex. In particular, $i$ is admissible. The case $p_k = p_{k-1}$ works analogously.
\end{proof}
We are mostly interested in the case where either $P$ or $P'$ is the partially ordered set with one element, i.e. the corresponding category of filtered simplicial sets is just the category of simplicial sets. We obtain the following immediate corollary of \Cref{propIncInProd}.
\begin{corollary}\label{corProdIncs}
	Let $X \xhookrightarrow{i} Y$ be a cofibration in $\textnormal{sSet}_{P}$ and $X' \xhookrightarrow{i'} Y'$ be a cofibration in $\textnormal{sSet}$. Then if either of the two is an FSAE, so is the morphism in $\textnormal{s\textbf{Set}}_P \cong \textnormal{s\textbf{Set}}_{P \times \star}$: 
	$$ Y \otimes X' \cup_{X \otimes X'} X \otimes Y' \hookrightarrow Y \otimes Y'$$
\end{corollary}
This result has a series of immediate consequences for the interactions of FSAEs with filtered simplicial mapping cylinders (see \Cref{propCylIncs}). These can be thought of as filtered, simplicial set analogues to the results found for example in (\cite[Ch. 2,$\paragraphmark 5$]{cohenCourse}), \cite[Sec. 6]{whitehead1939simplicial}, \cite[Ch. 6.3]{kampsPo}). They are the decisive arguments when it comes to verifying the Eckmann-Siebenmann axioms (\Cref{axEckSieb}) for the construction of a Whitehead group.
\begin{proposition}\label{propCylIncs}
Let $A \xhookrightarrow{s} B$ be an FSAE, $B' \xhookrightarrow{i} B$ be a cofibration and $B \xrightarrow{f} Z$ be a morphism in $\textnormal{s\textbf{Set}}_P$.
Then all of the following morphisms in $\textnormal{s\textbf{Set}}_P$ (illustrated in \Cref{fig:cylSaes}) are FSAEs.
 \begin{enumerate}	
\item $i_0,i_1: Z \hookrightarrow Z \otimes \Delta^1$ \label{propCylIncs1}
\item $Z \hookrightarrow M_f$	\label{propCylIncs2}
\item $B \cup_{A \hookrightarrow M_{f\circ s}} M_{f\circ s} \hookrightarrow M_{f}$ \label{propCylIncs3}
\item $M_{f \circ i} \hookrightarrow M_f$ \label{propCylIncs4}
\end{enumerate}
In particular, for any elementary stratum preserving simplicial homotopy $Y_0 \otimes \Delta^1 \xrightarrow{H} Z$ with $f:= H_0$ and $g:= H_1$, denote by $\hat{M}_H$ the pushout given by $$ \begin{tikzcd}
Y_0 \otimes \Delta^1 \arrow[r, hook] \arrow[d] & M_H \arrow d\\
Y_0 \arrow[r, hook] & \hat M_H
\end{tikzcd}.$$
Then the compositions 
\begin{enumerate}
	\setcounter{enumi}{4} 
	\item $M_f,M_g\hookrightarrow M_H \to \hat M_H$ \label{propCylIncs5}
\end{enumerate} are both FSAEs.
\end{proposition}
\begin{figure}[H]
	\includegraphics[width=\linewidth]{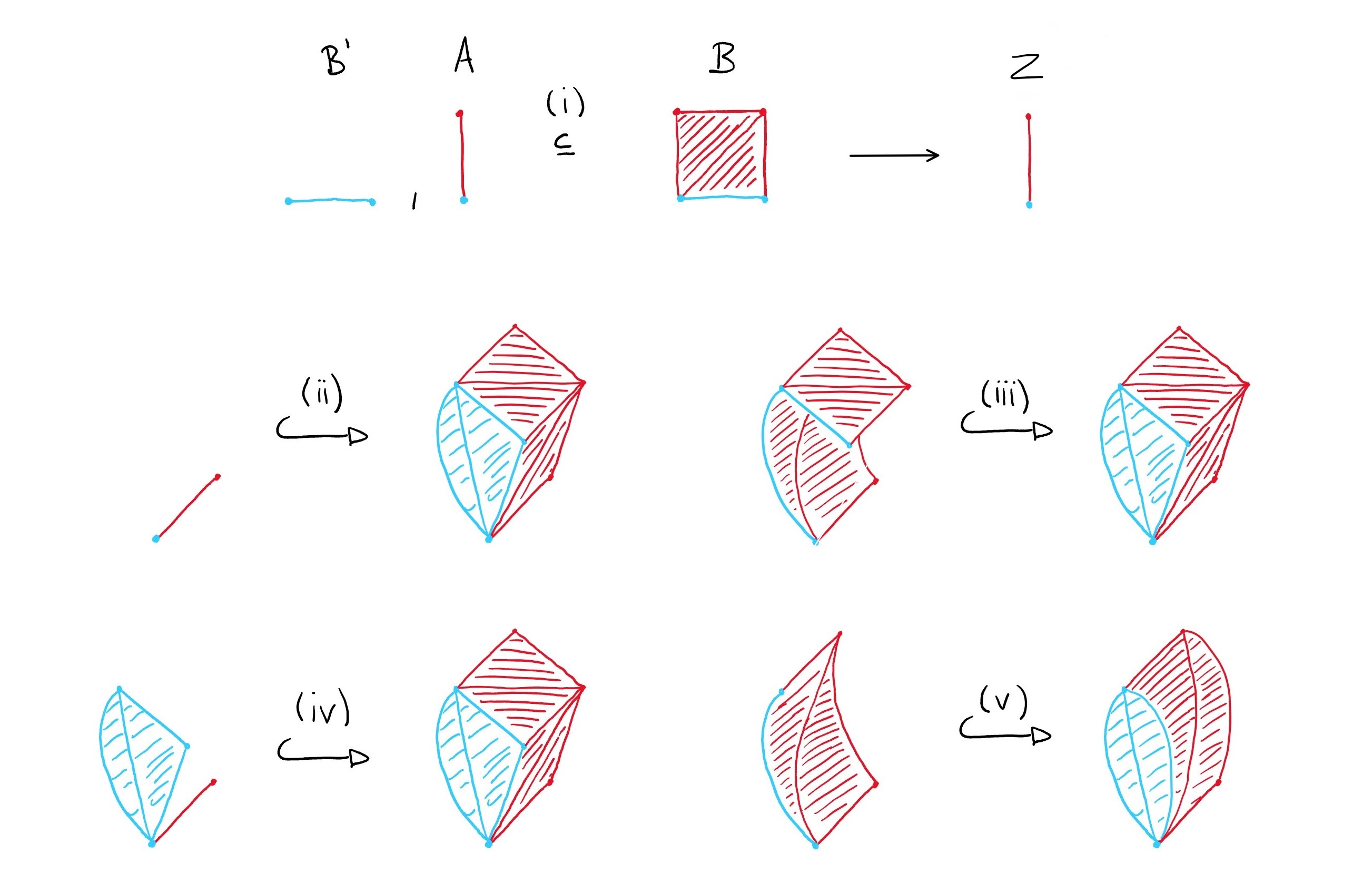}
	\caption{An illustration of some examples of the FSAEs \ref{propCylIncs1} to \ref{propCylIncs5} in \Cref{propCylIncs}.}
	\label{fig:cylSaes}
\end{figure}
\begin{proof}
For \ref{propCylIncs1} take $Y = Z$,$ X= \emptyset$, $X' = \Delta^0$ and $i' = i_0,i_1$ in \Cref{corProdIncs}. For \ref{propCylIncs2} note that $Z \hookrightarrow M_f$ is given by the pushout diagram
$$\begin{tikzcd}
B \arrow[r,hook, "i_0"] \arrow[d,"f"] & B \otimes \Delta^1 \arrow d\\
Y \arrow[r,hook] & M_f
\end{tikzcd}.$$
By \ref{propCylIncs1}, $i_0$ is an FSAE, hence, so is $Y \hookrightarrow M_f$. For \ref{propCylIncs3} first consider the case of \Cref{corProdIncs}, where $X = A, Y = B$ and $X'= \Delta^0 \sqcup \Delta^0 , Y'= \Delta^1, i'=i_0 \sqcup i_1$. We obtain that $$ B \underset{A \xhookrightarrow{i_1} A \otimes \Delta^1}{\cup} A\otimes \Delta^1 \underset{{A \xhookrightarrow{i_0} A \otimes \Delta^1} }{\cup} B \hookrightarrow B\otimes \Delta^1$$ is an FSAE. Now, take the inclusion of $B$ at $0$ into the previous homotopy pushout $j_0$ and consider the following commutative diagram. We omit some of the pushout subscript, to simplify notation from here on out.
$$\begin{tikzcd}
B \arrow [r, "j_0", hook] \arrow[d, "f"] & B \cup A \otimes \Delta^1 \cup B \arrow[d, "f'"] \arrow[r, hook] & B \otimes \Delta^1 \arrow[d]\\
 Z \arrow[r, hook ] & B \underset{A \hookrightarrow M_{f\circ s}}{\cup} M_{f \circ s} \arrow[r, hook] & M_f
\end{tikzcd}. $$ The upper horizontal composition is just the inclusion of $B$ into $B \otimes \Delta^1$ at $0$ and the lower horizontal composition is $Z \hookrightarrow M_f$. In particular, the inner left and the outer square are cocartesian. Hence, so is the inner right square. Thus, as we have just seen that the upper right horizontal is an FSAE, the same holds for the morphism in \ref{propCylIncs3}, by stability under pushouts. For \ref{propCylIncs4} take $X = B', Y = B$ and $X' = \Delta ^0, Y' = \Delta^1, i' = i_0$. Then the result follows by a similar pushout argument as we used for \ref{propCylIncs3}.\\
\\
We prove \ref{propCylIncs5} for $f$. The case of $g$ works analogously. First note that by \ref{propCylIncs1}, \ref{propCylIncs3} and \ref{propCylIncs4} the cofibrations $ M_f \hookrightarrow Y_0 \otimes \Delta^1 \cup M_f \hookrightarrow M_H$ as well as their composition are FSAEs. The $\cup$ stands for the pushout of $Y_0 \xhookrightarrow{i_0} Y_0 \otimes \Delta^1$ along $Y_0 \hookrightarrow M_f$ here. Now, consider the commutative diagram \begin{equation}\label{diagLetItEnd}
	\begin{tikzcd}
	Y_0 \otimes \Delta^1 \arrow[r, hook] \arrow[d] & Y_0 \otimes \Delta^1 \underset{...}{\cup} M_f \arrow [d] \arrow[r, hook] & M_H \arrow[d] \\
	Y_0 \arrow[r, hook] & M_f \arrow[r, hook] & \hat M_H
	\end{tikzcd}.
\end{equation} 
The outer square is cocartesian by definition. The inner left square fits into $$ \begin{tikzcd}\label{diagItDidEnd}
Y_0 \arrow[r, hook, "i_1"] \arrow[d, hook] & Y_0 \otimes \Delta^1 \arrow [d, hook] \arrow[r, hook] & Y_0 \arrow[d, hook] \\
M_f \arrow[r, hook] & Y_0 \otimes \Delta^1 \underset{...}{\cup} M_f \arrow[r] & M_f
\end{tikzcd}$$
The horizontal compositions of \ref{diagItDidEnd} are the identity making the outer square cocartesian. The left inner square is cocartesian by definition. In particular, the right inner square, which is the left inner square in \eqref{diagLetItEnd} is cocartesian. Hence, so is the inner right square in \eqref{diagLetItEnd}. By stability under pushouts, this shows that $M_f \hookrightarrow \hat M_H$ is an FSAE.
\end{proof}
In particular, \Cref{propCylIncs} shows that with our specification of elementary collapses we do not run into problems such as we illustrated in \Cref{exStrictBad}.
\section{The Whitehead group of a filtered simplicial set}
 In \Cref{subsecEckSiebApp} we have reviewed a general setting that Whitehead groups can be constructed in. Then, in in \Cref{secElandFSAE} we begun a detailed analysis of a class of expansions of $P$-filtered simplicial sets, FSAEs, which work as a filtered analogue to the expansion in the classical simple homotopy theory. We now show that (finite) FSAEs fit into this general setting (\Cref{subsecEckSiebApp}) and thus obtain a Whitehead group and torsion for filtered simplicial sets. We then study these invariants in detail, showing that they behave much like the classical ones (\Cref{subsecWHdetails}). Finally, we obtain the analogue to the result that every finite CW-complex is simple homotopy equivalent to a finite simplicial complex for our setting (\Cref{subsecSSvSC}).
\subsection{The Eckmann-Siebenmann approach for filtered simplicial sets}\label{subsecEckSiebAppHolds}
Recall the setting of \Cref{subsecEckSiebApp}. Given a category $\mathcal C$ contained in a larger category $\hat{ \mathcal C}$ and a class of morphisms $\Sigma$ in $\mathcal C$ fulfilling certain axioms, it allows us to construct a Whitehead group and a Whitehead torsion. In this subsection, we apply this theory to the case where $ \hat{ \mathcal C}$ is the category of finite filtered simplicial sets in $\textnormal{s\textbf{Set}}_P$, $\mathcal{C}$ is the subcategory of $\hat{ \mathcal C}$ with the same objects morphisms taken to be only cofibrations in the Douteau model structure - i.e. inclusions of filtered simplicial subsets - and $\Sigma$ is the class of (finite) FSAEs. \\
\\
Recall that we denote by $\mathcal H \textnormal{s\textbf{Set}}_P^{\textnormal{fin}}$ the full subcategory of $\mathcal H \textnormal{s\textbf{Set}}_P$ given by finite filtered simplicial sets. Further, let $\mathcal C$ ($\mathcal C'$) be the subcategory of $\textnormal{s\textbf{Set}}_P$ given by finite (arbitrary) filtered simplicial sets and morphisms cofibrations and $\Sigma$ $(\Sigma')$ the class of morphism in $\mathcal C$ ($\mathcal C'$)that are FSAEs. As every FSAE is in particular an anodyne extension and these are precisely the trivial cofibrations in the Douteau-Henrique model structure on $\textnormal{s\textbf{Set}}_P$ \cite[Thm. 2.14]{douSimp}, by the universal property of the localization of a category we obtain a commutative (on the nose) diagram of functors
\begin{equation*}\label{diagLotsofFun}
	\begin{tikzcd}
	\mathcal C(\Sigma^{-1}) \arrow[r] \arrow[d] & \mathcal H \textnormal{s\textbf{Set}}_P^{\textnormal{fin}} \arrow [d, hook]\\
	\mathcal C'(\Sigma'^{-1}) \arrow[r] & \mathcal{H}\textnormal{s\textbf{Set}}_P
	\end{tikzcd}.
\end{equation*}
All of the functors above are given by the identity on objects. It turns out, all of them are in fact fully faithful. Before we prove this, we start by proving two lemmata. Denote by $Q$ and $Q'$ the structure functors $\mathcal C \to \mathcal C(\Sigma^{-1})$ and $ \mathcal C' \to \mathcal C'(\Sigma'^{-1}$ respectively.
\begin{lemma}\label{lemComHo}
	 Let $X \xhookrightarrow{f,g} Y$ be two cofibrations in $\textnormal{s\textbf{Set}}_P$ that are homotopic through a (elementary, stratum preserving) simplicial homotopy. Then $$Q'(f) = Q'(g).$$If both $X$ and $Y$ are finite, then also $$Q(f) = Q(g).$$
\end{lemma}
\begin{proof}
	 The argument both in the finite and in the infinite case are identical, so we prove the finite one.
	 We fist show the case where the two are homotopic through a cofibration $H: X \otimes \Delta^1\hookrightarrow Y$. We then show an additional lemma to complete the proof of this one. Consider the two cofibrations $X \xhookrightarrow{i_0,i_1} X \otimes \Delta^1$. If we can show that $Q(i_0) = Q(i_1)$, we are done since then $$Q(f) = Q(H) \circ Q(i_0) = Q(H) \circ Q(i_1) = Q(g).$$ 
	Now, note that $\textnormal{sd}(\Delta) = \Lambda^2_{2}$. Consider the following diagram in $\textnormal{s\textbf{Set}}$. $$ \begin{tikzcd}
	\Delta^0 \arrow[r, "i", hook] \arrow[d, hook, "sd(i)"] & \Delta^1 \arrow[d, "d_2", hook]\\
	 \textnormal{sd}(\Delta^1) = \Lambda_2^2\arrow[r, hook] & \Delta^2
	\end{tikzcd}, $$
	for $i$ the inclusion at $0$ or $1$. All of the maps in the diagram are FSAEs. In particular, if we apply $X \otimes (-)$ by \Cref{propIncInProd}, we obtain commutative diagrams of FSAEs 
	$$\begin{tikzcd}
	X \arrow[r, "i_j", hook] \arrow[d, hook, "i_j'"] & X \otimes \Delta^1\arrow[d, hook]\\
	X \otimes \Lambda^2_2\arrow[r, hook] & X \otimes \Delta^2
	\end{tikzcd}, $$
	$j = 0,1$. By commutativity and the fact that all arrows involved become isomorphisms under $Q$, it suffices to show that $$Q(i_0') = Q(i_1').$$ Let $t: X \otimes \textnormal{sd}(\Delta^1) \xrightarrow{\sim} X \otimes \textnormal{sd}(\Delta^1)$ be the isomorphism induced by the swap of endpoints $\textnormal{sd}(\Delta^1) \to \textnormal{sd}(\Delta^1)$. Denote by $i_m: X \hookrightarrow X \otimes \textnormal{sd}(\Delta^1)$ the inclusion induced by sending $ \Delta^0$ to the middle vertex in $\textnormal{sd}(\Delta^1)$. The diagram in $\mathcal C$ $$ \begin{tikzcd}
	X \arrow[r, "i_m" ,hook] \arrow[d , hook, "i_m"] & X \otimes \textnormal{sd}(\Delta^1) \arrow ["t", d]\\
	X \otimes \textnormal{sd}(\Delta^1) \arrow[r , "1"] & X \otimes \textnormal{sd}(\Delta^1) 
	\end{tikzcd}$$
	commutes.
	By the same usage of \Cref{propIncInProd} as before, one shows that $i_m$ is an FSAE. In particular, by commutativity of the last diagram, $$Q(t) = Q(1_{X \otimes \textnormal{sd}(\Delta^1)})$$ and it follows $$ Q(i'_0) = Q(t)Q(i'_1) = Q(1)Q(i'_1) =Q(i_1')$$ completing the first part of this proof.
\end{proof}
\begin{lemma}\label{lemMapCylCom}
Let $X \xhookrightarrow{f} Y$ be a cofibration in $\textnormal{s\textbf{Set}}_P$. Then the mapping cylinder splitting $$\begin{tikzcd}
X \arrow[r, "f"] \arrow[rd, hook] & Y \arrow[d, hook]\\
& M_f
\end{tikzcd}
$$
commutes in $\mathcal C'(\Sigma'^{-1})$. If $X$ and $Y$ are finite, it also commutes in $\mathcal C(\Sigma^{-1}))$. 
\end{lemma}
\begin{proof}
	Denote by $H$ the structure map $ X \otimes \Delta^1\to M_f$. Since $f$ is a cofibration, so is this (by stability under pushouts). $X \xhookrightarrow{i_1} X \otimes \Delta^1\xrightarrow{H} M_f$ is the lower diagonal map in the statement. $X \xhookrightarrow{i_0} X \otimes \Delta^1\xrightarrow{H} M_f$ is the composition of the upper right with the vertical. In particular, by the first part of the proof of \Cref{lemComHo}, the result follows.
\end{proof}
\begin{proof}[proof of \Cref{lemComHo}; part two] We still have to show that $Q(f) = Q(g)$ if the two are homotopic through a simplicial homotopy $H: X \otimes \Delta^1\to Y$ that is not necessarily a cofibration. To show this, consider the following diagram in $\textnormal{s\textbf{Set}}_P$.
$$
\begin{tikzcd}
 & M_f \arrow[rd, hook]\\
X \arrow[r, bend right = 40, "g"] \arrow[r, "f" , bend left = 40, swap] \arrow[ru, hook] \arrow[rd, hook] & Y \arrow[u, hook] \arrow[d, hook] \arrow[hook, r]& \hat M_H \\
 & M_g \arrow[ru, hook]
\end{tikzcd}
$$
Here $\hat{M}_H$ and the arrows into it are the ones constructed in \Cref{propCylIncs} (i.e. $\hat{M}_H$ is the mapping cylinder of $H$ where the $X \otimes \Delta^1$ at $1$ was collapsed to $X$.) By \Cref{propCylIncs} \Cref{propCylIncs5}, all arrows in the right part of the diagram are FSAEs and the upper right and lower right triangles are commutative by construction. By \Cref{lemMapCylCom} the upper left and lower left triangle commute in $\mathcal C(\Sigma^{-1})$. Passing to $\mathcal C(\Sigma^{-1})$ we now have a diagram where: \begin{itemize}
	
	\item The outer square commutes.
	\item The upper left and lower left triangle commute.
	\item The upper right and lower right triangle commute and consist only of isomorphisms.
\end{itemize}
In particular, a quick diagram chase shows that in $\mathcal C(\Sigma^{-1})$ we have $Q(f) = Q(g)$.
\end{proof}
\begin{proposition}\label{propEqCharOfHo}
	All of the functors in \eqref{diagLotsofFun} are fully faithful. In particular, the horizontals are isomorphisms of categories.
\end{proposition}
\begin{proof} We prove the finite case as the infinite one is basically just a simpler version of this one.\\
	\\
We start by showing fullness. By \Cref{propHoCharFin} it suffices to show this that morphisms in $\mathcal{H}\textnormal{s\textbf{Set}}_P^{\textnormal{fin}}$ of shape $[f]$ for some morphism $f:X \to Y$ in $\textnormal{s\textbf{Set}}_P^{\textnormal{fin}}$.
 Using mapping cylinder factorization we obtain the diagram
$$\begin{tikzcd}
X \arrow[r, "f"] \arrow[rd, hook] & Y \arrow[d, hook]\\
 & M_f
\end{tikzcd},
$$
which is commutative in $\mathcal{H}\textnormal{s\textbf{Set}}_P$. As, $M_f$ is again finite, by \Cref{propCylIncs}, $Y \hookrightarrow M_f$ is a (finite) FSAE, hence induces an isomorphism in $\mathcal C(\Sigma^{-1})$. So it suffices to show that the class of $X \hookrightarrow M_f$ in $\mathcal H\textnormal{s\textbf{Set}}_P$ comes from an arrow in $\mathcal C(\Sigma^{-1})$. But this is obvious, as $X \hookrightarrow M_f$ is a cofibration of finite filtered simplicial sets.\\
\\		
It remains to show the faithfulness. First note that by the same argument as in \Cref{subsecEckSiebApp} (i.e. the fact that we can always change direction in a zigzag by using pushouts) every morphism in $\mathcal C(\Sigma^{-1})$ is of the shape $Q(s)^{-1}Q(f)$ where $s$ is a (finite) FSAE and $f$ a morphism in $\mathcal C$. So suppose we are given two such morphisms $Q(s)^{-1}Q(f),Q(s')^{-1}Q(g)$ with $[s]^{-1}[f] = [s']^{-1}[g]$. We need to show $Q(s)^{-1}Q(f)=Q(s')^{-1}Q(g)$. By invertibility of $Q(s),Q(s')$ we may without loss of generality assume that $f,g$ have the same source $X$ and target $Y$ and that $s = s' = 1_Y$. In particular, we then have $$[f] = [g].$$ By \Cref{propHoCharFin} there then exists an $s \in \Sigma $ such that $s \circ f$ and $s \circ g$ are (stratum preserving, elementarily) simplicially homotopic. By \Cref{lemComHo} this implies $Q(s \circ f) = Q(s \circ g)$. As $Q(s)$ is invertible, we obtain $$Q(f) = Q(g).$$
\end{proof}
As a consequence of \Cref{propEqCharOfHo}, our choices of $\mathcal C$ and $\Sigma$ are promising candidates for constructing a Whitehead group (with respect to FSAEs i.e. with respect to elementary expansions) on $\mathcal{H}\textnormal{s\textbf{Set}}_P^{\textnormal{fin}}$ ($\mathcal{H}\textnormal{sSet}$). For $X \in \mathcal{C}$, $(\mathcal{C'})$ and $\Sigma$, $\Sigma'$ as before recall the definitions of a simple morphism and of $A(X), E(X)$ (\Cref{defSimpleSetting}). For the remainder of this section, we only be focus on the finite setting i.e. on $\mathcal{C},\Sigma$. Quite clearly, all of the arguments below also work in the infinite one, with the exception that $A(X)$ is in general probably not of set size in this case. If one is willing to work with groups of class size, this might produce a fruitful theory as well. \\
\\
In the finite setting however, $E(X)$ is clearly of set size, as the equivalence relation on $E(X)$ is up to isomorphism. We now show that the triple $\mathcal{C} \subset \hat{ \mathcal C},\Sigma$ satisfies the axioms in \Cref{defAxEckSieb}, i.e. admits a Whitehead group.
\begin{theorem}\label{thrmEckSiebAxAreTrue}
	Let $\hat{ \mathcal C}$ be the full subcategory of $\textnormal{s\textbf{Set}}_P$ given by finite filtered simplicial sets. Let $\mathcal{C}$ be the subcategory which has the same objects as $\mathcal{C}$ but only cofibrations in $\textnormal{s\textbf{Set}}_P$ as morphisms. Let $\Sigma$ be the class of FSAEs in $\hat{ \mathcal C}$. Then $\mathcal{C} \subset \hat{ \mathcal C}$, $\Sigma$ admits a Whitehead group (see \Cref{defAxEckSieb}).
\end{theorem}
\begin{proof}
\ref{eckAx0} of \Cref{defAxEckSieb} is obvious as FSAEs contain all isomorphism and are closed under composition by definition. Similiarly, as cofibrations are closed under pushouts and FSAEs are closed under pushouts by definition, we immediately obtain \ref{eckAx1}. It remains to show $\ref{eckAx2}$. So let $f,g:X \to Y$ be two cofibrations of simplicial sets, such that $Q(f)=Q(g)$. Then, in particular, $[f] = [g]$. By \Cref{propHoCharFin} and the fact that (finite) FSAEs are stable under composition we may hence assume without loss of generality that $f$ and $g$ are elementary simplicially homotopic through a simplicial homotopy $H$. 
By \Cref{propCylIncs} \ref{propCylIncs5}, we then have a commutative diagram $$ \begin{tikzcd}
 X \arrow[r, hook] \arrow[d, hook] & M_f \arrow[d, hook, "s_f"]\\
 M_g \arrow[r, hook ,"s_g"] & \hat M_H
\end{tikzcd}$$
in $\mathcal{C}$, with $s_f$ and $s_g$ FSAEs. Now, extend this diagram as follows: 
$$ \begin{tikzcd}
& & & X \arrow[llld , "f" , swap, hook'] \arrow[rd, hook] \arrow[ld, hook'] \arrow[rrrd, "g", hook ]& & &\\
Y \arrow[rd, hook, "i_1", swap] & & M_f \arrow[rd, hook, "s_f", swap] ,\arrow[ld, hook', "s'_f"]& & M_g \arrow[rd, hook, "s'_g", swap] \arrow[ld, hook' ,"s_g"] & &Y \arrow[ld, hook', "i_1"] \\
& Y \otimes \Delta^1\arrow[rd] & & \arrow[ld ]\hat M_H \arrow[rd]& & \arrow[ld] Y \otimes \Delta^1 & \\
& & R_f \arrow[rd] && R_g \arrow[ld] &&\\
& & & R & & &
\end{tikzcd}$$
All the squares ending in $R$s are obtained by pushing out. The upper left and upper right square commute by definition. Now, as both $f$ and $g$ are cofibrations, we have that $s'_f$ and $s'_g$ are FSAEs, by \Cref{propCylIncs} \ref{propCylIncs4}. Furher $i_1$ is an FSAE by \Cref{propCylIncs} \ref{propCylIncs4}. In particular, by stability under pushouts, every arrow, but the ones in the first row is an FSAE. Hence, by stability under composition and commutativity of the diagram, we have shown \ref{eckAx2}.
\end{proof}
Thus \Cref{thrmEckSiebAxAreTrue}, we can apply \Cref{thrmEckSieb} together with \Cref{propEqCharOfHo} to obtain two functors 
\begin{align}\label{diagFunWHM}
	\mathcal{H}\textnormal{s\textbf{Set}}_P^{\textnormal{fin}} \cong \mathcal C(\Sigma^{-1}) &\xrightarrow{A} \textnormal{\textbf{MonAb}};\\
	\label{diagFunWHG}
	\mathcal{H}\textnormal{s\textbf{Set}}_P^{\textnormal{fin}} \cong \mathcal C(\Sigma^{-1}) &\xrightarrow{E} \textnormal{\textbf{Ab}}.
\end{align} 
We give a more detailed analysis of these two functors in \Cref{subsecWHdetails}. It can also be helpful to go back to \Cref{thrmEckSieb}, where the basic constructions underlying $A$ and $E$ are listed.
We are now in shape to finally define the Whitehead group of a finite filtered simplicial set.
\begin{definition}\label{defWH}Let Let $X \in \textnormal{s\textbf{Set}}_P^{\textnormal{fin}}$. 
	\begin{itemize}
		\item Denote by $A_P$ the functor in \eqref{diagFunWHM}.\\ $A_P(X)$ is called the \textit{Whitehead monoid of the filtered simplicial set $X$}. 
		\item Denote by $Wh_P$ the functor in \eqref{diagFunWHG}.\\ $Wh_P(X)$ is called the \textit{Whitehead group of the filtered simplicial set $X$}. 
		\item For a morphism $\alpha:X \to Y$ in $\mathcal{H}\textnormal{s\textbf{Set}}_P^{\textnormal{fin}}$, let $\tau_P(\alpha) \in A_P(X)$ ($... \in Wh_P(X)$ if $\alpha$ is an isomorphism) be the simple morphism class of the morphism in $\mathcal C(\Sigma^{-1})$ corresponding to $\alpha$.\\ $\tau_P(\alpha)$ is called the \textit{Whitehead torsion of $\alpha$}.
		\item For a morphism $f: X \to Y$ in $\textnormal{s\textbf{Set}}_P$, define $\tau_P(f):= \tau_P([f]) \in A_P(X)$ ($ ... \in Wh_P(X)$ if $f$ is a weak equivalence). \\$\tau_P(f)$ is called the \textit{Whitehead torsion of $f$}.
	\end{itemize}
\end{definition}
In the following subsection, we investigate some of the properties of $A_P, Wh_P$ and $\tau_P$. In particular, we give a more explicit description of their functoriality for morphisms in $\textnormal{s\textbf{Set}}_P$. It turns out, in many ways they behave just like the classical theory. This of course stems from the fact, that most of what we have done here is a formalization of the arguments used there. More generally, it seems that if one keeps really good track of what we are actually using, our argument should generalize to any cofibrantly generated combinatorial simplicial model category, fulfilling an analogue of \Cref{propHoCharFin} for compact objects, and also the analogue of \Cref{propIncInProd}.
\subsection{Properties of $Wh_P$, $A_P$, $\tau_P$ and of simple equivalences}\label{subsecWHdetails}
The astute reader will have noticed that, while we have spent a great deal of time talking about FSAEs - the filtered simplicial analogue to expansions in the classical setting - we have yet to define what a simple equivalence is, in the setting of finite filtered simplicial sets. This is because to really say anything interesting about these objects, it is helpful to already have the theory of \Cref{thrmEckSieb} in place. Showing that one can use FSAEs of finite filtered simplicial sets to construct a Whitehead group (a Whitehead monoid) and a Whitehead torsion has been the content of the last subsection, of \Cref{propEqCharOfHo}, \Cref{thrmEckSiebAxAreTrue} and \Cref{defWH} to be more precise. Now, that these are set up, we can start our investigation of simple equivalences. It can be of interest to compare the results in this section to the ones found for example in \cite{whitehead1939simplicial} for a simplicial complex setting, \cite{whitehead1950simple}, \cite{cohenCourse} for a CW setting and \cite{kampsPo} for an abstract setting, which they and their proofs are clearly inspired by. \\
\\
First, we use \Cref{propEqCharOfHo} together with \Cref{defSimpleSetting} to define simple equivalences of filtered simplicial sets. Recall that we are in the setting of \Cref{subsecEckSiebApp} with $\hat{ \mathcal C}= \textnormal{s\textbf{Set}}_P^{\textnormal{fin}}$, $\mathcal{C}$ the subcategory given by cofibrations in $\hat{ \mathcal C}$ and $\Sigma$ the class of FSAEs in $\hat{ \mathcal C}$ and that we have a (canonical) isomorphism of categories $\mathcal C(\Sigma^{-1}) \cong \mathcal{H}\textnormal{s\textbf{Set}}_P^{\textnormal{fin}}$, by \Cref{propEqCharOfHo}.
\begin{definition}\hfill
	\begin{itemize}
	\item A morphism in $\alpha$ in $\mathcal{H}\textnormal{s\textbf{Set}}_P^{\textnormal{fin}}$ is called \textit{a simple equivalence} if and only if the morphism in $\mathcal C(\Sigma^{-1})$ corresponding to $\alpha$ is a simple morphism. That is, if it is a finite composition of morphisms of shape $[s]$ or $[s]^{-1}$, where $s$ is an FSAE in $\textnormal{s\textbf{Set}}_P^{\textnormal{fin}}$.
	\item Two objects in $X,Y \in \mathcal H\textnormal{s\textbf{Set}}_P$ are said to have the same \textit{simple homotopy type} or also to be \textit{simply equivalent} if there is a simple equivalence $X \xrightarrow{\alpha} Y$.
	\item A morphism in $\textnormal{s\textbf{Set}}_P^{\textnormal{fin}}$ is called a $\textit{simple equivalence}$ if $[f]$ is a simple equivalence.
	\item We say two morphisms $\alpha,\beta $ in $\mathcal{H}\textnormal{s\textbf{Set}}_P^{\textnormal{fin}}$ with the same source \textit{have the same simple morphism class} if their corresponding morphism in $\mathcal C(\Sigma^{-1})$ have the same simple morphism class. That is, if there is a simple equivalence $\sigma $ in $\mathcal{H}\textnormal{s\textbf{Set}}_P^{\textnormal{fin}}$ such that $\alpha = \sigma \circ \beta$. The equivalence class generated by this relation is called the \textit{simple morphism class of $\alpha$}, denoted by $\langle \alpha \rangle$. 
	\item The simple morphism class of a morphism $f$ in $\textnormal{s\textbf{Set}}_P^{\textnormal{fin}}$ is defined to be the simple morphism class of $[f]$, denoted by $\langle f \rangle$.
	\item A zigzag in $\textnormal{s\textbf{Set}}_P^{\textnormal{fin}}$ $$X = X^0 \leftrightarrow ... \leftrightarrow X^n = Y $$ with all arrows FSAEs is called a \textit{deformation} between $X$ and $Y$. We also say \textit{$X$ deforms into $Y$} and vice versa.
	\item A commutative diagram in $\textnormal{s\textbf{Set}}_P^{\textnormal{fin}}$ $$ \begin{tikzcd}
	& X \arrow[ld, hook', "a_0", swap] \arrow[d, hook] \arrow[rd, "a_n", hook]&\\
	Y^0\arrow[r, leftrightarrow] & ... \arrow[r, leftrightarrow]& Y^n
	\end{tikzcd},
	$$
	where the verticals are cofibrations, and the lower horizontal is a deformation from $Y_0$ to $Y_n$ is called a \textit{deformation from $a_0$ to $a_n$}. We also say \textit{$a_0$ deforms into $a_n$} and vice versa.
	\end{itemize}
\end{definition}
Clearly, every finite FSAE is a simple equivalence. Also, by definition, any deformation induces a simple equivalence, by taking the zigzag as a morphism in $\mathcal{H}\textnormal{s\textbf{Set}}_P^{\textnormal{fin}}$. In particular, we obtain a vast class of examples of simple equivalences, by using the results in \Cref{propCylIncs} together with the fact that simple equivalences fulfill a two out of three rule, by definition.
\begin{example}\label{exSimpleMorphismCyl}
Let $X \xrightarrow{f} X'$ be a morphism and $X \xhookrightarrow{a} Y$ be a cofibration in $\mathcal{H}\textnormal{s\textbf{Set}}_P^{\textnormal{fin}}$. Denote by $i_X,i_{X'}$ the respective inclusions into the mapping cylinder $M_f$. Then the following arrows are simple equivalences.
\begin{enumerate}
	\item $X \otimes \Delta^1\xrightarrow{p_X} X$ \label{exSimpleMorphismCyl1}
	\item $M_f \xrightarrow{p_X'} X'$ \label{exSimpleMorphismCyl2}
	\item $M_f\cup_{i_{X'},a} Y \xrightarrow{p_{X'} \cup a} X' \cup_{f,a} Y$ \label{exSimpleMorphismCyl3} 
\end{enumerate}
We have illustrated an example of the third in \Cref{fig:ExSimpleEq}. 
\begin{figure}[H]
	\includegraphics[width=\linewidth]{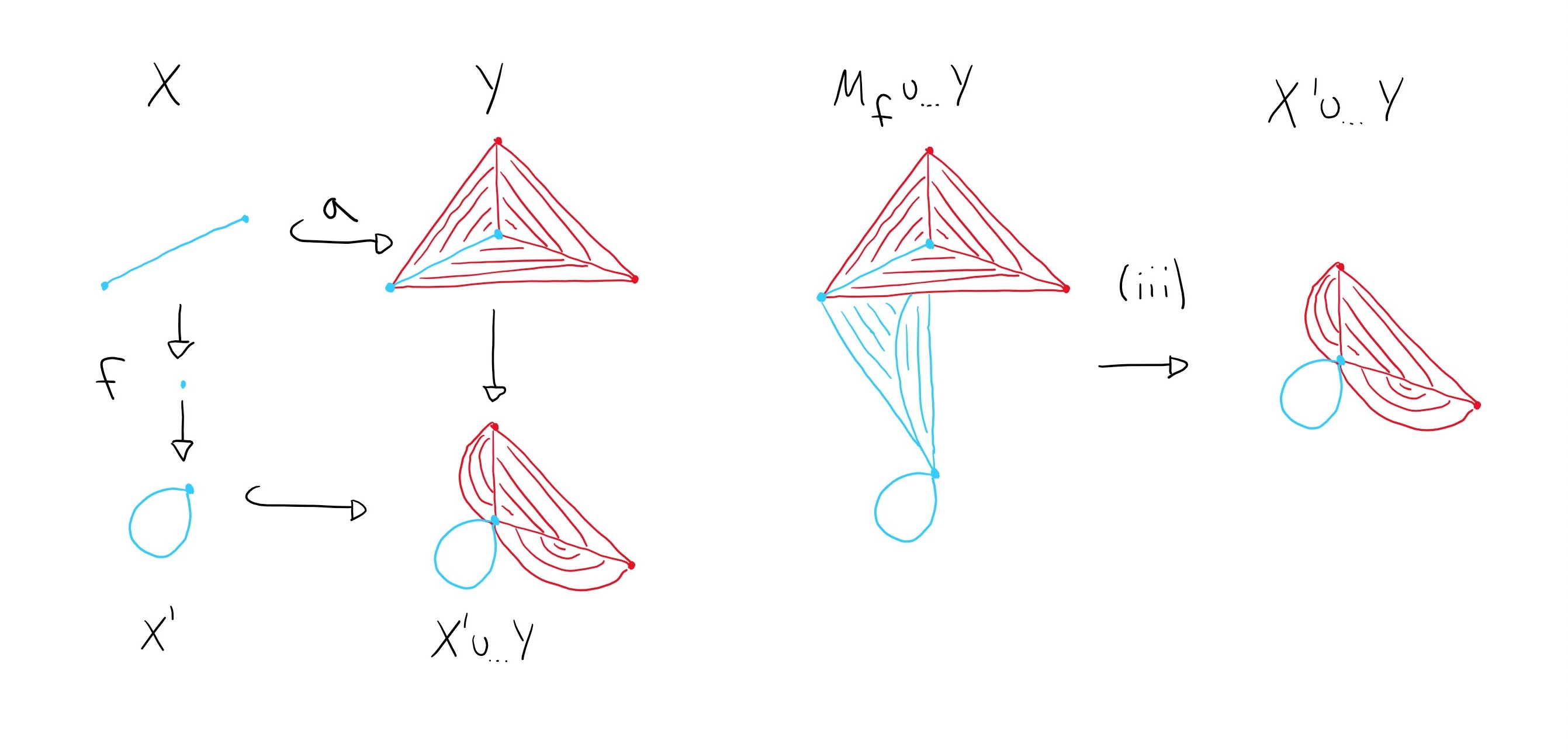}
	\caption{An example of the simple equivalence in \ref{exSimpleMorphismCyl3} for $P =\{0,1\}$ with the $0$-stratum marked in blue and the $1$-stratum marked in red. The pushout square is to the left, the simple equivalence to the right.}
	\label{fig:ExSimpleEq}
\end{figure}
The first two are an immediate consequence of the results in \Cref{propCylIncs} and the two out of three law. For the third one ... 
\end{example}
\begin{proof}
..., consider first the pushout square:
$$ \begin{tikzcd}
Y \arrow[r] \arrow[d, hook, "i_0"] & X' \cup_{f,g} Y \arrow[d, hook, "\phi"]\\
Y \otimes \Delta^1\arrow[r] & X' \cup_{i_0\circ f, a} Y \otimes \Delta^1
\end{tikzcd}.$$
As the left vertical is an FSAE, so is the right. Secondly, consider the pushout square 
$$ \begin{tikzcd}
 X \otimes \Delta^1\cup_{g,i_1} Y \arrow[r] \arrow[d, hook] & M_f \cup_{i_{X'},a} Y \arrow [d, hook, "\psi"]\\
 Y \otimes \Delta^1\arrow[r] & X' \cup_{i_0\circ f, a} Y \otimes \Delta^1
\end{tikzcd} $$
Again, as by $g$ is a cofibration, by \Cref{propIncInProd} the left vertical and hence also the right vertical is an FSAE. Now, consider the triangle diagram: 
\begin{equation}\label{diagAnotherTriangle}
 \begin{tikzcd}
M_f \cup_{i_{X'},a} Y \arrow[rr, "p_{X'}\cup a"] \arrow[rd, hook, "\psi", swap]& & X' \cup_{f,a} Y \arrow[ld, hook', "\phi"] \\
& X' \cup_{i_0\circ f, a} Y \otimes \Delta^1& 
\end{tikzcd}.
\end{equation}
It is clearly not commutative in $\textnormal{s\textbf{Set}}_P$. However, a left inverse to $\phi$ and hence an inverse in $\mathcal{H}\textnormal{s\textbf{Set}}_P^{\textnormal{fin}}$ is given by the collapsing the cylinder map, $X' \cup_{i_0\circ f, a} Y \otimes \Delta^1\to X' \cup_{f,a} Y$. The latter fits into a commutative diagram:
$$ \begin{tikzcd}
M_f \cup_{i_{X'},a} Y \arrow[rr, "p_{X'}\cup a"] \arrow[rd, hook, "\psi", swap]& & X' \cup_{f,a} Y \\
& X' \cup_{i_0\circ f, g} Y \otimes \Delta^1\arrow[ru]& 
\end{tikzcd}. $$
In particular, \eqref{diagAnotherTriangle} commutes in $\mathcal{H}\textnormal{s\textbf{Set}}_P^{\textnormal{fin}}$. Hence, by the two out of three property, we obtain that $M_f\cup_{i_{X'},a} Y \xrightarrow{p_{X'} \cup a} X' \cup_{f,a} Y$ is a simple equivalence.
\end{proof}

We begin with a few statements that are immediate consequences of the isomorphism of categories $\mathcal C(\Sigma^{-1}) \cong \mathcal{H}\textnormal{s\textbf{Set}}_P^{\textnormal{fin}}$ and the results in \Cref{subsecEckSiebApp}. For the following, let $X \in \textnormal{s\textbf{Set}}_P^{\textnormal{fin}}$.
\begin{remark}\label{remSimpleStuffA}
$\tau_P$ induces a one to one correspondence between the simple morphism classes with source $X$ and $A_P(X)$. In particular, $\alpha: X \to Y$ in $\mathcal{H}\textnormal{s\textbf{Set}}_P^{\textnormal{fin}}$ is a simple equivalence if and only if $\tau_P(\alpha) = 0$. Due to this, we do not distinguish between elements of $A_P(X)$ and simple morphism classes (of morphism with source $X$) and for the sake of notational brevity, we write $\langle f \rangle $ instead of $\tau_P(f)$ for the remainder of this subsection.

By \Cref{lemCharOfMorCs} every morphism in $\mathcal{H}\textnormal{s\textbf{Set}}_P^{\textnormal{fin}}$ is of the shape $[s]^{-1}\circ [a]$ where $a$ is a cofibration and $s$ is a FSAE. Hence, every simple morphism class in $A_P(X)$ is of the shape $\langle a \rangle$ for some cofibration $a:X \xhookrightarrow{} Y$ in $\textnormal{s\textbf{Set}}_P^{\textnormal{fin}}$. By \Cref{lemEqConE}, two such cofibrations have the same simple morphism class if and only if they differ up to a deformation. In particular, we can think also think of $A_P(X)$ as cofibrations in $\textnormal{s\textbf{Set}}_P^{\textnormal{fin}}$ with source $X$ modulo deformation. 

By the same argument as we used for the proof of \Cref{lemCharOfMorCs} (i.e. stability of FSAEs under pushout and composition), every deformation between object or morphism can be replaced by one that is of the shape $s: X \to S \xleftarrow{} Y:t$ with $s$ and $t$ FSAEs.
\end{remark}
\begin{remark}\label{remSimpleStuffB}
As a morphism in a model category is a weak equivalence if and only if it induces an isomorphism in the homotopy category, every simple equivalence in $\textnormal{s\textbf{Set}}_P^{\textnormal{fin}}$ is a weak equivalence. In particular, $f:X \to Y$ in $\textnormal{s\textbf{Set}}_P^{\textnormal{fin}}$ is a weak equivalence if and only if $\tau_P(f)$ ends up in $Wh_P(X)$. 
Hence, we can think of $Wh_P(X)$ as the simple morphism classes of weak equivalences, with source $X$ or alternatively by \Cref{remSimpleStuffA} as acyclic cofibrations with source $X$ modulo deformation. 
\end{remark}
Just as in the classical scenario, one then has the following characterizations of simple equivalences.
\begin{proposition}\label{propEqCharSimpleEq}
Let $f: X \to Y$ be a morphism in $\textnormal{s\textbf{Set}}_P^{\textnormal{fin}}$. Then the following are equivalent: 
\begin{enumerate}
		\item $f$ is a simple equivalence. \label{propEqCharSimpleEq1}
		\item $[f] = [t]^{-1}\circ [s]$ for $s,t$ finite FSAEs. \label{propEqCharSimpleEq2}
		\item $\tau_P(f) = 0$ \label{propEqCharSimpleEq3}
		\item $X \hookrightarrow M_f$ is a simple equivalence. \label{propEqCharSimpleEq4}
	\end{enumerate}
\end{proposition}
\begin{proof}
The first three statements are equivalent by \Cref{remSimpleStuffA} and \Cref{lemCharOfSimCs}. Clearly, weak equivalences fulfill a two out of three property by definition. Hence, by mapping cylinder factorization and \Cref{propCylIncs} (the fact that $Y \hookrightarrow M_f$ is an FSAE), we obtain that \ref{propEqCharSimpleEq1} is equivalent to \ref{propEqCharSimpleEq4}.
\end{proof}
Next up, we study the functoriality of $A_P$ and of $Wh_P$. By a slight abuse of notation, we denote any functoriality of the above by a $(-)_*$. That is, we use $f_*$ for $A_P([f]), Wh_P([f])$, for $f$ in $\textnormal{s\textbf{Set}}_P^{\textnormal{fin}}$ and $\alpha_*$ for $A_P(\alpha), Wh_P(\alpha)$ for $\alpha$ in $\mathcal{H}\textnormal{s\textbf{Set}}_P^{\textnormal{fin}}$. What is meant, will usually be obvious from context. So far, we really only have an explicit description of $\alpha_*\langle \beta \rangle$ after a choice of cofibrations $a$ and $b$ such that $[a]=\alpha$ ,$[b] = \beta$. By \Cref{thrmEckSieb}, $\alpha_*(\langle \beta \rangle$) is then given by $\langle b' \rangle$, where $b'$ is a pushout of $b$ along $a$ in $\textnormal{s\textbf{Set}}_P$. As is usually the case, for pushouts to interact well with homotopy properties, it suffices for one of the two maps to be a cofibration (and the objects to be cofibrant). This is reflected in \Cref{propDesOfFStar}. It is very reminiscent of the setting of homotopy pushouts (see for ex. \cite{kampsPo}), and also reflects in \ref{exSimpleMorphismCyl3} of \Cref{exSimpleMorphismCyl}. Before we prove this proposition, it is useful to know the following.
\begin{lemma}\label{lemInvOfs}
	Let $ X' \xrightarrow{s} X$ be an FSAE in $\textnormal{s\textbf{Set}}_P^{\textnormal{fin}}$ and $\langle \alpha \rangle \in A_P(X)$. Then $$ (s_*)^{-1}\langle \alpha \rangle = \langle \alpha \circ [s] \rangle $$
\end{lemma}
\begin{proof}
By \Cref{remSimpleStuffA}, we may without loss of generality assume that $\alpha = [a]$, for a cofibration $a$. Now, use \Cref{lemPushAlongf} to obtain: $$ s_*\langle a \circ s \rangle = s_*\langle s \rangle + \langle a \rangle = \langle a \rangle.$$
In particular, as $s_*$ is a bijection this shows the result.
\end{proof}
\begin{proposition}\label{propDesOfFStar}
Let $X \in \textnormal{s\textbf{Set}}_P^{\textnormal{fin}}$ and $f,g$ morphisms in $\textnormal{s\textbf{Set}}_P^{\textnormal{fin}}$ with source $f$. Consider a pushout square in $\textnormal{s\textbf{Set}}_P^{\textnormal{fin}}$
$$ \begin{tikzcd}
X \arrow[r,"f"] \arrow[d, "g"] & X' \arrow[d, "g'"]\\
Y \arrow[r] & Y'
\end{tikzcd}$$ and denote by the diagonal by $d$. If either $f$ or $g$ is a cofibration, then:
$$ f_*\langle g \rangle = \langle g' \rangle,$$
and 
$$ \langle f \rangle + \langle g \rangle = \langle d \rangle.$$
\end{proposition}
\begin{proof}
	We start with the case where $g$ is a cofibration. Denote by $i_{X'}$ and $i_X$ the respective inclusions into $M_f$. Then $[f] = [i_X']^{-1} \circ [i_X]$. Now, consider the pushout composition diagram \begin{equation}\label{diagDesOfStar}
	\begin{tikzcd}
	X \arrow[r, hook, "i_X"] \arrow[d, "g" , hook] & M_f \arrow[r, "p_{X'}"] \arrow[d, hook] & X' \arrow[d, "g'"]\\
	Y \arrow[r, hook] & M_f \cup_{i_X,g} X' \arrow[r] & Y' 
	\end{tikzcd}
	\end{equation} 
	Then as $[p_{X'}] = [i_{X'}]^{-1}$, by \Cref{lemInvOfs} $$f_*\langle g \rangle = \langle X' \xhookrightarrow{i_{X'}} M_f \hookrightarrow M_f \cup_{i_X,g} Y \rangle.$$ In \Cref{exSimpleMorphismCyl} we have shown that since $g$ is a cofibration $$M_f \cup_{i_X,g} Y \to Y' = X'\cup_{g,f}Y$$ is a simple morphism. Hence, by commutativity of \eqref{diagDesOfStar}, we obtain:
\begin{align*}
f_*\langle g \rangle & = \langle X' \xhookrightarrow{i_X'} M_f \hookrightarrow M_f \cup_{i_X,g} Y \rangle \\
&= \langle X' \xrightarrow{[p_{X'}]^{-1}} M_f \to M_f \cup_{i_X,g} Y \to X'\cup_{g,f}Y \rangle \\
&= \langle g' \rangle. 
\end{align*}
For the proof of the case where $f$ is a cofibration, again we refer to \eqref{diagDesOfStar} but with the roles of $f$ and $g$ swapped. In this case, the result follows from composability of pushouts and again the fact that if $f$ is a cofibration, then $$M_g \cup_{i_X,f} X' \to Y' = Y\cup_{f,g}X'$$ is a simple equivalence (\Cref{exSimpleMorphismCyl}). By the same fact, we also have that the diagonal of the left hand square in \eqref{diagDesOfStar} only differs from the diagonal of the rectangle by a simple equivalence, showing the addition formula.
\end{proof}
As two immediate corollaries of this and \Cref{propEqCharSimpleEq} we obtain:
\begin{corollary}\label{corStabofSim}
Simple equivalences are stable under pushout along cofibrations in $\textnormal{s\textbf{Set}}_P^{\textnormal{fin}}$. Simple equivalences that are also cofibrations are stable under pushouts along arbitrary morphisms in $\textnormal{s\textbf{Set}}_P^{\textnormal{fin}}$.
\end{corollary}
\begin{corollary}\label{corBasicFormula}
Let $f: X \to Y$ and $g: Y \to Z$ be morphisms in $\textnormal{s\textbf{Set}}_P^{\textnormal{fin}}$. Then $$ f_* \langle g \circ f \rangle = f_* \langle f \rangle + \langle g \rangle.$$
\end{corollary}
\begin{proof}
Take a mapping cylinder factorization of $f$, $f = p_Y \circ i_X $, and let $i_Y$ be the inclusion of $Y$ into $M_f$. Then set $\hat g : = g \circ p_Y$. Now, using \Cref{propDesOfFStar}, just as in \Cref{lemPushAlongf} one obtains:
$$(i_X)_*\langle \hat g \circ i_X \rangle = (i_X)_*\langle i_X \rangle + \langle \hat g \rangle .$$ Finally, applying $(p_Y)_* = (i_Y)_*^{-1}$ to the last equation gives:
\begin{align*}
	f_*\langle g \circ f \rangle = f_*\langle \hat g \circ i_X \rangle &= f_*\langle i_X \rangle + (p_Y)_*\langle \hat g \rangle \\ 
	&= f_*\langle p_Y \circ i_X \rangle + (i_Y)_*^{-1}\langle g \circ p_Y \rangle \\ 
	&= f_*\langle f \rangle + \langle g \circ p_Y \circ i_Y \rangle \\
	&= f_*\langle f \rangle + \langle g \rangle,
\end{align*} 
where the second equality follows from the fact that $p_Y$ is a simple equivalence, and the third equality follows from \Cref{lemInvOfs}.
\end{proof}
\begin{corollary}\label{corInvOfSimpleEq}
	Let $ X' \xrightarrow{s} X$ be a simple equivalence in $\textnormal{s\textbf{Set}}_P^{\textnormal{fin}}$ and $\langle \alpha \rangle \in A_P(X)$. Then $$ (s_*)^{-1}\langle \alpha \rangle = \langle \alpha \circ [s] \rangle $$
\end{corollary}
\begin{proof}
	The proof is identical to the one of \Cref{lemInvOfs}, but this time using \Cref{corBasicFormula}.
\end{proof}
The results in \Cref{corStabofSim} is of course reminiscent of the interaction of cofibrations and weak equivalences in a cofibrant model category. This makes on hope for analogue to the cube lemma (see for example \cite[Lem. 5.2.6]{hovey2007model}). In the classical setting, such a result follows from the sum theorem \cite[Prop. 23.1]{cohenCourse}. In the filtered case one can argue similarly. In fact, we now show the following result: 
\begin{proposition}\label{propBigSumFormula}
	Consider a commutative cube 
	\begin{equation}\label{diagComCube2}
	\begin{tikzcd}[row sep=2.5em]
	X^{0} \arrow[rr,"h_0"] \arrow[dr,swap, hook] \arrow[dd,swap] &&
	Y^{0} \arrow[dd,swap, near end] \arrow[dr,hook] \\
	& X^{1} \arrow[rr,crossing over, near start, "h_1"] && Y^{1}
	\arrow[dd] \\
	X^{2} \arrow[rr, near end, "h_2"] \arrow[dr,swap, "f_2" , hook] && Y^{2} \arrow[dr,swap, ,hook] \\
	& X \arrow[rr, "h"] \arrow[uu,< -,crossing over, "f_1", near end]&& Y
	\end{tikzcd},
	\end{equation}
in $\textnormal{s\textbf{Set}}_P^{\textnormal{fin}}$, where the left and the right face are cocartesian, the upper morphisms in these squares are cofibrations and $h_0$ is a weak homotopy equivalence. Denote by $f_0$ the diagonal in the left face.
Then: 
$$ \langle h \rangle = f_{1*}\langle h_1 \rangle + f_{2*} \langle h_2 \rangle - f_{0*} \langle h_0 \rangle .$$
\end{proposition}
\begin{proof}
We first reduce to the case where all maps in the left cocartesian face are cofibrations. Using mapping cylinder factorization on the left vertical and factoring this pushout face we obtain a commutative diagram:
	\begin{equation}\label{diagComCubeAdv}
	\begin{tikzcd}[row sep=2.5em]
	X^{0} \arrow[rr,"h_0"] \arrow[dr,swap, hook] \arrow[dd,swap, hook] &&
	Y^{0} \arrow[dd,swap, near end] \arrow[dr,hook] \\
	& X^{1} \arrow[rr,crossing over, near start, "h_1"] && Y^{1}
	\arrow[dd] \\
	M \arrow[rr, near end, "\tilde h_2"] \arrow[dd, "s"] \arrow[dr,swap, " \tilde f_2" , hook] && Y^2 \arrow[dr,swap, ,hook] \arrow[dd,"=", near end] \\
	& M' \arrow[rr, "\tilde h", crossing over, near start] \arrow[uu,< -,crossing over, "\tilde f_1", near end]&& Y \arrow[dd, "="] \\
	X^{2} \arrow[rr, near end, "h_2"] \arrow[dr,swap, "f_2" , hook] && Y^{2} \arrow[dr,swap, ,hook] \\
	& X \arrow[rr, "h"] \arrow[uu,< -,crossing over, "s'", near end]&& Y
	\end{tikzcd},
	\end{equation}
where $\tilde f_1, \tilde f_2$ are cofibrations, as they are pushouts of cofibrations, and $s$ is a simple equivalence. By \Cref{corStabofSim}, $s'$ is also a simple equivalence. Denote by $\tilde f_0$ the diagonal going from $X_0$ to $M'$. Assuming we have already shown the formula in the two cofibration case, we then obtain:
\begin{equation}\label{diagcofibCase}
	\langle \tilde h \rangle = \tilde f_{1*}\langle h_1 \rangle + \tilde f_{2*} \langle \tilde h_2 \rangle - \tilde f_{0*} \langle h_0 \rangle
\end{equation} Now, by \Cref{corInvOfSimpleEq} and commutativity of the diagram, we have 
\begin{align*}
\langle \tilde h \rangle &= (s')^{-1}_*	\langle h \rangle;\\
\langle \tilde h_2 \rangle &= (s)^{-1}_*	\langle h_2 \rangle.
\end{align*}
Together with \eqref{diagcofibCase} we obtain:
\begin{align*}
	\langle h \rangle & = s'_*( \tilde f_{1*}\langle h_1 \rangle + \tilde f_{2*} \langle \tilde h_2 \rangle - \tilde f_{0*} \langle h_0 \rangle )\\
	& = (s' \circ \tilde f_1)_*\langle h_1 \rangle + (s' \circ \tilde f_2)_*\langle \tilde h_2 \rangle - (s' \circ \tilde f_0)_* \langle h_0 \rangle\\
	& = f_{1*}\langle h_1 \rangle + f_{2*}(s_*\langle \tilde h_2 \rangle) - f_{0*}\langle h_0 \rangle \\
	& = f_{1*}\langle h_1 \rangle + f_{2*}\langle h_2 \rangle - f_{0*}\langle h_0 \rangle.
\end{align*}
So it remains to show the case where all the maps in the left square are cofibrations. We proceed similarly to the proof of \cite[Prop. 23.1]{cohenCourse}. To make simplify notation we omit the subscript of the pushout cups in the remainder of the proof, wherever possible. Denote by $\phi$ the morphism $X \hookrightarrow X \cup_{X^0, i_1} M_{h_0} = X \cup X^0 \otimes \Delta^1\cup Y^0$. $\phi$ is the bottom horizontal composition in the commutative diagam: $$ \begin{tikzcd}
&X_0 \arrow[r,"h_0"] \arrow[d, "i_0", hook] & Y_0 \arrow d\\
X \arrow[r, hook] &X \cup X^0 \otimes \Delta^1\arrow[r] & X \cup X^0 \otimes \Delta^1\cup Y^0
\end{tikzcd},
$$
where the right square is cocartesian and $i_0$ is a cofibration, due to our assumption that all the morphisms in the left square are cofibrations. By \Cref{propIncInProd} the first lower horizontal is an FSAE, in particular a weak equivalence. As weak equivalences are stable under pushout along cofibrations in a cofibrant model category \cite[Prop. 13.1.3]{hirschhornModel}, $\phi$ is also a weak equivalence. By \Cref{propDesOfFStar}, we have \begin{equation}\label{equPropSumFormula0}
	\langle \phi \rangle = f_{0*}\langle h_0 \rangle.
\end{equation} Now, consider the pushout square $$ \begin{tikzcd}
X \cup_{X^0} M_{h_0} \arrow[r, "\psi_1", hook] \arrow[d, "\psi_2"] & X \cup_{X^1} M_{h_1} \arrow[d]\\
X \cup_{X^2} M_{h_2} \arrow[r] & X \cup_{X} M_{h} = M_h 
\end{tikzcd}, $$
where we denote the diagonal by $\psi$. As $Y^0 \to Y^1$ is a cofibration, so is $\psi_1$. In particular, by \Cref{propDesOfFStar}, we obtain \begin{equation}\label{equPropSumFormula1}
	\langle \psi \rangle = \langle \psi_1 \rangle + \langle \psi_2 \rangle.
\end{equation}
Denote by $\eta_1, \eta_2, \eta$ the respective inclusions $X \hookrightarrow X \cup M_{h_1}$, etc. By \Cref{propDesOfFStar} \begin{align}
	\langle \eta \rangle &= \langle h \rangle \label{equPropSumFormula-1}\\
	\langle \eta_i \rangle &= f_{i_*}\langle h_i \rangle \textnormal{, for $i=1,2$}.\label{equPropSumFormula-2}
\end{align} 
Furthermore, these factor as: 
$$ \begin{tikzcd}
X \arrow[r, "\phi"] \arrow[rr, bend left = 60, "\eta_i", hook] & X \hookrightarrow X \cup M_{h_0} \arrow[r, "\psi_{i}"] & X \cup M_{h_i}
\end{tikzcd}.$$ with $i$ either $1,2$ or "empty". Applying \Cref{corBasicFormula} we obtain:
\begin{align}\label{equPropSumFormula2}
	\phi_*\langle \eta \rangle &= \phi_* \langle \phi \rangle + \langle \psi \rangle \\
	\label{equPropSumFormula3}
	\phi_*\langle \eta_i \rangle &= \phi_* \langle \phi \rangle + \langle \psi_i \rangle \textnormal{, for $i=1,2$}.
\end{align}
Note that as $\phi$ is a weak equivalence, $\langle \phi \rangle$ has a unique inverse.
Combining \eqref{equPropSumFormula1}, \eqref{equPropSumFormula2} and \eqref{equPropSumFormula3} we get: 
$$ \phi_*\langle \eta \rangle = \phi_*\langle \eta_1 \rangle + \phi_*\langle \eta_2 \rangle - \phi_* \langle \phi \rangle. $$ As $\phi$ is a weak equivalence, $\phi_*$ is an isomorphism. Hence, $$\langle \eta \rangle = \langle \eta_1 \rangle + \langle \eta_2 \rangle - \langle \phi \rangle.$$ Now, insert \eqref{equPropSumFormula0}, \eqref{equPropSumFormula-1} and \eqref{equPropSumFormula-2} into the last formula, to obtain :
$$\langle h \rangle = f_{1*}\langle h_1 \rangle + f_{2*} \langle h_2 \rangle - f_{0*} \langle h_0 \rangle.$$
\end{proof} 
As an immediate consequence of this, we have the classical cube lemma (see \cite[Lem. 5.2.6]{hovey2007model}) and a version of it for simple equivalences.
\begin{corollary}\label{corCubeForSim}
Consider a cube as in \Cref{propBigSumFormula}. Then if $h_0, h_1$ and $h_2$ are weak equivalences, so is $h$. Further, if $h_0, h_1$ and $h_2$ are simple equivalences then so is $h$.
\end{corollary}
\begin{proof}
By \Cref{remSimpleStuffB}, being a weak equivalence is equivalent to having invertible Whitehead torsion. As being invertible in a monoid is sustained under monoid morphisms and as sums of invertible elements are invertible the first statement follows by \Cref{propBigSumFormula}. For the second statement, repeat the same argument replacing invertible by being $0$.
\end{proof}
This nearly finishes our abstract investigation of the properties of the Whitehead monoid and torsion for filtered simplicial sets. We end the section with the following characterization of simple equivalences, summarizing some of the results in this subsection. This connects our results to the axiomatic approach taken in \cite[Ch. 6, \paragraphmark 3]{kampsPo} as well as a series of examples instrumental to the remainder of this work.
\begin{proposition}\label{propClassCharofSim}
	Denote by $\mathcal{S}$ the class of simple equivalences in $\textnormal{s\textbf{Set}}_P^{\textnormal{fin}}$. Then $\mathcal{S}$ is given by the smallest class containing all isomorphisms and elementary expansions that is closed under: 
	\begin{enumerate}
		\item The two out of three law;\label{propClassCharofSim1}
		\item Finite coproducts;\label{propClassCharofSim2}
		\item Pushouts of cofibrations in $\mathcal S$ along arbitrary morphisms in $\textnormal{s\textbf{Set}}_P^{\textnormal{fin}}$.\label{propClassCharofSim3}

	\end{enumerate}
	\begin{proof}
	First we show that $\mathcal{S}$ is in fact closed under these operations. For the two out of three law, this is immediate by definition. For finite coproducts, apply \Cref{corCubeForSim} to the case where $X^0$ and $Y^0$ are the empty filtered simplicial set. \ref{propClassCharofSim3} is part of the statement of \Cref{corStabofSim}. \\
	To show the converse inclusion, first note that, by \Cref{corSaeareExp}, the class generated in this fashion contains all FSAEs. In particular, by two out of three and \Cref{propIncInProd}, it contains all the mapping cylinder collapses $M_f \to Y$ for $X \to Y$ a morphism in $\textnormal{s\textbf{Set}}_P^{\textnormal{fin}}$. Hence, again by two out of three, it suffices to show that this class contains all simple equivalences that are also cofibrations. By \Cref{remSimpleStuffB} and \Cref{remSimpleStuffA} such a simple equivalence $a: X \to Y$ fits into a commutative diagram $$ \begin{tikzcd} X \arrow[rr, "a", hook] \arrow[rd, hook, "t"]& &Y \arrow[ld, "s", hook' ] \\
	&S&
	\end{tikzcd}$$
	in $\textnormal{s\textbf{Set}}_P^{\textnormal{fin}}$, where $s$ and $t$ are FSAEs. A final application of two out of three shows that then $a$ is contained in the class, as we have already noted that all (finite) FSAEs are.
	\end{proof}
\end{proposition}
\Cref{corCubeForSim} together with \Cref{propClassCharofSim} can be used to easily provide a useful range of examples of simple equivalences.
\begin{proposition}\label{propLvtSim}
	Let $X \in \textnormal{s\textbf{Set}}_P^{fin}$. Then both of the last vertex maps \begin{align*}
	\textnormal{sd}(X) &\longrightarrow X\\
	\textnormal{sd}_P(X) &\longrightarrow X
	\end{align*}
	are simple equivalences.
\end{proposition}
\begin{proof}
	By the fact that both subdivision functors preserve cofibrations and an inductive use of \Cref{corCubeForSim} it suffices to show this for $X = \Delta^{\mathcal J}$, for $\mathcal J$ a d-flag in $P$. In the case of $\textnormal{sd}$, \Cref{ExLotsOfFSAEs} \ref{ExLotsOfFSAE2} provides a section to the morphism that is an FSAE, hence the result follows by two out of three. In the case $\textnormal{sd}_P$, note that $\textnormal{sd}_P$ preserves simple equivalences. To see this, recall that as it has a right adjoint and sustains cofibrations, by \Cref{propClassCharofSim}, it suffices to show that it sends elementary expansions into simple equivalences. But this was already shown in \ref{ExLotsOfFSAEDout} of \Cref{ExLotsOfFSAEs}. Let $\mathcal J'$ be the non-degenerate flag in $P$ that $\mathcal J$ degenerates from. By \ref{ExLotsOfFSAE1} of \Cref{ExLotsOfFSAEs}, the degeneracy map is a simple equivalence. By the naturality of the last vertex map, we then have a commutative diagram 
	\begin{center}
		\begin{tikzcd}
			\Delta^{\mathcal J} \arrow[r] \arrow[d] & \Delta^{\mathcal J'} \arrow[d]\\
			\textnormal{sd}_P(\Delta^{\mathcal J}) \arrow[r] & \textnormal{sd}_P(\Delta^{\mathcal J'}).
		\end{tikzcd}
	\end{center}
	So again by two out of three, it suffices to show that the last vertex map is a simple equivalence for the case where $\mathcal J$ is non-degenerate. But, \ref{ExLotsOfFSAE3} of \Cref{ExLotsOfFSAEs} provides a section of the last vertex map that is given by a finite FSAE, completing the proof through a final appeal to two out of three.
\end{proof}
As an immediate corollary of this result we obtain by an application of two out of three:
\begin{corollary}\label{corSDRetainWeakEq}
	Both functors \begin{align*}
		\textnormal{s\textbf{Set}}_P^{fin} \xrightarrow{\textnormal{sd}} 	\textnormal{s\textbf{Set}}_P^{fin};\\
		\textnormal{s\textbf{Set}}_P^{fin} \xrightarrow{\textnormal{sd}_P} 	\textnormal{s\textbf{Set}}_P^{fin}
	\end{align*}
	preserve simple equivalences.
\end{corollary}
If one is entirely concerned with showing the last two statements in the weaker setting of weak equivalences, this can of course be done, just as in the classical setting and there is no restriction to finite simplicial sets necessary. Finally, we obtain an alternative description of the Whitehead torsion of a morphism in $\mathcal H\textnormal{s\textbf{Set}}_P^{fin}$.
\begin{corollary}\label{corLVdesTor}
	Let $\alpha: X\to Y $ be a morphism in $\mathcal H\textnormal{s\textbf{Set}}_P^{fin}$ and $n \in \mathbb N$ and \\$f:\textnormal{sd}_P^n(X) \to Y$ such that $\alpha = [f]\circ [\textnormal{l.v.}_P^n]^{-1}$ as in \Cref{propRepHoClasses}. Then
	$$\tau_P(\alpha) = (\textnormal{l.v.}^n_P)_* \big ( \tau_P(f) \big ).$$
\end{corollary}
\begin{proof}
	This is immediate from \Cref{propLvtSim} and \Cref{corInvOfSimpleEq}.
\end{proof}
\subsection{Filtered ordered simplicial complexes and simple homotopy type}\label{subsecSSvSC}
Recall the following classical result.
\begin{proposition}\cite[Prop. 7.2]{cohenCourse}\label{propCWsimSC}
	Every finite CW-complex has the simple homotopy type of a finite simplicial complex (of the same dimension).
\end{proposition}
In this subsection, we prove a filtered analogue of this proposition in the setting of $\textnormal{s\textbf{Set}}$. 
%Recall, \todo{add where simplicial complexes are introduced} that a simplicial set, such that if a simplex is non-degenerate so are all of its faces, is called a $\Delta$-set. A filtered simplicial set with this property, is called a \textit{filtered $\Delta$-set}.
% 
Recall that a \textit{filtered ordered simplicial complex} or short FOS-complex is equivalently a filtered simplicial-set, such that every non-degenerate simplex is uniquely determined by the set of its vertices (see \Cref{lemCharFOS}).
\begin{theorem}\label{thrmSSvSC}
	Every finite filtered simplicial set $X \in \textnormal{s\textbf{Set}}_P$ has the simple homotopy type of a finite filtered ordered simplicial complex (over $P$) of the same dimension.
\end{theorem}
This will be particularly useful later on, when we are going to describe the connection between the topological and the simplicial filtered homotopy categories. In practice, it allows us, to reduce to the setting of FOS-complexes and make use of a simplicial approximation theorem. To show this result, we first begin with some generalities on the relationship between FOS-complexes and filtered simplicial sets. Recall that a non-singular simplicial set is a simplicial set $S$ such that each of its non-degenerate simplices $\Delta^n \to S$ is given by an inclusion of simplicial sets. A $P$-filtered simplicial set is called \textit{non-singular} if its underlying simplicial set is non-singular.
\\
\\
Clearly, every FOS-complex is a filtered non-singular filtered simplicial set. Furthermore, in a certain sense, these properties improve under subdivision. Consider, for example, the following simplicial model of a singular $S^1$.
\begin{figure}[H]
	\centering
	\includegraphics[width=120mm]{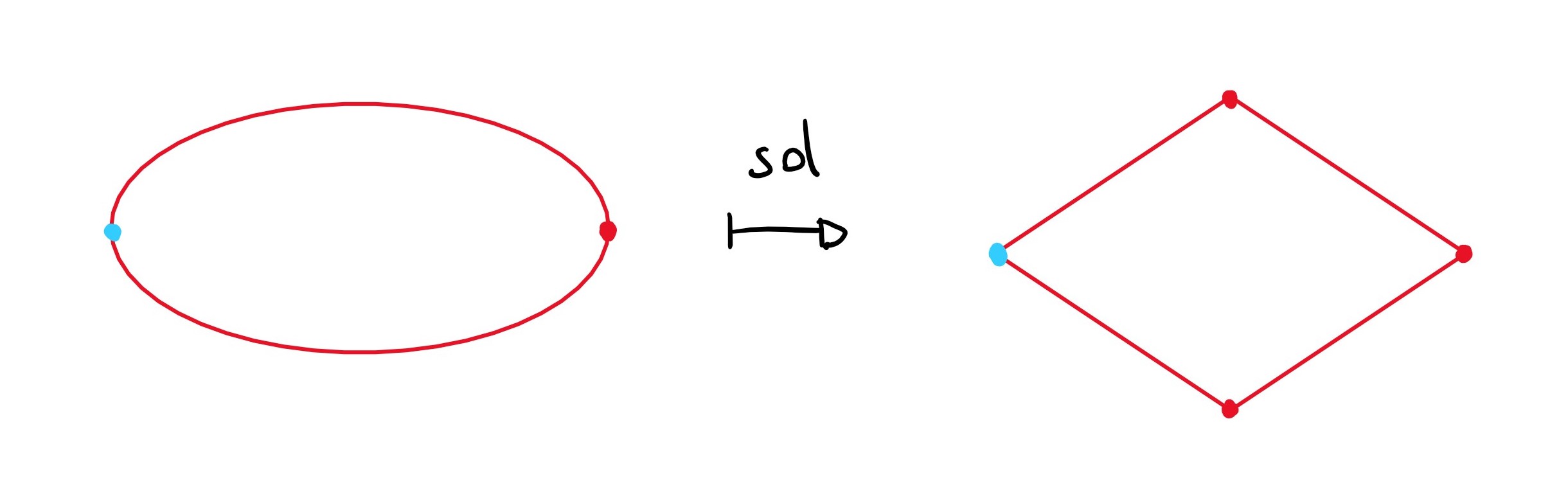}
	\caption{Triangulation of a $S^1$ filtered over $P = \{0,1\}$, with the standard color scheme. The left hand side is a non-singular filtered simplicial set. After subdividing once, we obtain an FOS-complex.}
	\label{fig:exFQStoFOS}
\end{figure}
This type of behavior is the content of the following lemma.
\begin{lemma}\label{lemFQSImproves}
	Let $K \in \textnormal{s\textbf{Set}}_P$ be a non-singular filtered simplicial set. Then $\textnormal{sd}(K)$ is an FOS-complex.
\end{lemma}
\begin{proof}
This has nothing to do with the filtrations, so we might just as well prove the non-filtered case.
Every simplex of $\textnormal{sd}(K)$ is given by a pair $(\Delta^{n} \xrightarrow{\tau} K, \sigma)$ with $\sigma \in \textnormal{sd}(\Delta^{n})$. 
That is, it is given by $$\Delta^k \xrightarrow{\sigma} \textnormal{sd}(\Delta^{n}) \xrightarrow{\textnormal{sd}(\tau)} \textnormal{sd}(K).$$ 
Furthermore, by choosing $n$ minimally, we may assume that $\sigma = (\sigma_0 \subset ... \subset \sigma_k)$ fulfills $\sigma_k = [n]$, as otherwise, the simplex comes from the boundary of $\textnormal{sd}(\Delta^n)$ contradicting minimality. Under the minimality assumption, a simplex is then uniquely determined by such a pair. In particular,
 the vertices of $\textnormal{sd}(K)$ are then given by pairs $(\Delta^{n} \xrightarrow{\tau} K, [n])$, where $\tau$ is non-degenerate. In other words, they correspond to the non-degenerate simplices of $K$. The vertex set, for a simplex $(\Delta^{n} \xrightarrow{\tau} K, \sigma)$ as above, is then given by $$\{\Delta^{\#\sigma_i} \xrightarrow{\sigma_i} \Delta^{n} \xrightarrow{\tau} K \mid i \in [k]\}.$$ This set is ordered by the containment relation of the $\sigma_i$. By the minimality assumption, its maximal element always corresponds to $\tau$. Hence, for two non-degenerate simplices to have the same vertex set, they have to come from representing pairs $(\tau, \sigma)$, $(\tau', \sigma')$ with $\tau = \tau'$. Now, as $\tau$ is an inclusion by assumption, and thus pushing forward with it is injective, this means that the underlying set of the flags $\sigma$ and $\sigma'$ in $[n]$ agree. As they are non-degenerate, this implies that $\sigma = \sigma'$. Hence, we have shown that every non-degenerate simplex in $\textnormal{sd}(K)$ is uniquely determined by its vertices, i.e. that $K$ is an FOS-complex.
\end{proof}
The advantage of non-singular filtered simplicial sets over FOS-complexes is that they are stable under certain pushouts in $\textnormal{s\textbf{Set}}_P$.
\begin{lemma}\label{lemPushoutofFQS}
	Consider a pushout diagram in $\textnormal{s\textbf{Set}}_P$ \begin{center}
		\begin{tikzcd}
			L \arrow[hook, r, hook] \arrow[d, hook]& K \arrow[d, hook] \\
			L' \arrow[hook, r] & K'
		\end{tikzcd}
	\end{center}
with the upper horizontal and the left vertical (and hence all arrows) cofibrations and $K$ and $L'$ (and hence $L$) non-singular filtered simplicial sets. Then $K'$ is also a non-singular filtered simplicial set.
\end{lemma}
\begin{proof}
This an easy exercise in composability of cofibrations.
\end{proof}
We now make use of a construction similar to the mapping cylinder of simplicial complexes, used in \cite{whitehead1939simplicial}. The essential idea is that by using this mapping cylinder construction, the attaching maps of a simplex can be replaced by cofibrations, allowing an application of \Cref{lemPushoutofFQS}. 
\begin{definitionconstruction}\label{conMcx}
Let $K \xrightarrow{f} K'$ be a morphism in $\textnormal{s\textbf{Set}}_P$ with $K$ and $K'$ FOS-complexes. For the remainder of this construction, we think of $f$ as a morphism of ordered simplicial complexes in the classical sense as in \Cref{subsecOrdered}, that is as a map on the vertices preserving the ordering and filtration and fulfilling that the image of ever simplex (thought of a a set of vertices) is again a simplex. By the fully faithful embedding from \Cref{corEmbedFos}, a construction in this setting transfers into the setting of simplicial sets. Now, denote by $M^{cx}_f$ the subcomplex of the join (recall \Cref{conJoin}) of $\textnormal {sd}(K')$ and $\textnormal {sd}(K)$ given by 
\begin{align*}
	M^{cx}_f:=\Big\{ \{\sigma_0 \subset ... \subset \sigma_k \} \star \{\tau_{k+1} \subset ... \subset \tau_{k+l}\} \mid \sigma_i \subset f(\tau_{k+j}) \textnormal{ for all $i,j$} \Big\}. 
\end{align*}
We have an induced filtration on the vertex set by the respective filtrations of $\textnormal{sd}(K)$ and $\textnormal{sd}(K')$. While such definition on the vertex set will generally not give a filtration for the join, since not any two vertices are contained in a common $1$-simplex in $\textnormal{sd}(P)$, it does define one on $M_f^{cx}$, as $f$ is stratum preserving. An ordering on $M_f^{cx}$ is then given by taking the two orderings on $\textnormal{sd}(K)$ and $\textnormal{sd}(K')$ and putting vertices in the latter before ones in the former, whenever they lie in a common simplex. Clearly, $M_f^{cx}$ comes with two inclusions $$ \textnormal {sd}(K') \hookrightarrow M_f^{cx} \hookleftarrow \textnormal{sd}(K).$$ Furthermore, the former inclusion comes with a retract \begin{align*}
	M^{cx}_f &\longrightarrow \textnormal{sd}(K') \\
	\{\sigma_0 \subset ... \subset \sigma_k\} \star \{\tau_{k+1} \subset ... \subset \tau_{k+l}\} &\longmapsto \{\sigma_0 \subset ... \subset \sigma_k \subset f(\tau_{k+1}), ... \subset f(\tau_{k+l})\}
\end{align*}
Hence, as for the usual mapping cylinder, one obtains a commutative mapping cylinder splitting diagram: \begin{center}
	\begin{tikzcd}
		\textnormal{sd}(K) \arrow[rr, "{\textnormal{sd}(f)}", bend left] \arrow[r, hook] & M^{cx}_f \arrow[r] & \textnormal{sd}(K')
	\end{tikzcd}.
\end{center}
Finally, note that if $K$ and $K'$ are both of dimension lower than $n$, then $M^{cx}_f$ is of dimension lower than $n+1$.
\end{definitionconstruction}
\begin{proposition}
	In the setting of \Cref{conMcx}, the inclusion $K' \hookrightarrow M_f^{cx}$ is an FSAE. In particular, whenever $K$ and $K'$ are finite, the retract $M^{cx}_f \longrightarrow \textnormal{sd}(K')$ is a simple equivalence.
\end{proposition}
\begin{proof}
Again we think of $f$ also as a map of ordered filtered simplicial complexes.
Using \Cref{propEqCharSaeTot} we construct a regular pairing on $(M_f^{cx})_{n.d.} \setminus (K')_{n.d.}$. Denote $B:= M_f^{cx}$ and $A:=K'$. We say that a non-degenerate simplex in $B_{n.d.} \setminus A_{n.d.}$, $\{\sigma_0 \subsetneq ... \subsetneq \sigma_k\} \star \{\tau_{k+1} \subsetneq ... \subsetneq \tau_{k+l}]$ is of type $I$ if $\{\sigma_{0},...,\sigma_{k}\} \neq \emptyset$ and $\sigma_{k} = f(\tau_{k+1})$. Note that as we assume the simplex not to lie in $A$, such a $\tau_{k+1}$ always exists. Else, we say it is of type $II$. Then \begin{align*}
	T: B_{II} &\longrightarrow B_I \\
	 \{\sigma_0 \subsetneq ...\subsetneq \sigma_k\}\star \{\tau_{k+1} \subsetneq ... \subsetneq \tau_{k+l}\} &\longmapsto \{\sigma_0 \subsetneq ... \subsetneq \sigma_k \subsetneq f(\tau_{k+1})\} \star \{\tau_{k+1} \subsetneq ... \subsetneq\tau_{k+l}\}
\end{align*}
As $f$ is stratum preserving, this in fact defines a proper pairing. To see that it is regular, we use \Cref{lem14Moss} and set \begin{align*}
	\Phi: B_{II} &\longrightarrow \mathbb{N} \\
	\{\sigma_0 \subsetneq ... \subsetneq \sigma_k\} \star \{\tau_{k+1} \subsetneq ... \subsetneq\tau_{k+l}\} &\longmapsto l.
\end{align*}
It is an easy verification very much similar to the one in the proofs of \Cref{ExLotsOfFSAEs} that this gives a map such that $$\sigma \prec \sigma' \implies f(\sigma) < f(\sigma').$$
\end{proof}
Using the mapping cylinder splitting from \Cref{conMcx} we obtain from this result: 
\begin{corollary}\label{corMcxFact}
	Let $K \xrightarrow{f} K'$ be a morphism of finite FOS-complexes in $\textnormal{s\textbf{Set}}_P$. Then $\textnormal{sd}(f)$ factors into a cofibration of finite FOS-complexes followed by a simple equivalence of FOS-complexes.
\end{corollary}
We are now in shape to prove \Cref{thrmSSvSC}.
\begin{proof}[Proof of \Cref{thrmSSvSC}]
	We prove the theorem via induction over the dimension of $X$. When $X$ is $0$-dimensional the result is trivial, as every $0$-dimensional filtered simplicial set is also an FOS-complex. Now, if $X$ is $(n+1)$-dimensional, then it it fits into a pushout diagram 
	\begin{equation}\label{prooThmSSvSCstart}
		\begin{tikzcd}
			\bigsqcup \partial \Delta^{\mathcal{J}_i} \arrow[r, hook] \arrow[d] &\bigsqcup \Delta^{\mathcal{J}_i} \arrow[d]\\
			\hat X \arrow[r, hook] &X
		\end{tikzcd}
	\end{equation}
where $\hat X$ is the $n$-skeleton of $X$ with the induced filtration and the disjoint unions (together with their induced maps) are given by the non-degenerate $n+1$-simplices of $X$. By the induction hypothesis, $\hat X$ has the simple homotopy type of a finite $n$-dimensional FOS-complex $K$. By \Cref{propRepHoClasses} this means that there is a $k \in \mathbb N$ and a simple equivalence $\textnormal{sd}^k_{P}(\hat X) \xrightarrow{a} K$. Now, as $\textnormal{sd}_{P}$ sustains colimits and cofibrations, this gives us a composition of pushout diagrams 
\begin{equation*}
	\begin{tikzcd}
	\bigsqcup \textnormal{sd}^k_{P}\big ( \partial \Delta^{\mathcal{J}_i} \big ) \arrow[r, hook] \arrow[d] &\bigsqcup \textnormal{sd}^k_{P} \big (\Delta^{\mathcal{J}_i} \big ) \arrow[d]\\
	\textnormal{sd}_{P}^k(\hat X) \arrow[r, hook] \arrow[d, "a"]& \textnormal{sd}^k_{P}(X) \arrow[d, "\tilde a"] \\
	K \arrow[r, hook] & \tilde X.
	\end{tikzcd}
\end{equation*}
By \Cref{propLvtSim}, $X$ is simple homotopy equivalent to $\textnormal{sd}^k_{P}(X)$. Furthermore, by \Cref{corStabofSim}, $\tilde a$ is also a simple equivalence. Hence, it suffices to show that $\tilde X$ has the simple homotopy type of a finite $n+1$-dimensional FOS-complex. By applying the mapping cylinder factorization of \Cref{corMcxFact} to $\bigsqcup \textnormal{sd}^k_{P}\big ( \partial \Delta^{\mathcal{J}_i} \big ) \to K$, we again obtain a new pushout composition diagram 
\begin{equation*}
\begin{tikzcd}
\textnormal{sd} \Big (\bigsqcup \textnormal{sd}^k_{P}\big ( \partial \Delta^{\mathcal{J}_i} \big ) \Big ) \arrow[r, hook] \arrow[d, hook] 
&\textnormal{sd}\Big ( \bigsqcup \textnormal{sd}^k_{P} \big (\Delta^{\mathcal{J}_i} \big ) \Big ) \arrow[d]
\\ \tilde K \arrow[r, hook] \arrow[d]& \tilde K' \arrow[d] \\
\textnormal{sd}(K) \arrow[r, hook] & \textnormal{sd}(\tilde X)
\end{tikzcd},
\end{equation*}
where $\tilde K$ is a finite FOS-complex of dimension $n+1$ (as it is given by the simplicial mapping cylinder from \Cref{conMcx}), the first left vertical is a cofibration and $\tilde K \to sd(K)$ is a simple equivalence. Thus, again by \Cref{corStabofSim}, $\tilde K' \to \textnormal{sd}(\tilde X)$ is a simple equivalence. As the latter target is simply equivalent to $\tilde X$ by \Cref{propLvtSim}, it suffices to show that $\tilde K'$ has the homotopy type of a finite $(n+1)$-dimensional FOS-complex. But by \Cref{lemPushoutofFQS}, $\tilde K'$ is a finite $(n+1)$-dimensional non-singular filtered simplicial set. Hence, by one last appeal to \Cref{propLvtSim}, this is simply equivalent to the finite $(n+1)$-dimensional FOS-complex $\textnormal{sd}(\tilde K')$, where we used \Cref{lemFQSImproves} to see that this is in fact an FOS-complex.
\end{proof}
\section{The Whitehead group in a topological filtered setting}
\label{secTopWh}
So far our Whitehead group is defined as a weak equivalence invariant of finite filtered simplicial sets and the Whitehead torsion as an invariant of morphisms in $\mathcal{H}\textnormal{s\textbf{Set}}_P^{fin}$. Recall, however, that in the classical setting (see for example \cite{cohenCourse}), the Whitehead group is a weak homotopy invariant of topological spaces (in fact it is even an invariant of the fundamental groupoid), and that, given a choice of finite CW-structure (or even a weakly homotopy equivalence to a finite CW-complexes), the Whitehead torsion is an invariant of the (weak) homotopy class of maps. Hence, while being purely combinatorially constructed, the classical Whitehead group measures things in the topological realm. One of course would hope that a similar result holds in the filtered setting. This turns out to be true. \\
\\
Classically, one explanation why the Whitehead group shifts to the topological setting is that the homotopy category of the category CW-complexes with cellular maps, $\mathcal H \textnormal{\textbf{CW}}$, is equivalent to the homotopy category of topological spaces $\mathcal H \textnormal{\textbf{Top}}$ (with respect to the Quillen Model Structure). In particular, one obtains an equivalence of categories from the full subcategory of $\mathcal H \textnormal{\textbf{CW}}^{fin}$ given by finite CW-complexes to the full subcategory of $\mathcal H \textnormal{\textbf{Top}}$ given by topological spaces, weakly equivalent to a finite CW-complex. One can then think of computing the Whitehead group and Whitehead torsion in the topological setting, as choosing an inverse to this equivalence (i.e. for each space a weakly equivalent CW-complex) and then composing this equivalence with the Whitehead group defined on CW-complexes. Of course, while the isomorphism type of the Whitehead groups does not depend on this choice, the Whitehead torsion still does. That is, for this construction to be well-defined, one needs to keep track of the choices of (weakly equivalent) CW-structure one makes.\\
\\
Now, to define simple homotopy theory in the filtered topological setting, we could take the approach of taking an inverse (up to isomorphism) to the restriction of the fully faithful functor $|-|_P:\mathcal H\textnormal{s\textbf{Set}}_P^{fin} \to \mathcal H \textnormal{\textbf{Top}}_{P}$ to its essential image (see \Cref{thrmFullyFaithful}), and then composing it with the Whitehead group functor. Note however that while this would make the Whitehead group well-defined up to natural isomorphism, the Whitehead torsion of a map would still depend on a choice of inverse. In fact, it would even depend on the choice of inverse whether the Whitehead torsion is $0$, i.e. whether a filtered map is a simple equivalence, in some sense. This problem of course already occurs in the classical setting. To make the Whitehead torsion well-defined, one needs to first make a choice of CW-structure (up to homotopy) for all spaces involved. This problem can somewhat be amended using the fact that every homeomorphism is a simple equivalence (see \cite[App. Main Theorem]{cohenCourse}), thus giving a notion of simple equivalence, only depending on the homeomorphism type of the spaces involved. However, as we have no such result in the filtered setting at hand, a priori, the construction of Whitehead torsion always depends on such choices. We thus restrict our-self to the setting, of realizations of filtered simplicial sets. A generalization to the ``up to weak equivalence'' setting is straightforward.\\
\\
Denote by $|\textnormal{s\textbf{Set}}^{fin}_P|$ the category with objects given by the objects in $\textnormal{s\textbf{Set}}_P$ and morphisms given by $$\textnormal{Hom}_{|\textnormal{s\textbf{Set}}^{fin}_P|}(X,Y) = \textnormal{Hom}_{\textnormal{\textbf{Top}}_{P}}(|X|_P,|Y|_P).$$ We call this the \textit{category of finitely triangulated spaces, filtered over $P$}. Further, we denote by $\mathcal H|\textnormal{s\textbf{Set}}^{fin}_P|$ the category with the same objects and morphisms given by $$\textnormal{Hom}_{\mathcal H|\textnormal{s\textbf{Set}}^{fin}_P|}(X,Y) = \textnormal{Hom}_{\mathcal H\textnormal{\textbf{Top}}_{P}}(|X|_P,|Y|_P).$$ Both, by definition, embed fully faithfully into $\textnormal{\textbf{Top}}_{P}$ and $\mathcal H \textnormal{\textbf{Top}}_{P}$ respectively. By \Cref{thrmFullyFaithful}, we immediately obtain:
\begin{corollary}\label{corSimApprox'}
	The realization functor $$|-|_P:\mathcal H\textnormal{s\textbf{Set}}_P^{fin} \longrightarrow \mathcal H\textnormal{\textbf{Top}}_{P}$$ induces an isomorphism of categories $$\mathcal H\textnormal{s\textbf{Set}}^{fin}_P \xrightarrow{\sim} \mathcal H|\textnormal{s\textbf{Set}}_P|^{fin}.$$ 
	In particular, every morphism $\phi: |X|_P \to |Y|_P$ in 
	$\mathcal H|\textnormal{s\textbf{Set}}_P|$ is of the shape 
	$$\phi = [|f|_P] \circ [|\textnormal{l.v.}_P^{n}|_P]^{-1},$$ for sufficiently large 
	$n \in \mathbb N$ and some $f: \textnormal{sd}_P^n(X) \to Y$.
\end{corollary} Such an $f$ as in \Cref{corSimApprox'} is called a \textit{simplicial approximation to $\phi$.} Thus, we can now think of the functors $A_P$ and $Wh_P$ constructed in \Cref{subsecEckSiebAppHolds} as functors being defined on $\mathcal H|\textnormal{s\textbf{Set}}_P|^{fin}$. 
\begin{lemma}\label{lemEquCHarTopSim}
Let $\phi$ be a morphism in $\mathcal H|\textnormal{s\textbf{Set}}_P|^{fin}$. Then the following are equivalent.
\begin{enumerate}
	\item The morphism in $\mathcal H\textnormal{s\textbf{Set}}_P^{fin}$ corresponding to $\phi$ is a simple equivalence.
	\item $f$ is a simple equivalence, where $f$ is any simplicial approximation of $\phi$.
\end{enumerate}
\end{lemma}
\begin{proof}
	This is an immediate consequence of \Cref{thrmSSvSC}, \Cref{propLvtSim} and the two out of three property for simple equivalences in $\mathcal H \textnormal{s\textbf{Set}}_P^{fin}$.
\end{proof}
Using the isomorphism of categories $\mathcal H \textnormal{s\textbf{Set}}_P^{fin} \cong \mathcal H|\textnormal{s\textbf{Set}}_P^{fin}|$ we can now transfer most of the nomenclature from the setting of stratum preserving simplicial maps to the setting of stratum preserving maps of the underlying realizations.
	\begin{definition}\hfill
		\begin{itemize}
			\item A morphism $\phi$ in $\mathcal{H}|\textnormal{s\textbf{Set}}_P^{\textnormal{fin}}|$ is called \textit{a simple equivalence} if any of the equivalent characterizations in \Cref{lemEquCHarTopSim} is fulfilled.
			\item Two objects in $X,Y \in \mathcal H|\textnormal{s\textbf{Set}}^{fin}_P|$ are said to have the same \textit{simple homotopy type} or also to be \textit{simply equivalent} if there is a simple equivalence $\phi:X \to Y$.
			\item A morphism $\varphi$ in $|\textnormal{s\textbf{Set}}_P^{\textnormal{fin}}|$ is called a $\textit{simple equivalence}$ if $[\varphi]$ is a simple equivalence.
			\item We say two morphisms $\phi,\phi' $ in $\mathcal{H}|\textnormal{s\textbf{Set}}_P^{\textnormal{fin}}|$ with the same source \textit{have the same simple morphism class} if their corresponding morphism in $\mathcal H\textnormal{s\textbf{Set}}_P^{fin}$ have the same simple morphism class. That is, if there is a simple equivalence $\sigma $ in $\mathcal{H}|\textnormal{s\textbf{Set}}_P^{\textnormal{fin}}|$ such that $\phi = \sigma \circ \phi'$. The equivalence class generated by this relation is called the \textit{simple morphism class of $\phi$}, denoted by $\langle \phi \rangle$. 
			\item The simple morphism class of a morphism $\varphi$ in $|\textnormal{s\textbf{Set}}_P^{\textnormal{fin}}|$ is defined to be the simple morphism class of $[\varphi]$, denoted by $\langle \varphi \rangle$.
			\item The \textit{Whitehead torsion}, $\tau_P(\phi)$, of a morphism $\phi \in \mathcal H|\textnormal{s\textbf{Set}}_P|$, is defined as the Whitehead torsion of the corresponding morphism in $\textnormal{s\textbf{Set}}_P$.
			\item The \textit{Whitehead torsion} $\tau_P(\varphi)$ of a morphism $\varphi \in|\textnormal{s\textbf{Set}}_P|$ is defined as $\tau_P([\varphi])$.
		\end{itemize}
	\end{definition} 
Summarizing results from the previous sections, we then have the following equivalent characterization of our Whitehead monoid, group and torsion.
\begin{proposition}\label{propCharOfWh}
	Let $X \in \textnormal{s\textbf{Set}}_P^{fin}$. Then there is a commutative diagram of bijections
\begin{center}
	\begin{tikzcd}
		A_P(X) \arrow[d] \arrow[r, leftarrow] & 
		{\sfrac{\Big \{a \mid X \xhookrightarrow{a} Y \in \textnormal{s\textbf{Set}}_P^{fin} \textnormal{ a cofibration} \Big \}}{\textnormal{deformation}} } \arrow[d, "\langle - \rangle"] \\
		
		{\Big \{\langle \alpha \rangle \mid X \xrightarrow{\alpha} Y \in \mathcal H\textnormal{s\textbf{Set}}_P^{fin} \Big\}} \arrow[d, "{|-|_P}"] & {\Big \{ \langle f \rangle \mid X \xrightarrow{f} Y \in \textnormal{s\textbf{Set}}_P^{fin} \Big \} } \arrow[l, hook'] \arrow[d, "|-|_P"]&\\
		 \Big \{\langle \phi \rangle \mid X \xrightarrow{\phi} Y \in \mathcal H|\textnormal{s\textbf{Set}}_P^{fin}| \Big\} & \Big \{\langle \varphi \rangle \mid X \xrightarrow{\varphi} Y \in \mathcal |\textnormal{s\textbf{Set}}_P^{fin}| \Big\} \arrow[l, hook']
	\end{tikzcd}.
\end{center}
Furthermore, they restrict to bijections with $Wh_P(X)$ if one adds the additional condition of the respective arrows being an isomorphism/weak equivalence respectively. Under these identifications, the Whitehead torsion of an arrow $\varphi: X\to Y \in |\textnormal{s\textbf{Set}}_P|$ is equivalently given by:
\begin{center}
	\begin{tikzcd}
		{\tau_P(\varphi)} \arrow[d, equal]\arrow[r, equal] & ... \arrow[d, equal]\\
		{ \langle [f] \circ [\textnormal{l.v.}_P^n] \rangle } \arrow[r, equal] \arrow[d, equal]& {(\textnormal{l.v.}_P^n)_*^{-1}(\langle f \rangle)} \arrow[d, equal]\\
		{ \langle [\varphi] \rangle }\arrow[r, equal] & \langle \varphi \rangle
	\end{tikzcd},
\end{center}
where $f: \textnormal{sd}^n_P(X) \to Y$ is a simplicial approximation of $\varphi$.
\end{proposition}
\begin{proof}
	The upper square was already described in detail in \Cref{remSimpleStuffA}. The lower square commutes by definition of simple morphism classes. Furthermore, the lower left vertical is a bijection by \Cref{corSimApprox'} and the definition of simple equivalences in $\mathcal H|\textnormal{s\textbf{Set}}_P|$. Hence, by commutativity, the lower left horizontal inclusion is onto, making all the maps involved bijective. The statement on the Whitehead group follows analogously to \Cref{remSimpleStuffB}, from the fact that a morphism is a weak equivalence in a model category if and only if its image in the homotopy category is an isomorphism. Finally, the statement on the Whitehead torsion is immediate from \Cref{corLVdesTor}.
\end{proof}
\subsection{Comparison to the classical Whitehead group}
The obvious question arises how our construction for the Whitehead group of a filtered simplicial set relates to the classical construction of the Whitehead group of a CW-complex. In this section, we are going to show that our construction can be thought of as a generalization from the case where $P$ is a one-point set to the case of arbitrary partially ordered sets. To be more specific:\\
\\
In the case where $P = \star$, $\textnormal{s\textbf{Set}}_P$ is isomorphic to $\textnormal{s\textbf{Set}}$ in the obvious way, and the Douteau model structure is the Kan-Quillen model structure on $\textnormal{s\textbf{Set}}$. The analogous statement can be made for $\textnormal{\textbf{Top}}_{\star}$, and the model structure on $\textnormal{\textbf{Top}}_{\star} \cong \textnormal{\textbf{Top}}$ is just the classical Kan-Quillen one (on $\Delta$-generated spaces, to be precise). Note that we used ``$\star$'' to avoid any ambiguity with pointed spaces. Let $\mathcal H \textnormal{\textbf{CW}}^{fin}$ denote the homotopy category of finite CW-complexes. By the cellular approximation theorem, $\mathcal H \textnormal{\textbf{CW}}^{fin}$ embeds fully faithfully into $\mathcal H \textnormal{\textbf{Top}}$. The realization of a simplicial set naturally carries the structure of a CW-complex. This induces a fully faithful embedding $$\mathcal H \textnormal{s\textbf{Set}}_\star^{fin} \cong \mathcal H|\textnormal{s\textbf{Set}}^{fin}_*|= \mathcal H|\textnormal{s\textbf{Set}}^{fin}| \hookrightarrow \mathcal H \textnormal{\textbf{CW}}^{fin}.$$ As every finite CW-complex has the simple homotopy type of a finite simplicial complex (see \cite[Prop. 7.2]{cohenCourse}), this is even an equivalence of categories. Now, consider the following diagram
\begin{equation}\label{diagWhAreIso}
	\begin{tikzcd}
		\mathcal H \textnormal{s\textbf{Set}}_\star \arrow[r, "\sim"] \arrow[rd, "Wh_{\star}", swap]&\mathcal H \textnormal{\textbf{CW}}^{fin} \arrow[d, "Wh"]\\
		{ }& \textnormal{\textbf{Ab}}
	\end{tikzcd},
\end{equation}
where by $Wh$ we refer to the classical Whitehead group functor. In this subsection, we are going to prove the following result. (Note that for some reason, some authors define the classical Whitehead torsion to live in the Whitehead group of the target space, not the source. We always mean the corresponding element in the source).
\begin{theorem}\label{thrmOldNewAgree}
	The diagram \eqref{diagWhAreIso} commutes up to a natural isomorphism, uniquely described by the property that, for $f\in|\textnormal{s\textbf{Set}}_\star^{fin}|$, $$ \tau_\star(f) \mapsto \tau(f),$$ with $\tau(f)$ the classical Whitehead torsion.
\end{theorem}
In particular, our theory turns out to be a straight up generalization of the non filtered setting. We need a result about the relationship between simple homotopy equivalence in the simplicial complex and the CW-complex setting, that seems to be somewhat folklore. Recall that an elementary expansion of a simplicial complex $K$ (not necessarily ordered) is a map $e$ obtained by a pushout in simplicial complexes as below 
\begin{center} 	
	\begin{tikzcd} \Lambda^n \arrow[r, hook] \arrow[d, hook] & \Delta^n \arrow[d] \\
	K \arrow[r,"e", hook] & K'
	\end{tikzcd},
\end{center}
where $\Lambda^n \hookrightarrow K$ is an inclusion of a full subcomplex. Note that both conditions, injectivity and fullness, are necessary as colimits in the category of simplicial complexes are generally rather non-geometrical in their behavior.
\begin{remark}\label{remElemAreSame}
	 Consider some arbitrary order of $K'$. This induces compatible orderings on the whole diagram. Then $e$ maps to an elementary expansion in the sense of \Cref{defElem}, under the fully faithful inclusion of ordered simplicial complexes into simplicial sets, $S^o$ from \Cref{corEmbedFos}. Conversely, every finite (F)SAE of ordered simplicial complexes maps to a composition of elementary expansion of simplicial complexes, under the forgetful functor.
\end{remark} 
For a fixed finite simplicial complex $K$, one then sets $$E_{scx}(K)=\sfrac{\big \{ K \xhookrightarrow{a} K' \mid \text{ s.t. } K' \text{ finite and } |a| \text{ is a homotopy equivalence} \big \}}{\sim_e},$$ where "$\sim_e$" is the equivalence relation induced by composition with elementary expansions (and isomorphisms). The equivalence class of $1_K$, $\langle 1_K \rangle$, turns this into a pointed set.
\begin{remark}\label{remPushoutInSim}
	We should say a few words on the matter of why we are not just applying the theory in \Cref{subsecEckSiebApp}. While at first sight it might be very tempting and in fact the authors of \cite{siebenmannInfinite} seem to think it is a straightforward application, the theory is not applicable. We have already noted in \Cref{remAxAreWrong} that pushouts in the category of inclusions of CW-complexes do not exits. The same argument holds for simplicial complexes. But even if one decides to work with the modified axioms that we adopted in \Cref{subsecEckSiebApp}, one runs into the problem of pushouts, even along cofibrations in the larger category of all simplicial complexes, being highly ungeometric. For example the pushout of two $1$-simplices along their boundary is not a $S^1$ as one might hope, but again a $1$-simplex. To obtain the correct geometrical thing, one needs to either pass to a larger combinatorial category, such as simplicial sets or allow for some notion of subdivision in the morphisms, i.e. pass to the piecewise linear setting (with the caveat, that even there not all pushouts do exist).
	In particular, a priori $E_{scx}(-)$, as defined above, neither carries a group structure, nor is it a functor. We circumvent these shortcoming, through artificially introducing subdivisions in the next construction. It might be that if one very carefully substitutes some mapping cylinder arguments through mapping cylinders of simplicial complexes, as used in \cite{whitehead1939simplicial}, such a step might not be necessary.
\end{remark}
\begin{definitionconstruction}\label{conWhSCtoSS}
Let $K$ be a finite simplicial complex. In \ref{ExLotsOfFSAE3} of \Cref{ExLotsOfFSAEs}, we have shown that, for any (F)SAE $f: K \hookrightarrow K'$, $\textnormal{sd}(f)$ is again a (F)SAE. In particular, by \Cref{remElemAreSame}, the induced map of unordered simplicial complexes $\textnormal{sd}(f): \textnormal{sd}(K) \hookrightarrow \textnormal{sd}(K')$ is again a composition of elementary expansions. Hence, $\textnormal{sd}$ induces a well-defined map $$E_{scx}(K) \xrightarrow{\textnormal{sd}} E_{scx}(\textnormal{sd}(K)).$$ Clearly, this is a map of pointed sets. Now, we define $\tilde E(K)$ as the colimit 
\begin{equation}
	E_{scx}(K) \to E_{scx}(\textnormal{sd}(K)) \to E_{scx}(\textnormal{sd}^2(K)) \to ... \to \tilde E(K). 
\end{equation} 
We make the analogous construction for $Wh_\star(X)$, for $X$ some ordering of $K$. We can again consider the colimit $$Wh_\star(X) \xrightarrow{\textnormal{sd}} Wh_\star(\textnormal{sd}(X)) \xrightarrow{\textnormal{sd}} Wh_\star(\textnormal{sd}^2(X)) \to ... \to \varinjlim Wh_\star (\textnormal{sd}^n(X)).$$ 
However, as $\textnormal{l.v.}$ is a simple equivalence (\Cref{propLvtSim}) and by its naturality together with the characterization of functoriality of the Whitehead group on simple morphisms (\Cref{corInvOfSimpleEq}), the subdivision map is given by $\textnormal{l.v.}^{-1}_*$. In particular, it is an isomorphism. Hence, we may identify $Wh_\star(X)$ with the colimit on the right. Now, under this identification, by the universal property of the colimit, we obtain a map of pointed sets $\Phi': \tilde E(K) \to Wh_*(X)$ by mapping $$\big \langle \textnormal{sd}^n(K) \xhookrightarrow{a} K' \big \rangle \longmapsto \big \langle \textnormal{sd}^{n+1}(K) = \textnormal{sd}^{n+1}(X) \xhookrightarrow{\textnormal{sd}(a)} \textnormal{sd}(K') \big \rangle.$$ Note that as subdivisions of a simplicial complex are naturally ordered, the right hand side is in fact a morphism of ordered simplicial complexes, and hence, by the fully faithfull embedding of the latter into simplicial sets, a morphism of simplicial sets. Furthermore, this naturally transfers a group structure onto $\tilde E(K)$ as follows. For two inclusions $a: \textnormal{sd}^n(K) \hookrightarrow K_0$, $b: \textnormal{sd}^m(K) \hookrightarrow K_1$, first pass to a common degree of subdivision, i.e. for the sake of simplicity we assume $n=m=0$. Then, subdivide once and take the pushout in simplicial sets 
\begin{center}
	\begin{tikzcd}
		\textnormal{sd}(X) = \textnormal{sd}(K) \arrow[d, hook, "a"] \arrow[r, hook] & \textnormal{sd}(K_0) \arrow[d, hook] \\
		\textnormal{sd}(K_1) \arrow[r, hook ,"b"] & L
	\end{tikzcd}.
\end{center}
Denote by $d$ the diagonal.
While $L$ might not necessarily be a ordered simplicial complex, by \Cref{lemPushoutofFQS} and \Cref{lemFQSImproves}, $\textnormal{sd}(L)$ in fact is one. Thus, we set $$[\langle a \rangle ] + [\langle b \rangle ]:= [\langle d' \rangle ],$$ where $d'$ is the map of unordered simplicial complexes underlying $\textnormal{sd}(d)$. It is easily verified, using the fact that $\textnormal{sd}$ preserves pushouts in $\textnormal{s\textbf{Set}}$ and the same argument as in \Cref{lemAass}, that this in fact defines an abelian monoid structure on $\tilde E(K)$ and that $\Phi'$ is a monoid homomorphism. Showing that inverses exist is a little more subtle, but will follow from the group structure on $Wh_\star(X)$ in the end.
\end{definitionconstruction}
\begin{lemma}\label{lemPhiPrimeOnto}
	In the setting of \Cref{conWhSCtoSS}, $\Phi'$ is onto.
\end{lemma}
\begin{proof}
	We start by considering $Wh_*(X)$ from the simple morphism class perspective in \Cref{propCharOfWh}. Let $\langle \alpha: X \to Y \rangle$ be such a simple morphism class in $Wh_\star(X)$. By \Cref{thrmSSvSC}, we may without loss of generality assume that $Y$ is a ordered simplicial complex. By \Cref{corLVdesTor}, for some sufficiently large $n$ and some map $f$ of simplicials set $f: \textnormal{sd}^n(X) \to Y$, we have $$ \textnormal{l.v.}^{n}_* \langle f \rangle = \langle \alpha \rangle.$$ Hence, under the colimit identification of \Cref{conWhSCtoSS}, the two classes agree. Now, $f$ is a map of ordered simplicial complexes. Again, under the colimit identification, $$\langle f \rangle = \langle \textnormal{sd}(f) \rangle .$$ The latter however, by \Cref{corFac}, has the simple morphism class of some inclusion of ordered simplicial complexes $$\langle f' \rangle. $$ Now, finally forgetting about the order structure and then subdividing, this clearly lies in the image of $\Phi'$.
\end{proof}
We now construct the natural transformation in \Cref{thrmOldNewAgree}. 
\begin{definitionconstruction}\label{conWhSSvCW}
	Clearly,  the requirement on Whitehead torsions in \Cref{thrmOldNewAgree} already uniquely determines a map. Now, to see this is well-defined, 
	first note that as in the filtered simplicial set setting, where the Whitehead group can be thought of as simple morphism classes of stratum preserving maps (\Cref{propCharOfWh}), the classical Whitehead group can also be thought of as homotopy classes of maps of finite CW-complexes with a fixed source, modulo postcomposition with simple equivalences. For a source, see for example \cite[Sec. 6]{eckmann2006}, with the caveat that one needs to use the amended axioms of \Cref{subsecEckSiebApp}, for this to be correct. As clearly the realization of an elementary expansion of simplicial sets gives an elementary expansion of CW-complexes, passing from simple morphism classes in the $\textnormal{s\textbf{Set}}$-setting to simple morphism classes in the CW-setting by realization defines a map $Wh_\star(X) \to Wh(|X|)$ that fulfills the torsion condition (for $X \in \textnormal{s\textbf{Set}}_\star$). To see this in fact defines a natural transformation of groups, note that under the equivalence classes of inclusion of subsets (complexes) (see \cite[\paragraphmark 6]{cohenCourse} for the CW-case), this map simply corresponds to $$\langle a \rangle \mapsto \langle |a| \rangle.$$ In particular, as functoriality and addition in both settings is defined via pushouts, and $|-|$ preserves pushouts, this defines a natural transformation of abelian group valued functors
	\begin{align*}
		\Phi: Wh_*(X) \longrightarrow Wh(|X|).
	\end{align*}
\end{definitionconstruction}
\begin{lemma}\label{lemPhiOnto}
	$\Phi$ as in \Cref{conWhSSvCW} is onto.
\end{lemma}
\begin{proof}
This time, we take the perspective of simple morphism clases of morphisms in $\mathcal H|\textnormal{s\textbf{Set}}_\star^{fin}|$ and $\mathcal H \textnormal{\textbf{CW}}^{fin}$ respectively. For $Wh_\star$, this is given by \Cref{propCharOfWh}. Now, with respect to these identifications, a simple morphism class $\langle \phi \rangle $ is mapped to the simple morphism class of the underlying map of topological spaces. In particular, as $\mathcal H|\textnormal{s\textbf{Set}}^{fin}_\star| \to \mathcal H\textnormal{\textbf{CW}}^{fin}$ is fully faithful, the simple morphism class of every arrow $|X| \to T$ with target $T= |Y|$, the realization of a finite simplicial set, is met. But by \Cref{propCWsimSC}, the fact that every finite CW-complex has the simple homotopy type of a simplicial complex (and taking some ordering on this complex), and the embedding of ordered simplicial complexes into simplicial sets (\Cref{conEmbedFOS}), this is the case for every simple morphism class in the CW-setting.
\end{proof}
\begin{lemma}\label{lemRepPhiprimeprime}
In the setting of \Cref{conWhSCtoSS} and \Cref{conWhSSvCW}, denote by $\Phi''$ the composition $$\tilde E(K) \xrightarrow{\Phi'} Wh_\star(X) \xrightarrow{\Phi} Wh(|X|).$$ Then this is explicitly given by $$ \big \langle \textnormal{sd}^n(K) \xhookrightarrow{a} K' \big \rangle \longmapsto \big \langle |X| = |K| \xrightarrow{\sim} |\textnormal{sd}^n(K)| \xhookrightarrow{|a|} |K'| \big \rangle = |\textnormal{l.v.}^n|_*\big( \langle |a| \rangle \big) .$$ 
\end{lemma}
\begin{proof}
	By construction, $\Phi'(\langle a \rangle )$ for $a$ as above, is given by $$(\textnormal{l.v.}^{n+1})_*(\langle \textnormal{sd}(a) \rangle).$$ Hence, under $\Phi$ this maps to ${\big (|\textnormal{l.v.}^{n+1}|\big )_*}(|\textnormal{sd}(a)|)$. Now, $|\textnormal{l.v.}^{n+1}|$ is the realization of a simple equivalence (\Cref{propLvtSim}), and in particular a simple equivalence of CW-complexes. Let $l$ be a cellular homotopy inverse to it. Then as $|\textnormal{l.v.}^{n+1}|$ is homotopic to the sudivision homeomorphism $|\textnormal{sd}^{n+1}(K)| \xrightarrow{\sim} |K|$, a homotopy inverse to $l$ is homotopic to the inverse of the latter and hence, the latter is also a simple homotopy equivalence. Now, applying the analogue to \Cref{corInvOfSimpleEq} for the cellular Whitehead group (see \cite[Prop. 22.4]{cohenCourse}) we obtain:
	\begin{align*}
		\Phi'' \big( \textnormal{l.v.}^{n+1}_*(\langle \textnormal{sd}(a) \rangle ) \big) &= l_*^{-1} \big( \langle |\textnormal{sd}(a)| \rangle \big) \\
		&= \Big\langle |K| \xrightarrow{\sim} |\textnormal{sd}^{n+1}(K)| \xhookrightarrow{|\textnormal{sd}(a)|} |\textnormal{sd}(L)|\big) \Big \rangle \\
		&= \Big\langle |K| \xrightarrow{\sim} |\textnormal{sd}^{n+1}(K)| \xhookrightarrow{|\textnormal{sd}(a)|} |\textnormal{sd}(L)| \xrightarrow{\sim} |L| \big) \Big \rangle \\
		&= \Big\langle |X| = |K| \xrightarrow{\sim} |\textnormal{sd}^{n}(K)| \xhookrightarrow{|a|} |(L)| \Big \rangle \\
		&= |\textnormal{l.v.}^n|_*\big( \langle |a| \rangle \big) 
	\end{align*} 
	where the third equality follows again from the fact that the subdivision isomorphisms are simple equivalences, the fourth from their naturality and the final one again from an appeal to the CW-analogue of \Cref{corInvOfSimpleEq}.
\end{proof}
%
%Clearly, the realization of every elementary expansion of simplicial complexes, is an elementary expansion of CW-complexes, with respect to the induced CW-structures. Hence, for every simplicial complex $K$, this induces a map of pointed sets $$E_{scx}(K) \longrightarrow Wh(|K|)$$, where we think of the latter as inclusions of subcomplexes modulo elementary expansion, as in \cite{cohenCourse}. The subdivision construction $E_{scx}(K) \to E_{scx}(\textnormal{sd}(K))$ then fits into a commutative diagram:
%\begin{center}
%\begin{tikzcd}
%	E_{scx}(K) \arrow[d] \arrow[r] & Wh(|K|) \arrow[d, "{s}_*"] \\
%	E_{scx}(\textnormal{sd}(K)) \arrow[r] & Wh(|\textnormal{sd}(K)|) \\
%\end{tikzcd}
%Where $s$ is the subdivision isomorphism $|\textnormal{sd}(K)| \to |K|$. To see this, note first that just as in the simplicial set case, we can analogously think of $Wh(|K|)$ $(Wh(|\textnormal{sd(K)})|)$ as the simple morphism classes (in the CW-sense) of homotopy equivalences of finite CW-complexes with source $|K|$, $(\textnormal{sd}(|K|))$. This is for example shown in \cite[Sec. 6]{eckmann2006}, with the caveat that one needs to take the modified axioms we used in \Cref{subsecEckSiebApp} (see \Cref{remAxAreWrong}). Note that this is the analogous statement to \Cref{propCharOfWh}. Then, in this perspective, as $s$ is a subdivision homeomorphism and hence a simple homotopy equivalence (for example, because it is homotopic to a last vertex map, which is simple by \Cref{propLvtSim}) $s_*^{-1}$, is given by precomposition with $[s]$ (see for example ).
Next, we show $\Phi$ is injective. For this, we use a somewhat well known, but seemingly badly documented result on the relationship between simple equivalences in the simplicial complex and the cellular sense.
\begin{proposition}\label{propKerSCvCW}
	Let $K$ be a finite simplicial complex. Using the notation from \Cref{conWhSCtoSS} and \Cref{lemRepPhiprimeprime},
	the kernel (in the sense of pointed sets) of the map $E_{scx}(K) \to \tilde{E}(K) \xrightarrow{\Phi''} Wh(|K|)$ is trivial.
\end{proposition}
\begin{proof}
	A complete proof would essentially come down to replicating much of what Whitehead did in \cite{whitehead1939simplicial}, in the original simplicial setting. We instead refer there for most of the details. Without loss of generality, we may clearly assume $|K|$ to be connected. In the proof of \cite[Thm. 20]{whitehead1939simplicial} it is effectively shown that an inclusion of finite simplicial complexes $K \hookrightarrow K'$ belongs to $\langle 1_K \rangle $ if a certain class element associated to it in the algebraic Whitehead group of $\pi_1(|K|)$, $Wh(\pi_1(|K|))$, disappears. The proof back then essentially used the same method as in the cellular setting, but without the availability of the language of CW-complexes. If one explicitly tracks down the construction there, and uses \Cref{lemRepPhiprimeprime}, one sees that the element is precisely the one given by first applying $\Phi''$ and then using the natural isomorphism $Wh(|K|) \cong Wh(\pi_1(|K|))$, constructed for example in \cite[\paragraphmark 21]{cohenCourse}. Hence, the statement.
\end{proof}
As a corollary we obtain (using that $\textnormal{l.v.}$ is an isomorphism): 
\begin{corollary}\label{corPhiPrimPrimeInj}
	$\Phi''$ from \Cref{lemRepPhiprimeprime} has trivial kernel.
\end{corollary}
We can now comple the proof of \Cref{thrmOldNewAgree}.
\begin{corollary}
	Let $\Phi$ be as in \Cref{conWhSSvCW}, $\Phi'$ as in \Cref{conWhSCtoSS} and $\Phi''$ as in \Cref{lemRepPhiprimeprime}. Then all of them are isomorphisms of abelian monoids. In particular, $\tilde E_{scx}(K)$ is an abelian group and \Cref{thrmOldNewAgree} holds.
\end{corollary}
\begin{proof}
	This is just putting together what we already know. We have $$\Phi'' = \Phi \circ \Phi'.$$ By \Cref{lemPhiPrimeOnto}, $\Phi'$ is onto. By \Cref{corPhiPrimPrimeInj}, $\Phi''$ and hence also $\Phi'$ has trivial kernel. In particular, $\Phi'$ is an map from an abelian monoid into an abelian group that is onto and has trivial kernel. Such a map is always an isomorphism. Hence, $\Phi$ is injective. By \Cref{lemPhiOnto} it is also onto, i.e. also an isomorphism. Thus, finally the same holds for $\Phi''$.
\end{proof}

\bibliographystyle{alpha}
\bibliography{TowSimStratHo}
\begin{appendices}
\chapter{Appendix}
\subsection{A result on pullbacks of colimits}
The following result on base changes is useful for local constructions on certain subspaces of realizations of filtered simplicial sets. Recall that limits in the category of $\Delta$-generated spaces are taken by taking limits in the naive category of topological spaces (i.e. arbitrary ones) and then applying $k_{\Delta}$, i.e. putting the final topology with respect to simplices on the limit (see \cite{duggerDelta}).
\begin{proposition}\label{AppPropPullback}
	Let $X \in \textnormal{s\textbf{Set}}$ be a locally finite simplicial set, together with a map $|X| \xrightarrow{\varphi_X} A$ in $\textnormal{\textbf{Top}}$. Let $B \xrightarrow{f} A$ be another map in $\textnormal{\textbf{Top}}$. Let $|\Delta_i|$ be the diagram, given by the realizations of the non-degenerate simplices of $X$. Further, denote by $f^*(|X|)$ the total space of the base change of $\varphi_X$ along $f$ and by $f^*(|\Delta_i|)$ the total space of the base change of $|\Delta_i| \to X \to A$ along $f$. Consider the commutative diagram induced by the universal property of the pullback and colimit:
	$$ \begin{tikzcd}
	\varinjlim(f^*(|\Delta_i|)) \arrow[rr] \arrow[rd] & & f^*(|X|) \arrow[ld] \\
	&B&
	\end{tikzcd}.$$
	Then the horizontal map is a homeomorphism. Furthermore, if $f$ is such that, for each non-degenerate simplex of $X$, $\Delta_i$, the fiber product space in the naive category of topological space (that is topological spaces that are not necessarily $\Delta$-generated) of $|\Delta^i| \to |X| \to A$ and $f$ is a $\Delta$-generated space, then so is the fiber product space in the naive category of topological space of $f$ and $\varphi_X$, denoted $f^*(|X|)^{naive}$. In particular, then, $$f^*(|X|)^{naive} = f^*(|X|).$$ 
\end{proposition}
\begin{proof}
	To show that $\varinjlim(f^*(|\Delta_i|)) \longrightarrow f^*(|X|)$ is a homeomorphism, it really suffices to show that it is a quotient map. That it is bijective follows immediately from the fact that in all of the categories involved, limits and colimits are given by limits and colimits in $\textbf{Set}$ equipped with a topology and in $\textbf{Set}$ base change commutes with arbitrary colimits. We show that $\varinjlim(f^*(|\Delta_i|)^{naive}) \longrightarrow f^*(|X|)^{naive}$ is a homeomorphism. Then, as colimits in the $\Delta$-generated category are computed in the naive topological one and as $k_{\Delta}$ preserves colimits, the result follows. It suffices to show that $$\bigsqcup f^*(|\Delta_i|)^{naive} \longrightarrow f^*(|X|)^{naive}$$ is a quotient map. As pullbacks in naive topological spaces commute with arbitrary coproducts, this map fits into a commutative diagram in the naive category of topological spaces
	$$\begin{tikzcd}
	{\bigsqcup f^*(|\Delta_i|)^{naive}} \arrow[r] \arrow[d] & \bigsqcup {|\Delta_i|} \arrow[d]\\
	{f^*(|X|)}^{naive} \arrow[r] \arrow[d] & {|X| }\arrow[d, "\varphi_X"]\\
	B \arrow[r, "f"] &A	
	\end{tikzcd}$$
	with both squares cartesian. Hence, it suffices to show that the quotient map $\bigsqcup |\Delta_i| \to |X|$ stays a quotient map, under pulling back.
	However, it is a well known fact that proper maps between locally compact Hausdorff spaces are universally closed (i.e. every pullback of such a map is closed). See for example \cite[\href{https://stacks.math.columbia.edu/tag/005R}{Tag 005R}]{stacks-project} together with \cite[Ch.9 Prop. 7]{bourbaki1966general}. As $X$ is locally finite, this is the case, i.e. the quotient map is proper, and $|X|$ is locally compact Hausdorff. As a surjective, closed map is a quotient map, this finishes the proof.
\end{proof}
\subsection{A basic set theoretic manipulation}
\begin{lemma}\label{lemSetPerspective}
	For arrows in \textnormal{\textbf{Set}} $A \to B$, $X \to Y$, $A \to A'$ and $B' = B \cup_{A} A'$ the induced diagram
	$$ \begin{tikzcd}
		B \times X \cup_{A \times X} A \times Y \arrow[r] \arrow[d] & B \times Y \arrow d\\
		B' \times X \cup_{A' \times X} A' \times Y \arrow[r] & B' \times Y
	\end{tikzcd}$$
is a pushout diagram.
\end{lemma}
\begin{proof}
Really all we are going to need is the fact that in $\textbf{Set}$ the product preserves colimits as $\textbf{Set}$ is a close monoidal category. Now, consider the following composition of commutative diagrams.
\begin{center}

\begin{tikzcd}
	A \times Y \arrow[r] \arrow[d]&B \times X \cup_{A \times X} A \times Y \arrow[r] \arrow[d] & B \times Y \arrow d\\
	A' \times Y \arrow[r]&B' \times X \cup_{A' \times X} A' \times Y \arrow[r] & B' \times Y
\end{tikzcd}
\end{center}
With the obvious arrows. As $ - \times Y$ preserves colimits, the larger outer commutative square is a pushout. So it suffices to show that the left hand square is a pushout square. This follows from the natural isomorphisms:
\begin{align*}
(B \times X \cup_{A \times X} A \times Y)\cup_{A \times Y} A' \times Y & \cong B \times X \cup_{A \times X} A' \times Y \\ & \cong B \times X \cup_{A \times X} (A' \times X \cup_{A' \times X} A' \times Y)\\
& \cong (B \times X \cup_{A \times X} A' \times X) \cup_{A' \times X} A' \times Y
\\ &\cong B' \times X \cup_{A' \times X} A' \times Y
\end{align*}
\end{proof}
\end{appendices}
		
\chapter*{\normalsize Declaration of Authorship\\Eigenständigkeitserklärung}
\vspace{3 cm}
I, Lukas Waas, hereby declare that this thesis is my own work and that I used no sources other than those indicated.\\
\vspace{1 cm}\\
Ich, Lukas Waas, erkläre hiermit, dass ich diese Arbeit eigenständig verfasst habe und keine anderen Quellen als die angegebenen verwendet habe.\\
% Ort und Datum 
\vspace{3 cm}\\
\begin{tabular}{p{7cm}p{.5cm}l}
	\dotfill \\
	Place and date\\ 
	Ort, Datum
\end{tabular}% 
\begin{tabular}{p{7cm}p{.5cm}l}
	\dotfill \\ 
	Signature\\
	Unterschrift
\end{tabular}
\hfill
\end{document}